\documentclass{amsart}

\usepackage{amsmath,amsthm,amssymb}
\usepackage{enumerate}
\usepackage{graphicx}
\usepackage{float}
\usepackage[margin=1.5in]{geometry}
\numberwithin{equation}{section}

\theoremstyle{plain}
\newtheorem{thm}{Theorem}[section]
\newtheorem{cor}[thm]{Corollary}
\newtheorem{lem}[thm]{Lemma}
\newtheorem{prop}[thm]{Proposition}
\newtheorem{conj}[thm]{Conjecture}
\newtheorem{ques}[thm]{Question}
\theoremstyle{definition}
\newtheorem{defn}[thm]{Definition}
\newtheorem{fact}[thm]{Fact}
\theoremstyle{remark}
\newtheorem{rem}[thm]{Remark}

\usepackage{tikz}
\usetikzlibrary{arrows,shapes,trees}
\usepackage{graphicx}
\usepackage{hyperref}
\usepackage{amssymb}
\usepackage{amsfonts}
\usepackage{amsthm}
\usepackage{amsmath}
\usepackage{epstopdf}
\usepackage{mathrsfs}

\newcommand{\N}{\mathbb N}

\newcommand{\Z}{\mathbb Z}
\newcommand{\C}{\mathbb C}

\newcommand{\R}{\mathbb R}
\newcommand{\F}{\mathcal F}
\newcommand{\W}{\mathcal W}

\newcommand{\E}{\mathcal{E}}
\newcommand{\V}{\mathcal{V}}

\newcommand{\e}{\varepsilon}

\renewcommand{\t}{\tilde}
\renewcommand{\l}{\ell}
\newcommand{\la}{\lambda}

\newcommand{\p}{\partial}

\title{Characterizing symmetric spaces by their Lyapunov spectra}
\author{Clark Butler}
\thanks{The material in this paper is based upon work supported by the National Science Foundation Graduate Research Fellowship under Grant \# DGE-1144082.}
\begin{document}

\begin{abstract}
We prove that closed negatively curved locally symmetric spaces are characterized up to isometry among all homotopy equivalent negatively curved manifolds by the Lyapunov spectra of the periodic orbits of their geodesic flows. This is done by constructing a new invariant measure for the geodesic flow that we refer to as the horizontal measure. We show that the Lyapunov spectrum of the horizontal measure alone suffices to locally characterize these locally symmetric spaces up to isometry. We associate to the horizontal measure a new invariant, the horizontal dimension. We tie this invariant to extensions of curvature pinching rigidity theorems for complex hyperbolic manifolds to pinching rigidity theorems for the Lyapunov spectrum. Our methods extend to give a rigidity theorem for smooth Anosov flows $f^{t}$ orbit equivalent to the geodesic flow $g^{t}_{X}$ of a closed negatively curved locally symmetric space $X$:  $f^{t}$ is smoothly orbit equivalent to $g^{t}_{X}$ if and only if its Lyapunov spectra on all periodic orbits are a multiple of the corresponding Lyapunov spectra for $g^{t}_{X}$. 
\end{abstract}
\maketitle

\section{Introduction}\label{sec:intro}

There is a rich collection of dynamical invariants which characterize closed negatively curved locally symmetric spaces up to isometry. One classical result for surfaces is Katok's theorem \cite{Kat82} stating that a closed negatively curved surface has constant negative curvature if and only if the Liouville measure is the measure of maximal entropy; another is the characterization by Croke \cite{Cr} and Otal \cite{O} of closed negatively curved surfaces up to isometry by the lengths of their closed geodesics. Katok conjectured that closed negatively curved locally symmetric spaces are in general characterized by the Liouville measure being the measure of maximal entropy, however little progress on this conjecture has been made since the work of Flaiminio \cite{Fla95} establishing a local version of the conjecture for constant negative curvature manifolds. Hamenst\"adt \cite{Ham99} generalized the theorem of Croke and Otal to higher dimensional closed negatively curved manifolds in the special case that one of them is locally symmetric. 

Another spectacular characterization is the minimal entropy rigidity theorem of Besson, Courtois, and Gallot \cite{BCG2}, which states that if $X$ is a closed negatively curved locally symmetric space and $Y$ is a closed negatively curved Riemannian manifold homotopy equivalent to $X$, then $Y$ is isometric to $X$ if and only if $h_{\mathrm{top}}(g_{X}) = h_{\mathrm{top}}(g_{Y})$ and $\mathrm{vol}(X) = \mathrm{vol}(Y)$, i.e., the geodesic flows $g^{t}_{X}$ and $g^{t}_{Y}$ of $X$ and $Y$ have equal topological entropy and $X$ and $Y$ have equal volume. The dynamical invariants that we will consider in this paper share  features of the topological entropy, the marked length spectrum, and the entropy of the Liouville measure. 

Given a closed Riemannian manifold $Y$, we consider the geodesic flow $g^{t}_{Y}: T^{1}Y \rightarrow T^{1}Y$ on its unit tangent bundle $T^{1}Y$. Set $n = \dim Y$. By applying the multiplicative ergodic theorem \cite{Osel} to the derivative cocycle $Dg^{t}_{Y}$ of the geodesic flow, we may associate to each $g^{t}_{Y}$-invariant ergodic probability measure $\nu$ a list of real numbers 
\[
\la_{1}(g_{Y},\nu) \leq  \dots \leq \la_{2n-1}(g_{Y},\nu), 
\]
such that for $\nu$-a.e.~ unit tangent vector $v \in T^{1}Y$ and every unit vector $\xi \in T_{v}T^{1}Y$, we have 
\[
\lim_{t \rightarrow \infty}\frac{\|Dg^{t}_{Y}(\xi)\|}{t} = \la_{i}(g_{Y},\nu),
\]
 for some $1 \leq i \leq 2n-1$. These are known as the \emph{Lyapunov exponents} of $g^{t}_{Y}$ with respect to $\nu$; they describe all possible asymptotic exponential growth rates of vectors $\xi \in T_{v}T^{1}Y$ for $\nu$-a.e.~ $v \in T^{1}Y$. Equivalently, they describe the possible asymptotic exponential growth rates of Jacobi fields along a randomly chosen geodesic in $Y$, where this random choice is made according to the distribution $\nu$. We write $\vec{\la}(g_{Y}, \nu)$ for the vector in $\R^{2n-1}$ whose components are the Lyapunov exponents $\la_{i}(g_{Y},\nu)$, written in increasing order. For any choice of ergodic $g^{t}_{Y}$-invariant measure $\nu$ we always have $\la_{n}(g_{Y},\nu) = 0$; this corresponds to the fact that $g^{t}_{Y}$ acts isometrically on the direction of its flow. The Lyapunov spectra of ergodic $g^{t}_{Y}$-invariant measures will be the primary invariants associated to $g^{t}_{Y}$ that will be studied in this paper. We note that the Lyapunov spectrum of the Liouville measure $m_{Y}$ for $g^{t}_{Y}$ is connected to the entropy $h_{m_{Y}}(g_{Y})$ of $g^{t}_{Y}$ with respect to $m_{Y}$ by the \emph{Pesin entropy formula} \cite{Pes76}, 
\[
h_{m_{Y}}(g_{Y}) = \sum_{i=1}^{n-1} \la_{2n-i}(g_{Y},m_{Y}),
\]
with the sum being taken over all positive exponents of $g^{t}_{Y}$ with respect to $m_{Y}$. 
 
Periodic points of $g^{t}_{Y}$ correspond to closed geodesics in $Y$. For each periodic point $p \in T^{1}Y$ of $g^{t}_{Y}$ we let $\nu^{(p)}$ denote the unique $g^{t}_{Y}$-invariant probability measure supported on the orbit of $p$. Let $\l(p)$ denote the period of $p$, which is the length of the closed geodesic generated by $p$. The Lyapunov spectrum $\vec{\la}(g_{Y},\nu^{(p)})$ is determined exclusively by the absolute values of the eigenvalues of the linear map $Dg^{\l(p)}_{p}: TT^{1}Y_{p} \rightarrow TT^{1}Y_{p}$. When $Y$ is negatively curved the periodic points of $g^{t}_{Y}$ are dense in $T^{1}Y$. It is thus natural to ask to what extent the Lyapunov spectra  $\vec{\la}(g_{Y},\nu^{(p)})$ associated to these periodic orbits determines $Y$. This is the subject of the main theorem of this work. 

We note that the Lyapunov spectra of periodic orbits are in some sense orthogonal to the lengths of these orbits. One cannot determine the length of an orbit using the Lyapunov spectrum and vice versa. In the extreme case of a negatively curved locally symmetric space detailed below, the Lyapunov spectrum does not actually depend on the choice of  periodic orbit at all.

When $X$ is a negatively curved locally symmetric space, the Lyapunov spectrum does not depend on the choice of invariant measure $\nu$:  there is a fixed vector $\vec{\la}(g_{X}) \in \R^{2n-1}$ such that $\vec{\la}(g_{X},\nu) = \vec{\la}(g_{X})$ for all $g^{t}_{X}$-invariant ergodic probability measures $\nu$. Hence we will omit the choice of invariant measure when we refer to the Lyapunov spectrum of $g^{t}_{X}$. The vector $\vec{\la}(g_{X})$ depends only on the isometry type of the universal cover of $X$. For example, any negatively curved locally symmetric space with universal cover the real hyperbolic space $\mathbf{H}^{3}_{\R}$ of constant negative curvature $K \equiv -1$ has Lyapunov spectrum $(-1,-1,0,1,1)$,  while any locally symmetric space with universal cover the  complex hyperbolic plane $\mathbf{H}^{2}_{\C}$ has Lyapunov spectrum $(-2,-1,-1,0,1,1,2)$, provided we normalize $\mathbf{H}^{2}_{\C}$ to have sectional curvatures $-4 \leq K \leq -1$. The Lyapunov spectra of negatively curved locally symmetric spaces in general are described in Section \ref{subsec:neg curved}. 

Recall that a Riemannian manifold $Y$ is \emph{homothetic} to another Riemannian manifold $X$ if there is a constant $c > 0$ such that $Y$ is isometric to the Riemannian manifold $X_{c}$ obtained by scaling the metric of $X$ by $c$. Our main theorem is the following.

\begin{thm}\label{periodic loc symm}
Let $X$ be a closed negatively curved locally symmetric space with $\dim X \geq 3$. Let $Y$ be a closed negatively curved Riemannian manifold that is homotopy equivalent to $X$. Suppose that for each periodic point $p$ of $g^{t}_{Y}$ there exists a constant $\omega(p) > 0$ such that 
\begin{equation}\label{periodic exponent equality}
\vec{\la}(g_{Y},\nu^{(p)}) = \omega(p)\vec{\la}(g_{X}).
\end{equation}
Then $Y$ is homothetic to $X$. If there is some $p$ such that $\omega(p) = 1$ then $Y$ is isometric to $X$. 
\end{thm} 

Note that we only need to assume that the Lyapunov spectrum of $g^{t}_{Y}$ at a given periodic point $p$ is some multiple $\omega(p)$ of the Lyapunov spectrum of $g^{t}_{X}$, where $\omega(p)$ is allowed to depend in an arbitrary fashion on $p$. As part of the conclusion of the theorem one obtains that $\omega(p)$ is actually constant in $p$. 

It is instructive to consider the special case in which $Y$ is also locally symmetric, with the same universal cover as $X$. By the discussion above, the Lyapunov spectra $\vec{\la}(g_{Y},\nu^{p}) = \vec{\la}(g_{Y})$ do not depend on the choice of periodic point $p$, and since $X$ and $Y$ have the same universal cover we must have $\vec{\la}(g_{Y}) = \vec{\la}(g_{X})$. Hence \eqref{periodic exponent equality} is trivially satisfied with $\omega(p) = 1$ for all $p$. We thus obtain a new proof of the Mostow rigidity theorem as a corollary of Theorem \ref{periodic loc symm}. 

\begin{cor}{\cite{Mos73}}\label{mostow rigidity}
Let $X$ and $Y$ be two $n$-dimensional closed negatively curved locally symmetric spaces, $n \geq 3$. Suppose that we have a homotopy equivalence $\kappa: X \rightarrow Y$. Then $\kappa$ is homotopic to a homothety $\sigma: X \rightarrow Y$.
\end{cor}

\begin{rem}\label{mostow remarks}
The concluding step of Theorem \ref{periodic loc symm} relies on Theorem \ref{minimal rigidity} below, which is a corollary of the minimal entropy rigidity theorem \cite{BCG2}. Hence a direct application of Theorem \ref{periodic loc symm} to Corollary \ref{mostow rigidity} would yield a proof that depended on the minimal entropy rigidity theorem, which is undesirable as the Mostow rigidity theorem is itself a direct corollary of this theorem \cite{BCG2}. However, when $Y$ is locally symmetric the use of Theorem \ref{minimal rigidity} can be avoided. We explain this in Section \ref{mostow mod} at the end of the paper. 
\end{rem}

Let $\mathbf{H}$ be a negatively curved symmetric space, $\dim \mathbf{H} \geq 3$. We write $\vec{\la}(\mathbf{H})$ for the common Lyapunov spectrum of the geodesic flows of all negatively curved locally symmetric spaces with universal cover $\mathbf{H}$. Motivated by Theorem \ref{periodic loc symm}, we pose the following conjecture.

\begin{conj}\label{conjecture}
Let $Y$ be a closed negatively curved Riemannian manifold, $\dim Y \geq 3$. Suppose that for each periodic point $p$ of $g^{t}_{Y}$ there exists a constant $\omega(p) > 0$ such that 
\begin{equation}
\vec{\la}(g_{Y},\nu^{(p)}) = \omega(p)\vec{\la}(\mathbf{H}).
\end{equation}
Then $Y$ is locally symmetric with universal cover homothetic to $\mathbf{H}$.
\end{conj}

In \cite[Theorem 1.1]{Bu1} we established the case $\mathbf{H}  = \mathbf{H}^{n}_{\R}$ of this conjecture, which corresponds to the case in which $\mathbf{H}$ has constant negative curvature. The question is whether this theorem can be extended to cover the case of nonconstant negative curvature. The proof of Theorem \ref{periodic loc symm} offers strong evidence in favor of a positive answer to Conjecture \ref{conjecture}, as the hypothesis that $Y$ is homotopy equivalent to $X$ is only used in the second half of the proof starting in Section \ref{subsec:synchro}, after many important properties of $g^{t}_{Y}$ have already been established using only equation \eqref{periodic exponent equality}.

Our proof of Theorem \ref{periodic loc symm} owes the greatest debt to Hamenst\"adt's hyperbolic rank rigidity theorem \cite{Ham91} and Connell's generalization to the minimal Lyapunov exponent rigidity theorem \cite{Con}. In particular the geometric techniques we use in Section \ref{sec:orbit to conj} are directly inspired by the proofs of these theorems, using the special metrics on unstable manifolds that Hamenst\"adt introduced to prove her theorem. 

Given a smooth manifold $S$, Riemannian metrics on $S$ may be described as smooth sections of the bundle $\mathrm{Sym}^{2}(TS)$ of symmetric 2-tensors on $TS$. For a Riemannian manifold $X$ with underlying smooth manifold $S$, we let $\eta_{X}:S \rightarrow \mathrm{Sym}^{2}(TS)$ denote its metric tensor. Given a closed negatively curved locally symmetric space $X$ of nonconstant negative curvature, we write $\mathcal{U}_{X}$ for a certain $C^2$ open neighborhood of $\eta_{X}$ in the space of smooth sections of $\mathrm{Sym}^{2}(TS)$, which we describe in Section \ref{subsec:geom rigidity}. As a shorthand, for another Riemannian manifold $Y$ with the same underlying manifold $S$ we write $Y \in \mathcal{U}_{X}$ if $\eta_{Y} \in \mathcal{U}_{X}$. 

Our next result is a local rigidity theorem which characterizes $X$ among all $Y \in \mathcal{U}_{X}$ by its Lyapunov spectrum with respect to a \emph{single} $g^{t}_{Y}$-invariant measure $\mu_{Y}$. We refer to this measure as the \emph{horizontal measure}. For the formal construction of this measure see Section \ref{subsec:hormeasure}; we give a brief description here. 

Given $Y \in \mathcal{U}_{X}$, we will construct a H\"older continuous function 
\[
\zeta_{Y}: T^{1}Y \rightarrow (0,\infty),
\] 
in a natural way out of the action of $Dg^{t}_{Y}$ on a certain $Dg^{t}_{Y}$-invariant subbundle of $T(T^{1}Y)$. We then solve the Bowen equation $P(s\zeta_{Y}) = 0$, $s > 0$, where $P(s\zeta_{Y})$ denotes the topological pressure of the function $s \zeta_{Y}$ with respect to the flow $g^{t}_{Y}$. We obtain a unique number $q_{Y} > 0$ such that $P(q_{Y} \zeta_{Y}) = 0$, which we refer to as the \emph{horizontal dimension} of $g^{t}_{Y}$. The horizontal measure $\mu_{Y}$ is then defined to be the unique equilibrium state of the potential $q_{Y}\zeta_{Y}$ with respect to $g^{t}_{Y}$. For the locally symmetric space $X$ itself, $\zeta_{X}$ is a constant function, $q_{X}$ is easily described in terms of the Lyapunov spectrum $\vec{\la}(g_{X})$, and $\mu_{X}$ coincides with the Liouville measure $m_{X}$, which is the invariant volume for $g^{t}_{X}$  normalized to satisfy $m_{X}(T^{1}X) = 1$.

\begin{thm}\label{loc symmetric rigidity}
Let $X$ be a closed negatively curved locally symmetric space of nonconstant negative curvature. Let $Y \in \mathcal{U}_{X}$. Then $Y$ is isometric to $X$ if and only if
\[
\vec{\la}(g_{Y},\mu_{Y}) = \vec{\la}(g_{X}).
\]
\end{thm}

The horizontal measure $\mu_{Y}$ is a $g^{t}_{Y}$-invariant measure that is specifically adapted to the nonconstant negative curvature case. In general it does not coincide with any well-known previously considered invariant measures, such as the Liouville measure $m_{Y}$ or the Bowen-Margulis measure of maximal entropy for $g^{t}_{Y}$. 


In the case that $X$ has constant negative curvature, the proper analogue of $\mu_{Y}$ is the Liouville measure $m_{Y}$. In our previous work \cite{Bu1}, we established that Theorem \ref{loc symmetric rigidity} holds for $m_{Y}$ when $X$ has constant negative curvature, $Y$ is homotopy equivalent to $X$, and $Y$ has strictly $1/4$-pinched negative curvature. An important question is whether Theorem \ref{loc symmetric rigidity} remains true in the nonconstant negative curvature case if we replace $\mu_{Y}$ with $m_{Y}$.

\begin{ques}\label{volume question}
Let $X$ be a closed negatively curved locally symmetric space of nonconstant negative curvature and let $Y \in \mathcal{U}_{X}$. Suppose that $\vec{\la}(g_{Y},m_{Y}) = \vec{\la}(g_{X})$. Is $Y$ isometric to $X$? 
\end{ques}

Question \ref{volume question} is of particular interest in light of recent rigidity results of Saghin and Yang \cite{SY18} and Gogolev, Kalinin, and Sadovskaya \cite{GKS18} concerning the volume exponents of irreducible hyperbolic toral automorphisms. 

Hernandez \cite{Her} and independently Yau and Zheng \cite{FY} proved that any $1/4$-pinched negatively curved metric on a closed complex hyperbolic manifold is isometric to the standard symmetric metric. Gromov \cite{Gro91} extended these theorems to obtain $1/4$-pinching rigidity for closed quaternionic hyperbolic manifolds as well. Our final geometric rigidity theorem addresses possible ways to generalize these rigidity theorems by weakening hypotheses on curvature pinching to hypotheses on pinching inequalities among the Lyapunov exponents of the geodesic flow. This line of inquiry is inspired by a question of Boland and Katok  \cite{Bol} asking whether $1/2$-pinching of the Lyapunov exponents characterizes $g^{t}_{X}$ among nearby smooth flows on $T^{1}X$ when $X$ is a closed complex hyperbolic manifold.

Set $n = \dim Y$ as before. We say that the Lyapunov spectrum of $g^{t}_{Y}$ with respect to a $g^{t}_{Y}$-invariant ergodic probability measure $\nu$ is \emph{$1/2$-pinched} if there is a constant $a > 0$ such that 
\[
a \leq |\la_{i}(g^{t}_{Y},\nu)| \leq 2a, \; \text{for} \; 1 \leq i \leq 2n-1, \, i \neq n,
\]
i.e., excluding the exponent corresponding to the flow direction of $g^{t}_{Y}$, the Lyapunov exponents of $g^{t}_{Y}$ are pinched in absolute value between $a$ and $2a$. 

Curvature pinching estimates on a negatively curved Riemannian manifold $Y$ give rise to pinching estimates on the Lyapunov exponents of $g^{t}_{Y}$: if the sectional curvatures of $Y$ satisfy $-b^{2} \leq K \leq -a^{2}$ for constants $b > a > 0$, then for any given $g^{t}_{Y}$-invariant ergodic probability measure $\nu$ we have  $a \leq |\la_{i}(g^{t}_{Y},\nu)| \leq b$ for all $i \neq n$. In particular, if $Y$ has $1/4$-pinched negative curvature then the Lyapunov spectrum of $g^{t}_{Y}$ with respect to any ergodic invariant measure $\nu$ is $1/2$-pinched. Theorem \ref{geom pinching rigidity} below gives a partial result toward understanding whether the curvature $1/4$-pinching hypothesis in the rigidity theorems of Hernandez, Yau and Zheng, and Gromov above can be weakened to a $1/2$-pinching hypothesis on the Lyapunov spectrum of a special choice of invariant measure for the geodesic flow. 

Recall that, for $Y \in \mathcal{U}_{X}$, we denote the horizontal measure for $g^{t}_{Y}$ by $\mu_{Y}$ and denote the horizontal dimension of $g^{t}_{Y}$ by $q_{Y}$.  Our theorem shows that, under the additional hypothesis of a lower bound on the horizontal dimension $q_{Y}$, one can obtain a Lyapunov spectrum $1/2$-pinching rigidity theorem for $g^{t}_{Y}$ with respect to the measure $\mu_{Y}$.
\begin{thm}\label{geom pinching rigidity}
Let $X$ be a closed negatively curved locally symmetric space of nonconstant negative curvature. Let $Y \in \mathcal{U}_{X}$. Suppose that $q_{Y} \geq q_{X}$ and that the Lyapunov spectrum of $g^{t}_{Y}$ with respect to $\mu_{Y}$ is $1/2$-pinched. Then $Y$ is homothetic to $X$.  
\end{thm}

Establishing the lower bound $q_{Y} \geq q_{X}$ under the hypothesis that $\vec{\la}(g_{Y},\mu_{Y}) = \vec{\la}(g_{X})$ is a critical step in the proofs of Theorems \ref{periodic loc symm} and \ref{loc symmetric rigidity}. This, together with the role of this lower bound in the hypotheses of Theorem \ref{geom pinching rigidity}, prompts the following question.

\begin{ques}\label{affinity question}
Let $X$ be a closed negatively curved locally symmetric space of nonconstant negative curvature. Let $Y \in \mathcal{U}_{X}$. Do we always have $q_{Y} \geq q_{X}$?  
\end{ques}

An affirmative answer to this question would give a full generalization of the $1/4$-curvature pinching rigidity theorems for nonconstant negative curvature locally symmetric spaces to $1/2$-pinching rigidity theorems for their Lyapunov spectra with respect to their horizontal measures.

We now describe the organization of the paper. In Section \ref{sec: prelim} we go over many different results that will be needed throughout the paper. We provide a more thorough accounting of the contents of this section at its beginning. In Section \ref{sec:dyn theorems}  we state the dynamical rigidity theorems from which our geometric rigidity theorems will be derived and reduce all of the major theorems of the paper to one core Theorem \ref{core theorem}. In Section \ref{sec:uniform} we construct normal form coordinates for the action of an Anosov flow on the quotient of the unstable foliation by a subfoliation satisfying certain growth inequalities. Sections \ref{sec:exponents}, \ref{sec:orbit to conj}, \ref{sec:pansuderiv}, and \ref{sec: vert} comprise the proof of Theorem \ref{core theorem}. Finally in Section \ref{mostow mod} we explain how Corollary \ref{mostow rigidity} can be proved without the use of Theorem \ref{minimal rigidity}.

This paper has benefited from numerous discussions with many people. We thank Jairo Bochi, Mario Bonk, Aaron Brown, Robert Bryant, Chris Connell, Jonathan DeWitt, Nicolas Gourmelon, Ursula Hamenst\"adt, Pierre Pansu, Ralf Spatzier, and Andrew Zimmer for discussions that benefited this work and shed light on several new aspects which were then pursued further. We extend special thanks to Amie Wilkinson for much encouragement as well as several extensive reviews of the results of the project as they were being developed. These discussions resulted in immeasurable improvement of the final results. Lastly, we extend great thanks to the anonymous referee who provided a thorough review that ultimately led to the strengthening of many of the results in the paper. 

\section{Preliminaries}\label{sec: prelim}

In this section we carefully define the objects that we will be considering throughout this paper and gather various preliminary results that will be required in the rest of the paper. Section \ref{regularity section} defines several notions of regularity for maps that will appear throughout this paper. Section \ref{foliation section} defines the notions of regularity for foliations that we will consider. Section \ref{flow section} treats the basics of continuous flows and introduces time changes for these flows. Section \ref{subsec: linear} provides a detailed analysis of linear cocycles over homeomorphisms of compact metric spaces as well as continuous flows on these spaces; we define dominated splittings and Lyapunov exponents, and determine how the Lyapunov exponents transform under time changes of the linear cocycle. Sections \ref{anosov} and \ref{thermo section} are a brief summary of well-known results on Anosov flows and the thermodynamic formalism that we will use.

Section \ref{weak expanding} marks a shift to the introduction of new concepts. We introduce flows with weak expanding foliations and time changes of Anosov flows that are only smooth along the center-unstable foliation. This leads to the critical Proposition \ref{time change Gibbs property} asserting that the Gibbs property for a special equilibrium state is preserved under a specific time change. Section \ref{subsec:hormeasure} defines the horizontal measure and horizontal dimension for a special class of Anosov flows including the geodesic flow $g^{t}_{Y}$ of a Riemannian manifold $Y \in \mathcal{U}_{X}$ from Theorems \ref{loc symmetric rigidity} and \ref{geom pinching rigidity}. Finally, Section \ref{subsec:neg curved} summarizes the geometric properties of negatively curved Riemannian manifolds that will be needed in this paper, as well as the special properties of negatively curved symmetric spaces of nonconstant negative curvature that will be required. 

Throughout the paper we will use $\asymp$ to denote proportionality up to a multiplicative constant that is independent of the parameters under consideration. So for two positive real-valued functions $f$ and $g$ we write $f \asymp g$ if the inequality $C^{-1}f \leq g \leq C f$ holds for a constant $C \geq 1$. 

\subsection{Regularity}\label{regularity section} For topological spaces $M$ and $N$ we write $C^{0}(M,N)$ for the space of continuous maps from $M$ to $N$ equipped with the compact-open topology. When $M$ and $N$ are metric spaces, this is the topology of uniform convergence on compact subsets. 

Let $(M,d)$ and $(N,\rho)$ be metric spaces and let $0 < \alpha \leq 1$ be a real number. We define a map $\psi: M \rightarrow N$ to be \emph{$\alpha$-H\"older} if there is a constant $K >0$ such that for every $x, y \in M$, 
\begin{equation}\label{Holder distance inequality}
\rho(\psi(x),\psi(y)) \leq K d(x,y)^{\alpha}. 
\end{equation}
In the case $\alpha = 1$, we will also say that $\psi$ is \emph{Lipschitz}. We refer to the minimum possible constant $K$ in inequality \eqref{Holder distance inequality} as the \emph{$\alpha$-H\"older constant} of $\psi$, or the \emph{Lipschitz constant} in the case $\alpha = 1$. 

For $0 < \alpha \leq 1$ we say that $\psi$ is $C^{\alpha}$ if it is locally $\alpha$-H\"older, i.e., for each $x \in M$ there is a neighborhood $U$ of $x$ and a constant $K = K_{x}$ such that $\psi: U \rightarrow N$ is $\alpha$-H\"older. We write $C^{\alpha}(M,N)$ for the space of $C^{\alpha}$ maps from $M$ to $N$, which we equip with the induced topology from $C^{0}(M,N)$. To avoid a conflict of notation with the standard definition of $C^1$ regularity in terms of differentiability below, we write $C^{1^{-}}(M,N)$ for the space of locally Lipschitz maps from $M$ to $N$. If $M$ is compact then $\psi:M \rightarrow N$ being $C^{\alpha}$ implies that $\psi$ is $\alpha$-H\"older.

For two $C^r$ manifolds $M$ and $N$, $1 \leq r \leq \infty$, we write $C^{r}(M,N)$ for the space of $C^r$ maps from $M$ to $N$. When $M$ is closed, we equip this space with the standard $C^r$ topology corresponding to uniform convergence of a map and its derivatives up to order $r$ in a fixed system of coordinate charts on $M$ and $N$. 

\subsection{Foliations}\label{foliation section} Let $M$ be a $C^s$ manifold, $s \geq 0$, and let $r$ be an integer, $s \geq r \geq 0$. Let $m = \dim M$ and let $1 \leq k \leq m-1$ be given. For $n \in \N$ we let $B_{n}$ denote the open unit ball in $\R^{n}$ centered at $0$. A \emph{$k$-dimensional $C^r$-foliation} $\mathcal{W}$ of $M$ is a decomposition $\{\mathcal{W}_{i}\}_{i \in I}$ of $M$ into connected $k$-dimensional $C^r$ submanifolds $\mathcal{W}_{i}$ of $M$ which are pairwise disjoint, such that for each $x \in M$ there is an open neighborhood $U$ of $x$ and a $C^r$ coordinate chart $\psi: U \rightarrow B_{k} \times B_{m-k}$ such that $\psi(x) = (0,0)$ and for each $i \in I$ and each connected component $W$ of $U \cap \mathcal{W}_{i}$ there is a unique $p \in B_{m-k}$ such that $\psi(W) = B_{k}\times \{p\}$. Here $I$ is an appropriate index set. In the case $r = 0$ we will simply refer to $\W$ as a \emph{foliation} of $M$. The components $\mathcal{W}_{i}$ are referred to as the \emph{leaves} of $\W$. We refer to $\psi$ as a \emph{$C^r$-foliation chart}, and we refer to the domain $U$ of a foliation chart as a \emph{foliation box} for $\W$. 


For $x \in M$ we write $\W(x)$ for the leaf of $\W$ containing $x$. For an open submanifold $U \subseteq M$, any $C^r$ foliation $\W$ of $M$ induces a $C^r$ foliation $\W_{U}$ of $U$ whose leaves are the connected components of the intersections $\W_{i} \cap U$ of the leaves of $\W$ with $U$. For $x \in U$, $\W_{U}(x)$ denotes the leaf of $\W$ inside of $U$ that contains $x$. 

A \emph{transversal} to $\W$ is an $(m-k)$-dimensional submanifold $S$ of $M$ such that $\W(x)\cap S = \{x\}$ for all $x \in S$. Let $S_{1}$, $S_{2}$ be two transversals to $\W$ such that for each $x \in S_{1}$ there is at most one point $y \in S_{2}$ such that $y \in \W(x)$. For those $x$ for which there is such a point $y$, we set $h^{\W}_{S_{1},S_{2}}(x):=y$ and refer to this as the \emph{$\W$-holonomy map} from $S_{1}$ to $S_{2}$. Note that we do not require $h^{\W}_{S_{1},S_{2}}$ to be defined on all of $S_{1}$. When the choice of transversals is understood we will write $h^{\W}$ for this map. This discussion may all be localized to an open subset $U$ of $M$: a transversal to $\W$ inside of $U$ is a transversal to $\W_{U}$. We will still write $h^{\W}$ for the $h^{\W_{U}}$ holonomy map when the context is clear. The holonomies of $\W$ between transversals inside of all open subsets $U \subseteq M$ are $C^r$ if and only if the foliation $\W$ itself is $C^r$.

For $p \in \R^{m-k}$ we write $i_{p}:\R^{k} \rightarrow \R^{k} \times \R^{m-k}$ for the inclusion map $i_{p}(x) = (x,p)$. 

\begin{defn}\label{uniform foliation}
Let $\W$ be a foliation of a $C^s$ manifold $M$ and let $s \geq r \geq 1$. We define $\W$ to have \emph{uniformly $C^r$ leaves}  if there is an atlas of foliation charts $(\psi_{j},U_{j})_{j \in J}$ for $\W$ with $\psi_{j}: U_{j} \rightarrow B_{k} \times B_{m-k}$ such that the compositions for $p \in B_{m-k}$, 
\[
\zeta_{j,p} = \psi^{-1}_{j} \circ i_{p} : B_{k} \rightarrow M,
\] 
are $C^r$ and depend continuously on $p$ inside of $C^{r}(B_{k},M)$. 

For another $C^r$ manifold $N$, we define $f: M \rightarrow N$ to be \emph{uniformly $C^r$ along $\W$} if for each chart $(\psi_{j},U_{j})$ the compositions $f \circ \zeta_{j,p}$ depend continuously on $p \in B_{m-k}$ inside of $C^{r}(B_{k},N)$. 
\end{defn}  

For the case $r = \infty$, we define $\W$ to have uniformly $C^{\infty}$ leaves if it has uniformly $C^r$ leaves for all $1 \leq r < \infty$, and we define $f:M \rightarrow N$ to be uniformly $C^{\infty}$ along $\W$ if it is uniformly $C^r$ along $\W$ for each $1 \leq r < \infty$. As a shorthand, we will say that $\W$ has \emph{smooth leaves} if it has uniformly $C^{\infty}$ leaves, and we will say that a map $f:M \rightarrow N$ is \emph{smooth along $\W$} if it is uniformly $C^{\infty}$ along $\W$. 

Given a foliation $\W$ of $M$ with uniformly $C^r$ leaves, we let $T\W = \{T\W_{i}\}_{i \in I}\subset TM$ denote the tangent bundle to this foliation, which is the union of the tangent bundles of each leaf $\W_{i}$. The foliation $\W$ has uniformly $C^1$ leaves if and only if $T\W$ is a continuous subbundle of $TM$. Hence we will always consider $T\W$ as a continuous subbundle of $TM$. A map $f: M \rightarrow N$ that is uniformly $C^1$ along $\W$ induces a continuous derivative map $D_{\W}f: T\W \rightarrow TN$ which is given by the standard derivative map $Df:T\W_{i} \rightarrow TN$ of the $C^1$ map $f: \W_{i} \rightarrow N$ on each leaf of $\W$.

\subsection{Flows}\label{flow section} Let $M$ be a compact metric space. We write $\mathrm{Homeo}(M)$ for the space of homeomorphisms of $M$, equipped with the standard compact-open topology inherited from $C^{0}(M,M)$. We consider this space with the group structure given by composition of homeomorphisms. A \emph{continuous flow} $\{f^{t}\}_{t \in \R}$ on $M$ is a continuous group homomorphism 
\[
\{f^{t}\}_{t \in \R}: \R \rightarrow  \mathrm{Homeo}(M).
\]
Continuity of this homomorphism is equivalent to continuity of the map $(t,x) \rightarrow f^{t}(x)$ from $\R \times M$ to $M$. We will typically drop the brackets in the notation and write $f^{t}$ for $\{f^{t}\}_{t \in \R}$, unless there is cause for confusion with the time-$t$ map $f^{t}: M \rightarrow M$ for a fixed $t \in \R$. 

Let $\gamma:M \rightarrow \R$ be a continuous function. The \emph{additive cocycle} generated by $\gamma$ over $f^{t}$ is the continuous function $\tau:\R \times M \rightarrow \R$ given by 
\[
\tau(t,x) = \int_{0}^{t} \gamma(f^{s}x)\,ds.
\]
Observe that $\tau$ satisfies the \emph{cocycle identity}
\[
\tau(t+s,x) = \tau(s, f^{t}(x)) + \tau(t,x).
\]
Conversely, if we are given a continuous function $\tau: \R \times M \rightarrow \R$ which is continuously differentiable with respect to the $\R$-coordinate, satisfies the cocycle identity above, and satisfies $\tau(0,x) = 0$ for all $x \in M$ then $\tau$ is an additive cocycle with generator 
\[
\gamma(x) = \left.\frac{\p}{\p t}\right|_{t=0}\tau(t,x).
\]

The discussion on time changes that follows is taken from \cite{Par86}. A continuous flow $g^{t}$ on $M$ is a \emph{time change} of $f^{t}$ if there is a continuous function $\gamma: M \rightarrow (0,\infty)$ such that, letting $\tau: \R \times M \rightarrow \R$ denote the additive cocycle over $g^{t}$ generated by $\gamma$, we have $g^{t}x = f^{\tau(t,x)}x$ for all $t \in \R$, $x \in M$. We will refer to $\gamma$ as the \emph{speed multiplier} for this time change. 

Conversely, given a continuous function $\gamma:M\rightarrow (0,\infty)$, we can construct the time change $g^{t}$ of $f^{t}$ with speed multiplier $\gamma$ by letting $\omega(t,x)$ denote the additive cocycle over $f^{t}$ generated by $\gamma^{-1}$ and then setting, for each fixed $x \in M$ and $t \in \R$, $\tau(t,x)$ to be the unique solution to the equation $\omega(\tau(t,x),x) = t$. We note here that the compactness of $M$ implies that this equation has a solution for each $t \in \R$; this can be derived for instance from the fact that the continuous function $\omega(1,x)$ of $x$ is bounded away from zero on $M$.  We then set $g^{t}x = f^{\tau(t,x)}x$. It is then easily verified that $g^{t}$ is a continuous flow and $\tau$ is the additive cocycle over $g^{t}$ generated by $\gamma$. 

The following is an easy consequence of the definitions. 

\begin{prop}\label{inverse additive}
Let $\gamma, \beta: M \rightarrow (0,\infty)$ be continuous functions, let $g^{t}$ be the time change of $f^{t}$ with speed multiplier $\gamma$, and let $h^{t}$ be the time change of $g^{t}$ with speed multiplier $\beta$. Let $\tau$ denote the additive cocycle over $g^{t}$ generated by $\gamma$, $\omega$ the additive cocycle over $h^{t}$ generated by $\beta$ and $\theta$ the additive cocycle over $h^{t}$ generated by $\gamma \beta$.  Then $h^{t}$ is the time change of $f^{t}$ with speed multiplier $\gamma \beta$, and for all $x \in M$ and $t \in \R$,
\begin{equation}\label{time change transitive}
\theta(t,x) = \tau(\omega(t,x),x).
\end{equation}
In particular, setting $h^{t} = f^{t}$, $\beta = \gamma^{-1}$, we have
\begin{equation}\label{time change inverse}
\tau(\omega(t,x),x) = \omega(\tau(t,x),x) = t.
\end{equation}
\end{prop}

\begin{proof}
Let $x \in M$ be given. Differentiating with respect to $t$, we obtain
\begin{align*}
\frac{\p}{\p t}\tau(\omega(t,x),x) &= \frac{\p \tau}{\p t}(\omega(t,x),x) \frac{\p \omega}{\p t}(t,x) \\
&= \gamma(g^{\omega(t,x)}x)\beta(h^{t}x) \\
&= \gamma(h^{t}x)\beta(h^{t}x). 
\end{align*}
Since $h^{t}x = g^{\omega(t,x)}x = f^{\tau(\omega(t,x),x)}x$, we conclude that $h^{t}$ is the time change of $f^{t}$ with speed multiplier $\gamma \beta$. Equation \eqref{time change transitive} follows, and equation \eqref{time change inverse} follows immediately from \eqref{time change transitive}.
\end{proof}

As a consequence of Proposition \ref{inverse additive}, time change defines an equivalence relation on the space of continuous flows on $M$. 

Let $N$ be another compact metric space. 

\begin{defn}[Orbit equivalence]\label{orbit equivalent}Two continuous flows $f^{t}$ on $M$ and $g^{t}$ on $N$ are \emph{orbit equivalent} if there is a homeomorphism $\varphi: M \rightarrow N$ such that the flow $\{\varphi^{-1} \circ g^{t} \circ \varphi\}_{t \in \R}$ on $M$ is a time change of $f^{t}$.  We say that two orbit equivalent flows $f^{t}$ and $g^{t}$ are \emph{conjugate} if the homeomorphism $\varphi$ may be chosen such that $\varphi^{-1} \circ g^{t} \circ \varphi = f^{t}$ for all $t \in \R$. In the first case we refer to $\varphi$ as an \emph{orbit equivalence} from $f^{t}$ to $g^{t}$, and in the second case we refer to $\varphi$ as a \emph{conjugacy} from $f^{t}$ to $g^{t}$.
\end{defn} 

Orbit equivalences between nontrivial flows $f^{t}$ and $g^{t}$ are never unique, as one can always modify the orbit equivalence by flowing by $f^{t}$. We give a definition below which makes this precise. 

\begin{defn}[Flow related orbit equivalences]\label{flow related}
Two orbit equivalences $\varphi_{i}: M \rightarrow N$, $i =1,2$ from $f^{t}$ to $g^{t}$ are \emph{flow related} if there is a continuous function $\eta: M \rightarrow \R$ such that $\varphi_{2}(f^{\eta(x)}x) = \varphi_{1}(x)$ for all $x \in M$.
\end{defn}

We now extend these definitions to flows on  manifolds. Let $M$ be a compact $C^{r+1}$ manifold, $r \geq 1$, and let $v_{f}: M \rightarrow TM$ be a $C^r$ vector field on $M$. Then $v_{f}$ generates a  flow $f^{t}$ on $M$ which is $C^r$ in the sense that the map $\R \times M \rightarrow M$ given by $(t,x) \rightarrow f^{t}(x)$ is $C^r$. If $v_{g}$ is another vector field on $M$ that generates a $C^r$ flow $g^{t}$, we say that $g^{t}$ is a $C^r$ time change of $f^{t}$ if there is a $C^r$ map $\gamma: M \rightarrow (0,\infty)$ such that $v_{g}(x) = \gamma(x)v_{f}(x)$ for all $x \in M$. This implies that $g^{t}x = f^{\tau(t,x)}x$, where $\tau$ is the additive cocycle over $g^{t}$ generated by $\gamma$. For another compact $C^r$ manifold $N$ we define $f^{t}: M \rightarrow M$ and $g^{t}:N \rightarrow N$ to be $C^r$ orbit equivalent if there is a $C^r$ diffeomorphism $\varphi:M \rightarrow N$ such that the flow $\{\varphi^{-1} \circ g^{t}\circ \varphi\}_{t \in \R}$ is a $C^r$ time change of $f^{t}$. We give the space of $C^r$ flows $f^{t}$ on $M$ the induced topology from its identification with $C^{r}(M,TM)$ via the generator $v_{f}$. 

These definitions further generalize to flows defined along foliations. We consider a compact $C^{r+1}$ manifold $M$ with a foliation $\W$ that has uniformly $C^{r+1}$ leaves, $r \geq 1$.  Let $v_{f}: M \rightarrow TM$ be uniformly $C^r$ along $\W$, with $v_{f}(x) \in T\W_{x}$ for each $x \in M$.  For each leaf $\W_{i}$ of $\W$, $v_{f}$ defines a $C^r$ vector field on $\W_{i}$ and therefore generates a $C^r$ flow on $\W_{i}$.  Consequently $v_{f}$ gives rise to a continuous flow $f^{t}$ on $M$ that fixes the leaves of $\W$ and is uniformly $C^r$ along $\W$ in the sense that the map $\R \times M \rightarrow M$ given by $(t,x) \rightarrow f^{t}(x)$ is uniformly $C^r$ along the product foliation $\R \times \W$ of $\R \times M$. As above, a flow $g^{t}$ on $M$ is a time change of $f^{t}$ that is uniformly $C^r$ along $\W$ if $v_{g}(x) = \gamma(x)v_{f}(x)$ for a function $\gamma: M \rightarrow (0,\infty)$ that is uniformly $C^r$ along $\W$. Finally, for another compact $C^{r+1}$ manifold $N$ with a foliation $\V$ with uniformly $C^{r+1}$ leaves and a continuous flow $g^{t}:N \rightarrow N$ that is uniformly $C^{r}$ along $\V$, an orbit equivalence $\varphi:M \rightarrow N$ from $f^{t}$ to $g^{t}$ is uniformly $C^r$ along $\W$ if $\varphi(\W(x)) = \V(\varphi(x))$ for each $x \in M$ and $\varphi^{-1} \circ g^{t} \circ \varphi$ is a time change of $f^{t}$ that is uniformly $C^r$ along $\W$.

\subsection{Linear Cocycles}\label{subsec: linear}  Let $M$ be a compact metric space and let $E$ be a $k$-dimensional vector bundle over $M$. We equip $E$ with a continuously varying family of inner products $\{\langle\;,\;\rangle_{x}\}_{x \in M}$ on the fibers $E_{x}$ over $x \in M$, which we will refer to as a \emph{Riemannian structure} on $E$. These inner products induce a continuously varying family of norms $\{\|\cdot \|_{x}\}_{x \in M}$ on the fibers of $E$.

Let $f: M \rightarrow M$ be a homeomorphism. Let $\pi: E \rightarrow M$ denote the projection map.  A \emph{linear cocycle} on $E$ over $f$ is a continuous, invertible bundle map $A: E \rightarrow E$ that covers $f$, i.e., $f \circ \pi = \pi \circ A$. We write $A_{x}: E_{x} \rightarrow E_{f(x)}$ for the linear map induced by $A$ on the fiber of $E$ over $x$, and write $A^{n}$ for the $n$th iterate of $A$, $n \in \Z$, which covers $f^{n}$. Given another homeomorphism $g: N \rightarrow N$ of a compact metric space $N$ with a $k$-dimensional vector bundle $F$ over $N$ and a linear cocycle $B: F \rightarrow F$ over $g$, we say that $A$ and $B$ are \emph{conjugate} if there is a homeomorphism $\Phi: E \rightarrow F$ covering a conjugacy $\varphi: M \rightarrow N$ from $f^{t}$ to $g^{t}$ such that $\Phi \circ A = B \circ \Phi$, and such that $\Phi_{x}: E_{x} \rightarrow F_{\varphi(x)}$ is an invertible linear map for each $x \in M$. 


For a linear transformation $T: V \rightarrow W$ between two $k$-dimensional inner product spaces $V$ and $W$, we write
\[
\|T\| = \sup_{\substack{v \in V \\ v \neq 0}} \frac{\|T(v)\|}{\|v\|},
\]
for the \emph{norm} of $T$, and write 
\[
\mathfrak{m}(T) = \inf_{\substack{v \in V \\ v \neq 0}} \frac{\|T(v)\|}{\|v\|}.
\]
for the \emph{conorm} of $T$. We let $\sigma_{1}(T) \leq \dots \leq \sigma_{k}(T)$ denote the singular values of $T$ with respect to these inner products, listed in increasing order. We have $\|T\| =\sigma_{k}(T)$ and $\mathfrak{m}(T) =\sigma_{1}(T)$. Assuming that $T$ is orientation-preserving, we let
\[
\mathrm{Jac}(T) =  \prod_{i=1}^{k} \sigma_{i}(T),
\]
denote the \emph{Jacobian} of the transformation $T$. For $V = W$ the Jacobian as defined here is the determinant, $\mathrm{Jac}(T) = \det(T)$.

Fix a Riemannian structure on $E$. As a consequence of the multiplicative ergodic theorem \cite{Osel}, for each $f$-invariant ergodic probability measure $\nu$ there are real numbers $\la_{1}(A,\nu) \leq \dots \leq \la_{k}(A^,\nu)$ such that for each $1 \leq i \leq k$ we have
\[
\la_{i}(A,\nu) = \lim_{n \rightarrow \infty} \frac{\log \sigma_{i}(A^{n}_{x})}{n},
\]
for $\nu$-a.e.~ $x \in M$.
\begin{defn}[Lyapunov spectrum]\label{definition lyap}
The numbers $\la_{i}(A,\nu)$, $1 \leq i \leq k$, are the \emph{Lyapunov exponents} of $A$ with respect to $\nu$. The vector $\vec{\la}(A,\nu) = (\la_{i}(A,\nu))_{i=1}^{k}$ is the \emph{Lyapunov spectrum} of $A$ with respect to $\nu$. 
\end{defn}





The Lyapunov exponents do not depend on the choice of Riemannian structure on $E$: if we let $\sigma_{i}'(A^{t}_{x})$ denote the singular values measured with respect to another Riemannian structure $\langle \, , \, \rangle'$, by the  continuity of $A$ and the  compactness of $M$ we have for each $x \in M$, $n \geq 1$, and $1 \leq i \leq k$, 
\[
\sigma_{i}'(A^{n}_{x}) \asymp\sigma_{i}(A^{n}_{x}),
\]
with the implied constant being independent of $x$ and $n$. This immediately implies that the Lyapunov exponents measured with respect to these two inner products are equal. Similarly, it is easy to see that if $A$ and $B$ are  linear cocycles over $f: M \rightarrow M$ and $g: N \rightarrow N$ which are conjugate by a bundle map $\Phi$ covering a homeomorphism $\varphi: M \rightarrow N$ then $\vec{\la}(A,\nu) = \vec{\la}(B,\varphi_{*}\nu)$. It follows that the Lyapunov spectrum is a conjugacy invariant. 

For an $A$-invariant continuous subbundle $L \subseteq E$, we write $A|_{L}$ for the restriction of $A$ to $L$, considered as a linear cocycle on the subbundle $L$.

\begin{defn}[Dominated splitting]\label{defn: domination}
Let an integer $l$ be given, $1 \leq l \leq k-1$. A \emph{dominated splitting} of index $l$ for a linear cocycle $A$ over $f$ is an $A$-invariant splitting $E = L \oplus V$ of $E$ into two continuous subbundles $L$ and $V$, with $\dim L = l$, such that there are constants $C > 0$, $\chi > 0$ for which we have for all $x \in M$ and $n \geq 1$, 
\begin{equation}\label{splitting inequality}
\frac{\|A^{n}_{x}|_{L_{x}}\|}{\mathfrak{m}(A^{n}_{x}|_{V_{x}})} \leq Ce^{-\chi n}.
\end{equation}
Here we equip $L$ and $V$ with the Riemannian structures induced from the Riemannian structure on $E$.
\end{defn}

The property of a splitting being dominated does not depend on the choice of Riemannian structure on $E$. For a different choice of Riemannian structure the inequality \eqref{splitting inequality} will also hold, with the same exponent $\chi$ but a possibly different constant $C$.

If a dominated splitting $E = L \oplus V$ of index $l$ exists for $A$, then it is unique in the sense that if $E = L' \oplus V'$ is another dominated splitting of index $l$ for $A$ then $L = L'$ and $V = V'$. This immediately implies that if $A$ and $B$ are two conjugate linear cocycles with respective dominated splittings $E= L^{A} \oplus V^{A}$ and $F = L^{B} \oplus V^{B}$ of index $l$, then the conjugacy $\Phi$ satisfies $\Phi(L^{A}) = L^{B}$ and $\Phi(V^{A}) = V^{B}$. 

There are several criteria for the existence of a dominated splitting. The first of these is the existence of a strictly invariant cone field for $A$. Let $H \subset E$ be a proper continuous subbundle. For each $x \in M$, each $v \in E$ can be written uniquely as $v = v' + v^{\perp}$, with $v' \in H_{x}$ and $v^{\perp} \in H_{x}^{\perp}$. For each $x \in M$ and each $\eta > 0$ we then define a cone around $H_{x}$ by 
\[
\mathcal{C}_{x,\eta} = \{v \in E_{x}: \|v^{\perp}\| \leq \eta \|v'\|\}. 
\]

\begin{prop}\label{cone criterion}
Let $A: E \rightarrow E$ be a linear cocycle over a homeomorphism $f:M \rightarrow M$. Let $1 \leq l \leq k-1$ be an integer. Suppose that there is a $(k-l)$-dimensional continuous subbundle $H \subset E$, an integer $n \geq 1$, and real numbers $ 0 < \eta' < \eta$ such that for all $x \in M$ we have $A^{n}(\mathcal{C}_{x,\eta}) \subseteq \mathcal{C}_{f(x),\eta'}$. Then $A$ admits a dominated splitting $E = L \oplus V$ of index $l$ with $V_{x} \subset \mathcal{C}_{x,\eta'}$ for each $x \in M$. 
\end{prop}

This proposition is standard; we refer to \cite{BG09} for a proof. It's easy to see, conversely, that if $A$ has a dominated splitting $E = L^{A} \oplus V^{A}$, then a strictly invariant family of cones $\mathcal{C}_{x,\eta}$ around $V_{x}^{A}$ as in Proposition \ref{cone criterion} always exists. The existence of such a family is stable under $C^0$-small perturbations of the linear cocycle $A$, from which it follows that if $B$ is a linear cocycle over $f:M \rightarrow M$ which is $C^0$-close to $A$ then $B$ admits a dominated splitting $E = L^{B} \oplus V^{B}$ with $L^{B}$ and $V^{B}$ being $C^0$-close to $L^{A}$ and $V^{A}$ respectively. 

The second criterion that we will require is a theorem of Bochi and Gourmelon \cite{BG09}. 

\begin{thm}{\cite[Theorem A]{BG09}}\label{domination gap criterion}
Let $A: E \rightarrow E$ be a linear cocycle over a homeomorphism $f:M \rightarrow M$. Let $1 \leq l \leq k-1$ be an integer. Then $A$ admits a dominated splitting of index $l$ if and only if there are constants $C >0$ and $\chi > 0$ such that for all $n \geq 1$,
\[
\frac{\sigma_{l}(A_{x}^{n})}{\sigma_{l+1}(A_{x}^{n})} \leq Ce^{-\chi n}.
\]
\end{thm}


The domination condition implies that, for large enough $n$, there is a uniform exponential gap between the largest singular value of $A^{n}$ on $L$ and the smallest singular value of $A^{n}$ on $V$. Theorem \ref{domination gap criterion} shows that this is equivalent to an exponential gap between the $l$th and $(l+1)$th singular values of $A^{n}$ on $E$. This in turn corresponds to a gap in the Lyapunov exponents of $A$, which we establish below.  

We will need an elementary linear algebra lemma. For a linear transformation $A: U \rightarrow W$ between two finite-dimensional inner product spaces $U$ and $W$ and a subspace $L \subset U$, we will consider the restriction $A|_{L}: L \rightarrow A(L)$ and its singular values, computed as a linear transformation from $L$ to $A(L)$. 

\begin{lem}\label{linear algebra}
Let $U$, $W$ be $k$-dimensional inner product spaces, $l \geq 2$, and let $A: U \rightarrow W$ be an invertible linear map. Let $L \subset U$ be an $l$-dimensional subspace, with $1 \leq l \leq k-1$. 

Suppose that 
\[
\|A|_{L}\| < \mathfrak{m}(A|L^{\perp}),
\]
and that $A(L^{\perp})$ is perpendicular to $A(L)$. Then for $1 \leq i \leq l$,
\[
\sigma_{i}(A|_{L}) = \sigma_{i}(A),
\]
and for $l+1 \leq i \leq k$, 
\[
\sigma_{i-l}(A|_{L^{\perp}}) = \sigma_{i}(A). 
\]
\end{lem}

\begin{proof}
By choosing appropriate orthonormal bases for $U$ and $W$, we may assume that $A$ is a linear automorphism of $\R^{k}$ with its standard orthonormal basis $\{e_{1},\dots,e_{k}\}$, that $L$ corresponds to the subspace spanned by the first $l$ basis vectors $\{e_{1},\dots,e_{l}\}$, and that $A(L) = L$. Then by hypothesis we have $A(L^{\perp}) = L^{\perp}$ as well. 

Consider the singular value decompositions of $A|_{L}$ and $A|_{L^{\perp}}$. We obtain orthonormal bases $\{v_{1},\dots,v_{k}\}$ and $\{w_{1},\dots,w_{k}\}$ of $\R^{k}$ from these decompositions with $v_{i},w_{i} \in L$ for $1 \leq i \leq l$, and $v_{i},w_{i} \in L^{\perp}$ for $l+1 \leq i \leq k$, such that 
\[
A(v_{i}) = \sigma_{i}(A|_{L})w_{i},
\]
for $1\leq i \leq l$, and
\[
A(v_{i}) = \sigma_{i-l}(A|_{L^{\perp}})w_{i},
\]
for $l+1 \leq i \leq k$. Let $A^{*}$ denote the transpose of $A$; the singular values $\sigma_{i}(A)$ are the eigenvalues of $\sqrt{A^{*}A}$ with corresponding eigenvectors $v_{i}$. Hence the singular values $\sigma_{i}(A|_{L})$ and  $\sigma_{i-l}(A|_{L^{\perp}})$ above must each correspond to a singular value $\sigma_{j}(A)$ for some $1 \leq j \leq k$. Since the singular values for $A$, $A|_{L}$, and $A|_{L^{\perp}}$ are all listed in increasing order, the hypothesis $\sigma_{k}(A|_{L}) < \sigma_{1}(A|_{L^{\perp}})$ implies that the correspondence in the conclusion of the lemma holds. 
\end{proof}

\begin{lem}\label{domination exponent gap}
Let $1 \leq l \leq k-1$ be a given integer and let $\chi > 0$. Let $E = L \oplus V$ be a dominated splitting of index $l$  for a linear cocycle $A$ on $E$ over $f$ as above, such that the inequality \eqref{splitting inequality} holds with exponent $\chi$. Let $\nu$ be an $f$-invariant ergodic probability measure. Then 
\begin{enumerate}
\item for $1 \leq i \leq l$, $\la_{i}(A,\nu) = \la_{i}(A|_{L},\nu)$,
\item for $l + 1 \leq i \leq k$, $\la_{i}(A,\nu) = \la_{i-l}(A|_{V},\nu)$,
\item $\la_{l}(A,\nu) \leq \la_{l+1}(A,\nu)-\chi$. 
\end{enumerate}
\end{lem}

\begin{proof}
We first claim that it suffices to prove the lemma in the case that $L$ and $V$ are orthogonal. Let $T_{x}: E_{x} \rightarrow E_{x}$ be the unique linear transformation that restricts to the identity on $L_{x}$ and restricts on $V_{x}$ to the orthogonal projection of $V_{x}$ onto $L_{x}^{\perp}$. Clearly $T_{x}$ depends continuously on $x$ because $L_{x}$ and $V_{x}$ depend continuously on $x$. We define an inner product $\langle\, ,\,\rangle'_{x}$ on $E_{x}$ by
\[
\langle v,w\rangle_{x}' = \langle T_{x}(v),T_{x}(w)\rangle, 
\]
for $v,w \in E_{x}$. This defines a continuous inner product on $E$ with respect to which $L$ and $V$ are orthogonal. As observed above, the Lyapunov exponents of $A$ with respect to $\nu$ and the exponent $\chi$ in inequality \eqref{splitting inequality} are independent of the choice of continuous inner product on $E$. Hence we may use this new inner product to verify the claims of the lemma. 

We thus assume that $L$ and $V$ are orthogonal. Since the splitting is dominated, there is a positive integer $N$ such that for $n \geq N$ we have $\|A^{n}_{x}|_{L_{x}}\| < \mathfrak{m}(A^{n}_{x}|_{V_{x}})$ for all $x \in M$. We then apply Lemma \ref{linear algebra} to the map $A_{x}^{n}: E_{x} \rightarrow E_{f^{n}x}$ to obtain that $\sigma_{i}(A^{n}_{x}|_{L_{x}}) = \sigma_{i}(A^{n}_{x})$ for $1 \leq i \leq l$ and $\sigma_{i-l}(A^{n}_{x}|_{V_{x}}) = \sigma_{i}(A^{n}_{x})$ for $l+1 \leq i \leq k$. We conclude that for $\nu$-a.e.~ $x \in M$ and each $1 \leq i \leq l$,
\[
\la_{i}(A|_{L},\nu) = \lim_{n \rightarrow \infty} \frac{\log \sigma_{i}(A^{n}_{x}|_{L_{x}})}{n} = \lim_{n \rightarrow \infty} \frac{\log \sigma_{i}(A^{n}_{x})}{n} = \la_{i}(A,\nu).
\]
This proves (1). Similarly, for $l+1 \leq i \leq k$ and $\nu$-a.e.~ $x \in M$,
\[
\la_{i-l}(A|_{V},\nu) = \lim_{t \rightarrow \infty} \frac{\log \sigma_{i-l}(A^{n}_{x}|_{V_{x}})}{t} = \lim_{t \rightarrow \infty} \frac{\log \sigma_{i}(A^{n}_{x})}{t} = \la_{i}(A,\nu),
\]
which proves (2). Lastly, the domination inequality \eqref{splitting inequality} implies that for all $x \in M$, 
\[
\limsup_{n \rightarrow \infty} \frac{\log \sigma_{l}(A^{n}_{x}|_{L_{x}})}{n} \leq \liminf_{n \rightarrow \infty} \frac{\log \sigma_{1}(A^{n}_{x}|_{V_{x}})}{n} - \chi.
\]
Applying this inequality to a $\nu$-generic point $x$ and then combining this with (1) and (2), we obtain (3). 
\end{proof}

We next discuss H\"older regularity of vector bundles. For this purpose we will be considering continuous subbundles of a trivial bundle $M \times \R^{q}$ over $M$, for some $q \geq 1$. We consider $\R^{q}$ with the Euclidean metric and equip $M \times \R^{q}$ with the product metric. We endow the Grassmannian $\mathrm{Gr}_{k}(\R^{q})$ of $k$-planes in $\R^{q}$ with the standard metric, 
\begin{equation}\label{Hausdorff}
d(V,W) = \max\left\{\sup_{v \in V, \|v\| = 1}\inf\{\|v-w\|: w \in W\}, \sup_{w \in W, \|w\| = 1}\inf\{\|v-w\|: v \in V\} \right\}.
\end{equation}
A continuous subbundle $E \subseteq M \times \R^{q}$ corresponds to a continuous map $\xi_{E}: M \rightarrow \mathrm{Gr}_{k}(\R^{q})$ given by $\xi_{E}(x) = E_{x}$. For  $0 < \alpha \leq 1^{-}$, we say that a subbundle $E \subseteq M \times \R^{q}$ is $C^{\alpha}$ if the map $\xi_{E}$ is $C^{\alpha}$; recall that $C^{1-}$ denotes locally Lipschitz maps. We define a linear cocycle $A: E \rightarrow E$ over a $C^{\alpha}$ homeomorphism $f: M \rightarrow M$ to be $C^{\alpha}$ if it is $C^{\alpha}$ in the metric on $E$ inherited from the product metric on $M \times \R^{q}$. 

A principal concern of ours will be the H\"older regularity of the subbundles of a dominated splitting.  Proposition \ref{holder splitting} below is a direct consequence of the H\"older section theorem \cite[Theorem 3.8]{HPS}. 

\begin{prop}\label{holder splitting}
Let $0 < \alpha \leq 1^{-}$ be given. Let $E$ be a $C^{\alpha}$ vector bundle over a compact metric space $M$. Let $A: E \rightarrow E$ be a linear cocycle over a  homeomorphism $f:M \rightarrow M$ such that $f$ and $f^{-1}$ are Lipschitz and $A$ and $A^{-1}$ are $C^{\alpha}$. Suppose that $A$ has a dominated splitting $E = L \oplus V$ of index $l$. 

Then there exists $0 < \beta \leq \alpha$ such that $L$ and $V$ are $C^{\beta}$ subbundles of $E$. The exponent $\beta$ depends only on $\alpha$, the exponent $\chi$ of inequality \eqref{splitting inequality}, the Lipschitz constants of $f$ and $f^{-1}$, and the quantity $\sup_{x \in M} \max\{\|A_{x}\|,\|A_{x}^{-1}\|\}$.
\end{prop}

One obtains a more precise characterization of $\beta$ from a theorem of Brin and Kifer \cite{BK87} when $A$ satisfies certain growth inequalities on $L$ and $V$. We consider below the same setting as in Proposition \ref{holder splitting}.

\begin{prop}\label{strong holder splitting}
Suppose further that $L$ and $V$ are orthogonal, and that there are constants $\la_{1} \leq \la_{2} < \chi_{1} \leq \chi_{2}$, and $a > \alpha^{-1}\max\{\la_{1},-\chi_{2},0\}$ 
such that for all $n \geq 1$,
\[
e^{\la_{1} n} \leq \mathfrak{m}(A_{x}|_{L_{x}}) \leq \|A_{x}|_{L_{x}}\| \leq e^{\lambda_{2}n},
\]
and
\[
e^{\chi_{1}n} \leq \mathfrak{m}(A_{x}|_{V_{x}}) \leq \|A_{x}|_{V_{x}}\| \leq e^{ \chi_{2}n},
\]
and such that the Lipschitz constants of $f^{n}$ and $f^{-n}$ are at most $e^{an}$. Then $L$ is $C^{\beta}$ with 
\[
\beta =  \left(\frac{\chi_{1}-\la_{2}}{a\alpha + \chi_{2} - \la_{2}}\right)\alpha,
\]
and $V$ is $C^{\eta}$ with 
\[
\eta = \left(\frac{\chi_{1}-\la_{2}}{a\alpha - \la_{1} + \chi_{1}}\right)\alpha.
\]
\end{prop}

\begin{proof}
We will prove the claim for the subbundle $L$. The claim for $V$ then follows by considering the linear cocycle $A^{-1}$ over $f^{-1}$. For each $x \in M$ we take the unique linear extension $\t{A}_{x}:\R^{q} \rightarrow \R^{q}$ of $A_{x}: E_{x} \rightarrow E_{f(x)}$ that satisfies $\t{A}_{x}|_{E_{x}^{\perp}} \equiv 0$. Since $E$ is a $C^{\alpha}$ subbundle of $M \times \R^{q}$, $\t{A}: M \times \R^{q} \rightarrow M \times \R^{q}$ is $C^{\alpha}$, with $\alpha$-H\"older constant $K$ determined by the $\alpha$-H\"older constants of $A$ and $E$. 

We set $\t{L}_{x} = L_{x} \oplus E_{x}^{\perp}$. We conclude from our hypotheses that for all $n \geq 1$, 
\[
\|\t{A}_{x}^{n}|_{\t{L}_{x}}\| = \|A_{x}^{n}|_{L_{x}}\| \leq e^{\la_{2} n},
\]
and since $\t{A}_{x}|_{V_{x}} = A_{x}|_{V_{x}}$, 
\[
\mathfrak{m}(\t{A}_{x}^{n}|_{V_{x}}) \geq e^{\chi_{1} n}.
\]
Since $L$ and $V$ are orthogonal, we also have for any $x \in M$, 
\[
\|\t{A}_{x}\| = \|\t{A}_{x}|_{V_{x}}\| \leq e^{\chi_{2}}.
\]
We then obtain the estimate for $x$, $y \in M$, and $n \geq 1$, 
\begin{align*}
\|\t{A}^{n}_{x} - \t{A}^{n}_{y}\| &\leq \|\t{A}_{f^{n-1}x}-\t{A}_{f^{n-1}y}\| \|\t{A}^{n-1}_{y}\| + \|\t{A}_{f^{n-1}x}\| \|\t{A}^{n-1}_{x} - \t{A}^{n-1}_{y}\| \\
&\leq Kd(f^{n-1}x,f^{n-1}y)^{\alpha} e^{\chi_{2}(n-1)} + e^{\chi_{2}} \|\t{A}^{n-1}_{x} - \t{A}^{n-1}_{y}\| \\
&\leq K e^{(a \alpha + \chi_{2})(n-1)} d(x,y)^{\alpha} + e^{\chi_{2}} \|\t{A}^{n-1}_{x} - \t{A}^{n-1}_{y}\|.
\end{align*}
Applying these estimates inductively, we obtain 
\begin{align*}
\|\t{A}^{n}_{x} - \t{A}^{n}_{y}\| &\leq Kd(x,y)^{\alpha} \sum_{j=0}^{n-1} e^{(a\alpha + \chi_{2})j+ (n-j-1) \chi_{2}} \\
&= Kd(x,y)^{\alpha}e^{(n-1)\chi_{2}} \sum_{j=0}^{n-1} e^{a \alpha j} \\
&\leq K' d(x,y)^{\alpha} e^{(a\alpha + \chi_{2}) n},
\end{align*}
for some constant $K'$ independent of $n$. By \cite[Theorem 5.2]{BK87}, we conclude that $\t{L}$ is $C^{\beta}$ with
\[
\beta = \left(\frac{\chi_{1}-\la_{2}}{a\alpha + \chi_{2} - \la_{2}}\right)\alpha.
\]
We remark that applying this theorem requires the exponent in the growth estimate for $\|\t{A}^{n}_{x} - \t{A}^{n}_{y}\|$ to satisfy $a \alpha + \chi_{2} > \max\{\la_{2},0\}$; this estimate is a consequence of the requirements $a > 0$, $\chi_{2} > \la_{2}$, and $a\alpha > -\chi_{2}$.  Since $L = \t{L} \cap E$, this implies that $L$ is a $C^{\beta}$ subbundle of $E$. 
\end{proof}

We next extend the concepts of this section to linear cocycles over flows. Let $E$ be a continuous vector bundle over a compact metric space $M$ as above. Let $f^{t}:M \rightarrow M$ be a continuous flow. A linear cocycle on $E$ over $f^{t}$ is a continuous flow $\{A^{t}\}_{t \in \R}$ on $E$ such that the time-$t$ map $A^{t}$ is a bundle map covering $f^{t}$ for each $t \in \R$. We write $A^{1}:=A$ and $f^{1}:=f$ for the time-1 maps of $A^{t}$ and $f^{t}$ respectively. 

We define an $A^{t}$-invariant splitting $E = L \oplus V$ to be dominated if the splitting is dominated for the time-1 map $A$. In this case inequality \eqref{splitting inequality} also holds for all $t \geq 0$, though possibly with a larger constant $C$. The analogues of Propositions \ref{cone criterion}, \ref{domination gap criterion}, \ref{holder splitting}, and  \ref{strong holder splitting} all hold for dominated splittings for linear cocycles over flows; this is easily deduced from those propositions by consideration of the time-1 map. Similarly, we define $A^{t}$ to be uniformly quasiconformal if the time-1 map $A$ is uniformly quasiconformal, and in this case inequality \eqref{equation uniform quasi} holds for $t \geq 0$ with a possibly larger constant $C$ as well.

There will be several points at which we will need to eliminate the constants $C$ in certain growth inequalities related to linear cocycles $A^{t}$ over a flow $f^{t}$ by choosing a new Riemannian structure on $E$. Proposition \ref{new metric} below is standard, going back to Anosov \cite{A69}. 

\begin{prop}\label{new metric}
Let $A^{t}:E \rightarrow E$ be a linear cocycle on $E$ over a continuous flow $f^{t}: M \rightarrow M$. Let $\{\|\cdot\|\}_{x \in M}$ be a Riemannian structure on $E$. Suppose that for $n \geq 1$, $0 < a \leq b \in \R$, and some constant $C \geq 1$,
\[
C^{-1}e^{an} \leq \mathfrak{m}(A^{n}_{x}) \leq  \|A^{n}_{x}\| \leq Ce^{bn}.
\]
Then there exists a Riemannian structure $\{\|\cdot\|_{x}'\}_{x \in M}$ on $E$ such that for all $t > 0$ and $x \in M$, 
\[
e^{at} \leq \mathfrak{m}'(A^{t}_{x}) \leq  \|A^{t}_{x}\|' \leq e^{bt},
\]
where $\mathfrak{m}'$ denotes the conorm in this Riemannian structure. 
\end{prop}

The new Riemannian structure built from Proposition \ref{new metric} is also referred to as a \emph{Lyapunov adapted metric}. 

Given an $f^{t}$-invariant ergodic probability measure $\nu$ on $M$, we define the Lyapunov spectrum of $A^{t}$ with respect to $\nu$ to be the Lyapunov spectrum $\vec{\la}(A,\nu)$ of the time-1 map of $A$. We will need to determine how the Lyapunov spectrum changes when we make time changes of $f^{t}$ that give rise to time changes of $A^{t}$. 

As a shorthand, for a continuous function $\gamma:M \rightarrow (0,\infty)$ and a probability measure $\nu$ on $M$ we set
\[
\nu(\gamma) = \int_{M}\gamma \, d\nu
\]
 Let $g^{t}$ be a time change of $f^{t}$ with speed multiplier $\gamma$. Consider an ergodic invariant probability measure $\nu$ for $f^{t}$. We define a new probability measure $\nu^{\gamma}$ on $M$ that is equivalent to $\nu$, with Radon-Nikodym derivative
\begin{equation}\label{radon}
\frac{d\nu^{\gamma}}{d\nu}(x) = \gamma(x)^{-1} \nu(\gamma^{-1})^{-1}.
\end{equation}
Proposition \ref{new invariant measure} below may be found in \cite[Chapter 10$\S$3]{CFS82}.

\begin{prop}\label{new invariant measure}
The measure $\nu^{\gamma}$ is the unique $g^{t}$-invariant ergodic probability measure that is equivalent to $\nu$. All $g^{t}$-invariant ergodic probability measures have the form $\nu^{\gamma}$ for some $f^{t}$-invariant ergodic probability measure $\nu$. 
\end{prop}

One also has Abramov's formula \cite{Abr59} for the entropy of $g^{t}$ with respect to $\nu^{\gamma}$, 
\begin{equation}\label{abramov}
h_{\nu^{\gamma}}(g) = \nu(\gamma^{-1})^{-1}h_{\nu}(f).
\end{equation}

Let $\tau$ be the additive cocycle over $g^{t}$ such that $g^{t}x = f^{\tau(t,x)}x$. Given a linear cocycle $A^{t}$ over $f^{t}$, we define a linear cocycle $B^{t}$ over $g^{t}$ by the formula 
\[
B^{t}_{x} = A^{\tau(t,x)}_{x},
\]
for $x \in M$ and  $t \in \R$. Then $B^{t}$ is a time change of $A^{t}$ with the same  speed multiplier $\gamma$.

\begin{prop}\label{time change lyap}
Let $A^{t}$ be a linear cocycle over a continuous flow $f^{t}$. Let $g^{t}$ be a time change of $f^{t}$ with speed multiplier $\gamma$, and consider the corresponding linear cocycle  $B^{t}$ over $g^{t}$. Then for any $f^{t}$-invariant ergodic probability measure $\nu$, we have
\[
\vec{\la}(g,\nu^{\gamma}) = \nu(\gamma^{-1})^{-1}\vec{\la}(f,\nu).
\] 
\end{prop}

\begin{proof}
By the Birkhoff ergodic theorem, for $\nu^{\gamma}$-a.e. $x \in M$, 
\[
\lim_{n \rightarrow \infty} \frac{1}{n}\tau(n,x) = \lim_{n \rightarrow \infty} \frac{1}{n}\sum_{i=0}^{n-1}\tau(1,g^{i}x) = \int_{M}\tau(1,x) \, d\nu^{\gamma}(x). 
\]
By Fubini's theorem and the $g^{t}$-invariance of $\nu^{\gamma}$, 
\begin{align*}
\int_{M}\tau(1,x)\,d\nu^{\gamma}(x) &= \int_{0}^{1}\int_{M}\gamma(g^{s}x)\,d\nu^{\gamma}(x)\,ds \\
&= \int_{0}^{1}\int_{M}\gamma(x) \, d\nu^{\gamma}(x) \, ds \\
&= \int_{M} \gamma \, d\nu^{\gamma} \\
&= \nu(\gamma^{-1})^{-1},
\end{align*} 
with the last line following from \eqref{radon}. For $1 \leq i \leq k$, we have for $\nu^{\gamma}$-a.e.~ $x$ (and hence $\nu$-a.e.~ $x$), 
\[
\la_{i}(B,\nu^{\gamma}) = \lim_{n \rightarrow \infty} \frac{1}{n}\log \sigma_{i}(B^{n}_{x}) = \lim_{n \rightarrow \infty} \frac{1}{n}\log \sigma_{i}(A^{\tau(n,x)}_{x}). 
\]
Since $\tau(n,x) \rightarrow \infty$ as $n \rightarrow \infty$, for $\nu$-a.e.~ $x$ we have
\begin{align*}
\lim_{n \rightarrow \infty} \frac{1}{n}\log \sigma_{i}(A^{\tau(n,x)}_{x}) &= \left(\lim_{n \rightarrow \infty} \frac{\tau(n,x)}{n}\right) \left(\lim_{n \rightarrow \infty} \frac{1}{\tau(n,x)}\log \sigma_{i}(A^{\tau(n,x)}_{x})\right) \\
&= \nu(\gamma^{-1})^{-1}\la_{i}(A,\nu),
\end{align*}
which completes the proof. 
\end{proof}

When $M$ is a $C^r$ manifold, $1 \leq r \leq \infty$, we can define higher regularity for vector bundles $E \subseteq M \times \R^{q}$ over $M$. We define $E$ to be $C^r$ if the map $\xi_{E}$ is $C^r$, or equivalently if $E$ is a $C^r$ submanifold of $M \times \R^{q}$. We define a linear cocycle $A: E \rightarrow E$ over a $C^r$ diffeomorphism $f: M \rightarrow M$ to be $C^r$ if $A$ is a $C^r$ diffeomorphism of $E$.  Similarly we define a linear cocycle $A^{t}: E \rightarrow E$ over a $C^r$ flow $f^{t}: M \rightarrow M$ to be $C^r$ if $A^{t}$ is $C^r$ as a flow on $E$. 

Let $\W$ be a $k$-dimensional foliation of a manifold $M$ with uniformly $C^r$ leaves, $1 \leq r \leq \infty$. We define $E$ to be uniformly $C^r$ along $\W$ if $\xi_{E}$ is uniformly $C^r$ along $\W$. If $E$ is uniformly $C^r$ along $\W$ then we have a partition $\{E|_{\W_{i}}\}_{i \in I}$ of $E$ into its restrictions to each leaf of $\W_{i}$. This defines a foliation $\W_{E}$ of $E$ with uniformly $C^r$ leaves, in the sense that $\{E|_{\W_{i}}\}_{i \in I}$ is a uniformly $C^r$ family of submanifolds of $M \times \R^{q}$. 

Suppose that we have a homeomorphism $f: M \rightarrow M$ that is uniformly $C^r$ along $\W$. We define a linear cocycle $A: E \rightarrow E$ over $f$ to be uniformly $C^r$ along $\W$ if it is uniformly $C^r$ along $\W_{E}$. Similarly, for a flow $f^{t}$ that preserves $\W$ and is uniformly $C^r$ along $\W$, we define a linear cocycle $A^{t}: E \rightarrow E$ to be uniformly $C^r$ along $\W$ if $A^{t}$ is uniformly $C^r$ along $\W_{E}$. Lastly, for convenience we will not always want to consider our vector bundles as subbundles of a trivial bundle. We will consider more generally the case that $M$ is a  $C^{r+1}$ Riemannian manifold, $1 \leq r \leq \infty$, and $E \subseteq TM$ is a subbundle of the tangent bundle of $M$.

In this case we can measure H\"older regularity of $E$ using the \emph{Sasaki metric} \cite{Sa58} on $TM$, defined as follows: let $m = \dim M$. The Levi-Civita connection on $TM$ defines a splitting $TTM_{(x,v)} = H_{(x,v)} \oplus V_{(x,v)}$ at each $x \in M$, $v \in TM_{x}$, such that $V_{(x,v)}$ is the tangent space of the fiber $TM_{x}$ at $v$ and the projection $D\pi_{(x,v)}:H_{(x,v)}\rightarrow TM_{x}$ induced from the projection $\pi:TM \rightarrow M$ is an isomorphism. The Riemannian structure defining the Sasaki metric is then obtained by declaring $H_{(x,v)}$ and $V_{(x,v)}$ to be orthogonal and  assigning the Riemannian inner product $\langle \; , \; \rangle_{x}$ from $TM_{x}$ to $V_{(x,v)}$ via the natural identification of these spaces, and defining the inner product on $H_{(x,v)}$ to be the pullback of $\langle \; , \; \rangle_{x}$ on $TM_{x}$ by $D\pi_{(x,v)}$. 

Equipped with this metric, $TM$ is a $C^r$ Riemannian manifold that is locally isometric to $U \times \R^{m}$ where $U$ is any small enough ball in $M$; here we take the metric on $U \times \R^{m}$ to be the product of the Riemannian metric on $U$ and the Euclidean metric on $\R^{m}$. One may then check that all of the notions of regularity we defined in this section are local in nature and therefore pass naturally to subbundles $E \subseteq TM$. 

We make a few other notes to conclude this section. Our first note concerns quotient bundles. Given a compact metric space $M$, a vector bundle $E$ over $M$, and a continuous subbundle $V \subset E$, we define the \emph{quotient bundle} $E/V$ over $M$ to be the bundle whose fibers are the quotient spaces $E_{x}/V_{x}$, $x \in M$. This bundle carries the quotient topology from the projection $p: E \rightarrow E/V$, which is a bundle map covering the identity on $M$. For all of the above discussions of regularity in this section, we identify $E/V$ with the subbundle $V^{\perp}$ via the isomorphism $p|_{V^{\perp}}:V^{\perp} \rightarrow E/V$. Then $E/V$ has the same regularity as $V$. 

For a bundle map $A: E \rightarrow E$ over a homeomorphism $f: M \rightarrow M$ with $A(V) = V$, we have an induced map $\bar{A}: E/V \rightarrow E/V$. We measure the regularity of this map by identifying it with the map $\check{A} := p|_{V^{\perp}} \circ \bar{A} \circ (p|_{V^{\perp}})^{-1}$ on $V^{\perp}$. Letting $\pi_{V^{\perp}}$ denote the orthogonal projection onto $V^{\perp}$, we observe that
\begin{equation}\label{quotient}
\check{A} = \pi_{V^{\perp}} \circ A |_{V^{\perp}}.
\end{equation}

We will also occasionally need to consider exterior powers of bundles. For $1 \leq l \leq k$, given a vector bundle $E$ over $M$ we form the $l$th exterior power bundle $\wedge^{l}E$ over $M$ whose fiber over $x \in M$ is the $l$th exterior power $\wedge^{l}E_{x}$ of the vector space $E_{x}$. A linear cocycle $A: E \rightarrow E$ over a homeomorphism $f:M \rightarrow M$ naturally induces a linear cocycle $\wedge^{l}A: \wedge^{l}E \rightarrow \wedge^{l}E$ over $f$. A Riemannian structure on $E$ also naturally induces a Riemannian structure on $\wedge^{l}E$. The bundle $\wedge^{l}E$, the linear cocycle $\wedge^{l}A$, and the Riemannian structure all have the same regularity as $E$, $A$, and the Riemannian structure on $E$. All that we will require from exterior power cocycles is the equality for $x \in M$, 
\begin{equation}\label{exterior}
\|\wedge^{l}A_{x}\| = \prod_{i=k-l+1}^{k}\sigma_{i}(A_{x}). 
\end{equation}
Note in particular that $\|\wedge^{k}A_{x}\| = \mathrm{Jac}(A_{x})$.



\subsection{Anosov flows}\label{anosov}
An \emph{Anosov flow} on a $C^{r+1}$ Riemannian manifold $M$, $r \geq 1$ is a $C^r$ flow $f^{t}: M \rightarrow M$, with $v_{f}(x) \neq 0$ for all $x \in M$, such that there exists a  $Df^{t}$-invariant splitting $TM = E^{u} \oplus E^{c} \oplus E^{s}$ into subbundles $E^{u}$, $E^{c}$, and $E^{s}$, with $E^{c}$ the line bundle spanned by $v_{f}$, as well as numbers $C \geq 1$ and $a > 0$ such that for all $t \geq 0$ and $x \in M$, 
\begin{equation}\label{unstable inequality}
\mathfrak{m}(Df^{t}|_{E^{u}_{x}}) \geq C^{-1}e^{ a t}.
\end{equation}
and
\begin{equation}\label{stable inequality}
\|Df^{t}|_{E^{s}_{x}}\| \leq C e^{-a t}.
\end{equation}
Hence $E^{u}$ is exponentially expanded by $Df^{t}$ and $E^{s}$ is exponentially contracted by $Df^{t}$. A standard reference for the claims made below regarding Anosov flows is \cite{HK}. For the discussion below we will always assume that our flows are \emph{transitive}, i.e., that they have a dense orbit. 

The bundles $E^{u}$, $E^{c}$, and $E^{s}$ are referred to as the \emph{unstable}, \emph{center}, and \emph{stable} bundles respectively. The sum bundles $E^{cu}:=E^{c} \oplus E^{u}$ and $E^{cs}:=E^{c} \oplus E^{s}$ are referred to as the \emph{center-unstable} and \emph{center-stable} bundles respectively. The subbundles $E^{u}$, $E^{s}$, $E^{c}$, $E^{cu}$, and $E^{cs}$ are uniquely integrable and tangent to foliations $\W^{u}$, $\W^{s}$, $\W^{c}$, $\W^{cu}$, and $\W^{cs}$ respectively, each of which have uniformly $C^r$ leaves. These foliations will be referred to as the unstable, stable, center, center-unstable, and center-stable foliations respectively. We note that $E^{c}$ has the same regularity as $v_{f}$, hence $\W^{c}$ is a $C^r$ foliation of $M$ whose leaves are the orbits of $f^{t}$. 

When $f^{t}$ is $C^2$, the derivative cocycle $Df^{t}: TM \rightarrow TM$ is a $C^1$ linear cocycle over $f^{t}$. The splittings $TM = E^{u} \oplus E^{cs}$ and $TM = E^{cu} \oplus E^{s}$ are both dominated splittings, from which we conclude by Proposition \ref{holder splitting} that each of the subbundles $E^{u}$, $E^{s}$, $E^{cu}$, and $E^{cs}$ are H\"older continuous. However, one can obtain sharper regularity results for the center-unstable bundle $E^{cu}$ and center-stable bundle $E^{cs}$ than those given by Proposition \ref{holder splitting}. 
\begin{defn}[$\beta$-bunching]\label{beta bunching}
For $\beta > 0$ we say that an Anosov flow $f^{t}$ is \emph{$\beta$-$u$-bunched} if there is a constant $C > 0$ such that for all $x \in M$ and $t \geq 0$, 
\begin{equation}\label{bunching inequality}
\frac{\|Df^{t}_{x}|_{E^{u}_{x}}\|}{\mathfrak{m}(Df^{t}_{x}|_{E^{u}_{x}})} < C\min \{\|Df^{t}_{x}|_{E^{s}_{x}}\|^{-1}, \mathfrak{m}(Df^{t}_{x}|_{E^{u}_{x}})\}^{\beta^{-1}},
\end{equation}
and we say that $f^{t}$ is \emph{$\beta$-$s$-bunched} if $f^{-t}$ is $\beta$-$u$-bunched. We say that $f^{t}$ is \emph{$\beta$-bunched} if it is both $\beta$-$u$-bunched and $\beta$-$s$-bunched.
\end{defn} 

The $\beta$-bunching conditions correspond to regularity properties of $E^{cu}$ and $E^{cs}$ \cite{Has2}. When $f^{t}$ is $\beta$-$u$-bunched for some $0 < \beta < 1$, the bundle $E^{cu}$ is $\beta$-H\"older; likewise when $f^{t}$ is $\beta$-$s$-bunched, the bundle $E^{cs}$ $\beta$-H\"older. When $\beta = 1$ the $1$-$u$-bunching condition implies that $E^{cu}$ is $C^1$, and similarly the $1$-$s$-bunching condition implies that $E^{cs}$ is $C^1$.  These bunching conditions are generically optimal for predicting the regularity of the bundles $E^{cu}$ and $E^{cs}$ \cite{HW99}.

The following fact about Anosov flows will be useful. 

\begin{fact}\label{structural stability} A $C^2$ Anosov flow $f^{t}: M \rightarrow M$ on a closed  Riemannian manifold $M$ is \emph{structurally stable}: any $C^2$ flow $g^{t}: M \rightarrow M$ on $M$ which is $C^1$ close enough to $f^{t}$ is itself an Anosov flow that is orbit equivalent to $f^{t}$ \cite[Theorem 18.2.3]{HK}. 

Furthermore, if $g^{t}:N \rightarrow N$ is \emph{any} Anosov flow on a closed  Riemannian manifold $N$ which is orbit equivalent to $f^{t}$ via a homeomorphism $\varphi: M \rightarrow N$, we may find a H\"older continuous orbit equivalence $\hat{\varphi}: M \rightarrow N$ from $g^{t}$ to $f^{t}$ which is flow related to $\varphi$ and is arbitrarily $C^0$ close to $\varphi$ \cite[Theorem 19.1.5]{HK}.
\end{fact}

We emphasize that the orbit equivalence $\varphi$ obtained from structural stability is rarely $C^1$ or even Lipschitz, even though both $f^{t}$ and $g^{t}$ are $C^2$. The second half of Fact \ref{structural stability} implies that we can always assume an orbit equivalence between Anosov flows is H\"older.

When $f^{t}$ has Anosov splitting $TM = E^{u} \oplus E^{c} \oplus E^{s}$, with $\dim E^{u} = k$, the largest $k$ Lyapunov exponents with respect to any given $f^{t}$-invariant ergodic probability measure $\nu$ will be positive. Furthermore, since the splitting $TM = E^{u}\oplus E^{cs}$ is dominated, by Proposition \ref{domination exponent gap} these exponents coincide with the Lyapunov exponents of the restriction $Df^{t}|_{E^{u}}$. We  introduce a special notation for these exponents.  

\begin{defn}[Unstable Lyapunov spectrum]\label{definition unstable lyap} Let $f^{t}$ be a $C^2$ Anosov flow with Anosov splitting $TM = E^{u} \oplus E^{c} \oplus E^{s}$. For an $f^{t}$-invariant ergodic probability measure $\nu$ we define the \emph{unstable Lyapunov spectrum} of $f^{t}$ with respect to $\nu$ to be the Lyapunov spectrum of $Df^{t}|_{E^{u}}$ with respect to $\nu$, 
\[
\vec{\la}^{u}(f,\nu):=\vec{\la}(Df|_{E^{u}},\nu).
\]
\end{defn}

\subsection{Thermodynamic formalism}\label{thermo section} We begin with a brief description of the thermodynamic formalism for suspension flows over subshifts of finite type, for which a standard reference is \cite{BR75}. Let $l \geq 1$ be an integer and let  $H = \{H_{ij}\}_{1 \leq i,j \leq l}$ be an $l \times l$ matrix with $H_{ij} = 0$ or $H_{ij} = 1$ for each $i$, $j$. The subshift of finite type with matrix $H$ is the space
\[
\Sigma = \{(x_{n})_{n \in \Z}: x_{n} \in \{1,\dots,l\}, \, H_{x_{n}x_{n+1}} = 1 \; \text{for $n \in \Z$}\}.
\]
We write $x = (x_{n})_{n \in \Z}$. We equip this space with the metric 
\[
d(x,y) = 2^{-\max \{N:\,  x_{n} = y_{n} \; \text{for $|n| \leq N$}\}}.
\]
We define $\sigma: \Sigma \rightarrow \Sigma$ to be the left shift map 
\[
\sigma((x_{n})_{n \in \Z}) = (x_{n+1})_{n \in \Z}
\]
Given a $C^{\beta}$ function $\psi: \Sigma \rightarrow (0,\infty)$, $0 < \beta \leq 1^{-}$, we define the suspension space
\[
\Sigma_{\psi} = \Sigma \times \R / (x,t+\psi(x)) \sim (\sigma(x),t).
\]
We give this space the product metric, taking the metric $d$ on $\Sigma$ and the Euclidean metric on $\R$. This makes $\Sigma_{\psi}$ into a compact metric space. We define a continuous flow on $\Sigma \times \R$ by
\[
\t{f}^{t}(x,s) = (x,s+t).
\]
This descends to a $C^{\beta}$ flow $f^{t}: \Sigma_{\psi} \rightarrow \Sigma_{\psi}$. 

We state here the conclusions of the thermodynamic formalism for these suspension flows.  We will assume that $\Sigma$ is \emph{topologically mixing}: there is an $n \geq 1$ such that all of the entries of $H^{n}$ are positive. 

Let $\zeta: \Sigma_{\psi} \rightarrow \R$ be a $C^{\alpha}$ function, $0 < \alpha \leq \beta$. We will often refer to $\zeta$ as a \emph{potential}. The \emph{topological pressure} of $\zeta$ with respect to $f^{t}$ is defined by
\begin{equation}\label{pressure definition}
P(\zeta) = \lim_{T \rightarrow \infty} \frac{1}{T}\log\sum_{\substack{(x,t): f^{t}x = x, \\ t \in (0,T]}} \exp\left(-\int_{0}^{t}\zeta(f^{s}x)\,ds \right). 
\end{equation}
For $\zeta \equiv 0$, $P(0) = h_{\text{top}}(f)$ is simply the topological entropy of $f^{t}$. 

The \emph{variational principle for pressure} states that $P(\zeta)$ may alternatively be described as
\begin{equation}\label{variational principle}
P(\zeta) = \sup_{\nu \in \mathcal{M}_{\text{erg}}(f^{t})} h_{\nu}(f) - \int_{\Sigma_{\psi}} \zeta \, d\nu,
\end{equation}
where the supremum is taken over the set $\mathcal{M}_{\text{erg}}(f^{t})$ of all $f^{t}$-invariant ergodic probability measures $\nu$, and $h_{\nu}(f)$ denotes the entropy with respect to the measure $\nu$. There is a unique $f^{t}$-invariant ergodic probability measure $\mu_{\zeta}$ which achieves the supremum in the variational principle. We refer to $\mu_{\zeta}$ as the \emph{equilibrium state} of $\zeta$ with respect to $f^{t}$. 

\begin{rem}
The negative sign in the exponent of equation \eqref{pressure definition} is not present in the reference \cite{BR75} that we use, as they use a different sign convention for the topological pressure. This is a minor detail, as the theory remains the same if one replaces $\zeta$ by $-\zeta$. Our sign convention is chosen such that the potentials we consider will always be positive. 
\end{rem}

Two $C^{\alpha}$ potentials $\zeta$ and $\omega$ are \emph{cohomologous} if there exists a $C^{\alpha}$ function $\xi: \Sigma_{\psi} \rightarrow \R$ such that for all $t \in \R$ and $x \in \Sigma_{\psi}$, 
\begin{equation}\label{cohomological equation}
\xi(f^{t}x) - \xi(x) = \int_{0}^{t}\zeta(f^{s}x) \, ds - \int_{0}^{t}\omega(f^{s}x) \, ds. 
\end{equation}
This equation has a measurable rigidity property: for a fixed equilibrium state $\nu$ (possibly with respect to another potential $\theta$) if the equation \eqref{cohomological equation} admits a measurable solution $\xi$ that holds for $\nu$-a.e.~ $x \in \Sigma_{\psi}$ and all $t \in \R$ then it admits a $C^{\alpha}$ solution that agrees $\nu$-a.e.~ with $\xi$ \cite{Liv}.

Lastly, two $C^{\alpha}$ potentials $\zeta$ and $\omega$ have the same equilibrium state $\mu_{\zeta} = \mu_{\omega}$ if and only if they are cohomologous up to an additive constant, or in other words, if and only if there is a $C^{\alpha}$ function $\xi$ and a constant $c \in \R$ such that for all $t \in \R$ and $x \in \Sigma_{\psi}$, 
\begin{equation}\label{cohomological equation plus}
\xi(f^{t}x) - \xi(x) + c = \int_{0}^{t}\zeta(f^{s}x) \, ds - \int_{0}^{t}\omega(f^{s}x) \, ds. 
\end{equation}

Let $g^{t}$ be a $C^{\beta}$ time change of $f^{t}$ on $\Sigma_{\psi}$. For each $x \in \Sigma$ we let $\eta(x)$ be the smallest positive real number such that 
\[
g^{\eta(x)}(x,0) = (x,\psi(x)). 
\]
We note that $\eta: \Sigma \rightarrow (0,\infty)$ is $C^{\beta}$. Let $\tau$ be the additive cocycle over $g^{t}$ such that $g^{t}(x,s) = f^{\tau(t,(x,s))}(x,s)$, which is $C^{\beta}$ since $g^{t}$ is $C^{\beta}$. We then have $\psi(x) = \tau(\eta(x),(x,0))$ for each $x \in \Sigma$. Now consider the suspension $\Sigma_{\eta}$ with roof function $\eta$ over the same subshift of finite type $\Sigma$, and consider the suspension flow $h^{t}$ on this space. We define a $C^{\beta}$ homeomorphism $\varphi: \Sigma_{\eta} \rightarrow \Sigma_{\psi}$ by 
\[
\varphi(x,s) = (x,\tau(s,x)).
\]
One may then easily verify that $\varphi \circ h^{t} = g^{t} \circ \varphi$. We conclude that $g^{t}$ is conjugate to a suspension flow $h^{t}$ over the same subshift of finite type $\Sigma$ with the new $C^{\beta}$ roof function $\eta$. Hence all of the conclusions of the thermodynamic formalism also apply to $g^{t}$. 

Let $f^{t}: M \rightarrow M$ be a topologically mixing $C^2$  Anosov flow on a smooth Riemannian manifold $M$; here topologically mixing means that for every pair of nonempty open sets $U$, $V \subseteq M$ there is a $T > 0$ such that for $t \geq T$ we have $f^{t}(U) \cap V \neq \emptyset$. By taking a Markov partition of this flow \cite{Bow73}, we obtain a topologically mixing subshift of finite type $\Sigma$ and a $C^{\alpha}$ roof function $\psi: \Sigma \rightarrow (0,\infty)$ such that there is a surjective $C^{\alpha}$ map $\theta: \Sigma_{\psi} \rightarrow M$ satisfying 
\begin{equation}\label{semiconjugacy}
\theta \circ g^{t} = f^{t} \circ \theta,
\end{equation}
where $g^{t}$ denotes the suspension flow on $\Sigma_{\psi}$. We refer to surjective maps $\theta$ satisfying equation \eqref{semiconjugacy} as \emph{semiconjugacies}. Using $\theta$, one may transfer the conclusions of the thermodynamic formalism above to the Anosov flow $f^{t}$, as in \cite{BR75}; this is done by using the fact that the surjection $\theta$ constructed in \cite{Bow73} is finite-to-one and identifies periodic orbits of $g^{t}$ with periodic orbits of $f^{t}$. In particular if one defines the topological pressure as in \eqref{pressure definition} then the variational principle \eqref{variational principle} for pressure holds, the cohomology relation \eqref{cohomological equation} has the same measurable rigidity property, and two potentials have the same equilibrium state if and only if the relation \eqref{cohomological equation plus} holds. 

If we take a H\"older time change $g^{t}$ of $f^{t}$, the flow $g^{t}$ may not be smooth on $M$. However, by the discussion above, there will be a new H\"older suspension flow $h^{t}: \Sigma_{\eta} \rightarrow \Sigma_{\eta}$ and a new finite-to-one semiconjugacy $\kappa: \Sigma_{\eta} \rightarrow M$ from $h^{t}$ to $g^{t}$. By the arguments above, the same conclusions of the thermodynamic formalism then continue to hold for $g^{t}$ as well. We will refer to a H\"older time change of an Anosov flow as a \emph{H\"older Anosov flow}. 

Suppose that $f^{t}$ is a H\"older Anosov flow. Consider a positive $C^{\alpha}$ function $\zeta: M \rightarrow (0,\infty)$. We define 
\begin{equation}\label{proto horizontal dim}
q_{\zeta} = \sup_{\nu \in \mathcal{M}_{\text{erg}}(f^{t})} \frac{h_{\nu}(f)}{\nu(\zeta)}.
\end{equation}
Observe that $q_{\zeta} < \infty$ because $\zeta$ is bounded away from $0$. Below we write $\mu = \mu_{q_{\zeta}\zeta}$ for the equilibrium state of $q_{\zeta}\zeta$ with respect to $f^{t}$. We let $\hat{f}^{t}$ be the H\"older time change of $f^{t}$ with speed multiplier $\zeta^{-1}$, and let $\hat{\mu} = \mu^{\zeta^{-1}}$ be the $\hat{f}^{t}$-invariant measure corresponding to $\mu$. 

\begin{prop}\label{proto synchro}
The following holds, 
\begin{enumerate}
\item $q_{\zeta}$ is the unique solution to the pressure equation $P(s\zeta) = 0$, $s \in \R$,
\item Equality in \eqref{proto horizontal dim} holds if and only if $\nu = \mu$.
\item $h_{\mathrm{top}}(\hat{f}) = q_{\zeta}$. 
\item $\hat{\mu}$ is the measure of maximal entropy for $\hat{f}^{t}$.
\end{enumerate}
\end{prop}

\begin{proof}
The function $s \rightarrow P(s\zeta)$ is analytic and strictly decreasing in $s$, with $P(s\zeta) \rightarrow -\infty$ as $s \rightarrow \infty$ \cite{BR75}. Since $P(0) = h_{\mathrm{top}}(f) > 0$, there is a unique positive solution $Q$ to the pressure equation $P(q\zeta) = 0$. 

We claim that $q = q_{\zeta}$. Since $P(q\zeta) = 0$, we have
\begin{equation}\label{variational application}
0 = \sup_{\nu \in \mathcal{M}_{\text{erg}}(f^{t})} h_{\nu}(f) - q\nu(\zeta),
\end{equation}
with equality if and only if $\nu$ is the equilibrium state of $q\zeta$ with respect to $f^{t}$. Rearranging and noting that $\zeta > 0$, this implies that
\[
q = \sup_{\nu \in \mathcal{M}_{\text{erg}}(f^{t})} \frac{h_{\nu}(f)}{\nu(\zeta)} = q_{\zeta},
\]
with equality if and only if $\nu$ is the equilibrium state of $q_{\zeta}\zeta = q\zeta$ with respect to $f^{t}$. This proves (1) and (2). 

Consider now the H\"older time change $\hat{f}^{t}$ of $f^{t}$ with speed multiplier $\zeta^{-1}$. From equation \eqref{abramov}, for any $f^{t}$-invariant ergodic probability measure $\nu$ with corresponding $\hat{f}^{t}$-invariant ergodic probability measure $\hat{\nu}$ we have
\begin{equation}\label{second abramov}
h_{\hat{\nu}}(\hat{f}) = \frac{h_{\nu}(f)}{\nu(\zeta)}.
\end{equation}
Hence
\begin{equation}\label{equalities above}
h_{\mathrm{top}}(\hat{f}) = \sup_{\hat{\nu} \in \mathcal{M}_{\text{erg}}(\hat{f}^{t})} h_{\hat{\nu}}(\hat{f}) = \sup_{\nu \in \mathcal{M}_{\text{erg}}(f^{t})} \frac{h_{\nu}(f)}{\nu(\zeta)} = q_{\zeta}.  
\end{equation}
This proves (3). Claim (4) then follows from (2) and \eqref{equalities above}. 
\end{proof}

We now return to the case of a $C^2$ topologically mixing Anosov flow. In general, given a foliation $\W$ of a manifold $M$ and a probability measure $\mu$ on $M$ one may define a disintegration $\{\mu_{x}\}_{x \in M}$ of $\mu$ along $\W$, with each measure $\mu_{x}$ assigning mass only to $\W(x)$. See \cite[Lemma 3.2]{AVW}. If the leaves of $\W$ are noncompact then the measures $\mu_{x}$ are only defined up to scaling by a positive constant. More precisely, if we consider a foliation box $U$  for $\W$, a transversal $T \subset U$ to $\W_{U}$, and the disintegration $\{\mu^{U}_{x}\}_{x \in T}$ of $\mu^{U}=\mu|_{U}$ with respect to the projection onto $T$ along $\W_{U}$, then for each $x \in T$ there is a constant $c_{x} > 0$ such that $c_{x}\mu_{x}|_{\W(x)} = \mu^{U}_{x}$, and the function $x \rightarrow c_{x}$ is measurable. 

\begin{defn}[Local product structure]\label{local product structure}
We say that a probability measure $\mu$ has \emph{local product structure} with respect to $\W$ if in any foliation chart $\psi: U \rightarrow B_{k} \times B_{m-k}$ there is a continuous function $\xi: B_{k} \times B_{m-k} \rightarrow (0,\infty)$ and measures $\nu_{k}$ on $B_{k}$, $\nu_{m-k}$ on $B_{m-k}$ such that 
\begin{equation}\label{product equation}
d(\psi_{*}\mu) = \xi d(\nu_{k} \times \nu_{m-k}).
\end{equation}
\end{defn}

For $f^{t}$ a $C^2$ transitive Anosov flow on $M$ and $\zeta:M \rightarrow \R$ a H\"older potential, the equilibrium state $\mu_{\zeta}$ will have local product structure with respect to all of the invariant foliations $\W^{u}$, $\W^{cu}$, $\W^{c}$, $\W^{cs}$, and $\W^{s}$, and the foliation chart $\psi$ and function $\xi$ in equation \eqref{product equation} can be taken to be $C^{\alpha}$ for each of these foliations for some $\alpha > 0$ \cite[Theorem 2]{H94}.


Lastly, we state the important \emph{Gibbs property} for equilibrium states $\mu= \mu_{\zeta}$ with respect to a $C^{\alpha}$ function $\zeta$. Fix $R > 0$ small enough that each ball $B^{u}(x,R) \subset \W^{u}(x)$ of radius $R$ inside the leaf $\W^{u}(x)$ is contained in a foliation box for $\W^{u}$. We define a measure $\mu_{x}^{u}$ on $B^{u}(x,R)$ by pulling back the disintegration of equation \eqref{product equation} in a foliation chart and restricting to $B^{u}(x,R)$. We normalize these measures such that $\mu_{x}^{u}(B^{u}(x,R)) = 1$. The Gibbs property then states that there is a $T > 0$ such that for each $r \leq R$, each $t \geq T$, and each $x \in M$, 
\begin{equation}\label{Gibbs property}
\mu_{f^{-t}x}^{u}(f^{-t}(B^{u}(x,r))) \asymp e^{-\int_{-t}^{0}\zeta(f^{s}x)\,ds - P(\zeta)t}\mu_{x}^{u}(B^{u}(x,r)).
\end{equation}
See \cite{H94}.

\subsection{Flows with weak expanding foliations}\label{weak expanding} Let $M$ be a closed $C^r$ Riemannian manifold and let $\mathcal{W}^{u}$ be a foliation of $M$ with uniformly $C^r$ leaves, $r \geq 1$.  Let $f: M \rightarrow M$ be a homeomorphism which is uniformly $C^r$ along $\mathcal{W}^{u}$ and preserves $\W^{u}$, i.e., $f(\W^{u}(x)) = \W^{u}(f(x))$ for $x \in M$. 

\begin{defn}[Expanding foliation]We will say that $\mathcal{W}^{u}$ is an \emph{expanding foliation} for $f$ if there are constants $a > 0$ and $c > 0 $ such that for every $x \in M$ and $n \geq 1$ we have 
\begin{equation}\label{expanding foliation}
\mathfrak{m}(Df^{n}_{x}|_{T\mathcal{W}_{x}^{u}}) \geq c e^{a n}. 
\end{equation}
\end{defn}

Note that inequality \eqref{expanding foliation} implies that $Df_{x}|_{T\W^{u}_{x}}$ is invertible for all $x \in M$, and therefore by the inverse function theorem $f^{-1}$ is also uniformly $C^r$ along $\W^{u}$. 

Now assume that $M$ is a closed $C^{r+1}$ Riemannian manifold, $r \geq 1$, and let $\W^{cu}$ be a foliation of $M$ with uniformly $C^{r+1}$ leaves. Let $v_{f} : M \rightarrow T\W^{cu}$ be a nonvanishing vector field on $M$ that is tangent to $\W^{cu}$ and uniformly $C^r$ along $\W^{cu}$. Then $v_{f}$ defines a flow $f^{t}: M \rightarrow M$ which preserves $\W^{cu}$ and is uniformly $C^r$ along $\W^{cu}$. We let $\W^{c}$ denote the orbit foliation of $f^{t}$, which is a $C^r$ subfoliation of $\W^{cu}$ that has uniformly $C^r$ leaves in $M$. We write $E^{cu} = T\W^{cu}$ and $E^{c} = T\W^{c}$. 

By the compactness of $M$ we have  $\|v_{f}(x)\| \asymp 1$ for $x\in M$. We can thus define a new Riemannian structure on $E^{c}$ by setting $v_{f}$ to have unit length. We may thus always assume that $\|v_{f}(x)\| = 1$ for all $x \in M$ and therefore $\|Df^{t}_{x}|_{E^{c}_{x}}\| = 1$ for all $x \in M$ and $t \in \R$. 

\begin{defn}[Weak expanding foliation]\label{weak expanding definition}We say that $\W^{cu}$ is a \emph{weak expanding foliation for $f^{t}$} if there is a $Df^{t}$-invariant codimension-1 subbundle $E^{u} \subset E^{cu}$ transverse to $E^{c}$ such that $E^{u}$ is exponentially expanded by $Df^{t}$: there are constants $c > 0$ and $a > 0$ such that for $t > 0$ and $x \in M$, 
\begin{equation}\label{flow expansion}
\mathfrak{m}(Df^{t}_{x}|_{E^{u}}) \geq ce^{a t}. 
\end{equation}
It follows that the splitting $E^{cu} = E^{u} \oplus E^{c}$ is dominated for the linear cocycle $Df^{t}|_{E^{cu}}$. 
\end{defn} 

\begin{prop}\label{unstable manifold flow}
The bundle $E^{u}$ is uniformly $C^{r}$ along $\W^{cu}$ and is tangent to a foliation $\W^{u}$ of $M$ which is uniformly $C^r$ along $\W^{cu}$. We have $f^{t}(\W^{u}(x)) = \W^{u}(f^{t}x)$ for all $x \in M$ and $t \in \R$. The submanifolds $\W^{u}(x)$ for $x \in M$ are characterized within $\W^{cu}(x)$ by
\[
\W^{u}(x) = \{y \in \W^{cu}(x): d(f^{-t}x,f^{-t}y) \rightarrow 0 \; \text{as $t \rightarrow \infty$}\}.
\] 
\end{prop}

\begin{proof}
The proof is a standard application of the $C^r$ section theorem \cite{HPS} to the time-1 map $f$, which we sketch below. Note that it suffices to prove the proposition for a power $f^{N}$ of $f$, hence by replacing $f$ with $f^{N}$ we may assume that the inequality \eqref{flow expansion} holds at $t = 1$ for some $a > 0$ with $c = 1$. Set $k = \dim E^{u}$. 

To establish the $C^r$ regularity of $E^{u}$ along $\W^{cu}$ we consider continuous sections of the bundle $\mathrm{Gr}_{k}(E^{cu})$ of $k$-planes in $E^{cu}$, which we write as $\Gamma^{0}(\mathrm{Gr}_{k}(E^{cu}))$. We equip this space with the metric 
\[
d(H^{1},H^{2}) = \sup_{x \in M} d(H^{1}_{x},H^{2}_{x}), 
\]
where $d$ denotes the distance \eqref{Hausdorff}. For each $\delta > 0$ we define an open subset of this space by
\[
\Omega_{\delta} = \{H \in \Gamma^{0}(\mathrm{Gr}_{k}(E^{cu})): d(H,E^{u}) \leq \delta\}. 
\]
The linear cocycle $Df|_{E^{cu}}$ induces a graph transform on sections, defined pointwise by
\[
Df_{\#}H_{x} = Df_{f^{-1}(x)}(H_{f^{-1}(x)}),
\]
for $x \in M$. For $\delta$ small enough, a standard calculation (see for instance \cite[Theorem 2.5]{HPS}) shows that the domination inequality \eqref{splitting inequality} for the splitting $E^{cu} = E^{u} \oplus E^{c}$ implies that $Df_{\#}(\Omega_{\delta}) \subset \Omega_{\delta}$ and $Df_{\#}$ is a contraction on $\Omega_{\delta}$. 

The $C^r$ section theorem \cite[Theorem 3.2]{HPS} then supplies a unique continuous $Df_{\#}$-invariant section of $\Omega_{\delta}$, which is $V$. Restricted to each $\W^{cu}$ leaf, $Df$ is $C^r$ and $\mathrm{Gr}_{k}(E^{cu})$ is a $C^r$ bundle. The $C^r$ section theorem implies that the section $V$ is $C^r$ provided that 
\begin{equation}\label{reg inequality}
\sup_{x \in M}\|(Df_{x}|_{E_{x}^{cu}})^{-1}\| \mathrm{Lip}(Df_{\#}|_{\Omega_{\delta}})^{r} < 1.
\end{equation}
Since $\|(Df_{x}|_{E_{x}^{cu}})^{-1}\| = 1$ for all $x \in M$ and $Df_{\#}|_{\Omega_{\delta}}$ is a contraction, this inequality holds and so we conclude that $E^{u}$ is $C^r$ over each $\W^{cu}$ leaf. The fact that $E^{u}$ is uniformly $C^r$ along $\W^{cu}$ follows from the fact that the invariant section of the theorem varies continuously in the $C^r$ topology with $Df$. The argument establishing the existence of the foliation $\W^{u}$ is identical to the standard argument that establishes the existence of unstable manifolds. See \cite[Theorem 4.1]{HPS}.

The equality $f^{t}(\W^{u}(x)) = \W^{u}(f^{t}x)$ for all $t \in \R$ and $x \in M$ follows from the equality $Df^{t}_{x}(E^{u}_{x}) = E^{u}_{f^{t}x}$. Inequality \eqref{flow expansion} clearly implies that $d(f^{-t}x,f^{-t}y) \rightarrow 0$ as $t \rightarrow \infty$ if $y \in \W^{u}(x)$. Conversely, suppose that $x \in M$ and $y \in \W^{cu}(x)$ are such that $d(f^{-t}x,f^{-t}y) \rightarrow 0$ as $t \rightarrow \infty$.  By replacing $x$ and $y$ with $f^{-T}x$ and $f^{-T}y$ for a large enough $T$ we may assume that they belong to a common foliation box $U$ for $\W^{u}$ and $\W^{c}$. There is then a unique $\eta  \in \R$ such that $f^{\eta}y \in \W_{U}^{u}(x)$. As $t \rightarrow \infty$ we have $d(f^{-t}x,f^{-t+\eta}y) \rightarrow 0$ as well. Since
\[
\eta = d(f^{-t}y,f^{-t+\eta}y) \leq d(f^{-t}x,f^{-t+\eta}y) + d(f^{-t}x,f^{-t}y) \rightarrow 0,
\]
we conclude that $\eta = 0$ and therefore $y \in \W^{u}(x)$. 
\end{proof}

Since $E^{u}$ is uniformly $C^r$ along $\W^{cu}$, from this point forward we will assume that we have a Riemannian structure on $E^{cu}$ that is uniformly $C^r$ along $\W^{cu}$, satisfies $\|v_{f}(x)\| = 1$ for all $x \in M$, and is such that $E^{u}$ and $E^{c}$ are orthogonal.

We next show that uniformly $C^r$ time changes of $f^{t}$ along $\W^{cu}$ also have $\W^{cu}$ as a weak expanding foliation. 

\begin{prop}\label{time change splitting}
Let $\gamma: M \rightarrow (0,\infty)$ be uniformly $C^r$ along $\W^{cu}$. Let $\hat{f}^{t}$ be the flow generated by the vector field $\gamma(x)v_{f}(x)$. Then $\hat{f}^{t}$ is uniformly $C^r$ along $\W^{cu}$ and has $\W^{cu}$ as a weak expanding foliation. 
\end{prop}

\begin{proof}
The fact that $\hat{f}^{t}$ is uniformly $C^r$ along $\W^{cu}$ is immediate from the definitions. The proof that $D\hat{f}^{t}|_{E^{cu}}$ has a dominated splitting $E^{cu} = \hat{E}^{u} \oplus E^{c}$ with $\hat{E}^{u}$ being uniformly expanded by $D\hat{f}^{t}$ is standard; it follows in a straightforward manner from the argument that a $C^r$ time change of an Anosov flow is also an Anosov flow. One simply performs the calculations only on the center-unstable manifolds.

To be more specific, write $\hat{f}^{t}x = f^{\tau(t,x)}x$, with $\tau(t,x)$ the additive cocycle over $\hat{f}^{t}$ generated by $\gamma$. We outline the calculation of \cite[Proposition 17.4.5]{HK}. For a fixed $x \in M$ and $t \in \R$ one calculates using the chain rule, in a local coordinate system in which $E^{u}$ and $E^{c}$ are orthogonal, that if the matrix of $Df^{t}$ is given in the $E^{cu} = E^{u} \oplus E^{c}$ frame by 
\[
\left(\begin{array}{cc}
Df^{t}_{x} & 0 \\
0 & 1 
\end{array}\right),
\]
then the matrix of $D\hat{f}^{t}$ in this frame is 
\[
\left(\begin{array}{cc}
Df^{\tau(t,x)}_{x} & \p_{u}\tau(t,x) \\
0 & 1 
\end{array}\right),
\]
with $\p_{u}\tau(t,x)$ the partial derivative of $\tau(t,x)$ along $E^{u}$. By compactness of $M$ and the fact that $\tau(t,x)$ is uniformly $C^1$ along $\W^{cu}$, there is a constant $K > 0$ such that $|\p_{u}\tau(t,x)| \leq Kt$ for $t > 0$. Using this bound, one calculates that there is an $\eta > 0$ such that the cone field $\mathcal{C}_{\eta}(x)$ around $E^{u}_{x}$ is contracted by $D\hat{f}^{N}$ for a large enough $N$, and therefore by Proposition \ref{cone criterion} $D\hat{f}^{t}$ admits a dominated splitting $E^{u} = \hat{E}^{u} \oplus E^{c}$ in which $\hat{E}^{u}$ is exponentially expanded. We refer to \cite{HK} for the details of this calculation, as they are identical to the ones performed in the setting there. 
\end{proof}

An important corollary of the chain rule computation for $D\hat{f}^{t}$ in Proposition \ref{time change splitting} is that, setting $\bar{E}^{u} :=E^{cu}/E^{c}$ and writing $\bar{D}\hat{f}^{t}$ for the induced map of $D\hat{f}^{t}$ on $\bar{E}^{u}$, we have
\begin{equation}\label{chain}
\bar{D}\hat{f}^{t}_{x} = \bar{D}f^{\tau(t,x)}_{x}.
\end{equation}
This in particular shows that time changes have the expected effect on Lyapunov exponents. Recall that if $\hat{f}^{t}$ is a time change of $f^{t}$ with speed multiplier $\gamma$ and $\nu$ is an $f^{t}$-invariant ergodic probability measure, we write $\nu^{\gamma}$ for the corresponding $\hat{f}^{t}$-invariant measure. We write $\vec{\la}^{u}(f,\nu) = \vec{\la}(Df|_{E^{u}},\nu)$ as before. 

\begin{lem}\label{smooth time exponents}
Let $\hat{f}^{t}$ be a time change of $f^{t}$ with speed multiplier $\gamma$ that is uniformly $C^r$ along $\W^{cu}$. Then for any $f^{t}$-invariant ergodic probability measure $\nu$ we have 
\[
\vec{\la}^{u}(\hat{f},\nu^{\gamma}) = \nu(\gamma^{-1})^{-1}\vec{\la}^{u}(f,\nu). 
\]
\end{lem}

\begin{proof}
The cocycle $D\hat{f}^{t}_{x}$ on $\hat{E}^{u}$ is conjugate to $\bar{D}f^{\tau(t,x)}_{x}$ over $\hat{f}^{t}$ by the projection $\hat{E}^{u} \rightarrow \bar{E}^{u}$, by \eqref{chain}. Likewise $\bar{D}f^{t}$ is conjugate to $Df^{t}$ on $E^{u}$. Since $\bar{D}f^{\tau(t,x)}_{x}$ is a time change of $\bar{D}f^{t}_{x}$ with speed multiplier $\gamma$, by Proposition \ref{time change lyap} we have
\begin{align*}
\vec{\la}^{u}(\hat{f},\nu^{\gamma}) &= \vec{\la}(\bar{D}\hat{f},\nu^{\gamma}) \\
&= \nu(\gamma^{-1})^{-1} \vec{\la}(\bar{D}f,\nu) \\
&= \nu(\gamma^{-1})^{-1}\vec{\la}^{u}(f,\nu).
\end{align*}
\end{proof}

We specialize now to the case of a $C^2$ topologically mixing Anosov flow $f^{t}: M \rightarrow M$. We let $\zeta:M \rightarrow (0,\infty)$ be a H\"older potential which is $C^2$ along $\W^{cu}$. We  let $q = q_{\zeta} > 0$ be such that $P(q\zeta) = 0$, as per Proposition \ref{proto synchro}.  Let $\mu$ be the equilibrium state of the potential $q \zeta$ with respect to $f^{t}$. Let $\hat{f}^{t}$ be the time change of $f^{t}$ with speed multiplier $\zeta^{-1}$, and let $\hat{\mu} = \mu^{\zeta^{-1}}$ be the $\hat{f}^{t}$-invariant probability measure equivalent to $\mu$, which is the measure of maximal entropy for $\hat{f}^{t}$ by Proposition \ref{proto synchro}. Let $\tau$ be the additive cocycle generated by $\zeta^{-1}$ over $\hat{f}^{t}$ such that $\hat{f}^{t}x = f^{\tau(t,x)}x$. 

By Proposition \ref{time change splitting} the flow $\hat{f}^{t}$ has $\W^{cu}$ as a weak expanding foliation and has dominated splitting $E^{cu} = \hat{E}^{u} \oplus E^{c}$. By Proposition \ref{unstable manifold flow} there is an $\hat{f}^{t}$-invariant foliation $\hat{\W}^{u}$ tangent to $\hat{E}^{u}$ which is uniformly $C^2$ along $\W^{cu}$. We will show that the disintegration $\{\hat{\mu}_{x}^{u}\}_{x \in M}$ of $\hat{\mu}$ along $\hat{\W}^{u}$ also has the Gibbs property \eqref{Gibbs property}. We let $\hat{B}^{u}(x,r)$ denote the ball of radius $r$ centered at $x$ inside of $\hat{\W}^{u}(x)$, and  fix $R > 0$ small enough that for $r \leq R$ the balls $\hat{B}^{u}(x,r)$ are each contained inside a foliation box for $\hat{\W}^{u}$, and such that the corresponding statement is also true for the Riemannian balls $B^{*}(x,r)$ of radius $r \leq R$ on the leaves of the foliations $\W^{*}$, $* \in \{c,u,cu\}$.  We then normalize the conditional measures such that $\hat{\mu}^{u}_{x}(\hat{B}^{u}(x,R)) = 1$ for each $x \in M$.

\begin{prop}\label{time change Gibbs property}
There is a $T > 0$ such that for all $x \in M$, $t \geq T$, and $r \leq R$ we have
\[
\hat{\mu}_{\hat{f}^{-t}x}^{u}(\hat{f}^{-t}(\hat{B}^{u}(x,r))) \asymp e^{-q t} \hat{\mu}_{x}^{u}(\hat{B}^{u}(x,r)). 
\]
\end{prop}

\begin{proof}
For $x \in M$, $0 < r \leq R$, and $\e > 0$ we set 
\[
B^{cu}_{\e}(x,r) = \bigsqcup_{\eta \in [-\e,\e]}f^{\eta}(B^{u}(x,r)).
\] 
We note that for each $y \in B^{u}(x,r)$ we have $B^{cu}_{\e}(x,r) \cap \W^{c}_{B^{cu}_{\e}(x,r)}(y) = B^{c}(y,\e)$ since we normalized our Riemannian metric such that the generator $v_{f}$ of $f$ has unit norm. We also note that the boxes $B^{cu}_{r}(x,r)$ are uniformly comparable to the Riemannian balls $B^{cu}(x,r) \subset \W^{cu}(x)$ in the sense that there is a constant $C \geq 1$ independent of $x$ and $r$ such that 
\begin{equation}\label{uniform comp}
B^{cu}(x,C^{-1}r) \subset B^{cu}_{r}(x,r) \subset B^{cu}(x,Cr).
\end{equation}
We will use the boxes $B^{cu}_{r}(x,r)$ instead of the Riemannian balls in this proof. By shrinking $R$ if necessary we can assume that each box $B^{cu}_{r}(x,r)$ with $r \leq R$ is also contained inside foliation charts for each of the foliations $\hat{\W}^{u}$, $\W^{u}$, $\W^{c}$, and $\W^{cu}$ under consideration.  

Let $\{\mu^{cu}_{x}\}_{x \in M}$ denote the disintegration of $\mu$ along $\W^{cu}$, which we normalize such that $\mu^{cu}_{x}(B^{cu}_{R}(x,R)) = 1$. Similarly we let $\{\mu^{u}_{x}\}_{x \in M}$ denote the disintegration of $\mu$ along $\W^{u}$, normalized such that $\mu^{u}_{x}(B^{u}(x,R)) = 1$. Lastly we let $\{\mu_{x}^{c}\}$ denote the disintegration of $\mu$ along the orbit foliation $\W^{c}$, which we can take to be the Riemannian arc length on each orbit. The local product structure of $\mu$ with respect to the foliations $\W^{u}$, $\W^{c}$, and $\W^{cu}$ (see Definition \ref{local product structure}) implies that for each $x \in M$ we can find a continuous function $\xi_{x}: B^{cu}_{R}(x,R) \rightarrow (0,\infty)$ such that 
\[
d\mu^{cu}_{x}|_{B^{cu}_{R}(x,R)} = \xi_{x}(d\mu^{u}_{x}|_{B^{u}(x,R)} \times d\mu^{c}_{x}|_{B^{c}(x,R)}),
\]
where the right side should be understood as corresponding to expressing $B^{cu}_{R}(x,R)$ as the product $B^{u}(x,R) \times B^{c}(x,R)$ by projecting along local $\W^{c}$ leaves onto $B^{u}(x,R)$ and projecting along local $\W^{u}$ leaves onto $B^{c}(x,R)$. Since we can cover $M$ by foliation boxes on which the functions $\xi_{x}$ are uniformly above and below away from zero, the compactness of $M$ implies that the functions $\xi_{x}$ are uniformly bounded above and below away from zero, independent of $x$. We thus have
\begin{equation}\label{local product use}
d\mu^{cu}_{x}|_{B^{cu}_{R}(x,R)} \asymp d\mu^{u}_{x}|_{B^{u}(x,R)} \times d\mu^{c}_{x}|_{B^{c}(x,R)},
\end{equation}
with the implied multiplicative constant being independent of $x$. 

Set $U_{x}:=B^{cu}_{R}(x,R)$. For each $y \in \W^{c}_{U_{x}}(x)$, we let $\theta_{y,x}: \W^{u}_{U_{x}}(y) \rightarrow \hat{\W}^{u}(y)$ denote the projection along the flowlines of $f^{t}$. We let $\Theta_{x}: U_{x} \rightarrow \W^{cu}(x)$ denote the resulting map on $U_{x}$, defined by taking its restriction to each slice $\W^{u}_{U_{x}}(y)$ to be $\theta_{y,x}$. The map $\Theta_{x}$ is a bi-Lipschitz diffeomorphism onto its image with bi-Lipschitz constant independent of $x$, and fixes the $\W^{c}$ foliation in the sense that $\Theta_{x}(\W^{c}_{U_{x}}(z)) \subset \W^{c}(z)$ for each $z \in U_{x}$.

We similarly define
\[
\hat{B}^{cu}_{\e}(x,r) = \bigsqcup_{\eta \in [-\e,\e]}\hat{B}^{u}(f^{\eta}x,r).
\]
For $t > 0$ we define $g^{-t}$ on $B^{cu}_{\e}(x,r)$ to be the map which restricts to $f^{\tau(-t,f^{\eta}x)}$ on $B^{u}(f^{\eta}x,r)$. Let $\{\nu_{x}\}_{x \in M}$ denote the conditionals of $\mu$ along $\hat{\W}^{u}$, normalized such that $\nu_{x}(\hat{B}^{u}(x,R)) = 1$. We observe that the conditionals of $\mu$ along the orbit foliation $\W^{c}$ are simply given by the Riemannian arc length on each orbit, since we assumed that the generator $v_{f}$ of $f$ has unit norm. We then have for any $r \leq R$, 
\[
\mu_{x}^{cu}(B^{cu}_{\e}(x,r)) \asymp \int_{-\e}^{\e}\mu^{u}_{x}(B^{u}(f^{\eta}x,r))\,d\eta,
\]
and 
\[
\mu_{x}^{cu}(\hat{B}^{cu}_{\e}(x,r)) \asymp \int_{-\e}^{\e}\nu^{u}_{x}(\hat{B}^{u}(f^{\eta}x,r))\,d\eta,
\]
where $d\eta$ here stands for the arc length measure on $\W^{c}$.

The disintegration of $\mu^{cu}_{x}$ along $\W^{c}$ is invariant under $f^{t}$, and is therefore given by a constant multiple of the Riemannian arc length on each leaf of $\W^{c}$. Since $\Theta_{x}$ fixes $\W^{c}$, it follows that $(\Theta_{x})_{*}\mu_{x}^{cu}|_{U_{x}}$ is equivalent to $\mu^{cu}_{x}$ on $\Theta_{x}(U_{x})$, with Radon-Nikodym derivative uniformly comparable to $1$, independent of $x$. 

For $y \in U_{x}$ let $h_{x}(y) \in \W^{c}_{U_{x}}(x)$ denote the unique intersection point of $\hat{\W}^{u}(y)$ with $\W^{c}_{U_{x}}(x)$ inside of $U_{x}$. On $U_{x}$ we then have for any $y \in U_{x}$ and $t > 0$,
\begin{equation}\label{graph time change}
\hat{f}^{-t}y = (\Theta_{\hat{f}^{-t}x} \circ f^{\tau(-t,h_{x}(y))} \circ \Theta_{x}^{-1})(y) = (\Theta_{\hat{f}^{-t}x} \circ g^{-t} \circ \Theta_{x}^{-1})(y)
\end{equation}
As remarked before, the fact that $\Theta_{x}$ is bi-Lipschitz onto its image and absolutely continuous with respect to $\mu^{cu}_{x}$ implies, when combined with equation \eqref{graph time change}, that for all $t > 0$, $\e < \frac{R}{2}$, and $r \leq R$, 
\begin{equation}\label{transfer measure}
\mu^{cu}_{x}(\hat{f}^{-t}(\hat{B}^{cu}_{\e}(x,r))) \asymp \mu^{cu}_{x}(g^{-t}(B^{cu}_{\e}(x,r))),
\end{equation}
with the implied multiplicative constant being independent of $x$. By the Gibbs property \eqref{Gibbs property}, there is a $T > 0$ such that for $t \geq T$ and $y \in \W^{c}_{U_{x}}(x)$, 
\begin{align*}
\mu^{u}_{f^{\tau(-t,y)}y}(f^{\tau(-t,y)}(B^{u}(y,r))) &\asymp e^{-q\int_{\tau(-t,y)}^{0}\zeta(f^{s}y)\,ds }\mu_{y}^{u}(B^{u}(y,r)) \\
&= e^{-qt}\mu_{y}^{u}(B^{u}(y,r)).
\end{align*}
Hence for $t \geq T$,
\[
\mu^{cu}_{x}(g^{-t}(B^{cu}_{\e}(x,r))) \asymp e^{-qt}\mu^{cu}_{x}(B^{cu}_{\e}(x,r)).
\]
We conclude that
\begin{align*}
\mu^{cu}_{x}(\hat{f}^{-t}(\hat{B}^{cu}_{\e}(x,r))) &\asymp e^{-qt}\mu^{cu}_{x}(B^{cu}_{\e}(x,r)) \\
&\asymp e^{-qt}\mu^{cu}_{x}(\hat{B}^{cu}_{\e}(x,r)) \\
&\asymp \e e^{-qt} \nu_{x}(\hat{B}^{cu}(x,r)),
\end{align*}
using again the fact that the conditional measures of $\mu^{cu}_{x}$ on $\W^{c}$ are multiples of the arc length. On the left-hand side we note that we have a bound $\zeta^{-1} \asymp 1$ on the speed multiplier for $\hat{f}^{-t}$, from which it follows that when restricted to any leaf $\W^{c}(x)$ the time changed flow $\hat{f}^{-t}$ is bi-Lipschitz, with Lipschitz constant independent of $t$. We thus have
\[
\mu^{cu}_{x}(\hat{f}^{-t}(\hat{B}^{u}_{\e}(x,r))) \asymp \e \nu_{\hat{f}^{-t}x}(\hat{f}^{-t}(\hat{B}^{u}(x,r))),
\]
as well. 

We conclude that
\begin{equation}\label{proto Gibbs}
 \nu_{\hat{f}^{-t}x}(\hat{f}^{-t}(\hat{B}^{u}(x,r))) \asymp e^{-qt} \nu_{x}(\hat{B}^{u}(x,r)).
\end{equation}
Since the $\hat{f}^{t}$-invariant measure $\hat{\mu}$ is equivalent to $\mu$ with Radon-Nikodym derivative $\frac{d\hat{\mu}}{d\mu} \asymp 1$ \eqref{radon}, this property holds for the disintegrations of $\hat{\mu}$ as well. Hence $\frac{d\hat{\mu}_{x}^{u}}{d\nu_{x}}\asymp 1$ for each $x \in M$, with implied constant independent of $x$. The proposition then follows from \eqref{proto Gibbs}.

\end{proof}

\subsection{The horizontal measure}\label{subsec:hormeasure} 
In this section we consider an Anosov flow $f^{t}$ which is obtained as a time change of another Anosov flow $h^{t}$ with speed multiplier $\gamma$ such that $\gamma$ is $C^r$ along $\W^{cu,h} = \W^{cu,f}$. We refer to such a flow as a $u$-$C^r$ Anosov flow. In the case that $r = \infty$ we will refer to such a flow as a $u$-smooth Anosov flow. As usual, we will always assume that our flows are transitive. We assume as well that $r \geq 3$. 

Let $k = \dim E^{u}$, $k \geq 2$, and let $1 \leq l \leq k-1$. 

\begin{defn}[$u$-splitting] A \emph{$u$-splitting of index $l$} for $f^{t}$ is a dominated splitting $E^{u} = L^{u} \oplus V^{u}$ of index $l$ for the linear  cocycle $Df^{t}|_{E^{u}}$ over $f^{t}$. We will sometimes refer to $L^{u}$ as the \emph{horizontal bundle} and $V^{u}$ as the \emph{vertical bundle}. Analogously, if instead $f^{t}$ is $C^r$ along $\W^{cs}$, then for $k = \dim E^{s}$ and $1 \leq l \leq k-1$ we define an \emph{$s$-splitting of index $l$} for an Anosov flow $f^{t}$ to be a dominated splitting $E^{s} = L^{s} \oplus V^{s}$ of index $l$ for the linear cocycle $Df^{t}|_{E^{s}}$ over $f^{t}$. 
\end{defn}

An $s$-splitting for $f^{t}$ may equivalently be thought of as a $u$-splitting for the inverse flow $f^{-t}$. 

We will need the following regularity result for dominated splittings associated to an expanding foliation, which we will apply in particular to the case of $u$-splittings. 

\begin{prop}\label{prop: splitting reg}
Let $M$ be a closed $C^{r}$ Riemannian manifold, $r \geq 2$, let $\W^{u}$ be a foliation of $M$ with uniformly $C^{r}$ leaves, and let $f^{t} :M \rightarrow M$ be uniformly $C^{r}$ along $\W^{u}$ and have $\W^{u}$ as an expanding foliation. 

Suppose that $Df|_{E^{u}}$ has a dominated splitting $E^{u} = L^{u} \oplus V^{u}$. Then $V^{u}$ is uniformly $C^{r-1}$ along $\W^{u}$. Furthermore $V^{u}$ is uniquely integrable and tangent to a foliation $\mathcal{V}^{u}$ of $M$ which is uniformly $C^{r-1}$ along $\W^{u}$.
\end{prop}

The proof of this proposition is standard, and is identical to the proof of Proposition \ref{unstable manifold flow}, treating $V^{u}$ as $E^{u}$ and $L^{u}$ as $E^{c}$. The domination inequality \ref{splitting inequality} will give the contraction in the graph transform. See \cite[Theorem 5.1]{HPS}. 

Recall that $Q_{x}^{u} = E^{u}_{x}/V^{u}_{x}$ is identified with  $(V^{u}_{x})^{\perp}$ for the purposes of its Riemannian structure. We will write $Df^{t}_{x}|_{Q^{u}_{x}}$ for the induced action of $Df^{t}_{x}$ on $Q_{x}^{u}$. We define a continuous function $\zeta_{f}: M \rightarrow (0,\infty)$  by
\begin{equation}\label{potential}
\zeta_{f}(x) = \left.\frac{d}{dt}\right|_{t=0} \log \text{Jac}\left(Df^{t}_{x}|_{Q^{u}_{x}}\right),
\end{equation}
for $x \in M$. By Proposition \ref{prop: splitting reg} we have that $V^{u}$ is uniformly $C^{r-1}$ along $\W^{cu}$. It follows that $Q^{u}$ is uniformly $C^{r-1}$ along $\W^{cu}$ and that $Df^{t}|_{Q^{u}}$ is uniformly $C^{r-1}$ along $\W^{cu}$. 

\begin{prop}
The function $\zeta_{f}$ is H\"older continuous on $M$ and uniformly $C^{r-1}$ along $\W^{cu}$. 
\end{prop}

\begin{proof}
By equation \eqref{exterior} we may rewrite $\zeta_{f}(x)$ as
\[
\zeta_{f}(x) =  \left.\frac{d}{dt}\right|_{t=0} \log \|\wedge^{l}Df^{t}_{x}|_{\wedge^{l}Q^{u}_{x}}\|.
\]
By Proposition \ref{holder splitting}, the subbundle $V^{u}$ is $C^{\beta}$ for some $0 < \beta \leq 1^{-}$. Hence the bundle $Q^{u} \cong (V^{u})^{\perp}$ over $M$ is $C^{\beta}$. Therefore the line bundle $\wedge^{l}Q^{u}$ is $C^{\beta}$. The flow $f^{t}$ induces a flow $\wedge^{l}Df^{t}$ on $\wedge^{l}Q^{u}$ which is uniformly $C^{r-1}$ along $\W^{cu}$. Identifying $\wedge^{l}Q^{u}$ with $\wedge^{l}(V^{u})^{\perp}$, the induced flow on $\wedge^{l}(V^{u})^{\perp}$ has a generator $\omega: \wedge^{l}TM \rightarrow T(\wedge^{l}TM )$ which is uniformly $C^{r-1}$ along the restriction of $\wedge^{l}E^{cu}$ to each leaf of $\W^{cu}$.  

Let $w_{x}$ denote either choice of unit vector spanning $\wedge^{l}(V^{u}_{x})^{\perp}$; in a neighborhood of any fixed $p \in M$ we can choose $w_{x}$ in a manner which is $C^{\beta}$ on $M$ and uniformly $C^{r-1}$ along $\W^{cu}$. We then have
\[
\zeta_{f}(x) = \log \|\omega(w_{x})\|,
\]
from which we conclude that $\zeta_{f}$ is $C^{\beta}$ and uniformly $C^{r-1}$ along $\W^{cu}$.

\end{proof}

We can therefore apply the thermodynamic formalism of Section \ref{thermo section} to $\zeta_{f}$ over $f^{t}$. Let $\nu$ be an $f^{t}$-invariant ergodic probability measure. By Fubini's theorem and the $f^{t}$-invariance of $\nu$, for each $t > 0$, 
\begin{align*}
\int_{M} \zeta_{f} \, d\nu &= \frac{1}{t}\int_{0}^{t}\int_{M} \zeta_{f}(f^{s}x)\, d\nu(x) \, ds \\
&= \int_{M} \frac{1}{t}\int_{0}^{t} \zeta_{f}(f^{s}x)\, ds \, d\nu(x) \\
&= \int_{M} \frac{1}{t} \log \mathrm{Jac}(Df^{t}|_{Q^{u}_{x}})\, d\nu(x). 
\end{align*}
Since $Df^{t}|_{Q^{u}}$ is conjugate to $Df^{t}|_{L^{u}}$ by the projection $L^{u} \rightarrow Q^{u}$, as $t \rightarrow \infty$ we have for $\nu$-a.e.~ $x \in M$,
\begin{align*}
\lim_{t \rightarrow \infty}  \frac{1}{t} \log \mathrm{Jac}(Df^{t}|_{Q^{u}_{x}}) &= \lim_{t \rightarrow \infty} \sum_{i=1}^{l}\frac{1}{t}\log \sigma_{i}(Df^{t}|_{L^{u}_{x}}) \\
&= \sum_{i=1}^{l} \la_{i}(Df|_{L^{u}},\nu) \\
&=\sum_{i=1}^{l} \la_{i}^{u}(f,\nu),
\end{align*}
with the final line following by applying Lemma \ref{domination exponent gap} to the dominated splitting $E^{u} = L^{u} \oplus V^{u}$. Hence by the dominated convergence theorem, 
\begin{equation}\label{convergence theorem}
\int_{M} \zeta_{f} \, d\nu = \sum_{i=1}^{l} \la_{i}^{u}(f,\nu).
\end{equation}

Setting $q_{f}:=q_{\zeta_{f}}$, equation \eqref{proto horizontal dim} then reads 
\begin{equation}\label{horizontal dim equation}
q_{f} = \sup_{\nu \in \mathcal{M}_{\text{erg}}(f^{t})} \frac{h_{\nu}(f)}{\sum_{i=1}^{l}\la^{u}_{i}(f,\nu)}.
\end{equation}

\begin{defn}[Horizontal dimension and measure]\label{defn horizontal dimension}
The \emph{horizontal dimension} of $f^{t}$ is the number $q_{f}$. The \emph{horizontal measure} for $f^{t}$ is the equilibrium state $\mu_{f}$ for $q_{f}\zeta_{f}$ with respect to $f^{t}$.
\end{defn}

The horizontal dimension and horizontal measure are invariants of conjugacies that are $C^1$ along $\W^{cu}$.

\begin{prop}\label{invariance hor measure}
Let $f^{t}:M \rightarrow M$, $g^{t}:N \rightarrow N$ be $u$-$C^2$ Anosov flows such that $\dim E^{u,f} = \dim E^{u,g} = k$ and such that there are $u$-splittings $E^{u,f} = L^{u,f} \oplus V^{u,f}$ and  $E^{u,g} = L^{u,g} \oplus V^{u,g}$ of index $l$, $1 \leq l \leq k-1$. Suppose that there is a  conjugacy $\varphi: M \rightarrow N$ from $f^{t}$ to $g^{t}$ that is uniformly $C^1$ along $\W^{cu,f}$. Then $q_{f} = q_{g}$ and $\varphi_{*}\mu_{f} = \mu_{g}$. 
\end{prop}

\begin{proof}
We first show that $q_{f} = q_{g}$. The ergodic invariant probability measures of $f^{t}$ and $g^{t}$ are in one-to-one correspondence: any invariant measure for $g^{t}$ can be written as $\varphi_{*}\nu$ for some invariant measure $\nu$ of $f^{t}$. Since $\varphi$ is a conjugacy, $h_{\varphi_{*}\nu}(g) = h_{\nu}(f)$. 

Since $\varphi$ is uniformly $C^1$ along $\W^{cu,f}$, $D\varphi:E^{cu,f} \rightarrow E^{cu,g}$ restricts on $E^{u,f}$ to a conjugacy from $Df^{t}|_{E^{u,f}}$ to $Dg^{t}|_{E^{u,g}}$. It follows that for any ergodic invariant probability measure $\nu$ for $f^{t}$ we have 
\[
\la_{i}^{u}(g,\varphi_{*}\nu)= \la_{i}^{u}(f,\nu),
\]
for $1 \leq i \leq k$. It follows that 
\[
q_{f} = \sup_{\nu \in \mathcal{M}_{\text{erg}}(f^{t})} \frac{h_{\nu}(f)}{\sum_{i=1}^{l}\la^{u}_{i}(f,\nu)} = \sup_{\nu \in \mathcal{M}_{\text{erg}}(f^{t})} \frac{h_{\varphi_{*}\nu}(g)}{\sum_{i=1}^{l}\la^{u}_{i}(g,\varphi_{*}\nu)} = q_{g}. 
\]

It remains to show that $\varphi_{*}\mu_{f} = \mu_{g}$. Since $\varphi$ is a homeomorphism it suffices to prove the equivalent statement that $\varphi^{-1}_{*}\mu_{g} = \mu_{f}$. We observe that $\varphi_{*}^{-1}\mu_{g}$ is the equilibrium state of the potential $q_{g}\zeta_{g} \circ \varphi$ with respect to the flow $f^{t}$. We will show that $q_{g}\zeta_{g} \circ \varphi$ is cohomologous to $q_{f} \zeta_{f}$ over the flow $f^{t}$, from which it follows that  $\varphi_{*}^{-1}\mu_{g} = \mu_{f}$.

The conjugacy $D\varphi$ maps the $u$-splitting $E^{u,f} = L^{u,f} \oplus V^{u,f}$ to the $u$-splitting $E^{u,g} = L^{u,g} \oplus V^{u,g}$. Hence $D\varphi$ restricts to a conjugacy from $Df^{t}|_{L^{u,f}}$ to $Dg^{t}|_{L^{u,g}}$. By taking Jacobians of the conjugacy equation 
\[
D\varphi|_{L^{u,f}} \circ Df^{t}|_{L^{u,f}} = Dg^{t}|_{L^{u,g}} \circ D\varphi|_{L^{u,f}},
\]
we conclude that for each $x \in M$ and $t \in \R$, 
\[
\mathrm{Jac}(D\varphi_{f^{t}_{x}}|_{L^{u,f}_{f^{t}_{x}}}) \mathrm{Jac}(Df^{t}_{x}|_{L^{u,f}_{x}}) = \mathrm{Jac}(Dg^{t}_{\varphi(x)}|_{L^{u,g}_{\varphi(x)}})\mathrm{Jac}( D\varphi_{x}|_{L^{u,f}_{x}}).
\]
Set $\xi(x):= \mathrm{Jac}(D\varphi_{x}|_{L^{u,f}_{x}})$. After taking logarithms and rearranging, we obtain for all $t \in \R$ and $x \in M$, 
\[
\xi(f^{t}x)-\xi(x) = \int_{0}^{t}\zeta_{g}(g^{s}\varphi(x))\, ds - \int_{0}^{t}\zeta_{f}(f^{s}x)\, ds. 
\]
Multiplying through by $q_{g} = q_{f}$, we conclude that $q_{g}\zeta_{g}\circ \varphi$ is cohomologous to $q_{f}\zeta_{f}$, as desired.
\end{proof}

The horizontal dimension is also invariant under time change, and the measure class of the horizontal measure is invariant as well. This is the justification for referring to $q_{f}$ as a dimensional quantity.

\begin{prop}\label{time change invariance}
Let $f^{t}$ be a $u$-$C^2$ Anosov flow with $u$-splitting $E^{u,f} = L^{u,f} \oplus V^{u,f}$, and let $g^{t}$ be a time change of $f^{t}$ with $u$-splitting $E^{u,g} = L^{u,g} \oplus V^{u,g}$ and speed multiplier $\gamma$ that is uniformly $C^2$ along $\W^{cu}$. Then $q_{f} = q_{g}$ and $\mu_{g} = \mu_{f}^{\gamma}$. 
\end{prop}

\begin{proof}
By Abramov's formula \eqref{abramov} and Proposition \ref{smooth time exponents} we have for any $f^{t}$-invariant measure $\nu$ and corresponding $\hat{f}^{t}$-invariant measure $\nu^{\gamma}$,
\[
\frac{h_{\nu^{\gamma}}(g)}{\sum_{i=1}^{l} \la_{i}^{u}(g,\nu^{\gamma})} = \frac{\nu(\gamma^{-1})^{-1}h_{\nu}(f)}{\nu(\gamma^{-1})^{-1}\sum_{i=1}^{l} \la_{i}^{u}(f,\nu)} = \frac{h_{\nu}(f)}{\sum_{i=1}^{l} \la_{i}^{u}(f,\nu)}.
\]
The conclusion $q_{g} = q_{f}$ immediately follows. The above also shows that the supremum in \eqref{horizontal dim equation} is achieved by $\mu_{f}^{\gamma}$, hence $\mu_{g} = \mu_{f}^{\gamma}$. 
\end{proof}

We now introduce another important equilibrium state for Anosov flows.

\begin{defn}[SRB measure]\label{SRB}
Let $f^{t}$ be a transitive $u$-$C^2$ Anosov flow. The \emph{SRB (Sinai-Ruelle-Bowen) measure} for $f^{t}$ is the equilibrium state $m_{f}$ of the potential 
\[
J_{f}(x) = \left.\frac{d}{dt}\right|_{t=0} \log \text{Jac}\left(Df^{t}_{x}|_{E^{u}_{x}}\right),
\]
with respect to $f^{t}$.
\end{defn}

The SRB measure for $f^{t}$ is characterized by the property that its conditional measures on the leaves of $\W^{u}$ are equivalent to the Riemannian volume on these leaves \cite{Led84}. 

The following lemma will be fundamental to several of our proofs in the subsequent sections.  

\begin{lem}\label{hor sup}
Let $f^{t}$ be a $u$-$C^2$ Anosov flow with $u$-splitting $E^{u} = L^{u} \oplus V^{u}$. Set $k = \dim E^{u}$, $l = \dim L^{u}$. Suppose that $\la_{k}^{u}(f,\mu_{f}) \leq 2\la_{1}^{u}(f,\mu_{f})$ and that
\begin{equation}\label{dim inequality}
q_{f} \geq \frac{l + 2(k-l)}{l}.
\end{equation}
Then equality holds in \eqref{dim inequality}, $\mu_{f}$ is the SRB measure for $f^{t}$, and there is a constant $c > 0$ such that $\vec{\la}^{u}_{i}(f,\mu_{f}) = c$ for $1 \leq i \leq l$ and $ \la^{u}_{i}(f,\mu_{f}) = 2c$ for $l+1 \leq i \leq k$. 
\end{lem}

\begin{proof}
We will use Ruelle's inequality \cite{R78},
\begin{equation}\label{ruelle}
h_{\mu_{f}}(f) \leq \sum_{i=1}^{k}\la_{i}(f,\mu_{f}),
\end{equation}
with equality if and only if the conditionals of $\mu_{f}$ on the leaves of $\W^{u}$ are equivalent to the Riemannian volume \cite{Led84}.

By equation \eqref{horizontal dim equation} and \eqref{ruelle}, we have
\[
q_{f} \leq \frac{\sum_{i=1}^{k}\la_{i}(f,\mu_{f})}{\sum_{i=1}^{l}\la_{i}^{u}(f,\mu_{f})}.
\]
Set $c = \la_{1}^{u}(f,\mu_{f})$. By the assumed inequality \eqref{dim inequality}, Ruelle's inequality, and the assumed inequality $\la_{k}^{u}(f,\mu_{f}) \leq 2c$, we conclude that 
\begin{align*}
\frac{l + 2(k-l)}{l}  &\leq q_{f} \\
&\leq \frac{\sum_{i=1}^{k}\la_{i}(f,\mu_{f})}{\sum_{i=1}^{l}\la_{i}^{u}(f,\mu_{f})} \\
&\leq 1 + \frac{2(k-l)c}{\sum_{i=1}^{l}\la_{i}^{u}(f,\mu_{f})} \\
&\leq 1 + \frac{2(k-l)c}{lc} \\
&= \frac{l + 2(k-l)}{l}.
\end{align*}
Thus all of the inequalities above are actually equalities. We first conclude that equality holds in \eqref{dim inequality} and that equality holds in Ruelle's inequality \eqref{ruelle}, from which it follows that the unstable conditionals of $\mu_{f}$ are equivalent to the Riemannian volume. This implies that $\mu_{f}$ is the SRB measure for $f^{t}$. Since the unstable Lyapunov exponents $\la_{i}^{u}(f,\mu_{f})$ are positive and increasing in $i$, equality in the third inequality above holds if and only if $\la_{i}^{u}(f,\mu_{f}) = 2c$ for all $l+1 \leq i \leq k$, and equality in the fourth inequality holds if and only if $\la_{i}^{u}(f,\mu_{f}) =c$ for all $1\leq i \leq l$.
\end{proof}

In the special case that $f^{t} = g^{t}_{Y}$ is the geodesic flow of a closed negatively curved Riemannian manifold $Y$ with a $u$-splitting $E^{u} = L^{u} \oplus V^{u}$, we write $q_{Y} = q_{g_{Y}}$ and $\mu_{Y} = \mu_{g_{Y}}$ and refer to $q_{Y}$ and $\mu_{Y}$ as the horizontal dimension and horizontal measure of $Y$ respectively. The flip map $v \rightarrow - v$ on $T^{1}Y$ is a smooth conjugacy from $g^{t}_{Y}$ to $g^{-t}_{Y}$, hence by Proposition \ref{invariance hor measure} we have $q_{g_{Y}} = q_{g_{Y}^{-1}}$ and $\mu_{g_{Y}} = \mu_{g_{Y}^{-1}}$. We conclude that the horizontal dimension and horizontal measure for $g^{t}_{Y}$ and $g^{-t}_{Y}$ coincide. 

\subsection{Negatively curved spaces}\label{subsec:neg curved}
In this section we review the properties of negatively curved Riemannian manifolds that we will require in this paper, in particular the structure of the negatively curved symmetric spaces. References for the general properties of negatively curved Riemannian manifolds discussed below are \cite{BGS} and  \cite{HH}. References for the discussion on the structure of the symmetric spaces of nonconstant negative curvature and the Carnot groups that model their horospheres are \cite{Pan85} and \cite{Bou18}.  

Let $Y$ be a simply connected complete Riemannian manifold. Let $\pi:T^{1} Y \rightarrow Y$ denote the projection. We write $g^{t}:=g^{t}_{Y}$ for the geodesic flow of $Y$ on $T^{1}Y$. We say that $Y$ has \emph{pinched negative curvature} if there are $b \geq a > 0$ such that the sectional curvatures of $Y$ satisfy $-b^{2} \leq K \leq -a^{2}$. A unit vector $v \in T^{1}Y$ defines a Busemann function $B_{v}: Y \rightarrow \R$ as follows: let $\gamma_{v}$ denote the unique geodesic with $\gamma_{v}'(0) = v$.  Then for $y \in Y$,
\[
B_{v}(y) = \lim_{t \rightarrow\infty} d(y,\gamma_{v}(t)) - t.
\]
The function $B_{v}$ is $C^2$ \cite{HH}. Its level sets in $Y$ are $C^2$ submanifolds that are known as \emph{horospheres}. We set $\mathcal{H}^{s}(v) :=B_{v}^{-1}(0)$ and refer to this as the \emph{stable horosphere} through $v$. We set $\mathcal{H}^{u}(v) = \mathcal{H}^{s}(-v)$ and refer to this as the \emph{unstable horosphere} through $v$. 

The gradient vector field $\nabla B_{v}:Y \rightarrow T^{1}Y$ is a $C^1$ unit vector field on $Y$ that satisfies $\nabla B_{v}(\pi(v)) = -v$. We set $R_{v} = -\nabla B_{v}$ and refer to $R_{v}$ as the \emph{radial vector field} in the direction of $v$. Note that $R_{v}(\pi(v)) = v$. We write $\W^{c}(v) = \bigcup_{t \in \R}g^{t}(v)$ for the $g^{t}$-orbit of $v$. We set $\W^{s}(v) = R_{v}(\mathcal{H}^{s}(v))$ and set $\W^{cs}(v) = R_{v}(Y) \subset T^{1}Y$. The integral curves of $R_{v}$ are unit speed geodesics. Thus this vector field generates a $C^1$ flow $f^{t}_{v}: Y \rightarrow Y$ such that for $w \in \W^{cs}(v)$ we have $(\pi \circ g^{t})(w) = (f^{t}_{v} \circ \pi)(w)$. Hence we can write $\W^{cs}(v) = \bigcup_{t \in \R} g^{t}(\W^{s}(v))$.  Using the flip map $v \rightarrow -v$ on $T^{1}Y$, we define $\W^{u}(v) = -\W^{s}(-v)$ and $\W^{cu}(v) = -\W^{cs}(-v)$. Then $\pi(\W^{u}(v)) = \mathcal{H}^{u}(v)$. 

The radial flow $R_{v}$ exponentially contracts the horospheres $\mathcal{H}^{s}(g^{t}v)$, $t \in \R$ \cite{HH}. From this we obtain that there is a $Dg^{t}$-invariant splitting $T^{1}Y = E^{u} \oplus E^{c} \oplus E^{s}$, with $E^{u}$ tangent to $\W^{u}$, $E^{c}$ tangent to $\W^{c}$, and  $E^{s}$ tangent to $\W^{s}$, such that $Dg^{t}$ exponentially expands $E^{u}$ and exponentially contracts $E^{s}$. Hence this splitting functions as an Anosov splitting for $g^{t}$ and we may treat $g^{t}$ as an Anosov flow on the noncompact Riemannian manifold $T^{1}Y$ with invariant foliations $\W^{*}$ described above. 

We write $\p Y$ for the boundary at infinity of $Y$, defined as equivalence classes $[\gamma]$ of geodesic rays $\gamma: [0,\infty) \rightarrow Y$, where $\gamma \sim \beta$ if there is a constant $C > 0$ such that $d(\gamma(t),\beta(t)) \leq C$ for all $t \geq 0$. The space $\p Y$ is homeomorphic to an $(n-1)$-sphere if $\dim Y = n$.   A vector $v \in T^{1}Y$ uniquely determines a point $\theta_{+}(v):=[\gamma_{v}]$ in $\p Y$. We also write $\theta_{-}(v) = [\gamma_{-v}]$. Conversely, given a point $x \in Y$ and two points $y \neq z \in \p Y$, there is a unique $v \in T^{1}Y$ such that $v \in T^{1}Y_{x}$, $\theta_{+}(v) = y$, and $\theta_{-}(v) = z$. 

We view the unstable horospheres of $Y$ as coordinate charts for $\p Y$. Given $v \in T^{1}Y$, we have a homeomorphism $P_{v}^{\mathcal{H}}: \mathcal{H}^{u}(v) \rightarrow \p Y \backslash \{\theta_{-}(v)\}$ defined for $x \in \mathcal{H}^{u}(v)$ by 
\[
P_{v}^{\mathcal{H}}(x) = \theta_{+}(-R_{-v}(x)) = \theta_{+}(\nabla B_{-v}(x)). 
\]
We then have a homeomorphism $P_{v}: \W^{u}(v) \rightarrow \p Y \backslash \{\theta_{-}(v)\}$ defined by $P_{v} = P_{v}^{\mathcal{H}} \circ \pi$, through which we can also view the unstable manifolds of $g^{t}_{Y}$ as coordinate charts for $\p Y$. For $v \in T^{1}Y$ and $w \notin \W^{u}(v)$ we then have a transition homeomorphism
\begin{equation}\label{transition maps}
P_{w}^{-1} \circ P_{v}: \W^{u}(v) \backslash \{x\} \rightarrow \W^{u}(w) \backslash \{y\},
\end{equation}
for $x = P_{v}^{-1}(\theta_{-}(w))$, $y = P_{w}^{-1}(\theta_{-}(v))$ -- if $w \in \W^{cu}(v)$ then we do not have to delete the points $x$ and $y$. The point $(P_{w}^{-1} \circ P_{v})(z)$ is the unique intersection point of $\W^{cs}(z)$ with $\W^{u}(w)$; hence $P_{w}^{-1} \circ P_{v}$ is given by the $\W^{cs}$-holonomy map from $\W^{u}(v) \backslash \{x\}$ to $\W^{u}(w) \backslash \{y\}$.

We specialize now to the case in which $Y$ is a negatively curved symmetric space. These spaces fit into four families: the real hyperbolic spaces $\mathbf{H}^{n}_{\R}$, the complex hyperbolic spaces $\mathbf{H}^{n}_{\C}$, the quaternionic hyperbolic spaces $\mathbf{H}^{n}_{\mathbb{H}}$, and the Cayley hyperbolic plane $\mathbf{H}^{2}_{\mathbb{O}}$. We normalize $\mathbf{H}^{n}_{\R}$ to have sectional curvatures $K \equiv -1$ and we normalize the other nonconstant negative curvature symmetric spaces to have sectional curvatures $-4 \leq K \leq -1$. We give further details below on the nonconstant negative curvature symmetric spaces. 

Set $X = \mathbf{H}^{n}_{\mathbb{K}}$, $\mathbb{K} \in \{\C,\mathbb{H},\mathbb{O}\}$, $n \geq 2$. Fix a vector $x \in T^{1}X$. The isometry group of $X$ acts transitively on the unstable horosphere $\mathcal{H}^{u}(x)$, identifying this horosphere as a homogeneous space for a 2-step Carnot group $G$,  which for $\mathbf{H}^{n}_{\C}$, $\mathbf{H}^{n}_{\mathbb{H}}$,  and $\mathbf{H}^{2}_{\mathbb{O}}$ is the complex, quaternionic, and octonionic Heisenberg group respectively. We recall from \cite{Pan85} that a 2-step Carnot group is a 2-step nilpotent Lie group $G$ whose Lie algebra $\mathfrak{g}$ splits as $\mathfrak{g} = \mathfrak{l} \oplus \mathfrak{v}$, with $\mathfrak{v}$ the center of $\mathfrak{g}$ and $[\mathfrak{l},\mathfrak{l}] = \mathfrak{g}$. The choice of transverse subspace $\mathfrak{l}$ to $\mathfrak{v}$ is not unique. Let $k = \dim X - 1$, $l = \dim \mathfrak{l}$. Then 
\[
k - l =  \dim \mathfrak{v} = \dim_{\R}\mathbb{K}-1,
\] 
where $\dim_{\R}\mathbb{K}$ is the dimension of the division algebra $\mathbb{K}$ as a vector space over $\R$. 

We write $0$ for the identity element in $G$, so that $TG_{0} = \mathfrak{g}$. There is a distinguished plane distribution $L^{x} \subset T\mathcal{H}^{u}(x)$ on $\mathcal{H}^{u}(x)$ which pulls back via the identification $G \rightarrow \mathcal{H}^{u}(v)$ to a left-invariant distribution $L$ on $G$. We set $\mathfrak{l} = L_{0}$. We let $V$ be the left-invariant distribution on $G$ with $V_{0} = \mathfrak{v}$. We equip $G$ with the left-invariant metric coming from the induced metric on $\mathcal{H}^{u}(x)$ from $X$, in which $L$ and $V$ are orthogonal. We write $T_{x}: G \rightarrow \W^{u}(x)$ for the identification of $G$ with $\mathcal{H}^{u}(x)$ composed with the identification of $\mathcal{H}^{u}(x)$ with $\W^{u}(v)$, chosen such that $T_{x}(0) = x$. The map $T_{x}$ is uniquely determined up to precomposition with isometries of $G$ that fix $0$. 

We write the geodesic flow $g^{t}$ on $\W^{u}(x)$ in these coordinates as $A^{t} = T_{g^{t}x}^{-1} \circ g^{t} \circ T_{x}$.  In this identification $A^{t}$ acts as an expanding automorphism on $G$. More precisely, for the left-invariant inner product $\langle\, , \rangle $ on $TG$, we have 
\[
\langle A^{t}v, A^{t}w \rangle = e^{t} \langle v, w \rangle, \; v, w \in L,
\]
and 
\[
\langle A^{t} v, A^{t}w\rangle = e^{4t} \langle v, w \rangle, \; v, w \in V.
\] 
Hence $A^{t}$ conformally expands $L$ by a factor of $e^{t}$ and conformally expands $V$ by a factor of $e^{2t}$. The $A^{t}$-invariant splitting $TG = L \oplus V$ corresponds to a global $Dg^{t}$-invariant splitting $E^{u} = L^{u} \oplus V^{u}$, with $Dg^{t}$ conformally expanding $L^{u}$ by a factor of $e^{t}$ and conformally expanding $V^{u}$ by a factor of $e^{2t}$. Switching to stable horospheres, we likewise have a global $Dg^{t}$-invariant splitting $E^{s} = L^{s} \oplus V^{s}$ with $Dg^{t}$ conformally contracting $L^{s}$ by a factor of $e^{-t}$ and conformally contracting $V^{s}$ by a factor of $e^{-2t}$

For a $C^1$ curve $\gamma$ in $G$, we let $\l(\gamma)$ denote its length in the left-invariant Riemannian metric on $G$. 

\begin{defn}[Carnot-Caratheodory metric] The \emph{Carnot-Caratheodory metric} $\rho_{G}$ on $G$ is defined by 
\begin{equation}\label{carnot metric}
\rho_{G}(x,y) = \inf_{\gamma} \l(\gamma),
\end{equation}
for $x$, $y \in G$, where the infimum is taken over all $C^1$ curves $\gamma$ in $G$ that join $x$ to $y$ and are tangent to $L$. As a shorthand we will refer to Carnot-Caratheodory metrics as CC-metrics. 
\end{defn}

The metric $\rho_{G}$ is invariant under left translations of $G$ and induces the Euclidean topology on $G$. The metric space $(G,\rho_{G})$ has Hausdorff dimension $l + 2(k-l)$ and the Hausdorff measure is the Riemannian volume $m_{G}$ on $G$. The automorphisms $A^{t}$ are conformal in the metric $\rho_{G}$ in the sense that
\[
\rho_{G}(A^{t}x,A^{t}y) = e^{t}\rho_{G}(x,y).
\]
For further discussion, see \cite{Bou18}.

Now let $X$ be closed negatively curved locally symmetric space of constant negative curvature. After rescaling the metric on $X$ by a positive constant if necessary, the universal cover $\t{X}$ will be isometric to one of the symmetric spaces $\mathbf{H}^{n}_{\mathbb{K}}$, $\mathbb{K} \in \{\C,\mathbb{H},\mathbb{O}\}$. The geodesic flow $g^{t}_{X}$ on $T^{1}X$ lifts to the geodesic flow $g^{t}$ on $T^{1}\mathbf{H}^{n}_{\mathbb{K}}$. For $x \in T^{1}X$, by identifying $\W^{u}(x)$ with a lifted leaf $\W^{u}(\t{x})$ for some choice of $\t{x} \in T^{1}\mathbf{H}^{n}_{\mathbb{K}}$ projecting to $x$, we obtain charts $T_{x}: G \rightarrow \W^{u}(x)$ with $T_{x}(0) = x$ that conjugate the action of $g^{t}_{X}$ on unstable manifolds to an automorphism $A^{t}$ on $G$ of the type described above. As before, these charts are only determined up to isometries of $G$ that fix $0$. 

The $Dg^{t}$-invariant splittings $E^{u} = L^{u} \oplus V^{u}$ and $E^{s} = L^{s} \oplus V^{s}$ on  $T^{1}\mathbf{H}^{n}_{\mathbb{K}}$ descend to a corresponding $u$-splitting of index $l$ and $s$-splitting of index $l$ for $Dg^{t}_{X}$. From our description of the action of $Dg^{t}$ on each of these subbundles, we conclude that for any $g^{t}_{X}$-invariant ergodic probability measure $\nu$ the Lyapunov spectrum $\vec{\la}(g_{X},\nu)$ is independent of $\nu$ and consists of one $0$ exponent (corresponding to the flow direction), $l$ exponents of value $1$, $l$ of value $-1$, $k-l$ of value $2$, and $k-l$ of value $-2$. By the variational principle the topological entropy of $g^{t}_{X}$ is then
\[
h_{\mathrm{top}}(g_{X}) = l + 2(k-l). 
\]
The potential $\zeta_{g_{X}}$ defined in \eqref{potential} is simply given by $\zeta_{g_{X}} \equiv l$. Therefore the horizontal measure $\mu_{X}$ is the measure of maximal entropy for $g^{t}_{X}$, which is just the Liouville measure $m_{X}$. Using \eqref{horizontal dim equation}, the horizontal dimension $q_{X}$ is then given by the formula 
\begin{equation}\label{hor dim formula}
q_{X} = \frac{h_{\mathrm{top}}(g_{X})}{l} = \frac{l+2(k-l)}{l}.
\end{equation}

\section{The dynamical rigidity theorems}\label{sec:dyn theorems}

In this section we will state the dynamical rigidity theorems \ref{periodic rigid}, \ref{dyn rigidity measure}, and \ref{dyn rigidity pinching} that serve as analogues of the geometric rigidity theorems \ref{periodic loc symm}, \ref{loc symmetric rigidity}, and \ref{geom pinching rigidity} stated in the introduction. We then explain how these geometric rigidity theorems can be deduced from their dynamical analogues. 

We begin by stating in Section \ref{core} the core dynamical theorem from which we will deduce Theorems \ref{periodic rigid}, \ref{dyn rigidity measure}, and \ref{dyn rigidity pinching}. In Sections \ref{section dyn rigidity} and \ref{section periodic} we state the dynamical rigidity theorems and reduce them to the core Theorem \ref{core theorem}. In Section \ref{subsec:geom rigidity} we deduce the geometric rigidity theorems from the dynamical rigidity theorems. This leaves only Theorem \ref{core theorem} to prove in the remainder of the paper. 

\subsection{The core theorem}\label{core}  We begin with an important definition.

\begin{defn}[$\beta$-fiber bunched]\label{first fiber bunched definition} Given an Anosov flow $f^{t}: M \rightarrow M$ on a Riemannian manifold $M$, a linear cocycle $A^{t}: L \rightarrow L$ on a vector bundle $L$ over $M$ equipped with a Riemannian structure, and a $\beta > 0$, we say that $A^{t}$ is \emph{$\beta$-fiber bunched} if there is a positive integer $N$ and a constant $0 < b < 1$ such that for all $x \in M$, 
\begin{equation}\label{fiber bunching inequality}
\|A^{N}_{x}\|\|(A^{N}_{x})^{-1}\|(\min \{\|Df^{N}_{x}|_{E^{s}_{x}}\|^{-1}, \mathfrak{m}(Df^{N}_{x}|_{E^{u}_{x}})\})^{-\beta} < b.
\end{equation}
\end{defn}
Note that if $A^{t}$ is $\beta$-fiber bunched then it is $\alpha$-fiber bunched for any $\alpha \geq \beta$.

Below we write $L^{cu} = E^{c} \oplus L^{u}$, $L^{cs} = E^{c} \oplus L^{s}$. Given a closed negatively curved locally symmetric space $X$ of nonconstant negative curvature with universal cover homothetic to $\mathbf{H}^{n}_{\mathbb{K}}$, $\mathbb{K} \in \{\C,\mathbb{H},\mathbb{O}\}$, we set $l(X) = \dim X - \dim_{\R}\mathbb{K}+1$. 

\begin{thm}\label{core theorem}
Let $f^{t}: M \rightarrow M$ be a smooth Anosov flow on a smooth Riemannian manifold $M$. Let $X$ be a closed locally symmetric space of nonconstant negative curvature, and suppose that we have an orbit equivalence $\varphi: T^{1}X \rightarrow M$ from $g^{t}_{X}$ to $f^{t}$. We assume that $f^{t}$ has a $u$-splitting $E^{u} = L^{u} \oplus V^{u}$ of index $l(X)$ and an $s$-splitting $E^{s} = L^{s} \oplus V^{s}$ of index $l(X)$. Assume further that there exists $\beta > 0$ such that 
\begin{enumerate}
\item $f^{t}$ is $\beta$-bunched,
\item $Df^{t}|_{L^{u}}$ and $Df^{t}|_{L^{s}}$ are $\beta$-fiber bunched, 
\item $Df^{t}|_{V^{u}}$ and $Df^{t}|_{V^{s}}$ are $1$-fiber bunched,
\item $L^{cu}$ and $L^{cs}$ are $\beta$-H\"older.
\end{enumerate}
Suppose that there exist constants $c_{1},c_{2} > 0$ such that
\[
\vec{\la}^{u}(f,\mu_{f}) = c_{1}\vec{\la}^{u}(g_{X}),
\]
and
\[
\vec{\la}^{u}(f^{-1},\mu_{f^{-1}}) = c_{2}\vec{\la}^{u}(g_{X}).
\]
Then  there is a $C^{\infty}$ orbit equivalence $\hat{\varphi}$ from $g^{t}_{X}$ to $f^{t}$ that is flow related to $\varphi$. 
\end{thm}

We will see below that hypotheses (1)-(4) are straightforward to verify in the settings that we consider. 

Recall that $X_{c}$ denotes the Riemannian manifold obtained by scaling the metric on $X$ by a constant $c > 0$. The geodesic flow $g^{t}_{X_{c}}$ is smoothly conjugate to the flow $g^{ct}_{X}$ and consequently $g^{t}_{X}$ is smoothly orbit equivalent to $g^{t}_{X_{c}}$. Hence in Theorem \ref{core theorem} it suffices to consider $X$ that are normalized to have sectional curvatures $-4 \leq K \leq -1$, as in Section \ref{subsec:neg curved}. We will always assume that $X$ has this normalization in what follows. 

\subsection{The dynamical rigidity theorems}\label{section dyn rigidity} We first give the dynamical analogue of Theorem \ref{loc symmetric rigidity}. We will need the following lemma. 

\begin{lem}\label{neighborhood lemma}
Let $X$ be a closed negatively curved locally symmetric space of nonconstant curvature. Then there is a $C^1$ open neighborhood $\mathcal{V}_{X}$ of $g^{t}_{X}$ in the space of smooth flows on $T^{1}X$ such that if $f^{t} \in \mathcal{V}_{X}$ then there is an orbit equivalence $\varphi$ from $g^{t}_{X}$ to $f^{t}$ and $f^{t}$ satisfies hypotheses (1)-(4) of Theorem \ref{core theorem} with $\beta = \frac{1}{5}$. 
\end{lem} 

\begin{proof}
The existence of the orbit equivalence $\varphi$ follows from structural stability \ref{structural stability}, once $\mathcal{V}_{X}$ is small enough. By the persistence of dominated splittings under perturbation (see the discussion after Proposition \ref{cone criterion}), once $f^{t}$ is $C^1$ close enough to $g^{t}_{X}$ there is a $u$-splitting $E^{u,f} = L^{u,f} \oplus V^{u,f}$ and an $s$-splitting $E^{s,f} = L^{s,f} \oplus V^{s,f}$ $C^0$-close to the corresponding $u$-splitting $E^{u,g} = L^{u,g} \oplus V^{u,g}$ and $s$-splitting $E^{s,g} = L^{s,g} \oplus V^{s,g}$ for $g^{t}_{X}$.

For a fixed $\beta > 0$, the $\beta$-bunching condition \eqref{bunching inequality} is easily seen to be $C^1$ open, since the derivative cocycle $Df^{t}$ will be $C^0$ close to the derivative cocycle $Dg^{t}_{X}$ and both $E^{u,f}$ and $E^{s,f}$ are $C^0$ close to $E^{u,g}$ and $E^{s,g}$ respectively. A direct computation using the structure of $g^{t}_{X}$ outlined in Section \ref{subsec:neg curved} shows that the flow $g^{t}_{X}$ is $\beta$-bunched for any $0 < \beta < \frac{1}{2}$, hence taking $\beta = \frac{1}{5}$ we conclude that the flow $f^{t}$ will be $\frac{1}{5}$-bunched if it is $C^1$ close enough to $g^{t}_{X}$. This gives hypothesis (1) of Theorem \ref{core theorem}.

For any fixed  $\beta > 0$, the fiber bunching inequality \eqref{fiber bunching inequality} is also $C^0$ open in the cocycle $A^{t}$ and $C^1$ open in the flow $f^{t}$. Since $Dg^{t}_{X}$ is conformal on $L^{u,g}$, $L^{s,g}$, $V^{u,g}$ and $V^{s,g}$, it is $\frac{1}{5}$-fiber bunched on these subbundles and therefore $Df^{t}$ will be $\frac{1}{5}$-fiber bunched on $L^{u,f}$, $L^{s,f}$, $V^{u,f}$ and $V^{s,f}$ as well for $f^{t}$ $C^1$ close enough to $g^{t}_{X}$. Since $\frac{1}{5}$-fiber bunching implies $1$-fiber bunching, this gives hypotheses (2) and (3) of Theorem \ref{core theorem}.

We next address the exponent of H\"older regularity for $L^{cu,f}$ and $L^{cs,f}$. These bundles have the same H\"older regularity as the quotient bundles $\bar{L}^{u,f} = L^{cu,f}/E^{c,f}$ and $\bar{L}^{s,f} = L^{cs,f}/E^{c,f}$. We define $\bar{V}^{u,f}$ and $\bar{V}^{s,f}$ analogously. We then have  dominated splittings $\bar{E}^{u,f} = \bar{L}^{u,f} \oplus \bar{V}^{u,f}$ and $\bar{E}^{s,f} = \bar{L}^{s,f} \oplus \bar{V}^{s,f}$  for $\bar{D}f^{t}$ on $\bar{E}^{u}$ and $\bar{E}^{s}$ respectively. We will perform the computation below for $\bar{L}^{u,f}$; the computation for $\bar{L}^{s,f}$ is analogous. 

Let $0 \leq \delta < \frac{1}{2}$ and $0 < \alpha < \frac{1}{2}$ be given. For $f^{t}$ $C^1$ close enough to $g^{t}_{X}$ (depending on $\delta$ and $\alpha$), we may assume that $f^{t}$ is $\alpha$-bunched and that there is a constant $C > 0$ such that for all $x \in T^{1}X$ and $t \geq 0$ the inequalities
\[
C^{-1}e^{(1-\delta)t} \leq \mathfrak{m}(\bar{D}f^{t}_{x}|_{\bar{L}^{u,f}_{x}}) \leq  \|\bar{D}f^{t}_{x}|_{\bar{L}^{u,f}_{x}}\| \leq Ce^{(1+\delta)t},
\]
and
\[
C^{-1}e^{(2-\delta)t} \leq \mathfrak{m}(\bar{D}f^{t}_{x}|_{\bar{V}^{u,f}_{x}}) \leq  \|\bar{D}f^{t}_{x}|_{\bar{V}^{u,f}_{x}}\| \leq Ce^{(2+\delta)t},
\]
hold, as well as the inequality 
\[
\max\{\|(Df^{t}_{x})^{-1}\|,\|Df^{t}_{x}\|\} \leq Ce^{(2+\delta)t}. 
\]
By replacing the norm $\|\cdot\|$ with an equivalent norm we can assume that the bundles  $\bar{L}^{u,f}$ and $\bar{V}^{u,f}$ are orthogonal to one another at the cost of the inequalities above being satisfied for a possibly larger constant $C$. By further individually modifying the restriction of this norm to $\bar{L}^{u,f}$ and $\bar{V}^{u,f}$ using Proposition \ref{new metric} we can then assume the above inequalities hold for all $t > 0$ without the constant $C$. By Proposition \ref{strong holder splitting} we then conclude that $\bar{L}^{u,f}$ is $\beta$-H\"older with exponent
\[
\beta = \left( \frac{2-\delta - (1+\delta)}{(2+\delta)\alpha+2+\delta - (1+\delta)}\right)\alpha = \frac{\alpha-2\delta \alpha}{1+(2+\delta)\alpha}.
\]
As $f^{t} \rightarrow g^{t}_{X}$ in the $C^1$ topology, we can take $\delta \rightarrow 0$ and $\alpha \rightarrow \frac{1}{2}$. We conclude that $\beta \rightarrow \frac{1}{4}$. Thus for $f^{t}$ $C^1$ close enough to $g^{t}_{X}$ the subbundles $\bar{L}^{u,f}$ and $\bar{L}^{s,f}$ will be $\frac{1}{5}$-H\"older. This provides hypothesis (4) of Theorem \ref{core theorem}.
\end{proof}

Let $\mathcal{V}_{X}$ be the $C^1$ open neighborhood of $g^{t}_{X}$ given by Lemma \ref{neighborhood lemma}.

\begin{thm}\label{dyn rigidity measure}
Let $X$ be a closed locally symmetric space of nonconstant negative curvature. Let $f^{t} \in \mathcal{V}_{X}$. Suppose that there exist constants $c_{1},c_{2} > 0$ such that
\[
\vec{\la}^{u}(f,\mu_{f}) = c_{1}\vec{\la}^{u}(g_{X}),
\]
and
\[
\vec{\la}^{u}(f^{-1},\mu_{f^{-1}}) = c_{2}\vec{\la}^{u}(g_{X}).
\]
Then $g^{t}_{X}$ is $C^{\infty}$ orbit equivalent to $f^{t}$.
\end{thm}

By the construction of $\mathcal{V}_{X}$, Theorem \ref{dyn rigidity measure} follows immediately from Theorem \ref{core theorem}. We also give the analogue of Theorem \ref{geom pinching rigidity} for flows. Below we set $k = \dim X - 1$. 

\begin{thm}\label{dyn rigidity pinching}
Let $X$ be a closed locally symmetric space of nonconstant negative curvature.  Let $f^{t} \in \mathcal{V}_{X}$. Suppose that $\min\{q_{f},q_{f^{-1}}\} \geq q_{X}$ and that 
\[
\la_{k}^{u}(f,\mu_{f}) \leq 2\la_{1}^{u}(f,\mu_{f}),
\]
and
\[
\la_{k}^{u}(f^{-1},\mu_{f^{-1}}) \leq 2\la_{1}^{u}(f^{-1},\mu_{f^{-1}}).
\]
Then  $g^{t}_{X}$ is $C^{\infty}$ orbit equivalent to $f^{t}$.
\end{thm}

We will show here that Theorem \ref{dyn rigidity measure} implies Theorem \ref{dyn rigidity pinching}. Set $l:=l(X)$.  By \eqref{hor dim formula}, the hypotheses of Theorem \ref{dyn rigidity pinching} imply that 
\[
\min\{q_{f},q_{f^{-1}}\}  \geq \frac{l + 2(k-l)}{l}.
\]
Then by Lemma \ref{hor sup}, there is a constant $c_{1} > 0$ such that $\la_{i}^{u}(f,\mu_{f}) = c_{1}$ for $1 \leq i \leq l$ and $\la_{i}^{u}(f,\mu_{f}) = 2c_{1}$ for $l+1 \leq i \leq k$. But this just implies that $\vec{\la}^{u}(f,\mu_{f}) = c_{1} \vec{\la}^{u}(g_{X})$. The same reasoning using $f^{-t}$ instead gives a constant $c_{2} > 0$ such that $\vec{\la}^{u}(f^{-1},\mu_{f^{-1}}) = c_{2} \vec{\la}^{u}(g_{X})$. Hence $f^{t}$ satisfies the hypotheses of Theorem \ref{dyn rigidity measure}, and so Theorem \ref{dyn rigidity pinching} reduces to Theorem \ref{dyn rigidity measure}.

\subsection{The periodic orbit theorem}\label{section periodic} The dynamical counterpart of Theorem \ref{periodic loc symm} is stated below. We recall that for a periodic point of a continuous flow $f^{t}:M \rightarrow M$, $\nu^{(p)}$ denotes the unique $f^{t}$-invariant  ergodic probability measure supported on the orbit of $p$. 

\begin{thm}\label{periodic rigid}
Let $f^{t}: M \rightarrow M$ be a smooth Anosov flow on a smooth Riemannian manifold $M$. Let $X$ be a closed negatively curved locally symmetric space, $\dim X \geq 3$. Suppose that there exists an orbit equivalence $\varphi: T^{1}X \rightarrow M$ from $g^{t}_{X}$ to $f^{t}$, and suppose that for each periodic point $p$ of $f^{t}$ there exists a constant $\omega(p) > 0$ such that
\[
\vec{\la}(f,\nu^{(p)}) = \omega(p)\vec{\la}(g_{X}).
\]
Then there is a $C^{\infty}$ orbit equivalence $\hat{\varphi}$ from $g^{t}_{X}$ to $f^{t}$ that is flow related to $\varphi$. 
\end{thm}

The reduction of Theorem \ref{periodic rigid} to Theorem \ref{core theorem} in the case that $X$ has nonconstant negative curvature is significantly more difficult than the previous reductions and occupies the rest of this section.  A significant part of this reduction relies heavily on work of Kalinin \cite{Kal} through Lemma \ref{approximate} below. The case in Theorem \ref{periodic rigid} in which $X$ has constant negative curvature is much simpler and can be deduced  through a straightforward simplification of the arguments from the nonconstant negative curvature case. Since the constant negative curvature case of Theorem \ref{periodic loc symm} was already established in our previous work \cite{Bu1}, we will not go over this simplification explicitly in this paper.

We will more generally establish the existence of a dominated splitting for an Anosov flow $f^{t}$ from a particular configuration of the Lyapunov exponents associated to its periodic orbits. Our arguments are similar to those of Velozo \cite{Vel}. We let $f^{t}: M \rightarrow M$ be a smooth transitive Anosov flow on a closed Riemannian manifold $M$. We assume that the Anosov splitting $TM = E^{u} \oplus E^{c} \oplus E^{s}$ of $f^{t}$ satisfies $\dim E^{u} = \dim E^{s} = k$, and let $1 \leq l \leq k-1$ be given. Let 
\[
\vec{\chi} = (-2,\dots,-2,-1,\dots,-1,0,1\dots,1,2\dots,2),
\] 
where there are $l$ $1$'s and $-1$'s, and $k-l$ $2$'s and $-2$'s. 

\begin{lem}\label{approximate}
Suppose that for each periodic point $p$ of $f^{t}$ there is a constant $\omega(p)>0$ such that
\[
\vec{\la}(f,\nu^{(p)}) = \omega(p)\vec{\chi}.
\]
Then for each $f^{t}$-invariant ergodic probability measure $\nu$ there is a constant $\omega(\nu)>0$ such that we have
\[
\vec{\la}(f,\nu) = \omega(\nu)\vec{\chi}.
\]
\end{lem}

\begin{proof}
First observe that $\la_{k+1}(f^{t},\nu) = \chi_{k+1} = 0$ for all $f^{t}$-invariant ergodic probability measures $\nu$, as this is the exponent in the flow direction $E^{c}$. For all $i \neq k+1$  we have $\chi_{i} \neq 0$ and therefore 
\[
\omega(p) = \frac{\la_{i}(f,\nu^{(p)})}{\chi_{i}}.
\]
By Kalinin's approximation theorem \cite{Kal}, for any $f^{t}$-invariant ergodic probability measure $\nu$ there is a sequence of periodic points $p_{n}$ of $f^{t}$ such that $\vec{\la}(f,\nu^{(p_{n})}) \rightarrow \vec{\la}(f,\nu)$. We set 
\[
\omega(\nu) :=\lim_{n \rightarrow \infty} \omega(p_{n}) = \lim_{n \rightarrow \infty}  \frac{\la_{i}(f,\nu^{(p_{n})})}{\chi_{i}} =\frac{\la_{i}(f,\nu)}{\chi_{i}},
\]
which holds independently of $i$. It follows that $\vec{\la}(f,\nu) = \omega(\nu)\vec{\chi}$.
\end{proof}

\begin{prop}\label{periodic dom split}
Suppose that for each periodic point $p$ of $f^{t}$ there is a constant $\omega(p)>0$ such that
\[
\vec{\la}(f,\nu^{(p)}) = \omega(p)\vec{\chi}.
\]
Then there is a $u$-splitting $E^{u} = L^{u} \oplus V^{u}$ of index $l$ and an $s$-splitting $E^{s} = L^{s} \oplus V^{s}$ of index $l$ for $f^{t}$. Furthermore, for every small enough $\e > 0$  there is a smooth time change $\hat{f}^{t}$ of $f^{t}$ for which there is a constant $C = C_{\e} \geq 1$ such that for every $x \in M$ and $t \geq 0$, 
\begin{equation}\label{first periodic inequality}
C^{-1}e^{(1-\e)t} \leq \mathfrak{m}(D\hat{f}^{t}_{x}|_{\hat{L}^{u}_{x}}) \leq  \|D\hat{f}^{t}_{x}|_{\hat{L}^{u}_{x}}\| \leq Ce^{(1+\e)t},
\end{equation}
and
\begin{equation}\label{second periodic inequality}
C^{-1}e^{(2-\e)t} \leq \mathfrak{m}(D\hat{f}^{t}_{x}|_{\hat{V}^{u}_{x}}) \leq  \|D\hat{f}^{t}_{x}|_{\hat{V}^{u}_{x}}\| \leq Ce^{(2+\e)t},
\end{equation}
and the same inequalities hold for $\hat{f}^{-t}$ with $\hat{L}^{s}$ and $\hat{V}^{s}$ replacing $\hat{L}^{u}$ and $\hat{V}^{u}$. Here $\hat{E}^{u} = \hat{L}^{u} \oplus \hat{V}^{u}$ and $\hat{E}^{s} = \hat{L}^{s} \oplus \hat{V}^{s}$ denote the corresponding $u$-splitting and $s$-splitting for $\hat{f}^{t}$. 
\end{prop}

\begin{proof}
We will show that $f^{t}$ has a $u$-splitting $E^{u} = L^{u} \oplus V^{u}$; the existence of the $s$-splitting $E^{s} = L^{s} \oplus V^{s}$ follows by applying our arguments to the inverse flow $f^{-t}$. Recalling that $\bar{E}^{u} = E^{cu}/E^{c}$, by considering the projection $E^{u} \rightarrow \bar{E}^{u}$ it suffices to show that $\bar{D}f^{t}: \bar{E}^{u} \rightarrow \bar{E}^{u}$ admits a dominated splitting $\bar{E}^{u} = \bar{L}^{u} \oplus \bar{V}^{u}$. Since it's easily checked that dominated splittings are preserved under time changes, it suffices to show that there is a time change $\hat{f}^{t}$ of $f^{t}$ such that the associated linear cocycle $\bar{D}\hat{f}^{t}:\bar{E}^{u} \rightarrow \bar{E}^{u}$ admits a dominated splitting. 

We begin with a few reductions. Set $\check{\chi}^{u} = \sum_{i=1}^{k} \chi_{i}^{u}$ and set $\vec{\chi}^{u} = (1,\dots,1,2,\dots,2)$ to be the vector consisting of the positive coordinates of $\vec{\chi}$. By Lemma \ref{approximate} we have $\vec{\la}(f,\nu) = \omega(\nu)\vec{\chi}$ for all $f^{t}$-invariant ergodic probability measures $\nu$, for corresponding constants $\omega(\nu)$. We may write
\begin{equation}\label{omega equation}
\omega(\nu) = \frac{\nu(\log \mathrm{Jac}(\bar{D}f))}{\check{\chi}^{u}}.
\end{equation}
We set
\[
\gamma(x) =  \check{\chi}^{u}\cdot \left(\left.\frac{d}{dt}\right|_{t=0}\log \mathrm{Jac}(\bar{D}f_{x}^{t})\right)^{-1}.
\]
Note that $\gamma$ is smooth along $\W^{cu}$. Let $g^{t}$ be the time change of $f^{t}$ with speed multiplier $\gamma$. Then $\bar{D}g^{t}$ is a time change of $\bar{D}f^{t}$ with the same speed multiplier $\gamma$. Let $\tau$ be the additive cocycle over $g^{t}$ generated by $\gamma$, and let $\xi$ denote the additive cocycle over $f^{t}$ generated by $\gamma^{-1}$. Then $f^{\tau(t,x)}x = g^{t}x$. By Lemma \ref{inverse additive} we have $\xi(\tau(t,x),x) = t$. Expanding this expression, we obtain  
\begin{equation}\label{Jacobian time change}
\check{\chi}^{u}t = \log \mathrm{Jac}(\bar{D}f^{\tau(t,x)}_{x}) = \log \mathrm{Jac}(\bar{D}g^{t}_{x}).
\end{equation}


Set 
\begin{equation}\label{omega change}
\omega(\nu^{\gamma}) = \omega(\nu)\nu( \gamma^{-1})^{-1} = \frac{\nu(\log \mathrm{Jac}(\bar{D}f))}{\nu(\gamma^{-1})\check{\chi}^{u}}
\end{equation}
For any $f^{t}$-invariant ergodic probability measure $\nu$, we have by Proposition \ref{time change lyap},
\[
\vec{\la}^{u}(g,\nu^{\gamma}) = \nu(\gamma^{-1})^{-1}\vec{\la}^{u}(f,\nu) 
= \omega(\nu^{\gamma})\vec{\chi}^{u}.
\]
On the other hand, by integrating equation \eqref{Jacobian time change} against the measure $\nu^{\gamma}$ at $t = 1$,
\[
\check{\chi}^{u} = \sum_{i=1}^{k} \la^{u}_{i}(g,\nu^{\gamma}).
\]
It follows that $\omega(\nu^{\gamma}) = 1$ for all $g^{t}$-invariant ergodic probability measures $\nu^{\gamma}$. Hence $\vec{\la}^{u}(g,\nu^{\gamma}) = \vec{\chi}^{u}$ for all such measures $\nu^{\gamma}$. 

Let $\gamma_{m} \rightarrow \gamma$ be a sequence of smooth functions $\gamma_{m}:M \rightarrow (0,\infty)$ converging uniformly to $\gamma$. Let $g^{t}_{m}$ be the time change of $f^{t}$ with speed multiplier $\gamma_{m}$. Then $g^{t}_{m}$ is a smooth flow, so applying Proposition \ref{time change lyap} again, but to the full flow,
\[
\vec{\la}(g_{m},\nu^{\gamma_{m}}) = \omega(\nu^{\gamma_{m}})\vec{\chi},
\]
with $\omega(\nu^{\gamma_{m}})$ defined as in \eqref{omega change}. By the explicit expression for $\omega(\nu^{\gamma_{m}})$ in \eqref{omega change} we conclude that as $m \rightarrow \infty$, $\omega(\nu^{\gamma_{m}}) \rightarrow 1$, uniformly in the measure $\nu$. In particular, for any given $\e > 0$ we will have for $m$ large enough, 
\begin{equation}\label{ratio 2}
|\la_{i}(g_{m},\nu^{\gamma_{m}})-\chi_{i}| < \frac{\e}{2},
\end{equation}
for each $1 \leq i \leq 2k+1$ and all $\hat{f}_{m}^{t}$-invariant ergodic probability measures $\nu^{\gamma_{m}}$. 

Since $\gamma_{m} \rightarrow \gamma$ uniformly, for $m$ large enough we have for all $x \in M$
\[
\left|\check{\chi}^{u} - \left.\frac{d}{dt}\right|_{t=0}\log \mathrm{Jac}(\bar{D}g_{m,x}^{t})\right| < \frac{\e}{2}.
\]
Integrating this inequality we obtain for all $t \geq 0$, 
\[
|\check{\chi}^{u}t - \log \mathrm{Jac}(\bar{D}g_{m,x}^{t})| < \frac{\e}{2} t, 
\]
which we rewrite as 
\begin{equation}\label{Jacobian inequality}
e^{(\check{\chi}^{u}-\frac{\e}{2})t} \leq \mathrm{Jac}(\bar{D}g^{t}_{m,x}) \leq e^{(\check{\chi}^{u}+\frac{\e}{2})t}.
\end{equation}
For any given $\e > 0$ we can choose $m$ large enough that inequalities \eqref{ratio 2} and \eqref{Jacobian inequality} hold for the associated time change $g^{t}_{m}$. By a further arbitrarily small smooth perturbation of $\gamma_{m}$, we may assume that $g^{t}_{m}$ is topologically mixing, as the only way for a transitive Anosov flow to fail to be topologically mixing is for it to be the suspension flow of an Anosov diffeomorphism, and this suspension flow property will be broken by any smooth time change not cohomologous to a constant multiple of the trivial additive cocycle $\eta(t,x) = t$ over the flow \cite[Theorem 1.8]{Pla72}. We set $\hat{f}^{t} := g^{t}_{m}$ in what follows, to simplify the notation, and continue to write $\bar{D}\hat{f}^{t}$ for the induced action of $D\hat{f}^{t}$ on $\bar{E}^{u}$. 

We will need the following lemma of Kalinin and Sadovskaya. For this lemma we let $f:M \rightarrow M$ be a homeomorphism of a compact metric space $M$ and let $a_{n}: M \rightarrow \R$ be a sequence of functions. We say that this sequence is \emph{subadditive} if for all $x \in M$ and $n, m \in \N$ we have
\[
a_{n+m}(x) \leq a_{m}(x) + a_{n}(f^{m}x). 
\]
For an $f$-invariant ergodic probability measure $\nu$ we define
\[
a(\nu) = \inf_{n \geq 1}\frac{1}{n}\int_{M} a_{n}\, d\nu = \lim_{n \rightarrow \infty} \frac{1}{n}\int_{M} a_{n}\, d\nu, 
\]
with the limit existing by Fekete's lemma.

\begin{lem}\cite[Proposition 4.9]{KS}\label{KS lemma}
Suppose that for all $f$-invariant ergodic probability measures $\nu$ we have $a(\nu) < 0$. Then there is an $N > 0$ such that $a_{n}(x) < 0$ for all $n \geq N$ and $x \in M$. 
\end{lem}

We apply Lemma \ref{KS lemma} to the following four subadditive sequences over the time-1 map $\hat{f}:M \rightarrow M$ of our time-changed Anosov flow $\hat{f}^{t}$ satisfying \eqref{ratio 2},
\begin{enumerate}
\item $a_{n}(x) = \log \|D\hat{f}^{n}_{x}|_{\hat{E}^{u}_{x}}\| - (2 + \e)n$,
\item  $a_{n}(x) = -\log \mathfrak{m}(D\hat{f}^{n}_{x}|_{\hat{E}^{u}_{x}}) + (1- \e)n$,
\item $a_{n}(x) =  \log \|D\hat{f}^{n}_{x}|_{\hat{E}^{s}_{x}}\|  + (2  -\e)n$,
\item $a_{n}(x) = -\log \mathfrak{m}(D\hat{f}^{n}_{x}|_{\hat{E}^{s}_{x}}) - (1 - \e)n$,
\end{enumerate} 
to obtain that there is a time $T = T_{\e} > 0$ such that for $t \geq T$ we have
\begin{equation}\label{unstable period inequality}
e^{(1-\e)t}\leq \mathfrak{m}(D\hat{f}^{t}_{x}|_{\hat{E}^{u}_{x}}) \leq \|D\hat{f}^{t}_{x}|_{\hat{E}^{u}_{x}}\| \leq e^{(2+\e)t},
\end{equation}
and 
\begin{equation}\label{stable period inequality}
e^{-(2+\e)t}\leq \mathfrak{m}(D\hat{f}^{t}_{x}|_{\hat{E}^{s}_{x}}) \leq \|D\hat{f}^{t}_{x}|_{\hat{E}^{s}_{x}}\| \leq e^{-(1-\e)t}.
\end{equation}
As a first application of these inequalities we can establish some regularity properties for the quotient bundles $\bar{E}^{u}$ and $\bar{E}^{s}$. 

\begin{lem}\label{reg lem}
The bundles $\bar{E}^{u}$ and $\bar{E}^{s}$ are $C^{\beta}$ for each $0 \leq \beta < 1$. 
\end{lem}

\begin{proof}
We will deduce this from the $\beta$-bunching criterion of Definition \ref{beta bunching}. For a given $\e > 0$ we choose a time change $\hat{f}^{t} = \hat{f}^{t}_{m}$ as above such that there is a $T = T_{\e}$ for which the inequalities \eqref{unstable period inequality} and \eqref{stable period inequality} hold for $t \geq T$. Using these inequalities we see that for the $\beta$-bunching inequality \eqref{bunching inequality} to hold for a given $\beta > 0$ for $\hat{f}^{t}$, it suffices to require that for $t \geq T$ we have 
\[
e^{(1+2\e)t} \leq e^{\beta^{-1} (1-\e)t},
\]
from which we see that it suffices to require 
\[
\beta  <  \frac{1-\e}{1+2\e}.
\]
The ratio on the right hand side converges to $1$ as $\e \rightarrow 0$. Since $\e$ can be taken as small as we desire, we conclude from the $\beta$-bunching criterion that $E^{cu}$ and $E^{cs}$ are $C^{\beta}$ subbundles of $TM$ for each $0 \leq \beta < 1$. Consequently $\bar{E}^{u} = E^{cu}/E^{c}$ and $\bar{E}^{s} = E^{cs}/E^{c}$ are $C^{\beta}$ bundles over $M$ for each $0 \leq \beta < 1$.
\end{proof}

Note in the proof of the lemma that we are using the key fact that the bundles $E^{cu}$, $E^{cs}$, and $E^{c}$ remain the same for all time changes of the original flow $f^{t}$. We now fix $\e$ small enough that the above discussion applies and such that we additionally have 
\begin{equation}\label{bad inequality}
\frac{l+2}{l+1} - 1 > 22k\e. 
\end{equation}
Thus for this chosen value of $\e$ satisfying \eqref{bad inequality} we will be considering  a smooth time change $\hat{f}^{t}$ of $f^{t}$ that is topologically mixing and satisfies the inequalities \eqref{ratio 2} and \eqref{Jacobian inequality}. We will show for this fixed choice of $\e$ that the linear cocycle $\bar{D}\hat{f}^{t}$ on $\bar{E}^{u}$ has a dominated splitting of index $l$.

We use Proposition \ref{new metric} to obtain a new Riemannian structure $\{\|\cdot\|_{\e}\}_{x \in M}$ on $TM$ such that inequalities \eqref{unstable period inequality} and \eqref{stable period inequality} hold for all $t > 0$; note that this structure depends on our fixed choice of $\e$. We have a uniform comparison of this structure with the original Riemannian structure by, for $v \in TM$ and a constant $C_{\e}$ depending only on $\e$,  
\begin{equation}\label{riemann compare}
C_{\e}^{-1}\|v\| \leq \|v\|_{\e} \leq C_{\e}\|v\|.
\end{equation}
We give $\bar{E}^{u}$ the induced Riemannian structure from the natural projection $\hat{E}^{u} \rightarrow \bar{E}^{u}$. We write $\mathfrak{m}_{\e}$ for the conorm in this Riemannian structure and write $\sigma_{i,\e}^{u}$ for the $i$th singular value on $\bar{E}^{u}$ in this structure. We let $d_{\e}$ denote the Riemannian metric on $M$ induced by this Riemannian structure. All of the computations that follow will be carried out in this fixed choice of modified Riemannian structure on $M$. 

If $\bar{D}\hat{f}^{t}$ does not admit a dominated splitting of index $l$ then, using Theorem \ref{domination gap criterion} and the fact that existence of a dominated splitting does not depend on the choice of Riemannian structure, we can find for any $\kappa > 0$ small enough a sequence of integers $n_{j} \rightarrow \infty$ and a sequence of points $x_{j} \in M$ such that for each $j \in \N$,
\begin{equation}\label{failure}
\frac{\sigma_{l,\e}(\bar{D}\hat{f}^{n_{j}}_{x_{j}})}{\sigma_{l+1,\e}(\bar{D}\hat{f}^{n_{j}}_{x_{j}})} \geq e^{-\kappa n_{j}}.
\end{equation}
We choose $\kappa$ small enough that $\kappa < \e$. 

Because we have assumed that $\hat{f}^{t}$ is topologically mixing, it has Bowen's specification property, as formulated in \cite[Theorem 18.3.14]{HK}. In particular, for any given $\delta > 0$ we can find an $\eta > 0$ such that for each $j \in \N$ there is a periodic point $p_{j}$ for $\hat{f}^{t}$ such that $d_{\e}(\hat{f}^{t}p_{j}, \hat{f}^{t}x_{j}) < \delta$ for $0 \leq t \leq 2n_{j}$ and $\hat{f}^{s_{j}}p_{j} = p_{j}$, where $0 \leq s_{j} \leq 2n_{j} + \eta$. The specified periodic point $p_{j}$ thus $\delta$-shadows $x_{j}$ for twice the length of the orbit on which we see inequality \eqref{failure} before closing this orbit with a segment of length at most $\eta = \eta(\delta)$ which depends only on $\delta$. We fix the threshold $\delta$ small enough that the standard shadowing estimates for points that are $\delta$-close along an orbit apply \cite[Theorem 6.2.8]{HK} (adjusted for flows) which together with inequalities \eqref{unstable period inequality} and \eqref{stable period inequality} gives that there is a constant $K_{\e}$ and a continuous offset function $\zeta_{j}:[0,2n_{j}] \rightarrow \R$ such that for each $j \in \N$ and $0 \leq t \leq 2n_{j}$, 
\[
d_{\e}(\hat{f}^{t}p_{j}, \hat{f}^{t+\zeta_{j}(t)}x_{j}) \leq K_{\e} (d_{\e}(p_{j},x_{j}) + d_{\e}(\hat{f}^{2n_{j}}p_{j},\hat{f}^{2n_{j}}x_{j}))e^{-(1-\e) \min\{t, 2n_{j}-t\}},
\]
and such that $|\zeta_{j}(t)| < C\delta$ for each $t$ for a uniform constant $C \geq 1$. We bound the distance terms on the right above by the diameter of $M$ in the metric $d_{\e}$ to obtain that there is a constant $K'_{\e}$ depending only on $\e$ such that for each $j \in \N$ and $0 \leq t \leq 2n_{j}$,
\begin{equation}\label{shadowing}
d_{\e}(\hat{f}^{t}p_{j}, \hat{f}^{t+\zeta_{j}(t)}x_{j}) \leq K_{\e}'e^{-(1-\e) \min\{t, 2n_{j}-t\}}.
\end{equation}

We will need to compare norms for $\bar{D}\hat{f}^{t}$ along the orbit of $p_{j}$ to norms along the orbit of $x_{j}$. This is the purpose of the next estimate. By Lemma \ref{reg lem} the bundle $\bar{E}^{u}$ is $C^{1-\e}$. We can then fix a family of identifications $I_{xy}:\bar{E}_{x}^{u} \rightarrow \bar{E}_{y}^{u}$ for points $x$ close to $y$ in $M$ that vary in a $(1-\e)$-H\"older fashion with the basepoint (see \cite[Section 2.2]{KS}). The estimates below are similar to those used in the proof of \cite[Proposition 4.2]{KS}.

To simplify notation we set $A^{t}:=\bar{D}\hat{f}^{t}$ and set $q_{m,j}=m+\zeta_{j}(m)$ for each $0 \leq m \leq 2n_{j}-1$ and write $n_{j,j} = n_{j}+ \zeta_{j}(n_{j})$. For a given $j \in \N$ we write
\begin{align*}
(A^{n_{j,j}}_{x_{j}})^{-1}I_{\hat{f}^{n_{j}}p_{j}\hat{f}^{n_{j,j}}x_{j}} A^{n_{j}}_{p_{j}} - I_{p_{j}x_{j}} &= \sum_{m=0}^{n_{j}-1}(A^{q_{m+1,j}}_{x_{j}})^{-1} I_{\hat{f}^{q_{m+1,j}}x_{j} \hat{f}^{m+1}p_{j}}A^{m+1}_{p_{j}}-(A^{q_{m,j}}_{x_{j}})^{-1} I_{\hat{f}^{q_{m,j}}x_{j} \hat{f}^{m}p_{j}}A^{m}_{p_{j}} \\
&=\sum_{m=0}^{n_{j}-1}(A^{q_{m,j}}_{x_{j}})^{-1} (A_{\hat{f}^{q_{m,j}}x_{j}}^{-1}I_{\hat{f}^{q_{m+1,j}}x_{j} \hat{f}^{m+1}p_{j}}A_{\hat{f}^{m}p_{j}}
-I_{\hat{f}^{q_{m,j}}x_{j} \hat{f}^{m}p_{j}}) A^{m}_{p_{j}}.
\end{align*}
The $(1-\e)$-H\"older property of the identifications gives for a constant $K_{\e}'' > 0$ depending only on $\e$ and each $ 0 \leq m \leq n_{j}$, 
\begin{align*}
\|A_{\hat{f}^{q_{m,j}}x_{j}}^{-1}I_{\hat{f}^{q_{m+1,j}}x_{j} \hat{f}^{m+1}p_{j}}A_{\hat{f}^{m}p_{j}}-I_{\hat{f}^{q_{m,j}}x_{j} \hat{f}^{m}p_{j}}\|_{\e} &\leq K''_{\e} d_{\e}(\hat{f}^{q_{m,j}}x_{j},\hat{f}^{m}p_{j})^{1-\e} \\
&\leq K'_{\e}K''_{\e}e^{-(1-\e)^{2}m},
\end{align*}
with the first line coming from the fact that $A^{t} = \bar{D}\hat{f}^{t}$ is the restriction of a smooth linear cocycle $Df^{t}$ to a $C^{1-\e}$ bundle $\bar{E}^{u}$ and where in the second line we've applied the inequality \eqref{shadowing} for $0 \leq m \leq n_{j}$ followed by using $|q_{m,j}-m| < C\delta$ and $|n_{j,j}-n_{j}| < C\delta$. We conclude using \eqref{unstable period inequality}, 
\begin{align*}
\|(A^{n_{j,j}}_{x_{j}})^{-1} I_{\hat{f}^{n_{j}}p_{j}\hat{f}^{n_{j,j}}x_{j}} A^{n_{j}}_{p_{j}} - I_{p_{j}x_{j}}\|_{\e} &\leq \sum_{m=0}^{n_{j}-1}K'_{\e}K''_{\e}e^{-(1-\e)^{2}m}\|(A^{n_{j,j}}_{x_{j}})^{-1}\|_{\e}  \| A^{m}_{p_{j}}\|_{\e} \\
&\leq \sum_{m=0}^{n_{j}-1}K'_{\e}K''_{\e}e^{-(1-\e)^{2}m}e^{-(1-\e)m} e^{(2+\e)m} \\
&\leq \sum_{m=0}^{n_{j}-1} K'_{\e}K''_{\e}e^{4\e m} \\
&\leq K_{\e}''' e^{4\e n_{j}},
\end{align*}
with the final constant $K_{\e}'''$ depending only on $\e$. The triangle inequality then gives
\[
\|(A^{n_{j,j}}_{x_{j}})^{-1} I_{\hat{f}^{n_{j}}p_{j}\hat{f}^{n_{j,j}}x_{j}} A^{n_{j}}_{p_{j}}\|_{\e} \leq \|I_{p_{j}x_{j}}\|_{\e}+K_{\e}''' e^{4\e n_{j}},
\]
We can then use the lower bound with conorms,
\[
\mathfrak{m}_{\e}((A^{n_{j,j}}_{x_{j}})^{-1})\mathfrak{m}_{\e}(I_{\hat{f}^{n_{j}}p_{j}\hat{f}^{n_{j,j}}x_{j}})\|A^{n_{j}}_{p_{j}}\|_{\e} \leq \|(A^{n_{j,j}}_{x_{j}})^{-1} I_{\hat{f}^{n_{j}}p_{j}\hat{f}^{n_{j,j}}x_{j}} A^{n_{j}}_{p_{j}}\|_{\e},
\]
Since $\mathfrak{m}_{\e}((A^{n_{j,j}}_{x_{j}})^{-1})^{-1} = \|A^{n_{j,j}}_{x_{j}}\|_{\e}$, we ultimately arrive at the estimate
\[
\frac{\|A^{n_{j}}_{p_{j}}\|_{\e}}{\|A^{n_{j,j}}_{x_{j}}\|_{\e}} \leq C\frac{\|A^{n_{j}}_{p_{j}}\|_{\e}}{\|A^{n_{j}}_{x_{j}}\|_{\e}} \leq C_{\e}\mathfrak{m}_{\e}(I_{\hat{f}^{n_{j}}p_{j}\hat{f}^{n_{j}}x_{j}})^{-1}(\|I_{p_{j}x_{j}}\|_{\e}+K_{\e}''' e^{4\e n_{j}}),
\] 
where $C$ is a uniform constant, since
\[
\frac{\|A^{n_{j}}_{x_{j}}\|}{\|A^{n_{j,j}}_{x_{j}}\|} \leq \mathfrak{m}(A^{n_{j,j}-n_{j}}_{x_{j}})^{-1} \leq C,
\]
since $|n_{j,j}-n_{j}| < C\delta$. 

Since the identifications $I$ have norm and conorm uniformly bounded below and above away from $0$ and $\infty$ by construction (with these bounds depending on the Riemannian structure $\|\cdot\|_{\e}$), by rearranging we can summarize our estimates as showing that there is a constant $H_{\e}$ depending only on $\e$ such that 
\begin{equation}\label{stepping stone}
\|A^{n_{j}}_{p_{j}}\|_{\e}\leq H_{\e}e^{4\e n_{j}}\|A^{n_{j}}_{x_{j}}\|_{\e}
\end{equation}
By using the linearity of the exterior product operator $\wedge^{i}$ and the basic bound $\|\wedge^{i}T\| \leq \|T\|^{i}$ for linear transformations $T$ we can repeat these same estimates (with appropriate adjustments) with $\wedge^{i}A$ in place of $A$ to obtain for each $1 \leq i \leq k$,
\begin{equation}\label{key dist estimate}
\|\wedge^{i}A^{n_{j}}_{p_{j}}\|_{\e} \leq  H_{\e}e^{4 k \e n_{j}}\|\wedge^{i}A^{n_{j}}_{p_{j}}\|_{\e} ,
\end{equation}
for a potentially larger constant $H_{\e}$. More precisely, the above estimate follows from the estimate for $1 \leq i \leq k$,
\[
\|\wedge^{i}A_{\hat{f}^{q_{m,j}}x_{j}}^{-1}\wedge^{i}I_{\hat{f}^{q_{m+1,j}}x_{j} \hat{f}^{m+1}p_{j}}\wedge^{i}A_{\hat{f}^{m}p_{j}}-\wedge^{i}I_{\hat{f}^{q_{m,j}}x_{j} \hat{f}^{m}p_{j}}\|_{\e} \leq (K_{\e}''')^{i}e^{4i\e n_{j}},
\]
and the same logic used to derive inequality \eqref{stepping stone}.


We are now able to put everything together to obtain a contradiction to the assumption that $\bar{D}\hat{f}^{t}$ does not admit a dominated splitting. We will do this by estimating the sum $\sum_{i=l+1}^{k}\la_{i}^{u}(\hat{f},\nu^{(p_{j})})$ from above and below for each $j \in \N$, letting $j \rightarrow \infty$, and then finally using the inequality \eqref{bad inequality} on $\e$. The estimate from below that we use is a simple consequence of inequality \eqref{ratio 2} for the measure $\nu^{(p_{j})}$, 
\begin{equation}\label{from below}
\sum_{i=l+1}^{k}\la_{i}^{u}(\hat{f},\nu^{(p_{j})}) \geq 2(k-l)(1-\e). 
\end{equation}
For an estimate from above we have
\begin{align*}
\sum_{i=l+1}^{k}\la_{i}^{u}(\hat{f},\nu^{(p_{j})}) &= \lim_{m \rightarrow\infty} \frac{1}{m s_{j}}\log \|\wedge^{(k-l)}\bar{D}\hat{f}^{m s_{j}}_{p_{j}}\|_{\e} \\
&\leq  \lim_{m \rightarrow\infty} \frac{1}{m s_{j}} \log  \|\wedge^{(k-l)}\bar{D}\hat{f}^{s_{j}}_{p_{j}}\|_{\e}^{m} \\
&= \frac{1}{s_{j}} \log  \|\wedge^{(k-l)}\bar{D}\hat{f}^{s_{j}}_{p_{j}}\|_{\e} \\
&\leq \frac{1}{s_{j}} \log \|\wedge^{(k-l)}\bar{D}\hat{f}^{n_{j}}_{p_{j}}\|_{\e} + \frac{1}{s_{j}}\log\|\wedge^{(k-l)} \bar{D}\hat{f}^{s_{j}-n_{j}}_{\hat{f}^{n_{j}}p_{j}}\|_{\e},
\end{align*}
for each $j \in \N$. Since this holds for all $j \in \N$, by combining this with \eqref{from below} we conclude that
\begin{equation}\label{contra inequality}
2(k-l)(1-\e) \leq \liminf_{j \rightarrow \infty} \frac{1}{s_{j}} \log \|\wedge^{(k-l)}\bar{D}\hat{f}^{n_{j}}_{p_{j}}\|_{\e} + \frac{1}{s_{j}}\log\|\wedge^{(k-l)} \bar{D}\hat{f}^{s_{j}-n_{j}}_{\hat{f}^{n_{j}}p_{j}}\|_{\e}
\end{equation}

We bound the first term on the right side of \eqref{contra inequality} using inequalities \eqref{failure} and \eqref{key dist estimate}. We start with some inequalities that hold for all $x \in M$ and times $t > 0$. For the starting Riemannian structure on $\bar{E}^{u}$, for all $x \in M$ we have for $t > 0$, 
\[
e^{(l+2(k-l) - \e)t} \leq \prod_{i=1}^{k}\sigma_{i}^{u}(\bar{D}\hat{f}_{x}^{t}) = \mathrm{Jac}(\bar{D}\hat{f}_{x}^{t}) \leq e^{(l+2(k-l) + \e)t},
\]
using inequality \eqref{Jacobian inequality} and recalling that $\check{\chi}^{u} = l + 2(k-l)$. Hence in the new Riemannian structure we have for $t > 0$,
\begin{equation}\label{new Jacobian inequality}
C_{\e}^{-k}e^{(l+2(k-l) - \e)t} \leq \prod_{i=1}^{k}\sigma_{i,\e}^{u}(\bar{D}\hat{f}^{t}_{x}) \leq C_{\e}^{k}e^{(l+2(k-l) + \e)t},
\end{equation}
with $C_{\e}$ the comparison constant from \eqref{riemann compare}. Since $\mathfrak{m}_{\e}(\bar{D}\hat{f}^{t}_{x}) \geq e^{(1-\e)t}$, we then have
\[
\prod_{i=1}^{k}\sigma_{i,\e}^{u}(\bar{D}\hat{f}^{t}_{x}) \geq e^{(1-\e)(l-1)t}\prod_{i=l}^{k}\sigma_{i,\e}^{u}(\bar{D}\hat{f}^{t}_{x}), 
\]
for all $t > 0$. Combining this with \eqref{new Jacobian inequality} and simplifying gives
\begin{equation}\label{sing unstable inequality}
\prod_{i=l}^{k}\sigma_{i,\e}^{u}(\bar{D}\hat{f}^{t}_{x}) \leq C_{\e}^{k}e^{(2(k-l)+1 + l\e)t},
\end{equation} 
for any $x \in M$ and $t > 0$. Since the singular values are nondecreasing in $i$, we can conclude from the left side of \eqref{new Jacobian inequality} that
\[
C_{\e}^{-k}e^{(l+2(k-l) - \e)t} \leq \sigma_{l+1,\e}^{u}(\bar{D}\hat{f}^{t}_{x})^{l+1}\prod_{i = l+2}^{k}\sigma_{i,\e}^{u}(\bar{D}\hat{f}^{t}_{x}) \leq \sigma_{l+1,\e}^{u}(\bar{D}\hat{f}^{t}_{x})^{l+1}e^{2(k-l-1)(1+\e)}.
\]
Note that the product on the right will be empty in the case $k = l+1$. The second inequality on the right follows from the norm bound $\|\bar{D}\hat{f}^{t}_{x}\|_{\e} \leq e^{(2+\e)t}$ for $t > 0$ derived from \eqref{unstable period inequality} in the new Riemannian structure $\|\cdot\|_{\e}$. By rearranging this inequality and simplifying the $\e$-term with a weaker inequality after taking the $(l+1)$th-root we arrive at the lower bound for $x \in M$ and $t > 0$,
\begin{equation}\label{sing lower inequality}
\sigma_{l+1,\e}^{u}(\bar{D}\hat{f}^{t}_{x}) \geq C_{\e}^{-k}e^{\left(\frac{l+2}{l+1}-2k\e\right)t}.
\end{equation}
Here we've also used that $C_{\e} \geq 1$.

We now specialize to $x = x_{j}$ and $t = n_{j}$ for a given $j \in \N$. Applying the assumed inequality \eqref{failure} (recalling that we chose $\kappa$ small enough that $\kappa < \e$) followed by the inequalities \eqref{sing unstable inequality} and finally \eqref{sing lower inequality} gives
\begin{align*}
\|\wedge^{(k-l)}\bar{D}\hat{f}^{n_{j}}_{x_{j}}\|_{\e} &= \prod_{i=l+1}^{k}\sigma_{i,\e}^{u}(\bar{D}\hat{f}^{n_{j}}_{x_{j}}) \\
&\leq \frac{e^{\kappa n_{j}}}{\sigma^{u}_{l+1,\e}(\bar{D}\hat{f}^{n_{j}}_{x_{j}})}\prod_{i=l}^{k}\sigma_{i,\e}^{u}(\bar{D}\hat{f}^{n_{j}}_{x_{j}}) \\
&\leq \frac{e^{\e n_{j}}}{\sigma^{u}_{l+1,\e}(\bar{D}\hat{f}^{n_{j}}_{x_{j}})}C_{\e}^{k}e^{(2(k-l)+1 + l\e)n_{j}} \\
&\leq C_{\e}^{2k}e^{\left( 2(k-l)+1 - \frac{l+2}{l+1} + (2k+l+1)\e\right)}.
\end{align*}
By \eqref{key dist estimate} we conclude that we have the estimate 
\begin{equation}\label{conclude estimate}
\|\wedge^{(k-l)}\bar{D}\hat{f}^{n_{j}}_{p_{j}}\|_{\e} \leq C_{\e}'e^{\left( 2(k-l)+1 - \frac{l+2}{l+1} + (6k+l+1)\e\right)},
\end{equation}
for a constant $C_{\e}'$ depending only on $\e$; this constant will be irrelevant in the inequality \eqref{contra inequality} once we take $j \rightarrow \infty$ (and therefore $s_{j} \rightarrow \infty$). We bound the second term of \eqref{contra inequality} using only the crude estimate
\[
\|\wedge^{(k-l)} \bar{D}\hat{f}^{s_{j}-n_{j}}_{\hat{f}^{n_{j}}p_{j}}\|_{\e} \leq e^{2(k-l)(1+\e)(s_{j}-n_{j})}, 
\]
from the right side of inequality \eqref{unstable period inequality} in the new Riemannian structure $\|\cdot\|_{\e}$. Putting this all together (and dispensing with the constant term $C_{\e}'$ as noted above), we conclude from \eqref{contra inequality} that 
\[
2(k-l)(1-\e) \leq \liminf_{j \rightarrow \infty} \frac{n_{j}}{s_{j}}\left( 2(k-l)+1 - \frac{l+2}{l+1} + (6k+l+1)\e\right) + \frac{(s_{j}-n_{j})}{s_{j}}(2(k-l)(1+\e)). 
\]
By construction $s_{j}$ is always within a uniformly bounded distance $\eta$ of $2n_{j}$, i.e., $|2n_{j}-s_{j}|$ is uniformly bounded independent of $j$. Thus the ratios $n_{j}/s_{j}$ and $(s_{j}-n_{j})/s_{j}$ both limit to $\frac{1}{2}$. It follows that 
\[
2(k-l)(1-\e) \leq \frac{1}{2}\left( 2(k-l)+1 - \frac{l+2}{l+1} + (6k+l+1)\e\right) + (k-l)(1+\e). 
\]
By canceling the $(k-l)$ terms from each side, collecting all the $\e$ terms on one side, and using both $l \leq k$ and $k \geq 1$, we can conclude from this inequality that
\[
\frac{l+2}{l+1} - 1 < 22k \e. 
\]
But this is precisely the inequality \eqref{bad inequality} we chose $\e$ to violate at the beginning of the proof. This produces a contradiction to our hypothesis that $\bar{D}\hat{f}^{t}$ does not admit a dominated splitting of index $l$. 

We have thus shown that $\hat{f}^{t}$ admits a $u$-splitting $\hat{E}^{u} = \hat{L}^{u} \oplus \hat{V}^{u}$. The inequality \eqref{ratio 2} implies that for every $\hat{f}^{t}$-invariant measure $\nu$ and all $1 \leq i \leq l$, 
\[
|\la_{i}(D\hat{f}|_{\hat{L}^{u}},\nu)- 1| < \frac{\e}{2}, 
\]
and for $l+1 \leq i \leq k$, 
\[
|\la_{i}(D\hat{f}|_{\hat{V}^{u}},\nu)- 2| < \frac{\e}{2}.
\]
Applying Lemma \ref{KS lemma} to the four subadditive sequences
\begin{enumerate}
\item $a_{n}(x) = \log \|D\hat{f}^{n}_{x}|_{\hat{V}^{u}_{x}}\| - (2 + \e)n$,
\item  $a_{n}(x) = -\log \mathfrak{m}(D\hat{f}^{n}_{x}|_{\hat{V}^{u}_{x}}) + (2- \e)n$,
\item $a_{n}(x) =  \log \|D\hat{f}^{n}_{x}|_{\hat{L}^{u}_{x}}\|  - (1  +\e)n$,
\item $a_{n}(x) = -\log \mathfrak{m}(D\hat{f}^{n}_{x}|_{\hat{L}^{u}_{x}}) + (1 - \e)n$,
\end{enumerate} 
gives inequalities \eqref{first periodic inequality} and \eqref{second periodic inequality} in the statement of the proposition with $2\e$ in place of $\e$; since $\e$ can be taken to be arbitrarily small it is harmless to replace $2\e$ with $\e$ in the notation. The inequalities for $\hat{f}^{-t}$ with $\hat{L}^{s}$ and $\hat{V}^{s}$ follow by applying the same calculation to $\hat{f}^{-t}$. 
\end{proof}

\begin{proof}[Theorem \ref{core theorem} $\Rightarrow$ Theorem \ref{periodic rigid}]:
For every $\e > 0$, we obtain from Proposition \ref{periodic dom split} a smooth time change $\hat{f}^{t}$ with $u$-splitting $\hat{E}^{u} = \hat{L}^{u} \oplus \hat{V}^{u}$ and $s$-splitting $\hat{E}^{s} = \hat{L}^{s} \oplus \hat{V}^{s}$ satisfying the inequalities \eqref{first periodic inequality} and \eqref{second periodic inequality} of Proposition \ref{periodic dom split}. The same calculations as those in Lemma \ref{neighborhood lemma} show that for $\e$ small enough $\hat{f}^{t}$ satisfies (1)-(4) of Theorem \ref{core theorem} with $\beta = \frac{1}{5}$; this is because as $\e \rightarrow 0$ the exponents in the growth inequalities for $D\hat{f}^{t}$ derived from Proposition \ref{periodic dom split} converge to the corresponding exponents in the growth inequalities for $Dg^{t}_{X}$. 

We conclude by Theorem \ref{core theorem} that $g^{t}_{X}$ is smoothly orbit equivalent to $\hat{f}^{t}$, with the orbit equivalence $\hat{\varphi}$ being flow related to $\varphi$. Since $\hat{f}^{t}$ is a smooth time change of $f^{t}$, this implies that $g^{t}_{X}$ is smoothly orbit equivalent to  $f^{t}$ by $\hat{\varphi}$.
\end{proof}

\begin{rem}\label{smoothness degree}
Smoothness of the flow $f^{t}$ is only assumed in Theorems \ref{dyn rigidity measure}, \ref{dyn rigidity pinching}, and \ref{periodic rigid} to simplify the exposition. A finite degree of smoothness suffices for the proofs; a careful examination reveals that these theorems hold under the hypothesis that $f^{t}$ is $C^r$ for some $r \geq 4$. In this case one ultimately concludes that the orbit equivalence $\varphi$ is $C^{r-1}$. 
\end{rem}

\subsection{The geometric rigidity theorems}\label{subsec:geom rigidity} We now show how Theorems \ref{periodic loc symm}, \ref{loc symmetric rigidity}, and \ref{geom pinching rigidity} follow from their respective dynamical analogues \ref{periodic rigid}, \ref{dyn rigidity measure}, and \ref{dyn rigidity pinching}.

We will first assume Theorem \ref{dyn rigidity measure} and use this to establish Theorem \ref{loc symmetric rigidity}.  Let $S$ denote the common underlying smooth manifold for $X$ and $Y$. $T^{1}X$ and $T^{1}Y$ are smooth submanifolds of $TS$ which will be $C^2$ close and the flows $g^{t}_{X}$ and $g^{t}_{Y}$ will be $C^1$ close if the Riemannian metrics for $X$ and $Y$ are $C^2$ close.  Let $\pi: T^{1}Y \rightarrow T^{1}X$ denote the smooth diffeomorphism given by radial projection in each of the fibers of $TS$. We now consider the smooth flow $f^{t} = \pi \circ g^{t}_{Y} \circ \pi^{-1}$ on $T^{1}X$. We define the neighborhood $\mathcal{U}_{X}$ in Theorem \ref{loc symmetric rigidity} by declaring $Y \in \mathcal{U}_{X}$ if $f^{t} \in \mathcal{V}_{X}$, where $\mathcal{V}_{X}$ is defined as in Theorem \ref{dyn rigidity measure}.

Assuming Theorem \ref{dyn rigidity measure}, we have the following.

\begin{lem}\label{implication}
Let $Y \in \mathcal{U}_{X}$ and suppose that $\vec{\la}(g_{Y},\mu_{Y}) = \vec{\la}(g_{X})$. Then $g^{t}_{X}$ is $C^{\infty}$ orbit equivalent to $g^{t}_{Y}$. 
\end{lem}

\begin{proof}
The flow $f^{t} = \pi \circ g^{t}_{Y} \circ \pi^{-1}$ belongs to $\mathcal{V}_{X}$ by hypothesis. Recalling that $\mu_{Y}=\mu_{g_{Y}} = \mu_{g_{Y}^{-1}}$, by Proposition \ref{invariance hor measure} we have $\pi_{*}\mu_{Y} = \mu_{f}$ and $\pi_{*}\mu_{Y} = \mu_{f^{-1}}$, from which it follows that $\mu_{f} = \mu_{f^{-1}}$. Furthermore $\vec{\la}(g_{Y},\mu_{Y}) = \vec{\la}(f,\mu_{f})$. It follows that
\[
\vec{\la}(f,\mu_{f}) = \vec{\la}(g_{X}),
\]
and therefore $g^{t}_{X}$ is $C^{\infty}$ orbit equivalent to $f^{t}$ by Theorem \ref{dyn rigidity measure}. Since $f^{t}$ is smoothly conjugate to $g^{t}_{Y}$, the conclusion follows. 
\end{proof}

By a lemma of Hamenst\"adt \cite[Lemma 4.5]{Ham99}, if $g^{t}_{X}$ is $C^{\infty}$ orbit equivalent to $g^{t}_{Y}$ then there is a constant $c > 0$ such that $g^{ct}_{X}$ is $C^{1}$ conjugate to $g^{t}_{Y}$, which implies that $g^{t}_{X_{c}}$ is $C^{1}$ conjugate to $g^{t}_{Y}$. To conclude the proof we apply the following theorem, which is a consequence of the minimal entropy rigidity theorem of Besson, Courtois, and Gallot.


\begin{thm}\cite[Theorem 1.3]{BCG2}\label{minimal rigidity}
Let $X$ be a closed negatively curved locally symmetric space of dimension $n \geq 3$ and let $Y$ be an $n$-dimensional closed negatively curved Riemannian manifold. Suppose that the geodesic flows $g^{t}_{X}$ and $g^{t}_{Y}$ are $C^1$  conjugate. Then $Y$ is isometric to $X$. 
\end{thm}

We conclude from Theorem \ref{minimal rigidity} that $Y$ is isometric to $X_{c}$. We then have 
\[
c\vec{\la}(g_{X}) = \vec{\la}(g_{X_{c}}) = \vec{\la}(g_{Y},\mu_{Y}) = \vec{\la}(g_{X}), 
\]
from which we conclude that $c = 1$. Therefore $Y$ is isometric to $X$. 

For Theorem \ref{geom pinching rigidity} we take $\mathcal{U}_{X}$ to be the same as in Theorem \ref{loc symmetric rigidity}. We will deduce Theorem \ref{geom pinching rigidity} directly from Theorem \ref{loc symmetric rigidity}. The hypotheses of Theorem \ref{geom pinching rigidity} imply that $f^{t} = \pi \circ g^{t}_{Y} \circ \pi^{-1}$ satisfies the hypotheses of Theorem \ref{dyn rigidity pinching}. We conclude from the argument directly below the statement of Theorem \ref{dyn rigidity pinching} that there is a constant $c > 0$ such that 
\[
\vec{\la}(g_{Y},\mu_{Y}) = c\vec{\la}(g_{X}) = \vec{\la}(g_{X_{c}}).   
\]
Applying Theorem \ref{loc symmetric rigidity}, we conclude that $Y$ is isometric to $X_{c}$, hence homothetic to $X$. 


Lastly we show that Theorem \ref{periodic rigid} implies Theorem \ref{periodic loc symm}. The hypothesis that $X$ is homotopy equivalent to $Y$ implies the existence of an orbit equivalence $\varphi$ from $g^{t}_{X}$ to $g^{t}_{Y}$ by a well-known argument. Since we will require an explicit description of this orbit equivalence in Section \ref{mostow mod}, we recall the construction here. 

Given two closed negatively curved manifolds $X$ and $Y$, a homotopy equivalence $\kappa: X \rightarrow Y$ lifts to a \emph{quasi-isometry} $\t{\kappa}:\t{X} \rightarrow \t{Y}$, in the sense that there are constants $K \geq 1$, $C > 0$ such that for every $x$, $y \in \t{X}$, 
\[
K^{-1}d_{\t{Y}}(\t{\kappa}(x),\t{\kappa}(y)) - C \leq d_{\t{X}}(x,y) \leq Kd_{\t{Y}}(\t{\kappa}(x),\t{\kappa}(y)) + C 
\]
For $v \in T^{1}\t{X}$ let $\gamma_{v}$ denote the unit-speed geodesic with $\gamma_{v}'(0) = v$. By the Morse-Mostow lemma \cite{Bou18}, there is a unique oriented geodesic $\sigma$ in $\t{Y}$ that is a bounded distance from $\t{\kappa}(\gamma_{v})$.

Fix a basepoint $x \in \t{X}$. Let $p_{v}$ denote the orthogonal projection of $x$ onto $\gamma_{v}$ and let $s_{v}$ denote the unique time such that $\gamma_{v}(s_{v}) = p_{v}$. Let $q_{v}$ denote the orthogonal projection of $\t{\kappa}(x)$ onto $\sigma$. We take unit speed parametrization of $\sigma$ such that $\sigma(0) = q_{v}$ and set $\t{\varphi}(v) = \sigma'(s_{v})$. Then $\t{\varphi}: T^{1}\t{X} \rightarrow T^{1}\t{Y}$ defines a homeomorphism giving an orbit equivalence between the geodesic flows of $\t{X}$ and $\t{Y}$. This homeomorphism is equivariant with respect to the action of $\pi_{1}(X)$ and $\pi_{1}(Y)$ on $T^{1}\t{X}$ and $T^{1}\t{Y}$ respectively and thus descends to an orbit equivalence $\varphi: T^{1}X \rightarrow T^{1}Y$ from $g^{t}_{X}$ to $g^{t}_{Y}$. 

Under the hypotheses of Theorem \ref{periodic loc symm} we thus have an orbit equivalence $\varphi:T^{1}X \rightarrow T^{1}Y$ from $g^{t}_{X}$ to $g^{t}_{Y}$. Taking $M = T^{1}Y$ and $f^{t} = g^{t}_{Y}$ in Theorem \ref{periodic rigid}, we conclude that $g^{t}_{X}$ is $C^{\infty}$ orbit equivalent to $g^{t}_{Y}$. The same argument with Theorem \ref{minimal rigidity} that was used above to show that Theorem \ref{dyn rigidity measure} implies Theorem \ref{loc symmetric rigidity} then establishes that $Y$ is homothetic to $X$. In particular, the Lyapunov spectrum of $g^{t}_{Y}$ with respect to an invariant ergodic probability measure $\nu$ is independent of $\nu$, so if $Y$ is isometric to $X_{c}$, $c > 0$, and there is any periodic point $p$ such that $\vec{\la}(g_{Y},\nu^{(p)}) = \vec{\la}(g_{X})$ then 
\[
c\vec{\la}(g_{X}) = \vec{\la}(g_{X_{c}}) = \vec{\la}(g_{Y},\nu^{(p)}) = \vec{\la}(g_{X}),
\]
from which we conclude that $c = 1$ and therefore $Y$ is isometric to $X$.

\section{Quotient normal forms}\label{sec:uniform}

\subsection{Holonomies for linear cocycles} 
Let $M$ be a $C^{r+2}$ closed Riemannian manifold and let $\mathcal{W}$ be a foliation of $M$ with uniformly $C^{r+2}$ leaves, $r \geq 1$. Let $0 < \alpha \leq 1^{-}$. We let $f: M \rightarrow M$ be $C^{\alpha}$ and uniformly $C^{r+2}$ along $\W$, with $\W$ as an expanding foliation for $f$. Let $E \subseteq TM$ be a $C^{\alpha}$ vector bundle over $M$ that is a subbundle of $TM$ and let $A$ be a $C^{\alpha}$ linear cocycle on $E$. We note that with minor modification the claims of this subsection apply equally well to the case that $E$ is a subbundle of a trivial bundle $M \times \R^{q}$.

\begin{defn}[Fiber bunching along a foliation]\label{second fiber bunched definition}We say that $A$ is \emph{$\alpha$-fiber bunched along $\mathcal{W}$} if there is a positive integer $N$ and a constant $0 < b < 1$ such that 
\begin{equation}\label{fiber bunching}
\|A^{N}_{x}\|\|(A^{N}_{x})^{-1}\|\mathfrak{m}(Df^{N}_{x}|T\mathcal{W}_{x})^{-\alpha}< b.
\end{equation}
\end{defn}

Our  definition of $\alpha$-fiber bunching for Anosov flows in Definition \ref{first fiber bunched definition} corresponds to the linear cocycle $A^{t}$ over the Anosov flow $f^{t}:M \rightarrow M$ being $\alpha$-fiber bunched along $\W^{u}$ \emph{and} the linear cocycle $A^{-t}$ over $f^{-t}$ being $\alpha$-fiber bunched along $\W^{s}$, which is the unstable foliation for $f^{-t}$. 

Note that inequality \eqref{fiber bunching} becomes more difficult to satisfy as $\alpha$ decreases. This condition is necessary to produce a continuous, $A$-equivariant collection of identifications $H_{xy}: E_{x} \rightarrow E_{y}$, $y \in \mathcal{W}(x)$, of the fibers of $E$ along $\mathcal{W}$, which we refer to as \emph{holonomies}. The construction of these maps is standard; while we work in a slightly different setting here, the required computations are exactly the same as those done by Kalinin-Sadovskaya \cite{KS}, and their proof can be adapted without modification to give the proposition below. As in \cite[Section 2.2]{KS}, we write $\{I_{xy}\}_{x,y\in \Delta}$ for a collection of identifications of $E_{x}$ with $E_{y}$ when $x$ is close to $y$ that vary in an $\alpha$-H\"older manner. Here $\Delta \subseteq M \times M$ denotes a neighborhood of the diagonal. 

\begin{prop}[\cite{KS}, Proposition 4.2]\label{existuhol}
Let $A$ be a $C^{\alpha}$ linear cocycle over $f$ on $E$ that is $\alpha$-fiber bunched along $\mathcal{W}$. Then for each $x \in M$, $y \in \mathcal{W}(x)$ there is a linear isomorphism $H_{xy}: E_{x} \rightarrow E_{y}$ with the following properties.
\begin{enumerate}
\item For $x \in M$ and $y$,$z \in \mathcal{W}(x)$ we have $H_{xx} = I_{x}$ and $H_{yz} \circ H_{xy} = H_{xz}$.
\item For every $x \in M$, $y \in \mathcal{W}(x)$, and $n \in \Z$ we have
\[
H_{xy} = (A^{n}_{y})^{-1} \circ H_{f^{n}x f^{n}y} \circ A^{n}_{x}.
\] 
\item There is an $r > 0$ and a constant $C > 0$ such that for $x \in M$ and $y \in \W(x)$ with $d(x,y) \leq r$ we have 
\[
\|H_{xy} - I_{xy}\| \leq C d(x,y)^{\alpha}. 
\]
\end{enumerate}
The family of linear maps $H$ satisfying the above properties is unique.
\end{prop}

We refer to the maps $H_{xy}$ as \emph{holonomies}. We will require a  more explicit statement of the uniqueness of the holonomy maps. The following may be extracted from the proofs of \cite{KS} (or alternatively derived directly from the statements of Proposition \ref{existuhol} above), 

\begin{prop}\label{unique uhol}
Let $A$ be a $C^{\alpha}$ linear cocycle over $f$ on $E$ that is $\alpha$-fiber bunched along $\mathcal{W}$. Let $x \in M$, $y \in \mathcal{W}(x)$ be given. Let $i_{n} \rightarrow \infty$ be a sequence of integers and let $J_{n}$ be a sequence of linear maps $J_{n}: E_{f^{-i_{n}}x} \rightarrow E_{f^{-i_{n}}y}$ satisfying 
\[
\|J_{n}-I_{f^{-i_{n}}x f^{-i_{n}y}}\| \leq C d(f^{-i_{n}}x,f^{-i_{n}}y)^{\alpha},
\]
for some constant $C > 0$. Then $A^{i_{n}}_{f^{-i_{n}}y} \circ J_{n} \circ A^{-i_{n}}_{x} \rightarrow H_{xy}$ as $t_{n} \rightarrow \infty$.
\end{prop}

From this point forward we will assume that both $E$ and $A$ are uniformly $C^{r+1}$ along $\W$ and that $A$ is $1$-fiber bunched. We will further assume that $E \subseteq T\W$. It is then possible to establish uniform $C^r$ regularity of the holonomy maps. We will do this by constructing an $A$-invariant flat connection $\nabla$ on $E$ along $\W$ that is uniformly $C^r$ along $\W$ and whose parallel transport maps on $E$ are the holonomies $H$.  

We denote by $\Gamma^{s}_{\mathcal{W}}(E)$ the space of sections of $E$ over $M$ that are uniformly $C^{s}$ along $\W$, $s$ an integer satisfying $0 \leq s \leq r+1$. A connection $\nabla$ on $E$ along $\W$ is defined to be a linear map 
\[
\nabla: \Gamma_{\mathcal{W}}^{1}(E) \rightarrow \Gamma_{\mathcal{W}}^{0}(T^{*}\W \otimes E),
\]
which satisfies the \emph{Leibniz} rule for functions $g: M \rightarrow \R$ that are uniformly $C^1$ along $\W$ and sections $X \in \Gamma_{\mathcal{W}}^{1}(E)$,
\[
\nabla(gX) = g \nabla(X) + dg \otimes X. 
\]
For $0 \leq s \leq r$ we say that $\nabla$ is \emph{uniformly $C^s$ along $\W$} if $\nabla(X)$ is uniformly $C^s$ along $\W$ whenever $X$ is uniformly $C^{s+1}$ along $\W$, or equivalently, 
\[
\nabla(\Gamma_{\mathcal{W}}^{s+1}(E)) \subseteq \Gamma_{\mathcal{W}}^{s}(T^{*}\W \otimes E).
\]
If $\nabla$ is uniformly $C^s$ along $\W$ with $s \geq 1$ then we can define the \emph{curvature tensor} $R^{\nabla}$ of $\nabla$ $X$, $Y \in \Gamma_{\W}^{2}(T\W)$ and $Z \in \Gamma_{\W}^{2}(E)$, 
\[
R^{\nabla}(X,Y)Z = \nabla_{X}\nabla_{Y}Z - \nabla_{Y}\nabla_{X}Z - \nabla_{[X,Y]}Z.
\]
This defines a section of $\mathrm{Hom}(T\W \otimes T\W  \otimes E, E)$ that is uniformly $C^{0}$ along $\W$. We say that $\nabla$ is \emph{flat} if $R^{\nabla} \equiv 0$.
\begin{rem}\label{simple flat}
Since the leaves $\mathcal{W}(x)$ for $x \in M$ are each diffeomorphic to $\R^{\dim E}$ (see for instance \cite[Theorem 6.1(d)]{HPS}) flatness of $\nabla$ implies that for any $C^1$ curve $\gamma$ joining $x \in M$ to $y \in \mathcal{W}(x)$ the induced parallel transport map $P_{\gamma}:E_{x} \rightarrow E_{y}$ does not depend on $\gamma$, where $P_{\gamma}$ is defined for $v \in E_{x}$ by taking $P_{\gamma}(v) \in E_{y}$ to be the value at $y$ of the unique vector field $V$ on $\gamma$ with $V(x) = v$ and $\nabla_{\gamma}V \equiv 0$.
\end{rem}

Given a connection $\nabla$ on $E$ that is $C^s$ along $\mathcal{W}$ for some $1 \leq s \leq r$, we define a new connection $A \cdot \nabla$ on $E$ by setting, for $X \in \Gamma^{1}_{\mathcal{W}}(E)$, 
\[
(A\cdot \nabla)(X) = ((Df^{*})^{-1} \otimes A^{-1})(\nabla(A(X)).
\]
Then $A \cdot \nabla$ defines a connection on $E$ that is uniformly $C^r$ along $\mathcal{W}$ as well. We say that $\nabla$ is \emph{$A$-invariant} if $A \cdot \nabla = \nabla$. 


If $\nabla_{1}$, $\nabla_{2}$ are two uniformly $C^r$ connections along $\mathcal{W}$ on $E$, their difference $S:=\nabla_{1}-\nabla_{2}$ is linear over functions $\psi: M \rightarrow \R$ that are uniformly $C^{1}$ along $\mathcal{W}$, and thus can be viewed as a section of the bundle $\mathrm{Hom}(E,T^{*}\mathcal{W} \otimes E)$, which is also uniformly $C^r$ along $\mathcal{W}$. Conversely, if $\nabla$ is a uniformly $C^r$ connection along $\mathcal{W}$ on $E$ and $S \in  \Gamma^{r}_{\mathcal{W}}(\mathrm{Hom}(E,T^{*}\mathcal{W} \otimes E))$ then $\nabla + S$ defines a uniformly $C^r$ connection along $\mathcal{W}$.

\begin{prop}\label{holonomy connect}
Suppose that $A$ is $1$-fiber bunched along $\mathcal{W}$. Then there exists a unique $A$-invariant connection $\nabla$ on $E$ that is uniformly $C^r$ along $\mathcal{W}$ which is flat on each leaf of $\W$. The parallel transport maps of $\nabla$ along  $\mathcal{W}$ coincide with the holonomy maps of Proposition \ref{existuhol}. 
\end{prop}

\begin{proof}
Let $\t{\nabla}$ be the connection on $T\W$ along $\W$ defined by taking the Levi-Civita connection of the Riemannian metrics on the leaves $\W(x)$ of $\W$ induced by the Riemannian metric on $M$. Then $\t{\nabla}$ is uniformly $C^{r}$ along $\W$ since $\W$ has uniformly $C^{r+2}$ leaves. Using the assumption that $E \subseteq T\W$, let $\nabla$ be the connection on $E$ along $\W$ defined by differentiating sections of $E$ using $\t{\nabla}$ and then linearly projecting the resulting section onto $E$. Then $\nabla$ is uniformly $C^r$ along $\W$ since $E$ is uniformly $C^{r+1}$ along $\W$. Let $k = \dim E$ and $m = \dim \W$. Let $N$ and $\theta$ be as in the fiber bunching inequality \eqref{fiber bunching} and set $S:= A^{N} \cdot \nabla - \nabla$.  Then $S$ defines a section of $\mathrm{Hom}(E,T^{*}M \otimes E)$ which is uniformly $C^{r}$ along $\mathcal{W}$. In particular $S$ is a continuous section of this bundle, so that $K_{0}= \sup_{x \in M} \|S_{x}\|$ is finite. For $j \geq 1$ and $v \in E$ we have 
\begin{align*}
\|A^{jN} \cdot S\| &\leq kl \|Df^{-jN}_{f^{jN}x}|_{E}\| \|A^{-jN}_{x}\| \|S_{f^{jN}x}\|\|A^{jN}_{x}\| \\
&\leq kl K_{0}\prod_{i=0}^{j-1}\|(Df^{N}_{f^{iN}x}|_{E})^{-1}\| \|(A^{N}_{f^{iN}x})^{-1}\| \|A^{N}_{f^{iN}x}\| \\
&< C_{0} b^{j}, 
\end{align*}
for a uniform constant $C_{0}$, by the fiber bunching inequality. Hence the infinite series $\sum_{j=0}^{\infty} A^{jN} \cdot S$ converges uniformly to a $C^0$ section $T$ of $\mathrm{Hom}(E,T^{*}\W \otimes E)$. 

We will show that $T$ is uniformly $C^r$ along $\mathcal{W}$. By the compactness of $M$ we can cover $M$ by finitely many open sets $U_{1},\dots,U_{p}$ such that each $U_{q}$ belongs to a  foliation chart for $\mathcal{W}$ defining it as a foliation with uniformly $C^{r+2}$ leaves as in Definition \ref{uniform foliation} and such that both the tangent bundle $T\mathcal{W}$ and $E$ can be trivialized in a uniformly $C^{r+1}$ fashion over each $U_{q}$. We then consider the neighborhood $B(s) \subset T\W$ of the zero section in $T\W$ whose fiber over each $x \in M$ is the ball $B_{x}(s)$ of radius $s$ inside of $T\W_{x}$, with $s$ chosen small enough that for each $x \in M$ there is some $U_{q}$ such that $B(x,s) \subset U_{q}$. We identify $B_{x}(s)$ with the ball $B(x,s)$ centered at $x$ of radius $s$ inside of $\W(x)$ via the $C^{r+1}$ Riemannian exponential map $\exp_{x}:B_{x}(s) \rightarrow B(x,s)$. We consider this as a coordinate chart on $B(x,s)$ and pull back all objects under consideration from $B(x,s)$ to $B_{x}(s)$ via $\exp_{x}$. 

Let $x \in M$ be fixed. Using our trivializations, on each $U_{q}$ we can find a family of sections $O_{q}$ of the orthonormal frame bundle of $T\W|_{U_{q}}$ that are uniformly $C^{r+1}$ along $\W$, and the same for $E|_{U_{q}}$. For a given $n \geq 0$, if we are given $x \in M$, by our choice of $s$ for each $0 \leq i \leq n$ we can find an integer $q(i)$ such that $B(f^{iN}x,s) \subset U_{q}$. We can thus isometrically identify each of the tangent spaces $T\W_{f^{iN}x}$ with $\R^{m}$ using the orthonormal frames constructed previously, and can identify each ball $B(f^{iN}x,s)$ with the ball $B_{\R^{m}}(0,s)$ of radius $s$ centered at $0$ in $\R^{m}$. The maps $f^{N}$ are defined in these coordinates at $f^{iN}x$, $0 \leq i \leq n-1$ on a sufficiently small neighborhood $V$ of $0$ of uniform size independent of $i$. We write $F_{i}$ for the map $f^{N}$ from this neighborhood of $f^{iN}x$ to $B_{\R^{m}}(0,s)$ in these coordinates. For a \emph{fixed} $n \in \N$, the composition $F^{n} = F_{n-1} \circ \dots \circ F_{0}$ will be well defined on a small enough neighborhood $V_{n} \subset V$ of $0$ in $B_{\R^{m}}(0,s)$. We write $GL_{m}(\R)$ for the group of invertible linear transformations of $\R^{m}$, and write $D_{i}: V \rightarrow GL(\R^{m})$ for the $C^{r+1}$ map whose image is the derivative matrix $D_{y}F_{i}$ at $y \in V$. We then write for $y \in V$, 
\[
D^{n}_{0}(y) = D_{y}F^{n} = D_{n-1}(f^{n-1}y) \cdots  D_{0}(y). 
\]

Since $E\subseteq T\W$ is a subbundle, the $C^{r}$ bundle map $A^{N}:E\rightarrow E$ then determines a $C^r$ section $A_{i}: V \rightarrow GL_{k}(\R)$ in these coordinates. We write $A^{n}_{0}(y) = A_{n-1}(f^{n-1}y)  \cdots A_{0}(y)$. The fiber bunching inequality translates into these coordinates as, for $y \in U$,
\[
\|D_{i}(y)^{-1}\| \|A_{i}(y)^{-1}\| \|A_{i}(y)\| < b.  
\]

Let $S_{n}$ denote the $C^r$ section of $\mathrm{Hom}(\R^{k}, \R^{m} \otimes \R^{k})$ corresponding to $S$ near $f^{nN}x$ in this coordinate system. Then the section $A^{(n-1) N} \cdot S$ near $x$ can be written in these charts as
\[
T_{n}(y) = ((D^{n}_{0}(y))^{-1} \otimes (A^{n}_{0}(y))^{-1})S_{n}(y)A^{n}_{0}(y). 
\]

We now recall the $C^r$ norm on $C^r$ functions $g: Z \rightarrow \R^{j'}$, where $Z \subseteq \R^{j}$ is an open subset of a Euclidean space of dimensions $j$ being mapped into one of dimension $j'$.  Let $I = (i_{1},\dots,i_{n_{1}})$ be a multi-index, with $i_{j} \geq 0$ being nonnegative integers. We set $|I| = \sum_{j=1}^{n_{1}}i_{j}$. For a multi-index $I$ with $|I| = r$, we define
\[
\frac{\p^{r} g}{\p x^{I}} = \frac{\p^{r}g}{\p x_{1}^{i_{1}}\dots \p x_{n_{1}}^{i_{n_{1}}}}.
\]
The $C^r$ norm of $g$ on $U$ is then defined by
\[
\|g\|_{r} = \sup_{|I| \leq r} \sup_{y \in Z} \left\|\frac{\p^{r}g}{\p x^{I}}(y)\right\|.
\]
We omit the domain $Z$ from the notation as it will be understood from the context. 

Applying the Leibniz rule, we have a bound on the $C^r$ norm in terms of the product of $C^r$ norms measured on $V_{n}$, 
\[
\|T_{n}\|_{r} \leq C\|(D^{n}_{0})^{-1}\|_{r}\|(A^{n}_{0})^{-1}\|_{r}\|S_{n}\|_{r}\|A^{n}_{0}\|_{r},
\]
with the constant $C = C_{r,k,m}$ depending only on $r$, $k$, and $m$. Since all of our coordinate charts and trivializations are uniformly $C^r$ along $\W$ and $S$ is uniformly $C^r$ along $\W$ as well, by the compactness of $M$ we obtain a uniform bound $\|S_{n}\|_{r} \leq C$ independent of $i$. Thus we have
\[
\|T_{n}\|_{r} \leq C\|(D^{n}_{0})^{-1}\|_{r}\|(A^{n}_{0})^{-1}\|_{r}\|A^{n}_{0}\|_{r},
\]
for a constant $C > 0$. 

Our key estimate is given in the following lemma. For a natural number $n$ we write $[0,n]$ for the set of integers $j$ satisfying $0 \leq j \leq i$. For $J \subseteq [0,n]$ we write $|J|$ for the cardinality of $J$. 

\begin{lem}\label{deriv product}
Let $r \geq 1$, $k \geq 1$, $n \geq 1$, and $m \geq 1$ be given with $n > r$. Let $A_{0},\dots,A_{n-1}$ be $GL_{m}(\R)$-valued $C^r$ functions defined on an open set $ \subseteq \R^{m}$. Suppose that there is a constant $K > 0$ such that $\|A_{j}\|_{r} \leq K$ for each $0 \leq j \leq n-1$, with the $C^r$ norm being measured on $U$. Then there is a constant $C_{m,r}$ depending only on $m$ and $r$ for which we have
\[
\left\| A_{n-1}\cdots A_{0}\right\|_{r} \leq C_{m,r}K^{r} \sup_{\substack{J \subseteq [0,n-1] \\ |J| = n-r}} \prod_{J} \|A_{j}\|_{0}.
\]
\end{lem}

\begin{proof}
Let $A_{0},\dots,A_{n-1}$ be $GL_{m}(\R)$-valued $C^r$ functions defined on $U$ with $\|A_{j}\|_{r} \leq K$ for each $0 \leq j \leq n-1$. For a multi-index $I = (i_{1},\dots,i_{m})$, set $A_{j}^{(I)} := \frac{\p}{\p x^{I}}A_{j}$. Then 
\begin{equation}\label{leibniz}
\frac{\p}{\p x^{I}}A_{n-1}\cdots A_{0} = \sum_{I_{0}+\dots + I_{n-1} = I} A_{n-1}^{(I_{n-1})}\cdots A_{0}^{(I_{0})},
\end{equation}
with the sum being taken over all multi-indices $I_{0},\dots,I_{n-1}$ which sum to $I$. Since $|I| \leq r$, we have that $I_{j} = (0,\dots,0)$ for all but at most $r$ indices $j$ for each individual  term of the sum on the right side of equation \eqref{leibniz}. From this it follows that for each individual product on the right-hand side, using the bound $\|A_{j}^{(I_{j})}\|_{0} \leq K$ for those $I_{j} \neq (0,\dots,0)$, 
\[
\left\|A_{n-1}^{(I_{n-1})}\cdots A_{0}^{(I_{0})}\right\|_{0} \leq  \sup_{\substack{J \subseteq [0,n-1] \\ |J| \geq n-r}} K^{|J|}\prod_{J} \|A_{j}\|_{0}.
\]
Since $\|A_{j}\|_{0} \leq K$ as well, this implies that
\[
\left\|A_{n-1}^{(I_{n-1})}\cdots A_{0}^{(I_{0})}\right\|_{0} \leq  K^{r}\sup_{\substack{J \subseteq [0,n-1] \\ |J| = n-r}} \prod_{J} \|A_{j}\|_{0}.
\]

Hence we conclude by the triangle inequality that 
\[
\left\|\frac{\p}{\p x^{I}}A_{n-1}\cdots A_{0}\right\|_{0} \leq C_{m,r}K^{r} \sup_{\substack{J \subseteq [0,n-1] \\ |J| = n-r}} \prod_{J} \|A_{j}\|_{0},
\]
with $C_{m,r}$ depending only on $k$ and $r$. This implies that
\[
\left\| A_{n-1}\cdots A_{0}\right\|_{r} \leq C_{m,r}K^{r} \sup_{\substack{J \subseteq [0,n-1] \\ |J| = n-r}} \prod_{J} \|A_{j}\|_{0}.
\]
\end{proof}

Using Lemma \ref{deriv product}, we obtain the upper bound for $i > r$, with a uniform constant $K$, 
\begin{equation}\label{triple count}
\|T_{n}\|_{r} \leq K\sup_{\substack{J_{1},J_{2},J_{3} \subset [0,n-1] \\ |J_{q}| = n-r}} \prod_{j \in J_{1}}\|D_{j}^{-1}\|_{0}  \prod_{j \in J_{2}}\|A_{j}^{-1}\|_{0}\prod_{j \in J_{3}}\|A_{j}\|_{0},
\end{equation}
with the $C^0$ norms also being measured on the neighborhood $U$. Assume now that $n \geq 3r$ and set $J'=J_{1} \cap J_{2} \cap J_{3}$.  Then $|J_{1} \cap J_{2} \cap J_{3}| \geq n-3r$ since $|J_{i}| = n-r$ for each $i$. The complements $\hat{J}_{i} = J_{i} \backslash J'$ then satisfy $|\hat{J}_{i}| \leq 2r$ for each $i$. Note that
\[
\sup_{\hat{J}_{1},\hat{J}_{2},\hat{J}_{3}} \prod_{j \in \hat{J}_{1}}\|D_{j}^{-1}\|_{0}  \prod_{j \in \hat{J}_{2}}\|A_{j}^{-1}\|_{0}\prod_{\hat{J} \in J_{3}}\|A_{j}\|_{0} \leq K'
\]
for some uniform constant $K'$ that depends only on $r$ and the uniform upper bounds on $\|D_{j}^{-1}\|_{0}$, $\|A_{j}^{-1}\|_{0}$ and $\|A_{j}\|_{0}$ that we obtain from the compactness of $M$ and the facts that $f$ is uniformly $C^{r+2}$ along $\W$ and $A$ is uniformly $C^{r+1}$ along $\W$. 

We can thus remove the integers in each $\hat{J}_{i}$ from the product $\eqref{triple count}$ at the cost of increasing the constant $K$ by a uniform amount. Hence, by the fiber bunching inequality, for a (possibly larger) constant $K$ we have
\begin{align*}
\|T_{n}\|_{r} &\leq K\sup_{\substack{J \subset [0,n-1]\\|J| =n-3r}}\prod_{j \in J}\|D_{j}^{-1}\|_{0} \|A_{j}^{-1}\|_{0}\|A_{j}\|_{0} \\
&< K b^{n-3r}. 
\end{align*}

The computation above is valid for any $n \geq 3r$, though the neighborhood $V_{n}$ of $0$ must be chosen smaller as $n$ becomes larger. However, the point $0$ lies in all of these neighborhoods, and the above bound shows that for all $n \geq 3r$ the $r$th order derivatives of $T_{n}$ at $0$ are bounded by $K b^{n-3r}$. This estimate is independent of the point $x$, and therefore we conclude that the $r$th order derivatives along $\W$ of $A^{nN} \cdot S$ are all bounded in norm by $K b^{n-3r}$. It follows that all $r$th order derivatives of the series 
\[
T=\sum_{n=0}^{\infty} A^{nN} \cdot S,
\]
converge uniformly along $\W$, and therefore $T$ also defines a section of $\mathrm{Hom}(E, T^{*}\W \otimes E)$ which is uniformly $C^r$ along $\W$. 

We claim that $T$ satisfies 
\begin{equation}\label{connect invariance}
A \cdot T = T - A \cdot \nabla + \nabla.
\end{equation}
The $n$th partial sum of the series $\sum_{n=0}^{\infty}A^{n N} \cdot S$ is represented by the difference of connections $A^{nN} \cdot \bar{\nabla} - \bar{\nabla}$. Set $\check{S} = A \cdot \nabla - A$. Then we also have
\[
\sum_{j=0}^{(n-1)N} A^{j} \cdot \check{S} = A^{nN} \cdot \bar{\nabla} - \bar{\nabla} = \sum_{j=0}^{n-1}A^{jN} \cdot S.
\]
From this it follows that the series $\sum_{n=0}^{\infty} A^{n} \cdot S$ converges uniformly to the same limit $T$. Equation \eqref{connect invariance} clearly follows from this series representation of $T$. Set $\nabla:=\bar{\nabla} + T$; then $\nabla$ defines a uniformly $C^r$ connection on $E$ that is uniformly $C^r$ along $\mathcal{W}$ which is $A$-invariant by equation \eqref{connect invariance}. 

Now that we know $\nabla$ is at least $C^1$ along $\mathcal{W}$, its curvature tensor $R^{\nabla}$ along $\mathcal{W}$ is defined. We claim that $R^{\nabla}\equiv 0$, i.e., $\nabla$ is flat. The $A$-invariance of $\nabla$ implies that for $X$, $Y \in \Gamma^{r+1}_{\mathcal{W}}(T\mathcal{W})$ and $Z \in \Gamma^{r+1}_{\mathcal{W}}(\E)$, 
\[
R^{\nabla}(X,Y)(Z) = A^{n}(R^{\nabla}(Df^{-n}(X),Df^{-n}(Y))(A^{-n}(Z))), 
\] 
which implies that at $x \in M$, for any $n \geq 1$,
\begin{align*}
\|R^{\nabla}(X,Y)(Z)(x)\| &\leq \|A^{n}_{x}\| \|R^{\nabla}\| \|(Df^{n}_{x})^{-1}\|^{2} \|(A^{n}_{x})^{-1}\| \\
&\leq C b^{n},
\end{align*}
by the fiber bunching inequality. Letting $n \rightarrow \infty$, we conclude that $R^{\nabla}(X,Y)(Z) = 0$. Hence $R^{\nabla}\equiv 0$.

We thus have unique parallel transport maps $P_{xy}: E_{x} \rightarrow E_{y}$ for $y \in \W(x)$ defined as in Remark \ref{simple flat}. By the $A$-invariance of $\nabla$, these maps satisfy the same equations as the holonomy maps $H_{xy}$ satisfy in the statement of Proposition \ref{existuhol}. Furthermore these maps depend in a Lipschitz fashion on $x$ and $y$ since $\nabla$ is $C^1$ along $\W$. Thus by the uniqueness part of Proposition \ref{existuhol} we conclude that $P_{xy} = H_{xy}$.
\end{proof}




\subsection{Quotient normal forms}  


We start with a general lemma. Given a $k$-dimensional $C^r$ foliation $\mathcal{F}$ of an $m$-dimensional $C^r$ manifold $M$, $r \geq 0$, we let $M/\F$ denote the topological space obtained as the quotient of $M$ by the equivalence relation $\sim$ given by $x \sim y$ if $y \in \F(x)$. We let $\pi_{\F}: M \rightarrow M/\F$ denote the quotient map. We let $B:=B_{m-k}$ denote the open unit ball in $\R^{m-k}$.

\begin{lem}\label{quotient manifold}
Suppose that $M/\F$ is Hausdorff and that for each $x \in M$ there is an open neighborhood $U$ of $x$ and a $C^r$ embedding $\zeta_{x}: B \rightarrow U$, with $\zeta_{x}(0) = x$, such that $\pi_{\F} \circ \zeta_{x}: B \rightarrow M/\F$ is injective. Then $M/\mathcal{F}$ is a $C^r$ manifold and the quotient map $\pi_{\F}$ is $C^r$. 
\end{lem}

\begin{proof}
Fix $x \in M$. By continuity of $\F$, for any open subset $W \subseteq B$ the set 
\[
\{y \in M: \F(y) \cap \zeta(W) \neq \emptyset \},
\]
is open in $M$. Hence $\pi_{\F} \circ \zeta_{x}$ is an open map, which implies that it is a homeomorphism onto its image because of the injectivity hypothesis. This shows that $M/\F$ is locally Euclidean, which implies that it is a manifold since we assume that $M/\F$ is Hausdorff.

Now assume that $r \geq 1$ and consider $\zeta_{x}$, $\zeta_{y}$ such that $\pi_{\F}( \zeta_{x}(B)) \cap \pi_{\F}( \zeta_{y}(B)) \neq \emptyset$. The transition map on the intersection is given by the $\F$-holonomy map $h^{\F}: S_{x} \rightarrow S_{y}$ between appropriate open subsets $S_{x} \subseteq \zeta_{x}(B)$ and $S_{y} \subseteq \zeta_{y}(B)$. Since $\F$ is $C^r$, we have that $h^{\F}$ is $C^r$ and therefore the lemma follows.
\end{proof}

We now return to the setting of Proposition \ref{prop: splitting reg} with a higher regularity assumption. Specifically, $M$ be a closed $C^{r+3}$ Riemannian manifold for an $r \geq 1$. Let $\W^{u}$ be a foliation of $M$ with uniformly $C^{r+3}$ leaves and let $f: M \rightarrow M$ be uniformly $C^{r+3}$ along $\W^{u}$ and have $\W^{u}$ as an expanding foliation. Set $E^{u}:=T\W^{u}$ and assume that we have a dominated splitting $E^{u} = L^{u} \oplus V^{u}$ for $Df|_{E^{u}}$ as in that proposition. Then Proposition \ref{prop: splitting reg} implies that $V^{u}$ is $C^{r+2}$ along $\W^{u}$ and therefore $\V^{u}$ defines a foliation of $M$ which is uniformly $C^{r+2}$ along $\W^{u}$. We consider $\V^{u}$ as a $C^{r+2}$ foliation of the $C^{r+2}$ manifold $M_{\W^{u}}$ whose connected components are the leaves of $\W^{u}$. We set $\mathcal{Q}^{u}:=M_{\W^{u}}/\mathcal{V}$ and write $\pi_{\V^{u}}: M_{\W^{u}} \rightarrow \mathcal{V}^{u}$ for the quotient map. We set $Q^{u} = E^{u}/V^{u}$ to be the quotient bundle over $M$, which is uniformly $C^{r+2}$ along $\W^{u}$ and is therefore $C^{r+2}$ over $M_{\W^{u}}$. We let $\mathcal{Q}^{u}(x) = \W^{u}(x)/\mathcal{V}^{u}$ be the quotient of the leaf $\W^{u}(x)$ by the foliation $\mathcal{V}^{u}$, for $x \in M$.

\begin{prop}\label{prop: quotient reg}
The space $\mathcal{Q}^{u}$ is a $C^{r+2}$ manifold and the projection $\pi_{\mathcal{V}^{u}}: M_{\W^{u}} \rightarrow \mathcal{Q}^{u}$ is $C^{r+2}$. The derivative $D\pi_{\mathcal{V}^{u}}: TM_{\W^{u}} \rightarrow T\mathcal{Q}^{u}$ induces a $C^{r+1}$ bundle map $\bar{D}\pi_{\mathcal{V}^{u}}: Q^{u} \rightarrow T\mathcal{Q}^{u}$ that is an isomorphism on each fiber. 
\end{prop}

\begin{proof}
We give each leaf of $\V^{u}$ the induced Riemannian metric from $M$. By continuity and compactness, for each sufficiently small $\delta > 0$ there is an $s > 0$ such that for any $x \in M$ the ball $B(x,s)$ of radius $s$ centered at $x$ inside of $\W^{u}(x)$ is contained inside a foliation box for $\V^{u}$ and for any $C^{r+2}$ transversal $J$ inside of $B(x,s)$ to $\mathcal{V}^{u}$ through $x$ which is tangent to $L^{u}$ at $x$ and whose tangent spaces make angle at most $\delta$ with $L^{u}$ inside of $B(x,s)$, each leaf of the restriction of $\mathcal{V}^{u}$ to $B(x,s)$ intersects $J$ exactly once. We will show that each leaf of $\mathcal{V}^{u}$ itself intersects $J$ exactly once. 

Suppose that there is some $y \neq z \in J$ such that $z \in \mathcal{V}^{u}(y)$. Since $f^{-n}$ exponentially contracts the leaves of $\W$, there is some $n$ such that $d(f^{-n}y,f^{-n}z) < s$ and such that there is a geodesic $\gamma$ from $f^{-n}y$ to $f^{-n}z$ inside of $\mathcal{V}^{u}(f^{-n}y)$ that remains entirely inside of $B(f^{-n}x,s)$. Since the splitting $E^{u} = L^{u} \oplus V^{u}$ is dominated, the angle that $J$ makes with $L^{u}$ is strictly decreased under iteration by $Df^{-n}$ for $n$ large enough. Since $f^{-n}(J)$ is tangent to $L^{u}$ at $f^{-n}x$ and $f^{-n}y  \in f^{-n}(J)$, we conclude that $\mathcal{V}^{u}(f^{-n}y)$ intersects $f^{-n}(J)$ exactly once inside of $B(f^{-n}x,s)$. But by the choice of $n$, $f^{-n}z \in f^{-n}(J)$ as well and lies in the same component of the restriction of $\mathcal{V}^{u}$ to $B(f^{-n}x,s)$ as $f^{-n}y$. This is a contradiction.

Thus the projection $\pi_{\V^{u}}: J \rightarrow \mathcal{Q}^{u}(x)$ is injective. The conclusions of this proposition will then follow from Lemma \ref{quotient manifold} once we also show that $\mathcal{Q}^{u}(x)$ is Hausdorff. Let $x \in M$, $y \in \W^{u}(x)$ with $y \notin \V^{u}(x)$. Letting $\delta$ and $s$ be as in the first paragraph, by applying $f^{-n}$ for a sufficiently large power of $n$ we can assume that $d(x,y) < s$ and that there is a transversal $J$ to $\V^{u}$ through $x$, making angle at most $\delta$ with $L^{u}$ inside of $B(x,s)$, such that the restriction $\V^{u}_{B(x,s)}(y)$ of $\V^{u}(y)$ to $B(x,s)$ intersects $J$ in exactly one point $z \neq x$. 

Let $J_{x}$ and $J_{z}$ be disjoint open sets inside of $J$ with $x \in J_{x}$, $z \in J_{z}$. Since each leaf of $\mathcal{V}^{u}$ intersects $J$ exactly once, we have that $\pi_{\V^{u}}(J_{x}) \cap \pi_{\V^{u}}(J_{z}) = \emptyset$. Since $\pi_{\V^{u}}(z) = \pi_{\V^{u}}(y)$, it follows that the projections of $x$ and $y$ to $\mathcal{Q}^{u}(x)$ have disjoint open neighborhoods and therefore that $\mathcal{Q}^{u}(x)$ is Hausdorff.
\end{proof}

The $C^{r+2}$ manifolds $\mathcal{Q}^{u}(x)$ are the connected components of $\mathcal{Q}^{u}$. We write $\bar{f}: \mathcal{Q}^{u} \rightarrow \mathcal{Q}^{u}$ for the induced $C^{r+2}$ diffeomorphism from $f$, which satisfies $\pi_{\V^{u}} \circ f = \bar{f} \circ \pi_{\V^{u}}$.

We now assume that $Df|_{L^{u}}$ is $1$-fiber bunched along $\W^{u}$. This implies that the action $Df|_{Q^{u}}$ of $Df$ on the quotient bundle  $Q^{u}$ is $1$-fiber bunched along $\W^{u}$ as well when we give $Q^{u}$ the quotient Riemannian structure from the projection isomorphism $(V^{u})^{\perp} \rightarrow Q^{u}$. The bundle $Q^{u}$ is $C^{r+2}$ along $\W^{u}$ by Proposition \ref{prop: splitting reg}, hence we can apply Proposition \ref{connect invariance} to obtain a  $Df|_{Q^{u}}$-invariant flat connection $\nabla$ on $Q^{u}$ that is uniformly $C^{r}$ along $\W^{u}$, for which the parallel transport of $\nabla$ coincides with the holonomies of $Df|_{L^{u}}$ under the quotient identification of $L^{u}$ with $Q^{u}$. For $x\in M$, $y \in \W^{u}(x)$ we write $P_{xy}^{u}: Q_{x}^{u} \rightarrow Q_{y}^{u}$ for the parallel transport maps of $\nabla$ on $Q^{u}$. 

\begin{lem}\label{holonomy vert}
Let $x \in M$ and let $S_{1}, S_{2} \subseteq \W^{u}(x)$ be any two local transversals to $\mathcal{V}^{u}$  such that $\pi_{\mathcal{V}^{u}}(S_{1}) = \pi_{\mathcal{V}^{u}}(S_{2}) \subseteq \mathcal{Q}^{u}(x)$. Identify the tangent bundles $TS_{1}$ and $TS_{2}$ with the restrictions of $Q^{u}$ to $S_{1}$ and $S_{2}$ respectively. Then the derivative of the $\mathcal{V}^{u}$-holonomy map $h^{\mathcal{V}^{u}}$ from $S_{1}$ to $S_{2}$ is given by, for $x \in S_{1}$ and $y = h^{\V^{u}}(x) \in S_{2}$, 
\[
Dh^{\mathcal{V}^{u}}_{x} = P_{xy}^{u}:Q_{x}^{u} \rightarrow Q_{y}^{u}.
\]
\end{lem}

\begin{proof}
Since $Q^{u}$ is uniformly $C^{r+2}$ along $\W^{u}$, for $s > 0$ small enough we can find a family of identifications $I_{xy}: Q_{x}^{u} \rightarrow Q_{y}^{u}$, $y \in \W^{u}(x)$, $d(x,y) < s$ that is uniformly $C^{r+2}$ along $\W^{u}$. Let $S_{1}$, $S_{2}$ be two given local transversals to $\mathcal{V}^{u}$ inside of $\W^{u}(x)$ with $\pi_{\V^{u}}(S_{1}) = \pi_{\V^{u}}(S_{2})$. Let $h^{\V^{u}}: S_{1} \rightarrow S_{2}$ denote the $\V^{u}$-holonomy map, and for each $x \in M$ let $y=h^{\V^{u}}(x)$ and let $Dh^{\V^{u}}_{x}$ denote the derivative considered as a map $Q_{x}^{u} \rightarrow Q_{y}^{u}$. Our goal is to show that $Dh^{\V^{u}}_{x} = P_{x y}^{u}$.  

For each $n \geq 0$ we consider the holonomy map $h^{\V^{u},n}: f^{-n}(S_{1}) \rightarrow f^{-n}(S_{2})$ defined by $h^{\V^{u}}  = f^{n} \circ  h^{\V^{u},n} \circ f^{-n}$. Since $E^{u} = L^{u} \oplus V^{u}$ is a dominated splitting, for $i = 1,2$, all tangent spaces of the transversal $f^{-n}(S_{i})$ make a uniformly small angle with $L^{u}$ for $n$ large enough. Since $\mathcal{V}^{u}$ is a uniformly $C^1$ subfoliation of $\W^{u}$, this implies that the holonomy maps $h^{\V^{u},n}$ are uniformly $C^1$ in $n$. In particular, there is a constant $C \geq 1$ such that,
\[
\left\|Dh^{\V^{u},n}_{f^{-n}x} - I_{f^{-n}x f^{-n}y}\right\| \leq Cd(f^{-n}x,f^{-n}y).
\]
Then, by Proposition \ref{unique uhol}, for each $x \in S_{1}$ we have 
\[
Dh^{\V^{u}}_{x} = Df^{n}_{f^{-n}y} \circ Dh^{\V^{u},n}_{f^{-n}x}\circ Df^{-n}_{x} \rightarrow P_{x y}^{u},
\]
as $n \rightarrow \infty$, which completes the proof.
\end{proof}

Let $X$ and $Y$ be $C^{r+1}$ vector fields on a component $\mathcal{Q}^{u}(x)$ of $\mathcal{Q}^{u}$, $x \in M$. Let $S$ be a $C^{r+2}$ local transversal to $\V^{u}$ inside of $\W^{u}(x)$. As above, we identify the tangent bundle $TS$ with the restriction $Q^{u}|_{S}$ of $Q^{u}$ to $S$. With this identification $\nabla$ induces a flat connection $\t{\nabla}^{S}$ on $TS$. Over the projection $\pi(S)$, we uniquely lift $X$ and $Y$ to $C^{r+1}$ vector fields $\t{X}$, $\t{Y}$ on $TS$. We then define
\begin{equation}\label{project connect}
\bar{\nabla}_{X}^{S}Y = D\pi(\t{\nabla}_{\t{X}}^{S}\t{Y}). 
\end{equation}
This defines a $C^r$ flat connection $\nabla^{S}$ on $\pi_{\V^{u}}(S)$. 

Consider two transversals $S_{1}$, $S_{2}$ with $\pi(S_{1}) \cap \pi(S_{2}) \neq \emptyset$. By enlarging these transversals, we may assume for simplicity that $\pi(S_{1}) = \pi(S_{2})$. By Lemma \ref{holonomy vert}, the holonomy $h^{\V^{u}}: S_{1} \rightarrow S_{2}$ maps $\t{\nabla}^{S_{1}}$ to $\t{\nabla}^{S_{2}}$. Hence $\bar{\nabla}^{S_{1}} = \bar{\nabla}^{S_{2}}$. We conclude that we have a flat, $D\bar{f}$-invariant $C^r$ connection $\bar{\nabla}$ on $\mathcal{Q}^{u}(x)$ for which equation \eqref{project connect} holds for any smooth local transversal $S$ to $\V^{u}$.

For a $C^1$ connection $\nabla$ on the tangent bundle $TM$ of a $C^3$ manifold $M$, the \emph{torsion tensor} is the section $T^{\nabla}$ of $\mathrm{Hom}(TM \otimes TM, TM)$ defined on $C^2$ vector fields $X$, $Y$ on $M$ by 
\[
T^{\nabla}(X,Y) = \nabla_{X}Y-\nabla_{Y}X - [X,Y].
\]
$\nabla$ is \emph{torsion-free} if $T^{\nabla} \equiv 0$. 

\begin{prop}\label{quotient connect}
The connection $\bar{\nabla}$ on $T\mathcal{Q}^{u}$ is torsion-free.
\end{prop}

\begin{proof}
The torsion tensor $T^{\bar{\nabla}}$ is a continuous section of $\mathrm{Hom}(T\mathcal{Q}^{u} \otimes T\mathcal{Q}^{u},T\mathcal{Q}^{u})$. By considering lifts of $\bar{\nabla}$ on a transversal to $\V^{u}$ as in equation \eqref{project connect}, we conclude from $\nabla$ being uniformly $C^r$ along $\W^{u}$ (and the compactness of $M$) that there is a constant $C$ such that $\|T^{\nabla}_{p}\| \leq C$ for all $p \in \mathcal{Q}^{u}$.  

The fiber bunching inequality \eqref{fiber bunching} for $Df$ on $Q^{u}$ implies that there is a constant $0<b < 1$ such that for $n$ large enough we have for all $x \in M$, 
\[
\|D\bar{f}^{n}_{x}\|\|(D\bar{f}^{n}_{x})^{-1}\|^{2} < b^{n}.
\]
For $n$ large enough and for each pair of unit vectors $v$, $w \in T\mathcal{Q}_{p}^{u}$,  the $D\bar{f}$-invariance of $T^{\nabla}$ implies that
\begin{align*}
\|T_{p}(v,w)\| &= \|D\bar{f}^{n}(T_{\bar{f}^{-n}p}^{\nabla}(D\bar{f}^{-n}(v),D\bar{f}^{-n}(w))\| \\
&\leq \|D\bar{f}^{n}_{p}\| \|T_{\bar{f}^{-n}p}^{\nabla}\| \|(D\bar{f}^{n}_{p})^{-1}\|^{2} \\
&\leq Cb^{n}.
\end{align*}
Letting $n \rightarrow \infty$, we conclude that $T_{p}^{\nabla} = 0$ for all $p \in \mathcal{Q}^{u}$, i.e., $T^{\nabla}\equiv 0$. We conclude that $\bar{\nabla}$ is torsion-free. 
\end{proof}

We obtain from $\bar{\nabla}$ a special family of charts on the connected components of $\mathcal{Q}^{u}$. For each $x \in M$ we endow $T\mathcal{Q}^{u}(x)$ with the unique $\bar{\nabla}$-invariant Riemannian structure $\{\langle \; , \; \rangle^{x}_{y}\}_{y \in \mathcal{Q}^{u}(x)}$ for which $D\pi_{\V^{u}}:Q_{x}^{u} \rightarrow T\mathcal{Q}_{\bar{x}}^{u}(x)$, $\bar{x} = \pi_{\V^{u}}(x)$, is an isometry. We write $\bar{d}_{x}$ for the resulting Riemannian metric on $\mathcal{Q}^{u}(x)$. Note $\bar{d}_{x} = \bar{d}_{y}$ for each $y \in \mathcal{Q}^{u}(x)$ by the flatness of $\bar{\nabla}$. Since $\bar{\nabla}$ is also torsion-free in addition to being $\bar{\nabla}$-invariant, it is the Levi-Civita connection for the metric $\bar{d}_{x}$. 

\begin{lem}\label{quotient complete}
For each $x \in M$ the metric $\bar{d}_{x}$ on $\mathcal{Q}^{u}(x)$ is complete.
\end{lem}

\begin{proof}
We can cover $M$ by finitely many open boxes $U_{1},\dots,U_{p}$ such that each box $U_{i}$ is subfoliated by a foliation $\mathcal{S}_{i}$ with uniformly $C^{r+2}$ leaves that is uniformly transverse to the foliation $\mathcal{V}^{u}$ within $U_{i}$; this can be done for instance by covering $M$ with foliation boxes for $\mathcal{V}^{u}$ and taking $\mathcal{S}_{i}$ to be the pullback by the foliation chart of the foliation of $\R^{m}$ by planes perpendicular to the image of the foliation $\V^{u}$. On a leaf $S$ of one of the local foliations $\mathcal{S}_{i}$,  $\bar{\nabla}$ has the relation \eqref{project connect} to $\t{\nabla}^{S}$. Since $\nabla$ is uniformly $C^r$ along $\W$ and since each of the $\mathcal{S}_{i}$ are uniformly transverse to $\mathcal{V}^{u}$, by the compactness of $M$ there is a constant $s > 0$ such that for each $x \in M$ there is some $i$ such that $x$ belongs to one of the leaves $S$ of the local foliation $\mathcal{S}_{i}$, the ball $B_{x,i} \subset$ of radius $s$ in the induced Riemannian metric on $S$ fully belongs to $S$, and $\t{\nabla}^{S}$ is complete on $B_{x,i}$. 

Suppose there is some $x \in M$ and some geodesic $\gamma:[0,q) \rightarrow \mathcal{Q}^{u}(x)$ in the metric $\bar{d}_{x}$ starting from $x$ that cannot be extended further. The $Df|_{Q^{u}}$-invariance of $\nabla$ implies that $f^{-n}\circ \gamma:[0,q_{n})$ is a geodesic in $\mathcal{Q}^{u}(f^{-n}x)$ with the metric $\bar{d}_{f^{-n}x}$ for each $n \geq 0$ (with the curve reparametrized $f^{-n}\circ \gamma$ being reparametrized by arc length). Since $Df^{-1}|_{Q^{u}}$ is uniformly contracting, as $n \rightarrow \infty$ we have $q_{n} \rightarrow 0$. Thus when $n$ is large enough we can assume that $f^{-n}\circ \gamma$ is contained within a leaf $S$ of one of the local foliations $\mathcal{S}_{i}$, and we can further assume that $f^{-n}\circ \gamma \subset B_{x,i}$.  This contradicts the completeness of the connection $\t{\nabla}^{S}$ on $B_{x,i}$. 
\end{proof}

By Lemma \ref{quotient complete} we can write $\exp_{x}: T\mathcal{Q}^{u}_{\bar{x}}(x) \rightarrow \mathcal{Q}^{u}(x)$ for the exponential map and write 
\[
e_{x} = \exp_{x} \circ \pi_{\V^{u}}: Q_{x}^{u} \rightarrow \mathcal{Q}^{u}(x). 
\]

\begin{prop}\label{quotient charts}
The following holds for the charts $e_{x}$, $x \in M$, 
\begin{enumerate}
\item The metric $\bar{d}_{x}$ on $\mathcal{Q}^{u}(x)$ is flat. 
\item The map $e_{x}:Q_{x}^{u} \rightarrow \mathcal{Q}^{u}(x)$ is an isometry.
\item The maps $e_{y}$ for $y \in \W^{u}(x)$ depend uniformly continuously on $y$ in the $C^{r}$ topology. 
\item We have $\bar{f} \circ e_{x} =  e_{f(x)} \circ Df|_{Q_{x}^{u}}$. 
\item For $y \in \W^{u}(x)$ the map $e_{y}^{-1} \circ e_{x}:Q_{x}^{u} \rightarrow Q_{y}^{u}$ is an affine transformation. 
\end{enumerate}
\end{prop}

\begin{proof}
Let $x \in M$ be given. Since $\bar{\nabla}$ is flat and is the Levi-Civita connection for $\bar{d}_{x}$, we conclude that $\bar{d}_{x}$ is flat. 

We next check that $\mathcal{Q}^{u}(x)$ is simply connected. There is clearly an $s > 0$ such that for any $y \in \mathcal{Q}^{u}$ the ball $B(y,s)$ centered at $y$ is contractible. If $\gamma$ is any closed loop in $\mathcal{Q}^{u}(x)$ which is nontrivial in $\pi_{1}(\mathcal{Q}^{u}(x))$, then $\bar{f}^{-n} \circ \gamma$ will be nontrivial in $\pi_{1}(\mathcal{Q}^{u}(f^{-n}x))$ for each $n \geq 1$. But since $\bar{f}^{-1}$ exponentially contracts $\mathcal{Q}^{u}$, for $n$ large enough $f^{-n} \circ \gamma$ will be contained inside of $B(f^{-n}(\gamma(0)),s)$ and will therefore be contractible, a contradiction.

Hence $\mathcal{Q}^{u}(x)$ is simply connected. Since $D\pi_{\V^{u}}$ induces an isometry of $T\mathcal{Q}_{\bar{x}}^{u}(x)$ with the vector space $Q^{u}_{x}$ which is isometric to a Euclidean space, by the Cartan-Hadamard theorem and the flatness of $\bar{d}_{x}$ the exponential map $\exp_{x}:T\mathcal{Q}_{x}^{u} \rightarrow \mathcal{Q}^{u}(x)$ is an isometry. Claim (2) follows immediately from this. Claim (3) then follows from the fact that $\bar{\nabla}$ descended from the connection $\nabla$ that was uniformly $C^r$ along $\W^{u}$. 
 
 For (4), since $\bar{f}$ preserves $\nabla$ we must have that $e_{f(x)}^{-1} \circ \bar{f} \circ e_{x}$ preserves an affine connection on $Q_{x}^{u}$, and is therefore an affine transformation of $Q_{x}^{u}$. Since it preserves $0$, it is actually a linear transformation. Taking the derivative at $0$  and noting that $D_{0}\exp_{\bar{x}} = I$, we obtain
 \begin{align*}
 D_{0}(e_{f(x)}^{-1} \circ \bar{f} \circ e_{x}) &= D_{\bar{f}(\bar{x})}\pi_{\V}^{-1} \circ D\bar{f}_{\bar{x}} \circ D_{x}\pi_{\V} \\
 &= Df_{x}|_{Q_{x}^{u}}.
 \end{align*}
 This establishes (4). Claim (5) follows from the fact that $e_{y}^{-1} \circ e_{x}$ maps an affine connection on $Q_{x}^{u}$ to an affine connection on $Q_{y}^{u}$, hence must itself be affine. 
\end{proof}

The discussion of this section applies equally well to the case in which $Df|_{E^{u}}$ is itself $1$-fiber bunched. We will view this as analogous to the degenerate case in which $V^{u} = \{0\}$, so that $Q^{u} = E^{u}$ and $\mathcal{Q}^{u}(x) = \W^{u}(x)$ for each $x \in M$. We obtain a $Df|_{E^{u}}$-invariant uniformly $C^{r}$ connection $\nabla$ on $E^{u}$ along $\W^{u}$ which is flat and torsion-free on each leaf $\W^{u}(x)$, since $E^{u}|_{\W^{u}(x)} = T\W^{u}(x)$ is the tangent bundle to the leaf. For each $x \in M$ we then obtain $C^r$ charts $e_{x}: E_{x}^{u} \rightarrow \W^{u}(x)$ satisfying the properties of Proposition \ref{quotient charts}.

\begin{rem}
In the case that $Df|_{E^{u}}$ satisfies an absolute form of the $1$-fiber bunching inequality \eqref{fiber bunching}, the case $Q^{u} = E^{u}$ in Proposition \ref{quotient charts} is a well-known result in normal forms theory going back to Katok and Guysinsky \cite{GK98}. The geometric approach to the construction of normal forms using invariant connections is due to Feres \cite{Fer04}. Remaining in the case $Q^{u} = E^{u}$, under our pointwise fiber bunching inequality the normal forms of \ref{quotient charts} were constructed by Sadovskaya \cite{S05} using a Taylor series method. To our knowledge our approach is the first construction of these normal forms in the pointwise fiber bunched case using an invariant connection instead. 

In the case $Q^{u} = E^{u}/V^{u}$, with $V^{u} \neq \{0\}$, that we consider in this paper, we point also to work of Melnick \cite{Mel19} in which smooth invariant connections on quotient bundles are constructed in the case of nonuniformly contracting foliations. These results resemble ours, however we are able to derive much stronger properties of these connections under weaker hypotheses in our setting because we assume the existence of a dominated splitting instead of a measurable invariant splitting. 
\end{rem}

\subsection{Extension to flows with weak expanding foliations}

The results from the previous section extend directly to flows $f^{t}$ with weak expanding foliations by consideration of the time-1 map $f = f^{1}$. The key observation is that for any cocycle $A^{t}$ over $f^{t}$, the holonomies for $A = A^{1}$ are also the holonomies for the time $t$-map $A^{t}$ for any $t > 0$. We let $f^{t}:M \rightarrow M$ be a flow on $M$ which is uniformly $C^r$ along a foliation $\W^{cu}$ with uniformly $C^r$ leaves, and suppose that $\W^{cu}$ is a weak-expanding foliation for $f^{t}$. We let $E$ be a $C^{\alpha}$ vector bundle over $M$.

\begin{prop}\label{flow holonomies} 
Suppose that the linear cocycle $A$ over $f$ is $C^{\alpha}$ and $\alpha$-fiber bunched along $\W^{u}$. Let $H^{u}_{xy}$, $x \in M$, $y \in \W^{u}(x)$ be the holonomies for $A$. Then for all $t \in \R$ we have
\begin{equation}\label{holonomy equivariance flow}
H_{xy}^{u} = (A^{t}_{y})^{-1} \circ H_{f^{t}x f^{t}y}^{u} \circ A^{t}_{x},
\end{equation}
If, furthermore, $A^{t}$ is $1$-fiber bunched and both $A^{t}$ and $E$ are uniformly $C^r$ along $\W^{cu}$, $r \geq 3$, then the invariant connection $\nabla$ for $A$ from Proposition \ref{holonomy connect} is also $A^{t}$-invariant. 
\end{prop}

\begin{proof}
For each $n \geq 1$, the linear cocycle $A^{1/n}$ over $f^{1/n}$ is uniformly $C^{\alpha}$ along $\W^{cu}$ and $\alpha$-fiber bunched along $\W^{u}$. Applying Proposition \ref{existuhol}, we obtain holonomies for $A^{1/n}$, which must also be holonomies for $(A^{1/n})^{n} = A$. Hence we conclude by the uniqueness part of \ref{existuhol} that equation \eqref{holonomy equivariance flow} holds for $t = 1/n$ for each $n \in \N$, and thus holds for all rational $t \in \R$. By continuity we conclude that this equation holds for all $t$. The second statement follows from equation \eqref{holonomy equivariance flow} since the parallel transport maps of $\nabla$ are given by the holonomy maps $H^{u}$.   
\end{proof}

As a consequence, if we have a dominated splitting $E^{u} = L^{u} \oplus V^{u}$ for the linear cocycle $Df^{t}|_{E^{u}}$ over $f^{t}$ such that $Df^{t}|_{L^{u}}$ is $1$-fiber bunched along $\W^{u}$, then by viewing this as a dominated splitting for the time-$t$ map $Df^{t}|_{E^{u}}$ over $f^{t}$ for any $t > 0$, which has $\W^{u}$ as an expanding foliation, we can apply all of the results of this section to this time-$t$ map. In particular, setting $\mathcal{Q}^{u}(x) = \W^{u}(x)/\mathcal{V}^{u}$, we obtain charts $e_{x}: Q_{x}^{u} \rightarrow \mathcal{Q}^{u}(x)$ from Proposition \ref{quotient charts} conjugating $\bar{f}^{t}:\mathcal{Q}^{u}(x)\rightarrow \mathcal{Q}^{u}(f^{t}x)$ to $Df^{t}|_{Q^{u}_{x}}$. By Proposition \ref{flow holonomies} these charts work for all values of $t$ simultaneously.

\section{From horizontal exponents to horizontal quasiconformality}\label{sec:exponents} 

We open this section with a definition. Let $M$ be a compact metric space and let $E$ be a vector bundle over $M$. 

\begin{defn}[Uniform quasiconformality]\label{defn: uniform quasi}
Let $A: E \rightarrow E$ be a linear cocycle over a homeomorphism $f:M \rightarrow M$. We say that $A$ is \emph{uniformly quasiconformal} if there is a constant $C \geq 1$ such that for every $x \in M$ and $n \geq 1$, 
\begin{equation}\label{equation uniform quasi}
C^{-1} \leq \frac{\|A_{x}^{n}\|}{\mathfrak{m}(A_{x}^{n})}\leq C. 
\end{equation}
We say that $A$ is \emph{conformal} if $\mathfrak{m}(A_{x}) = \|A_{x}\|$ for every $x \in M$. 
\end{defn}

We devote this section to the proof of Proposition \ref{exp to conf} below.  We write $\t{M}$ for the universal cover of $M$ below and write $\t{f}^{t}$ for the lift of $f^{t}$ to a flow on $\t{M}$. We also recall the definition \ref{local product structure} of local product structure for an $f^{t}$-invariant measure $\mu$. Below we set $k = \dim E^{u}$ and recall that $L^{cu}:=E^{c} \oplus L^{u}$.

\begin{prop}\label{exp to conf}
Let $f^{t}:M \rightarrow M$ be a smooth Anosov flow on a closed Riemannian manifold $M$ with a $u$-splitting $E^{u} = L^{u} \oplus V^{u}$ of index $l$, $2 \leq l < k$. Suppose that there is some $\beta > 0$ such that $f^{t}$ is $\beta$-$u$-bunched, that $L^{cu}$ is $\beta$-H\"older, and that $Df^{t}|_{L^{u}}$ is $\beta$-fiber bunched. Suppose further that there is a simply connected Riemannian manifold $Y$ of pinched negative curvature for which there is a uniformly continuous orbit equivalence $\varphi: \t{M} \rightarrow  T^{1}Y$ from $\t{f}^{t}$ to the geodesic flow $g^{t}_{Y}$. Lastly, suppose that there is a fully supported $f^{t}$-invariant ergodic probability measure $\mu$ with local product structure such that $\la_{1}^{u}(f,\mu) = \la_{l}^{u}(f,\mu)$. 

Then $Df^{t}|_{L^{u}}$ is uniformly quasiconformal. 
\end{prop}


We will require several preparatory steps. The first few steps will not require use of the orbit equivalence $\varphi$. 

Since $Df^{t}|_{L^{u}}$ is $\beta$-fiber bunched, it is $\beta$-fiber bunched along $\W^{u}$ and $\W^{s}$ in the sense of inequality \eqref{fiber bunching}. Hence by Proposition \ref{existuhol} we have stable and unstable holonomy maps $H^{s}_{xy}: L^{s}_{x} \rightarrow L^{s}_{y}$ for $y \in \W^{s}(x)$ and $H^{u}_{xz}: L^{u}_{x} \rightarrow L^{u}_{z}$ for $z \in \W^{u}(x)$ respectively. 

We next recall a crucial proposition from \cite{Bu1}. For the center-stable foliation $\W^{cs}$ of $f^{t}$, we write $h^{cs}:=h^{\W^{cs}}:S_{1} \rightarrow S_{2}$ for the holonomy map of this foliation between transversals $S_{1}$ and $S_{2}$. On a foliation box $U \subset M$ for $\W^{cs}$, $\W^{s}$, and $\W^{c}$, for any $x \in U$ and $y \in \W^{cs}_{U}(x)$ we let $\{z\} = \W^{s}_{U}(x) \cap \W^{c}_{U}(y)$ and define
\begin{equation}\label{defn center holonomy}
H^{cs}_{xy} = H^{c}_{zy} \circ H^{s}_{xz},
\end{equation}
where $H^{c}_{zy} = Df^{t_{0}}_{z}$, with $t_{0}$ close to $0$ defined by $f^{t_{0}}z = y$.

Below we take the plaques $\W^{u}_{U}$ of the unstable foliation for $f^{t}$ as transversals to the center-stable foliation $\W^{cs}_{U}$ for $f^{t}$ on the foliation box $U$. 

\begin{prop}\label{wcs diff}\cite[Lemma 4.4]{Bu1}
Let $f^{t}: M \rightarrow M$ be a $C^2$ Anosov flow. Let $E^{u} = L^{u} \oplus V^{u}$ be a dominated splitting for $Df^{t}|_{E^{u}}$. Suppose that for some $\beta > 0$ we have that $f^{t}$ is $\beta$-$u$-bunched, that $L^{cu}$ is $\beta$-H\"older, and that $Df^{t}|_{L^{u}}$ is $\beta$-fiber bunched.

Let $U$ be a foliation box for all of the invariant foliations $\W^{*}$ of $f^{t}$. Then for each $x \in M$, $y \in \W^{cs}_{U}(x)$ the center-stable holonomy map $h^{cs}:\W^{u}_{U}(x) \rightarrow \W^{u}_{U}(y)$ is differentiable along $L^{u}$. For $z \in W^{u}_{U}(x)$ and $w = h^{cs}(z)$ the derivative is given by
\[
Dh^{cs}_{z}|_{L^{u}} = H^{cs}_{z w}:L^{u}_{z} \rightarrow L^{u}_{w}.
\]
In particular $h^{cs}$ maps $C^1$ curves tangent to $L^{u}$ to $C^1$ curves tangent to $L^{u}$. 
\end{prop}

For a general $y \in \W^{cs}(x)$ we still have a well-defined center-stable holonomy map $h^{cs}: U_{x} \rightarrow U_{y}$ for $U_{x}$, $U_{y}$ small enough neighborhoods of $x$ and $y$ inside of $\W^{u}(x)$ and $\W^{u}(y)$ respectively. We can choose $\eta$ such that $f^{\eta}y \in \W^{s}(x)$ and then choose $t$ large enough that $f^{t+\eta}y$ and $f^{t}x$ lie in a common foliation box for all of the invariant foliations of $f^{t}$. Proposition \ref{wcs diff} then implies that the center-stable holonomy map $h^{cs}$ from a neighborhood of $f^{t}x$ inside of $f^{t}(U_{x})$ to a neighborhood of $f^{t+\eta}y$ inside of $f^{t+\eta}(U_{y})$ is differentiable along $L^{u}$ with derivative $H^{cs}_{zw}$, $z$ near $f^{t}x$, $w = h^{cs}(z)$. This establishes that the original center-stable holonomy map $h^{cs}$ is also differentiable along $L^{u}$ on a neighborhood of $x$, with derivative $Df^{t+\eta} \circ H^{cs}_{zw} \circ Df^{-t}$. Hence Proposition \ref{wcs diff} establishes differentiability along $L^{u}$ for $\W^{cs}$ holonomy over arbitrary distances, with this derivative being given by pre- and post-composition of $H^{cs}$ with iterates of $Df^{t}$. 

There is a minor discrepancy between the hypotheses of Proposition \ref{wcs diff} above and the hypotheses of Lemma 4.4 in \cite{Bu1}. In \cite{Bu1} we assumed that $E^{u}$ is $\beta$-H\"older and that $L^{u}$ is $\beta$-H\"older, whereas above we assume that $f^{t}$ is $\beta$-$u$-bunched, which only implies that $E^{cu}$ is $\beta$-H\"older, and we only assume that $L^{cu}$ is $\beta$-H\"older, which is weaker than assuming that $L^{u}$ is $\beta$-H\"older . However, in the proof of Lemma 4.4 only the $\beta$-H\"older continuity of the projections of $E^{u}$ and $L^{u}$ to the tangent space $TS$ of a smooth transversal $S$ to the orbit foliation $\W^{c}$ of $f^{t}$ is used. The $\beta$-H\"older continuity of these projections is directly implied by the $\beta$-H\"older continuity of $E^{cu}$ and $L^{cu}$. Hence Lemma 4.4 of \cite{Bu1} holds under these slightly weaker hypotheses. 

Since we have a dominated splitting $E^{u} = L^{u} \oplus V^{u}$ for $Df^{t}|_{E^{u}}$ and $Df^{t}|_{L^{u}}$ is $1$-fiber bunched along $\W^{u}$ (since it is $\beta$-fiber bunched along $\W^{u}$ for some $0 < \beta \leq 1$), we can apply the results of Section \ref{sec:uniform} to the action of $f^{t}$ on $\W^{u}$. We use the notation of that section, in particular letting $\V^{u}$ denote the resulting foliation tangent to $V^{u}$ and letting $\mathcal{Q}^{u}(x) = \W^{u}(x)/\V^{u}$ denote the quotients of the unstable foliation by the vertical foliation.

For a continuous curve $\gamma: [0,\infty) \rightarrow M$ in a space $M$, we write $\gamma(s) \rightarrow \infty$ as $s\rightarrow \infty$ if for each compact set $K \subset M$ there is an $N > 0$ such that $\gamma(s) \notin K$ for $s \geq N$. We first prove two lemmas within the context of the assumptions of Proposition \ref{exp to conf}.

\begin{lem}\label{go infinity}
Let $x \in M$ and let $\gamma:[0,\infty) \rightarrow \mathcal{Q}^{u}(x)$ be a geodesic in the metric $\bar{d}_{x}$. Let $\alpha: [0,\infty) \rightarrow \W^{u}(x)$ be a $C^1$ curve such that $\pi_{\V^{u}} \circ \alpha = \gamma$. Then $\gamma(s) \rightarrow \infty$ in $\mathcal{Q}^{u}(x)$ and $\alpha(s) \rightarrow \infty$ in $\W^{u}(x)$ as $t \rightarrow \infty$.  
\end{lem}

\begin{proof}
Consider the chart $e_{x}: Q^{u}_{x} \rightarrow \mathcal{Q}^{u}(x)$ of Proposition \ref{quotient charts}. The curve $e_{x}^{-1} \circ \gamma$ inside of $Q^{u}_{x}$ is a geodesic in the metric on $Q^{u}_{x}$ induced by the inner product, hence is a straight line in $Q^{u}_{x}$. It follows that $e_{x}^{-1}(\gamma(s)) \rightarrow \infty$ as $s \rightarrow \infty$. Since $e_{x}$ is an isometry, we conclude that $\gamma(s) \rightarrow \infty$ as $t \rightarrow \infty$. 

Let $K$ be a compact subset of $\W(x)$. Its projection $\pi_{\V}(K)$ is compact in $\mathcal{Q}(x)$, hence there is an $N > 0$ such that $\gamma(s) \notin \pi_{\V^{u}}(K)$ for $s \geq N$. This implies that $\alpha(s) \notin K$ for $s \geq N$. The lemma follows. 
\end{proof}

\begin{lem}\label{u foliation}
Let $x \in M$ and let $X: \W^{u}(x) \rightarrow L^{u}$ be a nowhere vanishing $H^{u}$-invariant vector field. Then there exists a continuous foliation $\mathcal{X}$ of $W^{u}(x)$ by $C^1$ curves tangent to $X$. Furthermore, every curve of the foliation $\mathcal{X}$ is properly embedded in $\W^{u}(x)$. 
\end{lem}

\begin{proof}

Let $x \in M$ be given and let $X: \W^{u}(x) \rightarrow L^{u}$ be an $H^{u}$-invariant vector field. For $t \geq 0$ let $X^{t} = Df^{t}(X): \W^{u}(f^{t}(x)) \rightarrow L^{u}$. We can smooth $X^{t}$ to obtain a $C^1$ vector field $Z^{t}$ satisfying $\|Z^{t}-X^{t}\| \leq 1$ -- one can perform the smoothings locally and then glue using a partition of unity. Since $Z^{t}$ is a $C^1$ vector field it is uniquely integrable; let $\mathcal{Z}^{t}$ be the  $C^1$ foliation of $\W^{u}(f^{t}(x))$ by $C^2$ curves tangent to $Z^{t}$.

Observe that $Df^{-t}(Z^{t}) \rightarrow X$ as $t \rightarrow \infty$, since we have 
\[
\|Df^{-t}(Z^{t}) - X\| = \|Df^{-t}(Z^{t}-X^{t})\| \leq e^{-at}\|Z^{t}-X^{t}\| \leq e^{-at}.
\]
for some constant $a > 0$. This implies that the curves of the foliation $\mathcal{Z}^{t}$  converge to a continuous foliation $\mathcal{X}$ of $\W^{u}(x)$ by curves tangent to $X$. 

To prove the last claim, let $\gamma$ be any curve in the foliation $\mathcal{X}$. The vector field $X$ projects to a $\bar{\nabla}$-parallel vector field $\bar{X}$, where $\bar{\nabla}$ is the invariant connection on $\mathcal{Q}^{u}(x)$ constructed in Proposition \ref{quotient connect}. Thus $\bar{\gamma} = \pi_{\V^{u}} \circ \gamma$ is a $\bar{\nabla}$-geodesic. By Lemma \ref{go infinity} we conclude that $\bar{\gamma}(s) \rightarrow \infty$ in $\mathcal{Q}^{u}(x)$ as $t \rightarrow \infty$ and that $\gamma(s) \rightarrow \infty$ in $\W^{u}(x)$. It follows that the ray $\gamma:[0,\infty) \rightarrow \W^{u}(x)$ is properly embedded in $\W^{u}(x)$; an analogous argument considering $s \rightarrow -\infty$ as well establishes that the full curve $\gamma$ is properly embedded in $\W^{u}(x)$. 
\end{proof}

We will now incorporate the orbit equivalence $\varphi$. We write $\t{\W}^{*,f}$ for the invariant foliations of $f^{t}$ lifted to $\t{M}$ and write $\W^{*,g}$ for the invariant foliations of the geodesic flow $g^{t} = g^{t}_{Y}$ on $Y$, as described in Section \ref{subsec:neg curved}.

\begin{lem}\label{one intersection}
Let $f^{t}: M \rightarrow M$ be a smooth Anosov flow on a closed Riemannian manifold $M$ and assume that there is an orbit equivalence $\varphi: \t{M} \rightarrow T^{1}Y$ of $\t{f}^{t}$ with $g^{t}_{Y}$, where $Y$ is a complete simply connected Riemannian manifold of pinched negative curvature. Then for each $x \in \t{M}$ and $y \in \t{\W}^{cu,f}(x)$ there is a unique intersection point $\{z\} = \t{\W}^{c,f}(y) \cap \t{\W}^{u,f}(x)$. 
\end{lem}

\begin{proof}
Suppose that there were points $z_{1} \neq z_{2} \in \t{\W}^{c}(y)$ such that $z_{1}, z_{2} \in \t{\W}^{u}(x)$. We can then write $z_{2} = \t{f}^{t}z_{1}$ for some $t \neq 0$; without loss of generality we may assume that $t > 0$. Then $\t{f}^{t}(\t{\W}^{u}(x)) = \t{\W}^{u}(x)$. Hence $\t{f}^{-t}: \W^{u}(x) \rightarrow \W^{u}(x)$ is a Lipschitz contraction of a complete metric space, which by the contraction mapping theorem must have a fixed point $p$.  Then $p$ is a periodic point of $\t{f}^{t}$, from which it follows that $\t{\W}^{c,f}(p)$ is compact. Since the orbit equivalence $\varphi$ maps the orbit foliation $\t{\W}^{c,f}$ of $\t{f}^{t}$ to the orbit foliation $\W^{c,g}$ of $g^{t}$, we conclude that $g^{t}$ has a closed orbit, or equivalently $Y$ has a closed geodesic. Since a complete simply connected nonpositively curved Riemannian manifold has no closed geodesics \cite{BGS}, this is a contradiction.  
\end{proof}

\begin{prop}\label{accessible curves}
Suppose that $f^{t}: M \rightarrow M$ is a smooth Anosov flow satisfying the hypotheses of Proposition \ref{exp to conf}. Let $F \subseteq L^{u}$ be a nontrivial $Df^{t}$-invariant subbundle that is invariant under both $H^{u}$ and $H^{s}$ holonomies. Then for each $x \in M$ and $y \in \W^{u}(x)$ with $y \neq x$ there exists a continuous curve $\gamma: [0,1] \rightarrow \W^{u}(x)$, which is $C^1$ on $[0,1)$, such that $\gamma(0) = x$, $\gamma(1) = y$, and  $\gamma'(s) \in F$ for each $s \in [0,1)$. 
\end{prop}

\begin{proof}
We will use the geometric properties of $Y$ discussed in Section \ref{subsec:neg curved}. We pass to the universal cover $\t{M}$ of $M$ and consider the lift $\t{f}^{t}$ of  $f^{t}$ on $\t{M}$, together with the lifts $\t{\W}^{*,f}$ of the invariant foliations of $f^{t}$ to $\t{M}$. We consider the lifted $u$-splitting $\t{E}^{u} = \t{L}^{u} \oplus \t{V}^{u}$ and the lift $\t{F} \subset \t{L}^{u}$ of $F$. By the uniform continuity of $\varphi$ we conclude that for each $x \in \t{M}$ we have $\varphi(\t{\W}^{cs,f}(x)) = \W^{cs,g}(\varphi(x))$ and $\varphi(\t{\W}^{cu,f}(x)) = \W^{cu,g}(\varphi(x))$, as for instance if $y \in \t{\W}^{cs,f}(x)$ and $\varphi(y) \notin \W^{cs,g}(\varphi(x))$ then as $t \rightarrow \infty$ we would have since $\varphi$ is an orbit equivalence,
\[
d_{Y}(\varphi(f^{t}x),\varphi(f^{t}y)) = d_{Y}(g^{\alpha(t,x)}(\varphi(x)),g^{\alpha(t,y)}(\varphi(y)))  \rightarrow \infty,
\]
while $d_{\t{M}}(f^{t}x,f^{t}y)$ remains bounded independently of $t$, contradicting the uniform continuity of $\varphi$. Here $\alpha$ is an additive cocycle over $f^{t}$, for which the uniform continuity property of $\varphi$ implies that $\alpha(t,p) \rightarrow \infty$ as $t \rightarrow \infty$ for each $p \in T^{1}Y$. This shows that $\varphi(\t{\W}^{cs,f}(x)) \subseteq \W^{cs,g}(\varphi(x))$, and similar arguments give the other inclusions. 

For $x \in \t{M}$ we define a homeomorphism $\varphi_{x}:\t{\W}^{u,f}(x) \rightarrow \W^{u,g}(\varphi(x))$ by setting $\varphi_{x}(y)$ to be the unique intersection point of $\W^{c,g}(\varphi(y))$ with $\W^{u,g}(\varphi(x))$; the uniqueness of this intersection point is guaranteed by the structure of the invariant foliations $\W^{*,g}$ for $g^{t}$ discussed in Section \ref{subsec:neg curved}.  For $p \in T^{1}Y$ we likewise define $\varphi_{p}^{-1}: \W^{u,g}(p) \rightarrow \t{\W}^{u,f}(\varphi^{-1}(p))$ by setting $\varphi_{p}^{-1}(q)$ to be the unique intersection point of $\t{\W}^{c,f}(\varphi^{-1}(q))$ with $\t{\W}^{u,f}(\varphi^{-1}(p))$. The uniqueness of this intersection point is guaranteed by Lemma \ref{one intersection}.

It suffices to prove the proposition in the lifted foliations for an arbitrary $x \in \t{M}$ and $y \in \t{\W}^{u,f}(x)$. Choose a point $z \in \t{M}$ such that $\varphi(z) \in T^{1}Y$ satisfies $\theta_{-}(\varphi(z)) = \theta_{+}(\varphi(y))$ and $\theta_{+}(\varphi(z)) = \theta_{+}(\varphi(x))$. By construction we then have $\varphi(x) \in \W^{cs,g}(\varphi(z))$ and therefore $x \in \t{\W}^{cs,f}(z)$. By replacing $z$ with $\t{f}^{s}z$ for an appropriate $s \in \R$, we may assume that $x \in \t{\W}^{s,f}(z)$; here we are using the fact that the projections $\theta_{+}(\varphi(z))$ and $\theta_{-}(\varphi(z))$ depend only on the $g^{t}$ orbit of $\varphi(z)$, which is in turn determined only by the $\t{f}^{t}$-orbit of $z$ since $\varphi$ is an orbit equivalence. 

Choose a nowhere vanishing $H^{u}$-invariant vector field $X: \t{\W}^{u,f}(z) \rightarrow \t{L}^{u}$ that is tangent to $\t{F}$. By Lemma \ref{u foliation}, there exists a continuous foliation $\mathcal{X}$ of $\t{\W}^{u,f}(z)$ by $C^1$ curves tangent to $X$. Consider the curve $\eta:[0,\infty) \rightarrow \t{\W}^{u,f}(z)$ in this foliation $\mathcal{X}$ starting at  $\eta(0) = z$. By Lemma \ref{u foliation}, $\eta$ is properly embedded in $\t{\W}^{u,f}(z)$ and therefore $\eta(s) \rightarrow \infty$ as $t \rightarrow \infty$. 

Let $\kappa = \varphi_{z} \circ \eta$. Then $\kappa$ is a continuous curve in $\W^{u,g}(\varphi(z))$ with $\kappa(0) = \varphi(z)$ and $\kappa(s) \rightarrow \infty$ as $t \rightarrow \infty$. Let $\kappa_{+} = \theta_{+} \circ \kappa$ be the projection of $\kappa$ to $\p Y$. Then we have $\kappa_{+}(s) \rightarrow \theta_{-}(\varphi(z)) = \theta_{+}(\varphi(y))$ as $s \rightarrow \infty$.  We now take the inverse projection 
\[
\theta_{+}^{-1}: \p Y \backslash \{\theta_{-}(\varphi(x))\} \rightarrow \W^{u,g}(\varphi(x)). 
\]
If the curve $\kappa_{+}$ contains $\theta_{-}(\varphi(x))$, we exchange $\eta$ for the ray $\eta_{-}:[0,\infty) \rightarrow \t{\W}^{u,f}(z)$ in the foliation $\mathcal{X}$ starting at $z$ and pointing in the direction of $-X$ instead. This curve is disjoint from $\eta$ except for its starting point that has projection 
\[
\theta_{-}(\varphi(z)) = \theta_{-}(\varphi(y)) \neq \theta_{-}(\varphi(x)),
\]
hence if we construct $\kappa_{+}$ using $\eta_{-}$ instead then it will not contain $\theta_{-}(\varphi(x))$. 

Hence we may assume that $\kappa_{+}$ does not contain $\theta_{-}(\varphi(x))$. Then $\theta_{+}^{-1} \circ \kappa_{+}$ defines a continuous curve on $\W^{u,g}(\varphi(x))$ with $\theta_{+}^{-1}(\kappa_{+}(0)) = \varphi(x)$ and $\theta_{+}^{-1}(\kappa_{+}(s)) \rightarrow \varphi(y)$ as $s \rightarrow \infty$. Setting $\sigma = \varphi_{\varphi(x)}^{-1} \circ \theta_{+}^{-1} \circ \kappa_{+}$ and recalling that $\varphi_{\varphi(x)}^{-1}: \W^{u,g}(\varphi(x)) \rightarrow \t{\W}^{u,f}(x)$ is a homeomorphism, we conclude that $\sigma(0) = x$ and $\sigma(s) \rightarrow y$ as $s \rightarrow \infty$. 

By the discussion after \eqref{transition maps}, the curve $\theta_{+}^{-1} \circ \kappa_{+}$ may equivalently be thought of as the image of $\kappa$ under $\W^{cs,g}$-holonomy from $\W^{u,g}(\varphi(z))$ to $\W^{u,g}(\varphi(x))$. This implies that $\sigma$ is the image of $\eta$ under $\t{\W}^{cs,f}$-holonomy, since $\varphi$ preserves the center-stable foliations. By Proposition \ref{wcs diff} and the hypothesis that $\t{F}$ is invariant under $H^{cs}$-holonomy, we conclude that $\sigma$ is a $C^1$ curve tangent to $\t{F}$ with $\sigma(0) = x$ and $\sigma'(s) \neq 0$ for all $s \in [0,\infty)$. The curve $\gamma(s) = \sigma((2/\pi)\arctan(s))$, $s \in [0,1]$, then has the desired properties of the proposition.
\end{proof}

\begin{proof}[Proof of Proposition \ref{exp to conf}]
By \cite[Lemma 2.5]{Bu1}, if there exists a fully supported $f^{t}$-invariant ergodic probability measure $\mu$ with local product structure such that we have $\la_{1}^{u}(f^{t},\mu) = \la_{l}^{u}(f^{t},\mu)$ -- or equivalently, in the terminology of \cite{Bu1}, the extremal Lyapunov exponents of $Df^{t}|_{L^{u}}$ with respect to $\mu$ are equal -- then either $Df^{t}|_{L^{u}}$ is uniformly quasiconformal, or there is a proper nontrivial $Df^{t}$-invariant subbundle $F \subset L^{u}$ such that $F$ is both $H^{s}$- and $H^{u}$-holonomy invariant.

Fix a point $x \in M$. Since $F$ is $H^{u}$-invariant, it projects to a $\bar{\nabla}$-parallel $C^{1}$ subbundle $\bar{F}$ of $T\mathcal{Q}^{u}(x)$. Since $\bar{\nabla}$ is a flat connection, by the Frobenius theorem there is a $C^{1}$ foliation $\mathcal{F}$ of $\mathcal{Q}^{u}(x)$ that is tangent to $\bar{F}$. If $F$ is a proper subbundle of $L^{u}$ then $\mathcal{F}$ is a nontrivial foliation of $\mathcal{Q}^{u}(x)$. 

Write $\pi:=\pi_{\V^{u}}$. There thus exists a point $y \in \W^{u}(x)$ such that $\pi(y) \notin \mathcal{F}(x)$. On the other hand, by Proposition \ref{accessible curves} there is a curve $\gamma:[0,1] \rightarrow \W^{u}(x)$ tangent to $F$ with $\gamma(0) = x$, $\gamma(1) = y$, and such that $\gamma$ is $C^1$ on $[0,1)$. Then $\pi \circ \gamma:[0,1] \rightarrow \mathcal{Q}^{u}(x)$ is a curve which is $C^1$ on $[0,1)$, tangent to $F$ on $[0,1)$, and satisfies $\pi(\gamma(s)) \rightarrow \pi(y)$ as $s \rightarrow 1$. This implies that $\pi\circ \gamma$ is contained entirely inside of $\mathcal{F}(\pi(x))$, which is impossible since we have assumed that $y \notin\mathcal{F}(\pi(x))$. Thus by contradiction we conclude that $Df^{t}|_{L^{u}}$ is uniformly quasiconformal. 
\end{proof}

We will need a variant of Proposition \ref{exp to conf} in Section \ref{sec: vert}, which we establish below. Let $k = \dim E^{u}$, assume that $k \geq 3$, and let $1 \leq l \leq k-2$.

\begin{prop}\label{one exp to conf}
Let $f^{t}:M \rightarrow M$ be a smooth Anosov flow on a closed Riemannian manifold $M$ with $u$-splitting $E^{u} = L^{u} \oplus V^{u}$ and $s$-splitting $E^{s} = L^{s} \oplus V^{s}$ each of index $l$. Suppose that $Df^{t}|_{V^{u}}$ and $Df^{t}|_{V^{s}}$ are $1$-fiber bunched and that there is a complete smooth Riemannian submanifold $S$ of $M$ which is tangent to $V^{u} \oplus E^{c} \oplus V^{s}$. Assume that the inclusion $S \hookrightarrow M$ is uniformly bounded in the $C^2$ norm. Suppose that there is a simply connected Riemannian manifold $Y$ of pinched negative curvature and a uniformly continuous orbit equivalence $\varphi: S \rightarrow T^{1}Y$ from $f^{t}|_{S}$ to  $g^{t}_{Y}$.

Under these hypotheses, if there is a fully supported $f^{t}$-invariant ergodic probability measure $\mu$ with local product structure such that we have $\la_{l+1}^{u}(f^{t},\mu) = \la_{k}^{u}(f^{t},\mu)$ then $Df^{t}|_{V^{u}}$ is uniformly quasiconformal. 
\end{prop}

\begin{proof}
Since $Df^{t}|_{V^{u}}$ is $1$-fiber bunched, we have stable and unstable holonomies $H^{s}$ and $H^{u}$ along $\W^{s}$ and $\W^{u}$ respectively for this linear cocycle over $f^{t}$. By \cite[Lemma 2.5]{Bu1}, under the hypothesis that $\la_{l+1}^{u}(f^{t},\mu) = \la_{k}^{u}(f^{t},\mu)$ we conclude that either $Df^{t}|_{V^{u}}$ is uniformly quasiconformal or we have a proper $Df^{t}$-invariant continuous subbundle $F \subset V^{u}$ which is invariant under both $H^{s}$ and $H^{u}$. 

We apply the results of Section \ref{sec:uniform} to the action of $f^{t}$ on $\mathcal{V}^{u}$, treating $\mathcal{V}^{u}$ as an expanding foliation for $f^{t}$ along which $Df^{t}|_{V^{u}}$ is $1$-fiber bunched. We obtain smooth flat charts $e_{x}: V^{u}_{x} \rightarrow \mathcal{V}^{u}(x)$ from Proposition \ref{quotient charts}, from which we conclude as we did previously that $F$ is $C^{\infty}$ along $\mathcal{V}^{u}$ and tangent to a foliation $\mathcal{F}$ which is also $C^{\infty}$ along $\mathcal{V}^{u}$. 

The restriction of $f^{t}$ to $S$ is an Anosov flow on $S$ with Anosov splitting $V^{u} \oplus E^{c} \oplus V^{s}$. The $1$-fiber bunching conditions on $Df^{t}|_{V^{u}}$ and $Df^{t}|_{V^{s}}$ imply that $f^{t}|_{S}$ is $1$-bunched. While $S$ is not compact, the graph transform method of Hasselblatt \cite{Has2} requires only uniform bounds on the $C^2$ norm of $f^{t}$; these bounds are supplied by the fact that we are restricting our flow to a smooth submanifold $S$ whose inclusion into $M$ has uniformly bounded $C^2$ norm. We conclude that $V^{cu} = V^{u} \oplus E^{c}$ and $V^{cs} = V^{s} \oplus E^{c}$ are $C^1$ subbundles of $TS$. Furthermore, for $y \in \V^{cs}(x)$ the derivative $Dh^{\V^{cs}}: V^{u}_{x} \rightarrow V^{u}_{y}$ of the center-stable holonomy along $V^{cs}$ is given by the center-stable holonomy $H^{cs}$ of the linear cocycle $Df^{t}|_{V^{u}}$, as in Proposition \ref{wcs diff}.

We apply the proof of Proposition \ref{accessible curves}, replacing $\t{M}$ and $\t{f}^{t}$ with $S$ and $f^{t}|_{S}$ respectively, to obtain for any $x \in S$, $y \in \V^{u}(x)$ a curve $\gamma:[0,1] \rightarrow \V^{u}(x)$ joining $x$ to $y$ which is $C^1$ and tangent to $\mathcal{F}(x)$ on $[0,1)$. Choosing $y \notin \mathcal{F}(x)$ gives a contradiction, from which it follows that $Df^{t}|_{V^{u}}$ is uniformly quasiconformal. 
\end{proof}

Lastly we remark that Propositions \ref{exp to conf} and \ref{one exp to conf} hold if we assume an orbit equivalence to $g^{-t}_{Y}$ instead of $g^{t}_{Y}$. The flip map $v \rightarrow -v$ on $T^{1}Y$ conjugates $g^{-t}_{Y}$ to $g^{t}_{Y}$, and therefore a flow which is orbit equivalent to $g^{-t}_{Y}$ is also orbit equivalent to $g^{t}_{Y}$.

\section{From orbit equivalence to conjugacy}\label{sec:orbit to conj}

\subsection{Hamenst\"adt metrics}\label{subsec:metrics}
We consider a closed $C^{r+1}$ Riemannian manifold $M$, $r \geq 1$, with a foliation $\W^{cu}$ with uniformly $C^{r+1}$ leaves, and a flow $f^{t}$ that is uniformly $C^r$ along $\W^{cu}$ and has $\W^{cu}$ as a weak expanding foliation. We consider each leaf $\W^{u}(x)$ with its induced Riemannian metric $d_{x}$, noting that $d_{x} = d_{y}$ for $y \in \W^{u}(x)$.

We suppose that $f^{t}$ satisfies the following inequality for some $a > 0$, all $x \in M$ and all $t \geq 0$, 
\begin{equation}\label{Hamenstadt inequality}
\mathfrak{m}(Df^{t}_{x}|_{E^{u}_{x}}) \geq e^{a t}.
\end{equation}
Note that if this inequality is satisfied only up to some constant $C$, we can always use Proposition \ref{new metric} to choose some new Riemannian structure on $E^{u}$ for which the inequality holds without this constant. Inequality \eqref{Hamenstadt inequality} implies that we have 
\begin{equation}\label{distance growth}
d_{f^{t}x}(f^{t}y,f^{t}z) \geq e^{at}d_{x}(y,z),
\end{equation}
for every $x \in M$, $y,z \in \W^{u}(x)$, and $t \geq 0$. 

We define for $x \in M$ and $y \in \W^{u}(x)$, 
\begin{equation}\label{beta equation}
\beta(x,y) = \sup\{t \in \R: d_{f^{t}x}(f^{t}(x), f^{t}(y)) \leq 1\}, 
\end{equation}
which is finite due to \eqref{distance growth}. 

\begin{defn}[Hamenst\"adt metric]\label{defn: Hamenstadt metric} The \emph{Hamenst\"adt metric} $\rho_{x}$ on $\W^{u}(x)$ is defined by, for $y, z \in \W^{u}(x)$, 
\begin{equation}\label{metric definition}
\rho_{x}(y,z) = e^{-a\beta(y,z)}.
\end{equation}
\end{defn}

Note that $\rho_{x}$ depends on the choice of parameter $a > 0$, which we will fix to be the maximal $a$ such that inequality \eqref{Hamenstadt inequality} is satisfied. For a proof that $\rho_{x}$ satisfies the triangle inequality we refer to \cite{Has4}, in which the case that $\W^{u}$ is the unstable foliation of an Anosov flow is treated. The proof of the triangle inequality is identical in this setting. 

These metrics were introduced by Hamenst\"adt \cite{Ham90} in the context where $f^{t}$ is a geodesic flow and $\W^{u}$ is the unstable foliation. Hasselblatt \cite{Has4} showed that her formulation of this metric extends to the setting of Anosov flows.  The metric $\rho_{x}$ satisfies 
\begin{equation}\label{conformal scaling}
\rho_{f^{t}x}(f^{t}y,f^{t}z) = e^{at}\rho_{x}(y,z)
\end{equation}
for every $t \in \R$, i.e., $f^{t}$ is conformal on the $\W^{u}$ foliation in this family of metrics. We note that $\rho_{x} = \rho_{y}$ as metrics on $\W^{u}(x)$ for $y \in \W^{u}(x)$. We denote the ball of radius $r$ centered at $y$ in the metric $\rho_{x}$ by $B_{\rho}(y,r)$. Similarly, we  denote the ball of radius $r$ centered at $y$ in the metric $d_{x}$ by $B_{d}(y,r)$.

The map $(x,y,z) \rightarrow \rho_{x}(y,z)$ for $x \in M$, $y,z \in \W^{u}(x)$ is clearly  continuous in $(x,y,z)$, since $\beta(y,z)$ is continuous in $y$, $z \in \W^{u}(x)$. This implies that the Hamenst\"adt metrics induce the Euclidean topology on the leaves of $\W^{u}$. Since closed balls in the metric $d_{x}$ are compact, we conclude that these assignments are uniformly continuous on all closed $d_{x}$-balls in $\W^{u}(x)$. This implies that for any \emph{fixed} $r > 0$ there is a constant $C(r) \geq 1$ such that for all $x \in M$ and $y \in \W^{u}(x)$ 
\begin{equation}\label{Hamenstadt inclusion 1}
B_{d}(y,C(r)^{-1}) \subseteq B_{\rho}(y,r) \subseteq B_{d}(y,C(r) ),
\end{equation}
and
\begin{equation}\label{Hamenstadt inclusion 2}
B_{\rho}(y,C(r)^{-1}) \subseteq B_{d}(y,r) \subseteq B_{\rho}(y,C(r)).
\end{equation}

This can be used to verify a useful property of the Hamenst\"adt metrics. A metric space $(W,\rho)$ is \emph{doubling} if there is an  $\l \in \N$ such that for every $r > 0$ and any $x \in W$ the ball $B(x,2r)$ may be covered by at most $\l$ balls of radius $r$. We refer to $\l$ as the \emph{doubling constant} of $(W,\rho)$.  

\begin{lem}\label{doubling spaces}
The metric spaces $(\W^{u}(x),\rho_{x})$ are doubling. The doubling constant may be chosen independent of $x$.
\end{lem}

\begin{proof}
We start with the case of unit scale. Set $k = \dim \W^{u}$. There is a constant $R > 0$ such that $B_{\rho}(x,1) \subseteq B_{d}(x,R)$ and $B_{d}(x,R^{-1}) \subseteq  B_{\rho}(x,\frac{1}{2}) $  for all $x \in M$. The Riemannian ball $B_{d}(x,2R)$ is bi-Lipschitz to the unit ball in $\R^{k}$, with bi-Lipschitz constant depending only on $R$ ( independent of $x$). Hence, since the unit ball in $\R^{k}$ is doubling with respect to the Euclidean metric, the metric space $B_{d}(x,2R)$ is doubling for some constant $\l$ that may be taken independent of $x$. 

Let $j \in \N$ be the least integer such that $2^{j}R^{-1} \geq R$. Then $B_{d}(x,R)$ may be covered by $\l^{j}$ $d_{x}$-balls of radius $R^{-1}$ by repeated application of the doubling property. It follows that $B_{\rho}(y,1)$ may be covered by $\l^{j}$ $\rho_{x}$-balls of radius $\frac{1}{2}$. Consider now the ball $B_{\rho}(y,r)$ for some $r > 0$, $y \in \W^{u}(x)$. Set $t = -a^{-1}\log r$ so that $r = e^{-at}$. Then by the conformality of $f^{t}$ in the Hamenst\"adt metrics,   
\[
f^{t}(B_{\rho}(y,e^{-at})) = B_{\rho}(f^{t}y,1). 
\]
We then cover $B_{\rho}(f^{t}y,1)$ by $\l^{j}$ $\rho_{f^{t}x}$-balls $B_{\rho}(z_{i},\frac{1}{2})$ of radius $\frac{1}{2}$. Since
\[
f^{-t}\left(B_{\rho}\left(z_{i},\frac{1}{2}\right)\right) = B_{\rho}\left(f^{-t}z_{i},\frac{e^{-at}}{2}\right) = B_{\rho}\left(f^{-t}z_{i},\frac{r}{2}\right),
\]
this gives a covering of $B_{\rho_{x}}(y,r)$ by $\l^{j}$ balls of radius $\frac{r}{2}$, which completes the proof. 
\end{proof}

The comparisons \eqref{Hamenstadt inclusion 1} and \eqref{Hamenstadt inclusion 2} can also be used to show that the Hamenst\"adt metrics transform nicely under flow conjugacies. Let $g^{t}$ and $f^{t}$ be two flows on $M$ and $N$ respectively which are uniformly $C^r$ along foliations $\W^{cu,g}$ and $\W^{cu,f}$ of $M$ and $N$ respectively that each have uniformly $C^{r+1}$ leaves. We assume that $\W^{cu,g}$ and $\W^{cu,f}$ are weak expanding foliations for $g^{t}$ and $f^{t}$ respectively.  We let $a > 0$ be a common constant such that $\mathfrak{m}(Dg^{t}_{x}|_{E^{u,g}_{x}}) \geq e^{at}$ and $\mathfrak{m}(Df^{t}_{y}|_{E^{u,f}_{y}}) \geq e^{at}$ for all $x \in M$, $y \in N$, and $t \geq 0$, and define the Hamenst\"adt metrics $\rho_{x,g}$ and $\rho_{y,f}$ on leaves $\W^{u,g}(x)$ and $\W^{u,f}(y)$ accordingly. 

Let $\varphi: M \rightarrow N$ be a conjugacy from $g^{t}$ to $f^{t}$  such that $\varphi(\W^{u,g}(x)) = \W^{u,f}(\varphi(x))$ for each $x \in M$.

\begin{prop}\label{conjugacy invariant Hamenstadt}
For each $x \in M$ the homeomorphism $\varphi: (\W^{u,g}(x),\rho_{x,g}) \rightarrow (\W^{u,f}(\varphi(x)),\rho_{x,f})$ is bi-Lipschitz, with Lipschitz constant independent of $x$. 
\end{prop}

\begin{proof}
By the local uniform continuity of the Hamenst\"adt metrics from the perspective of the Riemannian metrics, there is a constant $C \geq 1$ such that if $x \in M$, $y,z \in \W^{u,g}(x)$ are such that $\rho_{x,g}(y,z) = 1$ then 
\[
C^{-1} \leq d_{\varphi(x),f}(\varphi(y),\varphi(z)) \leq C.
\]
By the comparisons \eqref{Hamenstadt inclusion 1} and \eqref{Hamenstadt inclusion 2}, we conclude that there is a constant $K \geq 1$ such that
\[
K^{-1} \leq \rho_{\varphi(x),f}(\varphi(y),\varphi(z)) \leq K,
\]
for all $x \in M$, $y,z \in \W^{u,g}(x)$ with $\rho_{x,g}(y,z) = 1$. 

Now suppose that $\rho_{x,g}(y,z) = r$, $r > 0$. Set $t = -a^{-1}\log r$ so that $r = e^{-at}$. Then by the conformality of $g^{t}$ in the Hamenst\"adt metrics,
\[
\rho_{g^{-t}x,g}(g^{-t}y,g^{-t}z) = 1.
\]
Thus 
\[
K^{-1} \leq \rho_{\varphi(g^{-t}x),f}(\varphi(g^{-t}y),\varphi(g^{-t}z)) \leq K.
\]
But
\begin{align*}
\rho_{\varphi(g^{-t}x),f}(\varphi(g^{-t}y),\varphi(g^{-t}z)) &= \rho_{f^{-t}\varphi(x),f}(f^{-t}\varphi(y),f^{-t}\varphi(z)) \\
&= e^{-at}\rho_{\varphi(x),f}(\varphi(y),\varphi(z)) \\
&= r^{-1}\rho_{\varphi(x),f}(\varphi(y),\varphi(z)).
\end{align*}
Thus, recalling that $\rho_{x,g}(y,z) = r$,
\[
K^{-1} \rho_{x,g}(y,z) \leq \rho_{\varphi(x),f}(\varphi(y),\varphi(z)) \leq K \rho_{x,g}(y,z),
\]
which establishes the desired claim. 
\end{proof}

When we consider orbit equivalences instead of conjugacies, the situation is more complicated. Let $(W,\rho_{W})$ and $(Z,\rho_{Z})$ be two metric spaces. 

\begin{defn}[Quasisymmetric homeomorphism]\label{quasisymmetry defn}
A homeomorphism $\varphi: W \rightarrow Z$ is \emph{quasisymmetric} if there is a homeomorphism  $\eta:[0,\infty) \rightarrow [0,\infty)$ such that, for all $x, y ,z \in W$, 
\[
\rho_{W}(x,y)  \leq s\rho_{W}(x,z) \Rightarrow \rho_{Z}(\varphi(x),\varphi(y)) \leq \eta(s)\rho_{Z}(\varphi(x),\varphi(z)).
\]
When the specific form of $\eta$ is important, we will say that $\varphi$ is $\eta$-quasisymmetric. 
\end{defn}

For a constant $C > 0$ we say that a homeomorphism $\varphi: W \rightarrow Z$ is \emph{weakly $C$-quasisymmetric} if for all $x,y,z \in W$, 
\[
\rho_{W}(x,y)  \leq \rho_{W}(x,z) \Rightarrow \rho_{Z}(\varphi(x),\varphi(y)) \leq C\rho_{Z}(\varphi(x),\varphi(z)).
\]
When the metric spaces under consideration are \emph{doubling}, weak quasisymmetry implies quasisymmetry. 

\begin{lem}\cite[Theorem 10.19]{Hein01}\label{weak quasisymmetry}
Let $(W,\rho_{W})$ and $(Z,\rho_{Z})$ be doubling metric spaces, with $W$  connected. Let $\varphi: W \rightarrow Z$ be a weakly $C$-quasisymmetric homeomorphism. Then $\varphi$ is quasisymmetric, quantitatively.
\end{lem}

By ``quantitatively" we mean that the function $\eta$ of Definition \ref{quasisymmetry defn} may be determined purely from the constant $C$ and the doubling constants of $W$ and $Z$.

As before, we let $g^{t}$ and $f^{t}$ be two flows on closed $C^{r+1}$ Riemannian manifolds $M$ and $N$ respectively which are uniformly $C^r$ along foliations $\W^{cu,g}$ and $\W^{cu,f}$ of $M$ and $N$ respectively that each have uniformly $C^{r+1}$ leaves. We assume that $\W^{cu,g}$ and $\W^{cu,f}$ are weak expanding foliations for $g^{t}$ and $f^{t}$ respectively.  We let $a > 0$ be a common constant such that $\mathfrak{m}(Dg^{t}_{x}|_{E^{u,g}_{x}}) \geq e^{at}$ and $\mathfrak{m}(Df^{t}_{y}|_{E^{u,f}_{y}}) \geq e^{at}$ for all $x \in M$, $y \in N$, and $t \geq 0$, and define the Hamenst\"adt metrics $\rho_{x,g}$ and $\rho_{y,f}$ on leaves $\W^{u,g}(x)$ and $\W^{u,f}(y)$ accordingly. 

Let $\varphi: M \rightarrow N$ be an orbit equivalence from $g^{t}$ to $f^{t}$  such that $\varphi(\W^{cu,g}(x)) = \W^{cu,f}(\varphi(x))$ for each $x \in M$. For all $t \in \R$ and $ x\in M$ we may write $\varphi (g^{t}(x))= f^{\alpha(t,x)}(\varphi(x))$, with $\alpha(t,x)$ being an additive cocycle over $g^{t}$.

Let $x \in M$. Assume for the moment that the leaf $\W^{cu,f}(\varphi(x))$ has the property that for any $y$, $z \in \W^{cu,f}(\varphi(x))$ there is a unique intersection point $\{w\} = \W^{c,f}(y) \cap \W^{u,f}(z)$. This will be true if and only if $\W^{cu,f}(\varphi(x))$ contains no periodic points of $f^{t}$ by an argument similar to the proof of Lemma \ref{one intersection}. In this case we define a homeomorphism $\varphi_{x}: \W^{u,g}(x) \rightarrow \W^{u,f}(\varphi(x))$ by setting $\varphi_{x}(y)$ to be the unique intersection point of $\W^{c,f}(\varphi(y))$ and $\W^{u,f}(\varphi(x))$. Directly from this definition we obtain 
\begin{equation}\label{equivariance}
\varphi_{g^{t}x} (g^{t}(y)) = f^{\alpha(t,x)} (\varphi_{x}(y)).
\end{equation}
Furthermore, for any $y \in \W^{u,g}(x)$, if we define $\xi(x,y) \in \R$ to be the unique number such that  $f^{\xi(x,y)}\varphi_{x}(y) = \varphi(y)$ then 
\begin{equation}\label{base change}
\varphi_{y} = f^{\xi(x,y)} \circ \varphi_{x},
\end{equation}
since if $z \in \W^{u,g}(x)$ then $\varphi_{x}(z) \in \W^{u,f}(\varphi(x))$ and therefore $f^{\xi(x,y)}(\varphi_{x}(z)) \in \W^{u,f}(\varphi(y))$, which implies that $f^{\xi(x,y)}(\varphi_{x}(z))$ is the intersection point of $\W^{c,f}(\varphi(z))$ with $\W^{u,f}(\varphi(y))$. 

In the general case in which $\W^{u,f}(\varphi(x))$ may contain a periodic orbit, we define $\varphi_{x}$ locally and extend by the equivariance property \eqref{equivariance}.

\begin{prop}\label{induced map}
For each $x \in M$ there is a unique homeomorphism $\varphi_{x}:\W^{u,g}(x) \rightarrow \W^{u,f}(\varphi(x))$ such that $\varphi_{x}(x) = \varphi(x)$ and for every $y \in \W^{u,g}(x)$ and $t \in \R$, 
\begin{equation}\label{equivariance 2}
\varphi_{g^{t}x} (g^{t}(y)) = f^{\alpha(t,x)} (\varphi_{x}(y)).
\end{equation}
Furthermore, for each $x \in M$, $y \in \W^{u,g}(x)$ there is $\xi(x,y) \in \R$ such that
\begin{equation}\label{base change 2}
\varphi_{y} = f^{\xi(x,y)} \circ \varphi_{x}.
\end{equation}
\end{prop}

\begin{proof}
We write $B^{cu}$ for balls in $\W^{cu}$ in the Riemannian metric and $B^{u}$ for balls in $\W^{u}$ in the Riemannian metric. Choose an $r > 0$ small enough that for each $x \in M$ and each point $y \in B^{cu}(x,r)$ there is a unique intersection point $z$ of $\W^{c,g}(y)$ with $\W^{u,g}(x)$ when these foliations are restricted to the ball $B^{cu}(x,r)$. We assume the analogous statements are true for $f^{t}$ on $N$ with respect to this $r$ as well. By the  continuity of $\varphi$, there is an $r' < r$ such that for all $x \in M$ we have 
\[
\varphi(B^{u}(x,r')) \subset B^{cu}(\varphi(x),r).
\]
We define $\varphi_{x}$ on $\W^{u,g}(x) \cap B^{cu}(x,r')$ to be the composition of $\varphi$ with the uniquely defined projection along $\W^{c,f}$ inside of $B^{cu}(\varphi(x),r)$ onto $\W^{u,f}(\varphi(x))$. It's clear that $\varphi_{x}(x) = x$. 

Equation \eqref{equivariance 2} can be verified exactly as in \eqref{equivariance} when $y \in \W^{u,g}(x) \cap B^{u}(x,r')$ and $g^{t}y \in \W^{u,g}(g^{t}x) \cap B^{u}(g^{t}x,r')$.
We then extend $\varphi_{x}$ to all of $\W^{u,g}(x)$ by equation \eqref{equivariance 2}: for $y \in \W^{u,g}(x)$ we choose $t$ large enough that $g^{-t}y \in B^{u}(g^{-t}x,r')$ and define
\begin{equation}\label{equivariant extension}
\varphi_{x}(y) = f^{\alpha(t,x)}(\varphi_{g^{-t}x}(g^{-t}y)).
\end{equation}
It's easily checked, using equation \eqref{equivariance 2} on the balls of radius $r'$, that this definition is independent of the choice of $t$ and thus gives an extension of $\varphi_{x}$ to $\W^{u,g}(x)$ which continues to satisfy equation \eqref{equivariance 2}. Equation \eqref{base change 2} clearly also holds for $y \in B^{u}(x,\frac{r'}{2})$ (so that $x \in B^{u}(y,r')$)  by the same reasoning as was used to verify \eqref{base change}, and one can easily verify that \eqref{base change 2} in fact holds for all $y \in \W^{u,g}(x)$ using the  extension \eqref{equivariant extension} by equivariance. 

We next show that $\varphi_{x}$ is a homeomorphism. It's clearly a homeomorphism onto its image when restricted to $B^{u}(x,r')$, and using equation \eqref{equivariant extension} we conclude from this that $\varphi_{x}$ is a local homeomorphism. Hence it suffices to show that $\varphi_{x}$ is injective and surjective. For injectivity, suppose that $\varphi_{x}(y) = \varphi_{x}(z)$ for some $y$, $z \in \W^{u,g}(x)$. Then by \eqref{equivariance 2} we have $\varphi_{g^{-t}x}(g^{-t}y) = \varphi_{g^{-t}x}(g^{-t}z)$ for all $t \geq 0$. But once $t$ is large enough that $ g^{-t}y, g^{-t}z \in B^{u}(g^{-t}x,r')$, the fact that $\varphi_{g^{-t}x}$ is a homeomorphism onto its image in this ball implies that $g^{-t}y = g^{-t}z$, and therefore $y = z$.

For surjectivity, let $y \in \W^{u,f}(\varphi(x))$. Choose $t > 0$ large enough that $f^{\alpha(-t,x)}y$ is in the image of $\varphi_{g^{-t}x}$; this is possible because $\varphi_{g^{-t}x}$ is a homeomorphism of $B^{u}(g^{-t}x,r')$  onto a neighborhood of $\varphi(g^{-t}x)$ inside of $\W^{u,f}(\varphi(g^{-t}x))$, which contains a ball of radius $r''$ independent of $x$ and $t$. Equation \eqref{equivariance 2} then implies that $y$ is in the image of $\varphi_{x}$. 

\end{proof}


We will now show that the homeomorphisms $\varphi_{x}$ are $\eta$-quasisymmetric with respect to the Hamenst\"adt metrics on each of these leaves, with $\eta$ being independent of the choice of $x$. 

\begin{prop}\label{quasisymmetry orbit}
There is a homeomorphism $\eta: [0,\infty) \rightarrow [0,\infty)$ such that, for every $x \in M$, the map $\varphi_{x}: (\W^{u,g}(x), \rho_{x,g}) \rightarrow (\W^{u,f}(\varphi(x)), \rho_{\varphi(x),f})$ is $\eta$-quasisymmetric. One may  take $\eta(s) = C\max\{s^{\nu},s^{1/\nu}\}$ for some $C \geq 1$ and $\nu \in (0,1]$ independent of $x$. 
\end{prop}

\begin{proof}
We will show that $\varphi_{x}$ is weakly $K$-quasisymmetric for some $K > 0$ that is independent of $x \in M$. The conclusion then follows from Lemma \ref{weak quasisymmetry} and \ref{doubling spaces}. Throughout this proof $C$ will denote a uniform constant independent of the parameters under consideration; its value may change from line to line. 

By the uniform continuity of $\varphi$, $f^{t}$, and $g^{t}$, as well as the fact that the Hamenst\"adt metrics $\rho_{x,g}$ and $\rho_{x,f}$ depend continuously on the choice of point $x \in M$, for every $x \in M$ and $y \in \W^{u,g}(x)$ we have that if $\rho_{x,g}(x,y) = 1$ then 
\begin{equation}\label{comparability}
\rho_{\varphi(x),f}(\varphi_{x}(x), \varphi_{x}(y)) \asymp 1.
\end{equation}
Let $x \in M$, and let $y \in \W^{u,g}(x)$. Define $\beta = \beta(x,y)$ as in equation \eqref{beta equation}, and observe by the conformal scaling property \eqref{conformal scaling} that
\begin{align*}
\rho_{g^{\beta}x,g}(g^{\beta}x,g^{\beta}y) = e^{a\beta}\rho_{x,g}(x,y) = e^{a\beta}e^{-a \beta} = 1, 
\end{align*}
by the defining equation \eqref{metric definition} for the Hamenst\"adt metric. We then have from \eqref{comparability} above that 
\[
\rho_{\varphi(g^{\beta}x),f}(\varphi_{g^{\beta}x}(g^{\beta}x),\varphi_{g^{\beta}x}(g^{\beta}y)) \asymp 1.
\]
By using equivariance \eqref{equivariance 2} and the conformal scaling property \eqref{conformal scaling} for $f^{t}$, we deduce that
\begin{equation}\label{2nd alpha inequality}
\rho_{\varphi(x),f}(\varphi_{x}(x),\varphi_{x}(y)) \asymp e^{-a \alpha(\beta,x)}.
\end{equation}

We next take two points $y, z \in \W^{u,g}(x)$ with $\rho_{x,g}(x,y) \leq \rho_{x,g}(x,z)$, or equivalently $\beta(x,y) \geq \beta(x,z)$.  Inequality \eqref{2nd alpha inequality}  implies that
\[
\frac{ \rho_{\varphi(x),f}(\varphi_{x}(x),\varphi_{x}(y)) }{ \rho_{\varphi(x),f}(\varphi_{x}(x),\varphi_{x}(z)) } \asymp  e^{-a (\alpha(\beta(x,y),x) - \alpha(\beta(x,z),x))}.
\]
By the additivity of $\alpha$ we have 
\[
\alpha(\beta(x,y),x) - \alpha(\beta(x,z),x) = \alpha(\beta(x,y)-\beta(x,z),g^{\beta(x,z)}x).
\]
Continuity of the generator of $\alpha$ and the compactness of $M$ implies that there is a constant $b > 0$ such that $\alpha(t,x) \geq bt$ for all $t \geq 0$ and $x \in M$. This implies in particular that 
\[
 \alpha(\beta(x,y)-\beta(x,z),g^{\beta(x,z)}x) \geq b(\beta(x,y)-\beta(x,z)),
\]
which implies that for some constant $C \geq 1$, 
\[
\frac{ \rho_{\varphi(x),f}(\varphi_{x}(x),\varphi_{x}(y)) }{ \rho_{\varphi(x),f}(\varphi_{x}(x),\varphi_{x}(z)) } \leq C\frac{e^{-ab\beta(x,y)}}{e^{-ab\beta(x,z)}} = C\left(\frac{\rho_{x,g}(x,y)}{\rho_{x,g}(x,z)}\right)^{b}.
\]

Now let $w, y, z \in \W^{u,g}(x)$ be any given triple of points. By applying the previous computations with $w$ replacing $x$ and recalling that $\rho_{w ,g} = \rho_{x,g}$, we have the inequality
\begin{equation}\label{w inequality}
\frac{ \rho_{\varphi(w),f}(\varphi_{w}(w),\varphi_{w}(y)) }{ \rho_{\varphi(w),f}(\varphi_{w}(w),\varphi_{w}(z)) } \leq C\left(\frac{\rho_{x,g}(w,y)}{\rho_{x,g}(w,z)}\right)^{b},
\end{equation}
for $\rho_{x,g}(w,y) \leq \rho_{x,g}(w,z)$. By equation \eqref{base change 2}  we have 
\[
\rho_{\varphi(w),f}(\varphi_{w}(w),\varphi_{w}(y))  =  e^{a\xi(x,w)} \rho_{\varphi(x),f}(\varphi_{x}(w),\varphi_{x}(y))
\]
This implies, going back to inequality \eqref{w inequality}, that we have 
\[
\frac{ \rho_{\varphi(x),f}(\varphi_{x}(w),\varphi_{x}(y)) }{ \rho_{\varphi(x),f}(\varphi_{x}(w),\varphi_{x}(z)) } \leq C\left(\frac{\rho_{x,g}(w,y)}{\rho_{x,g}(w,z)}\right)^{b},
\]
if $\rho_{x,g}(w,y) \leq \rho_{x,g}(w,z)$, as the factors of  $e^{a\xi(x,w)}$ introduced on the top and bottom of the left side cancel. We conclude that there is a constant $C > 0$ such that, for each $x \in M$ and any $w,y,z \in W^{u,g}(x)$, 
\[
\rho_{x,g}(w,y) \leq \rho_{x,g}(w,z) \Rightarrow \rho_{\varphi(x),f}(\varphi_{x}(w),\varphi_{x}(y)) \leq C\rho_{\varphi(x),f}(\varphi_{x}(w),\varphi_{x}(z)), 
\]
which establishes that $\varphi_{x}$ is weakly $C$-quasisymmetric, as desired. The final claim on the specific form of $\eta$ follows from \cite[Theorem 11.3]{Hein01}.
\end{proof}

\subsection{Synchronization}\label{subsec:synchro} 
With our preparatory work complete, we begin the proof of Theorem \ref{core theorem}. We let $X$ be our given closed negatively curved locally symmetric space of nonconstant negative curvature and set $g^{t}:=g^{t}_{X}$. We let $f^{t}: M \rightarrow M$ be our smooth Anosov flow satisfying the hypotheses of the theorem. We let $\varphi: T^{1}X \rightarrow M$ be the given orbit equivalence from $g^{t}$ to $f^{t}$, which we may assume to be H\"older by Fact \ref{structural stability}. Let $TM = E^{u} \oplus E^{c} \oplus E^{s}$ be the Anosov splitting for $f^{t}$ and set $k = \dim E^{u} = \dim E^{s}$. Let $E^{u} = L^{u} \oplus V^{u}$ and $E^{s} = L^{s} \oplus V^{s}$ be the $u$-splitting and $s$-splitting for $f^{t}$ respectively, and set $l = \dim L^{u} = \dim L^{s}$. 

Hypotheses (1), (2), and (4) of Theorem \ref{core theorem} imply that there exists a $\beta > 0$ such that $f^{t}$ is $\beta$-$u$-bunched, $Df^{t}|_{L^{u}}$ is $\beta$-fiber bunched, and $L^{cu}$ is $\beta$-H\"older. Since $\vec{\la}^{u}(f,\mu_{f}) = c_{1}\vec{\la}^{u}(g_{X})$ and the first $l$ coordinates of $\vec{\la}^{u}(g_{X})$ are $1$, we conclude that $\la_{1}^{u}(f,\mu_{f}) = \la_{l}^{u}(f,\mu_{f})$. Lifting to the universal covers, we obtain a uniformly continuous orbit equivalence $\t{\varphi}: T^{1}\t{X} \rightarrow \t{M}$ between the lifted flows $\t{g}^{t}$ and $\t{f}^{t}$. We take $\t{X}$ to be the simply connected Riemannian manifold of pinched negative curvature in Proposition \ref{exp to conf}, and recall that since $\mu_{f}$ is an equilibrium state it has local product structure. We thus conclude by Proposition \ref{exp to conf} that $Df^{t}|_{L^{u}}$ is uniformly quasiconformal. 

Consider the quotient action $\bar{D}f^{t}$ of $Df^{t}$ on the bundle $\bar{E}^{u} = E^{cu}/E^{c}$ with dominated splitting $\bar{E}^{u} = \bar{L}^{u} \oplus \bar{V}^{u}$ induced from the $u$-splitting $E^{u} = L^{u} \oplus V^{u}$. The linear cocycles $Df^{t}|_{L^{u}}$ and $\bar{D}f^{t}|_{\bar{L}^{u}}$ over $f^{t}$ are conjugate by the projection identification $L^{u} \rightarrow \bar{L}^{u}$ and therefore $\bar{D}f^{t}|_{\bar{L}^{u}}$ is also uniformly quasiconformal. Setting $\bar{Q}^{u} = \bar{E}^{u}/\bar{V}^{u}$ and writing $\bar{D}f^{t}|_{\bar{Q}^{u}}$ for the quotient action on this bundle, we conclude that $\bar{D}f^{t}|_{\bar{Q}^{u}}$ is uniformly quasiconformal as well since it is conjugate to  $\bar{D}f^{t}|_{\bar{L}^{u}}$ via the projection $\bar{L}^{u} \rightarrow \bar{Q}^{u}$. 

For this next step we briefly return to the more general setting considered in Section \ref{subsec: linear}. Let $A^{t}: E \rightarrow E$ be a linear cocycle on a vector bundle $E$ over a continuous flow  $f^{t}:M \rightarrow M$ on a compact metric space $M$. For each $x \in M$ we write $\mathcal{R}E_{x}$ for the space of all inner products on $E_{x}$, which forms a fiber bundle $\mathcal{R}E$ over $M$. We write $\mathcal{C}E_{x}$ for the quotient of $\mathcal{R}E_{x}$ by scaling the inner products by positive constants, and write $\mathcal{C}E$ for the resulting fiber bundle over $M$. A Riemannian structure on $E$ is simply a continuous section of $\mathcal{R}E$ over $M$. We refer to a continuous section $\delta: M \rightarrow \mathcal{C}E$ as a \emph{conformal structure} on $E$. A Riemannian structure on $E$ induces a unique conformal structure on $E$. 

For each $t \in \R$, the linear map $A_{x}^{t}: E_{x} \rightarrow E_{f(x)}$ induces a map $(A_{x}^{t})^{*}: \mathcal{R}E_{f^{t}x} \rightarrow \mathcal{R}E_{x}$ by defining for an inner product $\kappa_{f^{t}x} \in \mathcal{R}E_{f^{t}x}$ and $v$, $w \in E_{f^{t}x}$,  
\begin{equation}\label{inner product map}
(A_{x}^{t})^{*}\kappa_{f^{t}x}(v,w) = \kappa_{f^{t}x}(A_{x}^{t}(v), A_{x}^{t}(w)). 
\end{equation}
Formula \eqref{inner product map} induces a corresponding quotient action $(A^{t}_{x})^{*}: \mathcal{C}E_{f(x)} \rightarrow \mathcal{C}E_{x}$. A conformal structure $\delta: M \rightarrow \mathcal{C}E$ is \emph{$A^{t}$-invariant} if $(A^{t}_{x})^{*}\delta_{f^{t}x} = \delta_{x}$ for each $x \in M$ and $t \in \R$. If $A^{t}$ preserves a conformal structure $\delta$, then $A^{t}$ will be conformal in any  Riemannian structure on $E$ that induces the conformal structure $\delta$.

Fix a Riemannian structure $\kappa$ on $E$. For each $x \in M$ we consider the space of inner products $\mathcal{R}_{1}E_{x} \subset \mathcal{R}E_{x}$ which are \emph{unimodular} with respect to $\kappa_{x}$, i.e., $\eta \in \mathcal{R}_{1}E_{x}$ if and only if in some (hence any) $\kappa$-orthonormal basis $\{e_{1},\dots,e_{k}\}$ of $E_{x}$ we have 
\[
\det(\kappa_{x}(e_{i},e_{j}))_{1\leq i,j\leq k} = 1.
\]
The projection $\mathcal{R}E_{x} \rightarrow \mathcal{C}E_{x}$ restricts to a smooth diffeomorphism $\mathcal{R}_{1}E_{x} \rightarrow \mathcal{C}E_{x}$. Under this identification a conformal structure $\delta$ on $E$ can be uniquely realized as a Riemannian structure $\delta: M \rightarrow \mathcal{R}_{1}E$. 

The fixed choice of background Riemannian structure $\kappa$ induces a Riemannian distance on each of the fibers $\mathcal{C}E_{x}$ of the bundle $\mathcal{C}E$ in which the induced maps $(A^{t}_{x})^{*}$ are isometries for each $x \in M$ and all $t \in \R$, see \cite[Section 2.3]{KS}. A measurable section $\delta: M \rightarrow \mathcal{C}E$ is \emph{bounded} if there is a uniform bound on the distance of $\delta_{x}$ from the conformal structure induced by $\kappa_{x}$ for all points $x \in M$ at which $\delta$ is defined.

We return to the setting of this section, taking $M$ to be our smooth closed Riemannian manifold, taking $E$ to be the bundle $\bar{Q}^{u}$ over $M$, and taking $A^{t} = \bar{D}f^{t}|_{\bar{Q}^{u}}$. Since $\bar{D}f^{t}|_{\bar{Q}^{u}}$ is uniformly quasiconformal, we conclude by \cite[Proposition 2.4]{KS2} that there is a $\mu_{f}$-a.e. defined measurable bounded $A^{t}$-invariant section $\delta_{0}:M \rightarrow \mathcal{C}\bar{Q}^{u}$.  By \cite[Lemma 2.4]{Bu1}, the section $\delta_{0}$ coincides $\mu_{f}$-a.e.~ with a $\bar{D}f^{t}$-invariant conformal structure (which we will also denote by $\delta_{0}$) on $\bar{Q}^{u}$ that is invariant under the stable and unstable holonomies $\bar{H}^{s}$ and $\bar{H}^{u}$ of $\bar{D}f^{t}|_{\bar{Q}^{u}}$ acting on $\mathcal{C}\bar{Q}^{u}$. Since these holonomy maps have a H\"older dependence on pairs of points within $\W^{s}$ and $\W^{u}$ by Proposition \ref{existuhol}(3) we conclude that $\delta_{0}$ is H\"older along $\W^{s}$ and $\W^{u}$ and is therefore H\"older as a section $\delta_{0}:M\rightarrow M$. The $\bar{H}^{u}$ holonomies along $\W^{u}$ are given by the smooth parallel transport maps of the $\bar{D}f^{t}$-invariant connection $\nabla$ on $\bar{Q}^{u}$ constructed in Proposition \ref{holonomy connect}. We conclude that $\delta_{0}$ is smooth along $\W^{u}$. Since $\bar{D}f^{t}|_{\bar{Q}^{u}}$ is smooth along $\W^{c}$ and $\delta_{0}$ is $\bar{D}f^{t}|_{\bar{Q}^{u}}$-invariant, we conclude that $\delta_{0}$ is also smooth along $\W^{c}$ and therefore $\delta_{0}$ is smooth along $\W^{cu}$. 

We consider $\bar{Q}^{u}$ with its induced Riemannian structure from the smooth Riemannian structure on $TM$. This induces an identification of $\mathcal{C}\bar{Q}^{u}$  with $\mathcal{R}_{1}\bar{Q}^{u}$ that is smooth along $\W^{cu}$. We let $\delta:M \rightarrow \mathcal{R}_{1}\bar{Q}^{u}$ be the Riemannian structure corresponding to the conformal structure $\delta_{0}$ above. We let 
\begin{equation}\label{unstable quotient potential}
\bar{\zeta}_{f}(x) = \left.\frac{d}{dt}\right|_{t=0} \log \mathrm{Jac}_{\delta}\left(\bar{D}f^{t}_{x}|_{\bar{Q}^{u}_{x}}\right),
\end{equation}
with the Jacobian being taken in the Riemannian structure $\delta$. The potential $q_{f}\bar{\zeta}_{f}$ is H\"older and smooth along $\W^{cu}$. Furthermore it is cohomologous to the potential $q_{f}\zeta_{f}$ \eqref{potential} whose equilibrium state is the horizontal measure $\mu_{f}$ since the cocycles $Df^{t}|_{Q^{u}}$ and $\bar{D}f^{t}|_{Q^{u}}$ are conjugate by the identification $Q^{u} \rightarrow \bar{Q}^{u}$ induced from the identification $E^{u} \rightarrow \bar{E}^{u}$ given by the projection $E^{u} \rightarrow E^{cu}/E^{c}$. Hence $q_{f}\bar{\zeta}_{f}$ has the same equilibrium state $\mu_{f}$. We let $h^{t}_{u}$ be the time change of $f^{t}$ with speed mutliplier $l \bar{\zeta}_{f}^{-1}$. Then $h^{t}_{u}$ is a $u$-smooth Anosov flow.

We may carry out this entire discussion replacing $f^{t}$ with $f^{-t}$. The hypotheses of Theorem \ref{core theorem} imply that $f^{-t}$ is $\beta$-$u$-bunched, that $Df^{-t}|_{L^{s}}$ is $\beta$-fiber bunched, that $L^{cs}$ is $\beta$-H\"older, and that $\la_{1}^{u}(f^{-1},\mu_{f^{-1}}) = \la_{l}^{u}(f^{-1},\mu_{f^{-1}})$. Applying Proposition \ref{exp to conf}, we conclude that $Df^{t}|_{L^{s}}$ is uniformly quasiconformal, from which it follows that $\bar{D}f^{t}|_{\bar{Q}^{s}}$ is uniformly quasiconformal. We then construct a H\"older Riemannian structure $\delta'$ on $\bar{Q}^{s}$, smooth along $\W^{cs}$, with respect to which $\bar{D}f^{t}|_{\bar{Q}^{s}}$ is conformal, in exactly the same manner as we constructed $\delta$ above. We then have a corresponding potential
\begin{equation}\label{stable quotient potential}
\bar{\zeta}_{f^{-1}}(x) = \left.\frac{d}{dt}\right|_{t=0} \log \mathrm{Jac}_{\delta'}\left(\bar{D}f^{-t}_{x}|_{\bar{Q}^{s}_{x}}\right),
\end{equation}
We let $h^{t}_{s}$ be the time change of $f^{t}$ with speed mutliplier $l \bar{\zeta}_{f^{-1}}^{-1}$. Then $h^{t}_{s}$ is a $u$-smooth Anosov flow that is smooth along $\W^{cs}$ and has $\W^{cs}$ as a weak expanding foliation.

Set $h^{t} = h^{t}_{u}$ and set $\gamma =  l \bar{\zeta}_{f}^{-1}$. We will only be considering this flow for the rest of the section; the arguments for $h^{t}_{s}$ are analogous. We let $\tau$ be the additive cocycle over $h^{t}$ generated by $\gamma$ and let $\omega$ be the additive cocycle over $f^{t}$ generated by $\gamma^{-1} = l^{-1} \bar{\zeta}_{f}$. We write $\|\cdot \|_{\delta}$ for the norm in the Riemannian structure $\delta$ on $\bar{Q}^{u}$. By conformality, for any $t \in \R$, $x \in M$, and $v \in \bar{Q}^{u}_{x}$ we have
\[
\|\bar{D}f^{t}(v)\|_{\delta} = \mathrm{Jac}_{\delta}\left(\bar{D}f^{t}_{x}|_{\bar{Q}^{u}_{x}}\right)^{l^{-1}}\|v\|_{\delta} = e^{\omega(t,x)}\|v\|_{\delta}.
\]
Therefore
\begin{align*}
\|\bar{D}h^{t}(v)\|_{\delta} &= \|\bar{D}f^{\tau(t,x)}(v)\|_{\delta}  \\
&= e^{\omega(\tau(t,x),x)}\|v\|_{\delta} \\
&= e^{t}\|v\|_{\delta},
\end{align*}
by Proposition \ref{inverse additive}. After pulling $\delta$ back to $\bar{L}^{u}$ by the projection $\bar{L}^{u} \rightarrow \bar{Q}^{u}$, we have the equation for any $t \in \R$ and $v \in \bar{L}^{u}$, 
\begin{equation}\label{conformal L}
\|\bar{D}h^{t}(v)\|_{\delta} = e^{t}\|v\|_{\delta}.
\end{equation}
Since $\bar{E}^{u} = \bar{L}^{u} \oplus \bar{V}^{u}$ is a dominated splitting for $\bar{D}h^{t}$, there are constants $\chi > 1$, $c > 0$ such that for all $t \geq 0$, 
\begin{equation}\label{chi inequality}
\mathfrak{m}(D\bar{f}^{t}|_{\bar{V}^{u}}) \geq c e^{\chi t}.
\end{equation}
Using Proposition \ref{new metric}, we can choose a new Riemannian structure on $\bar{V}^{u}$ for which this inequality holds with $c = 1$. We will denote this Riemannian structure by $\delta$ as well. We then extend $\delta$ to the full bundle $\bar{E}^{u}$ by declaring $\bar{L}^{u}$ and $\bar{V}^{u}$ to be orthogonal. We then have for all $x \in M$ and $t \geq 0$,
\begin{equation}\label{hamenstadt synchro inequality}
\mathfrak{m}_{\delta}(\bar{D}h^{t}_{x}|_{\bar{E}^{u}_{x}}) = e^{t}, 
\end{equation}
with $\mathfrak{m}_{\delta}$ denoting the conorm in $\delta$. 

By Proposition \ref{time change splitting} we have a $Dh^{t}$-invariant splitting $E^{cu} = E^{u,h} \oplus E^{c}$ and by Proposition \ref{unstable manifold flow} there is an unstable foliation $\W^{u,h}$ for $h^{t}$ that is smooth inside of $\W^{cu}$. The projection $E^{cu} \rightarrow \bar{E}^{u}$ induces an identification $E^{u,h} \rightarrow \bar{E}^{u}$ that is smooth along $\W^{cu}$. The dominated splitting $E^{u,h} = L^{u,h} \oplus V^{u,h}$ for $Dh^{t}$ corresponds to the dominated splitting $\bar{E}^{u} = \bar{L}^{u} \oplus \bar{V}^{u}$ for $\bar{D}h^{t}$ under this identification. We give $E^{u,h}$ the Riemannian structure $\{\langle \; , \; \rangle_{x}\}_{x \in M}$ that is isometric to $\delta$ under the identification $E^{u,h} \cong \bar{E}^{u}$.

By Propositions \ref{proto synchro} and \ref{time change invariance} we have that $q_{h} = q_{f}$, that the horizontal measure $\mu_{h}$ for $h^{t}$ is equivalent to the horizontal measure for $f^{t}$ and that the horizontal measure is the measure of maximal entropy for $h^{t}$ with $h_{\mathrm{top}}(h) = q_{f} l$. Furthermore, for all $x \in M$ and $t \in \R$, in the induced Riemannian structure on $Q^{u,h}$ from $E^{u,h}$, 
\[
q_{f}\log \mathrm{Jac}\left(Dh^{t}_{x}|_{Q^{u,h}_{x}}\right) = q_{f}l t.
\]
By \eqref{hamenstadt synchro inequality} we have for $x \in M$ and $t \geq 0$, 
\[
\mathfrak{m}(Dh^{t}_{x}|_{Q^{u,h}_{x}}) = e^{t}
\]
Hence the Hamenst\"adt metrics $\rho_{x,h}$ for $h^{t}$ on $\W^{u,h}(x)$ are defined with parameter $a = 1$ in \eqref{Hamenstadt inequality}.

\begin{rem}
The method of synchronizing an Anosov flow with respect to a certain potential has a long history, going back to Parry \cite{Par86} and Ghys \cite{Ghys87} in the case of the potential $J_{f}$ for the SRB measure in Definition \ref{SRB} when the Anosov splitting $TM = E^{u} \oplus E^{c} \oplus E^{s}$ is $C^1$ (and therefore $J_{f}$ is $C^1$). Our use of the synchronization method expands on the techniques of Hamenst\"adt \cite{Ham97}, who considered synchronizations for geodesic flows of negatively curved manifolds with respect to H\"older potentials that were only smooth along the center-unstable foliation. 
\end{rem}

We recall the notion of Hausdorff dimension for metric spaces. Given a metric space $(Z,\rho)$, for each $\alpha > 0$ and $\e > 0$ we can define a premeasure $\mathcal{H}_{\alpha}^{\e}$ on $(Z,\rho)$ by, for any subset $A \subset Z$, 
\[
\mathcal{H}_{\alpha}^{\e}(A) = \inf \sum_{B \in \mathcal{B}} (\mathrm{diam} B)^{\alpha},
\]
where $\mathcal{B}$ denotes any covering of $A$ by balls $B$ of radius at most $\e$. We then define $\mathcal{H}_{\alpha}(A) = \lim_{\e \rightarrow 0}\mathcal{H}_{\alpha}^{\e}(A)$ to be the \emph{$\alpha$-Hausdorff measure} of $A$. The \emph{Hausdorff dimension} of $(Z,\rho)$ is then defined to be
\[
\mathrm{Hd}(Z,\rho) = \inf\{\alpha > 0:\mathcal{H}_{\alpha}(Z) = 0\}.
\]

We will need the following lemma. As before we write $B_{\rho}(x,r)$ for the ball of radius $r$ in the Hamenst\"adt metric $\rho_{x,h}$ on $\W^{u,h}(x)$. Set $N_{f} = q_{f}l$.

\begin{lem}\label{Hausdorff dim computation}
For each $x \in M$ we have
\[
\mathrm{Hd}(B_{\rho}(x,1),\rho_{x,h}) = N_{f}. 
\]
\end{lem}

\begin{proof}
By Proposition \ref{time change Gibbs property}, there is an $R > 0$ and a $T > 0$ such that for any $x \in M$ and $t \geq T$ we have
\begin{equation}\label{Gibbs use}
\mu_{h^{-t}x,h}^{u}(h^{-t}(B_{d}(x,R))) \asymp e^{-N_{f} t} \mu_{x,h}^{u}(B_{d}(x,R)). 
\end{equation}
Here we are using the Riemannian balls $B_{d}$ inside $\W^{u,h}(x)$. Applying $h^{-t}$ to the comparison \eqref{Hamenstadt inclusion 1} of the Riemannian metric and Hamenst\"adt metric on $\W^{u,h}(x)$ at scale $R$, we obtain for any $y \in \W^{u,h}(x)$ and $t \in \R$,
\[
h^{-t}(B_{d}(y,C(R)^{-1})) \subseteq B_{\rho}(h^{-t}y,e^{-t}R) \subseteq h^{-t}(B_{d}(y,C(R))).
\]
Combining this with \eqref{Gibbs use}, we obtain for $t \geq T$,
\[
\mu_{h^{-t}x,h}^{u}(B_{\rho}(h^{-t}x,e^{-t}R)) \asymp e^{-N_{f}t},
\]
with the implied constant depending only on $R$. Setting $r = e^{-t}R$ and $y = h^{-t}x$, this then becomes
\[
\mu_{y,h}^{u}(B_{\rho}(y,r)) \asymp r^{N_{f}},
\]
which is valid for any $y \in M$ and $r \leq R' = e^{-T}R$, with implied constant depending only on $R$. 

Since the ball $B_{\rho}(x,1)$ has compact closure, the conditional measures $\mu_{y,h}^{u}$ for $y \in B_{\rho}(x,1)$ are all uniformly comparable, since they are all multiples of one another. Hence we obtain for $y \in B_{\rho}(x,1)$ and $r \leq R'$, 
\begin{equation}\label{r proportionality}
\mu_{x,h}^{u}(B_{\rho}(y,r)) \asymp r^{N_{f}},
\end{equation}
Let $ 0 < \e \leq R'$ be given, and let $B_{\rho}(x_{i},r_{i})$ be a covering of $B_{\rho}(x,1)$ by balls of radius $r_{i} \leq \e$. By a standard covering lemma \cite[Theorem 1.2]{Hein01}, by passing to a subcover we may assume that
\[
B_{\rho}\left(x_{i},\frac{r_{i}}{5}\right) \cap B_{\rho}\left(x_{j},\frac{r_{j}}{5}\right) = \emptyset,
\]
for $i \neq j$. Then 
\begin{align*}
\mu^{u}_{x,h}(B_{\rho}(x,1)) &\geq  \sum_{i} \mu^{u}_{x,h}\left(B_{\rho}\left(x_{i},\frac{r_{i}}{5}\right) \right) \\
&\geq c\sum_{i} \mu^{u}_{x,h}\left(B_{\rho}\left(x_{i},r_{i}\right) \right) \\
&\geq c\mu_{x,h}(B_{\rho}(x,1)),
\end{align*}
for some constant $c > 0$. The second inequality follows from \eqref{r proportionality}. Using \eqref{r proportionality} on the sum in that second line, we obtain
\[
\mu_{x,h}^{u}(B_{\rho}(x,1)) \asymp \sum_{i} r_{i}^{N_{f}}. 
\]
Letting $\e \rightarrow 0$, it follows that $0 < \mathcal{H}_{N_{f}}(B_{\rho}(x,1)) < \infty$ and therefore $\mathrm{Hd}(B_{\rho}(x,1),\rho_{x}) = N_{f}$. 
\end{proof}

A \emph{metric measure space} is a triple $(Z,\rho,\mu)$, with $(Z,\rho)$ a metric space and $\mu$ a Borel measure on $Z$. 

\begin{defn}[Ahlfors regular] A metric measure space is \emph{Ahlfors $N$-regular} for an exponent $N > 0$ if there is a constant $C \geq 1$ such that for every $x \in Z$ and $r \geq 0$ we have
\begin{equation}\label{Ahlfors regular}
C^{-1} r^{N}\leq \mu(B(x,r)) \leq C r^{N}.
\end{equation}
\end{defn}

An Ahlfors $N$-regular metric measure space has Hausdorff dimension $N$.

Let $G$ be the 2-step Carnot group on which the unstable manifolds of $g^{t}$ are modeled, from the discussion of Section \ref{subsec:neg curved}. The metric measure space $(G,\rho_{G},m_{G})$ is Ahlfors $N_{G}$-regular, where $\rho_{G}$ is the CC-metric on $G$, $m_{G}$ is a left invariant Riemannian volume on $G$, and $N_{G} = l + 2(k-l)$. The charts $T_{x}: (G,\rho_{G}) \rightarrow (\W^{u,g}(x),\rho_{x,g})$ for each $x \in T^{1}X$ are bi-Lipschitz with Lipschitz constant independent of $x \in T^{1}X$, and the metric measure space $(\W^{u,g}(x),\rho_{x,g},m_{x,g})$ is Ahlfors $N_{G}$-regular as well, where $m_{x,g}$ denotes the $N_{G}$-Hausdorff measure of $\rho_{x,g}$ on $\W^{u,g}(x)$ \cite{Ham90}.


\begin{prop}\label{lem:bound}
We have $q_{f} = q_{X}$ and consequently $N_{f} = N_{G}$. The horizontal measure $\mu_{h}$ is the SRB measure for $h^{t}$. 
\end{prop}

For the proof of this proposition we require some basic notions from quasiconformal analysis on metric spaces, which may be found in \cite{Hein01}.

Let $(Z,\rho)$ be a metric space and let $\gamma:I \rightarrow Z$ be a continuous map from an open interval $I \subseteq \R$ to $Z$, which we will refer to as a \emph{curve} in $Z$. A curve $\gamma$ is \emph{rectifiable} if it has finite length, and in this case we will usually assume that $\gamma$ is parametrized by arclength and identify it with its image. Given a Borel function $\beta: Z \rightarrow \R$, one may then define the line integral of $\beta$ on $\gamma$ in a manner analogous to that in Euclidean space. 

We assume now that $(Z,\rho,\mu)$ is a metric measure space. Let $1 \leq p < \infty$ be given. Consider a family $\Gamma$ of rectifiable curves in $Z$. We say that a Borel function $\beta: Z \rightarrow [0,\infty]$ is \emph{admissible} for $\Gamma$ if
\[
\int_{\gamma} \beta \, ds \geq 1,
\]
for each $\gamma \in \Gamma$. 
\begin{defn}[Modulus of a curve family]The \emph{$p$-modulus} of $\Gamma$ is defined by
\begin{equation}\label{mod defn}
\mathrm{mod}_{p}(\Gamma) = \inf_{\beta}\int_{Z} \beta^{p}\, d\mu,
\end{equation}
with the infimum being taken over all admissible functions $\beta$ for $\Gamma$. 
\end{defn}

The $p$-modulus is an outer measure on the set of all rectifiable curves in $Z$. A property is defined to hold for $\mathrm{mod}_{p}$-a.e.~ curve in $Z$ if the set of rectifiable curves for which this property fails has $p$-modulus zero. 

We will require the following simple lemma regarding curve families in the 2-step Carnot group $G$.

\begin{lem}\label{positive modulus}
For any open subset $U \subseteq G$ there is a family of rectifiable curves $\Gamma$, with $\gamma \subset U$ for each $\gamma \in \Gamma$, such that $\mathrm{mod}_{N_{G}}(\Gamma) > 0$. 
\end{lem}

\begin{proof}
Let $L \subset TG$ denote the left-invariant horizontal distribution on $G$, tangent to $\mathfrak{l}$ at the identity element. Let $Z: G \rightarrow L$ be a left-invariant vector field of unit length and let $\mathcal{Z}$ denote the foliation of $G$ by the integral curves of $Z$. Each of these curves is rectifiable for $\rho_{G}$, and is in fact a geodesic between any two of its points. We equip $G$ with a choice of left-invariant Riemannian metric for reference below. 

Let $U \subseteq G$ be an open subset of $G$. Fix a point $x \in U$ and let $B$ be a foliation box for $\mathcal{Z}$ containing $x$, chosen small enough that $B \subset U$. We let $\Gamma$ be the curve family consisting of the intersections of curves from $\mathcal{Z}$ with $B$. Let $\beta: G \rightarrow [0,\infty]$ be an admissible function for $\Gamma$. We let $S \subset B$ be a smooth transversal to $\mathcal{Z}$ in the foliation box. Since we are bounding the modulus below, we can restrict to those $\beta$ for which $\beta \equiv 0$ on $G\backslash B$.  For $x \in S$ we let $\gamma_{x}$ be the intersection of $\mathcal{Z}(x)$ with $B$. Let $p$ be the conjugate exponent to $N_{G}$, so that $p^{-1} + N_{G}^{-1} = 1$. Then by H\"older's inequality,
\[
\int_{B} \beta\, d\mu_{G} \leq \mu_{G}(B)^{p^{-1}}\left(\int_{B} \beta^{N_{G}}\, d\mu_{G} \right)^{N_{G}^{-1}}
\]
By rearranging this, applying Fubini's theorem, and then using the admissibility condition on $\beta$ we obtain
\begin{align*}
\int_{U} \beta^{N_{G}}\, d\mu_{G} &\geq \mu_{G}(B)^{-N_{G}p^{-1}}\left(\int_{B} \beta\, d\mu_{G} \right)^{N_{G}} \\
&= \mu_{G}(B)^{-N_{G}p^{-1}}\int_{S}\int_{\gamma_{x}}\beta\, ds \,d\nu \\
&\geq \mu_{G}(B)^{-N_{G}p^{-1}} \nu(S),
\end{align*}
with $\nu$ being a volume on $S$ comparable to the Riemannian volume. Since this lower bound does not depend on $\beta$, we conclude that 
\[
\mathrm{mod}_{N_{G}}(\Gamma) \geq \mu_{G}(B)^{-N_{G}p^{-1}}\nu(S) > 0.
\]
\end{proof}

We will use the following theorem of Tyson \cite{Tys}, as formulated in \cite[Theorem 15.10]{Hein01}.

\begin{thm}\label{Tyson theorem}
Let $(Z,\rho,\mu)$ be a compact Ahlfors $N$-regular metric measure space, $N > 1$, which has a family of rectifiable curves $\Gamma$ satisfying $\mathrm{mod}_{N}(\Gamma) > 0$. Let $\varphi: Z \rightarrow Y$ be a quasisymmetric homeomorphism of $(Z,\rho)$ onto a metric space $(Y,d)$. Then $\mathrm{Hd}(Y,d) \geq N$. 
\end{thm}

A simple inspection of the proof of Theorem \ref{Tyson theorem} given in \cite[Theorem 15.10]{Hein01} shows that it suffices to assume that there is an $R > 0$ such that the Ahlfors regularity inequalities \eqref{Ahlfors regular} for $(Z,\rho,\mu)$ hold only for balls of radius at most $R$ that meet some curve $\gamma \in \Gamma$ (see \cite[Remark 15.15]{Hein01}).  This is the version of Theorem \ref{Tyson theorem} that we will use in the proof of Proposition \ref{lem:bound} below. 

\begin{proof}[Proof of Proposition \ref{lem:bound}]
Fix $x \in T^{1}X$. By Proposition \ref{quasisymmetry orbit}, 
\[
\varphi_{x}: (\W^{u,g}(x),\rho_{x,g}, m_{x,g}) \rightarrow (\W^{u,h}(\varphi(x)), \rho_{\varphi(x),h},\mu_{\varphi(x),h}^{u}),
\]
is a quasisymmetric homeomorphism. Let $B_{x} = \varphi_{x}^{-1}(B_{\rho}(\varphi(x),1))$. We then have a restricted quasisymmetric homeomorphism $\varphi_{x}:B_{x} \rightarrow B_{\rho}(\varphi(x),1)$. 

Let $T_{x}: G \rightarrow \W^{u,g}(x)$ be our standard chart identifying $\W^{u,g}(x)$ with the 2-step Carnot group $G$ equipped with its left-invariant Carnot-Caratheodory metric $\rho_{G}$ and Lebesgue measure $m_{G}$, which is Alhfors $N_{G}$-regular. Let $B = T_{x}^{-1}(B_{x}) \subset G$ and consider the homeomorphism $\psi_{x} = \varphi_{x} \circ T_{x}$ from $B$ to $B_{\rho}(\varphi(x),1)$.

By Proposition \ref{Hausdorff dim computation}, the space $B_{\rho}(\varphi(x),1)$ has Hausdorff dimension $N_{f}$. The ball $B_{\rho}(\varphi(x),1)$ contains an open neighborhood of $\varphi(x)$, and therefore $B$ contains an open subset $U$ of $G$. Let $U_{0} \subset U$ be an open subset with closure contained in $U$. By Lemma \ref{positive modulus}, we can find a family of rectifiable curves $\Gamma$ inside of $U_{0}$ with $\mathrm{mod}_{N_{G}}(\Gamma) > 0$. Observe also that for $r$ sufficiently small we will have 
\[
m_{G}(B_{\rho_{G}}(y,r)) \asymp r^{N_{G}},
\]
for any $y \in U_{0}$. 

Hence by Theorem \ref{Tyson theorem} we conclude that $N_{f} \geq N_{G}$, from which it follows that $q_{f} \geq q_{X}$.  Since we assumed that $\vec{\la}^{u}(f,\mu_{f}) = c \vec{\la}^{u}(g_{X})$ for some $c > 0$, we have that $\la_{k}(f,\mu_{f}) = 2\la_{1}(f,\mu_{f})$ and therefore by Proposition \ref{hor sup} we must have that $q_{f} = q_{X}$ and that $\mu_{f}$ is the SRB measure for $f^{t}$. Hence $\mu_{h}$ is the SRB measure for $h^{t}$ and $N_{f} = N_{G}$, which completes the proof. 
\end{proof}

Set $N = N_{G} = q_{X}l$ in what follows. We have shown that the SRB measure, horizontal measure, and measure of maximal entropy for $h^{t}$ all coincide, with $h_{\mathrm{top}}(h) = N$. Since the SRB measure is the measure of maximal entropy, by \eqref{cohomological equation plus} there is a H\"older function $\xi: M \rightarrow (0,\infty)$ such that for all $x \in M$ and $t \in \R$, 
\[
\frac{\xi(h^{t}x)}{\xi(x)} = \frac{\mathrm{Jac}(Dh^{t}_{x}|_{E^{u,h}_{x}})}{e^{N t}} = \frac{\mathrm{Jac}(Dh^{t}_{x}|_{V^{u,h}_{x}})}{e^{2(k-l)t}},
\]
recalling that we declared the subbundles $L^{u,h}$ and $V^{u,h}$ to be orthogonal in our Riemannian structure. Hence after multiplying the Riemannian structure on $V^{u,h}$ by $\xi$, we may assume that
\begin{equation}\label{vertical Jacobian}
\mathrm{Jac}(Dh^{t}_{x}|_{V^{u,h}_{x}}) = e^{2(k-l)t},
\end{equation}
for every $x \in M$, $t \in \R$, and consequently
\begin{equation}\label{Jacobian scaling}
\mathrm{Jac}(Dh^{t}_{x}|_{E^{u,h}_{x}}) = e^{Nt}.
\end{equation}

For each $x \in M$, we then have a metric measure space $(\W^{u,h}(x),\rho_{x,h},m_{x,h})$, with $\rho_{x,h}$ being the Hamenst\"adt metric on $\W^{u,h}(x)$ and $m_{x,h}$ being the Riemannian volume $\W^{u,h}(x)$. We note that we have $m_{x,h} = m_{y,h}$ for $y \in \W^{u,h}(x)$, and that by equation \eqref{Jacobian scaling},
\begin{equation}\label{measure scaling}
h^{t}_{*}m_{x,h} = e^{Nt}m_{h^{t}x,h},
\end{equation}
for all $x \in M$ and $t \in \R$. 

\begin{prop}\label{Hamenstadt Ahlfors}
The metric measure space $(\W^{u,h}(x),\rho_{x,h},m_{x,h})$ is Ahlfors $N$-regular for each $x \in M$, with regularity constant independent of $x$. 
\end{prop}

\begin{proof}
By the continuity of the Hamenst\"adt metrics  there is a constant $C \geq 1$ such that for all $x \in M$
\[
C^{-1}\leq m_{x,h}(B_{\rho}(x,1)) \leq C.
\]
Since $m_{x,h} = m_{y,h}$ for $y \in \W^{u,h}(x)$, it follows that we also have
\[
C^{-1}\leq m_{x,h}(B_{\rho}(y,1)) \leq C.
\]
Let $r > 0$ and $y \in \W^{u,h}(x)$ be given and write $r = e^{t}$ for some $t \in \R$. Then by equation \eqref{measure scaling},
\[
m_{x,h}(B_{\rho}(y,r)) = e^{Nt}m_{f^{t}x}(B_{\rho}(h^{t}y,1)).
\]
We conclude that
\[
m_{x,h}(B_{\rho}(y,r)) \asymp e^{Nt} = r^{N},
\]
from which the proposition follows.
\end{proof}



\subsection{From orbit equivalence to conjugacy}
Our goal in this section is to show that $g^{t}$ and $h^{t}$ are conjugate. Let $m_{g}$ and $m_{h}$ be the measures of maximal entropy for $g^{t}$ and $h^{t}$ respectively, which are also their SRB measures. We have shown above that
\[
h_{\mathrm{top}}(g) = h_{\mathrm{top}}(h) = N.
\]
We let $\varphi$ be the orbit equivalence from $g^{t}$ to $f^{t}$ given to us in the hypotheses of Theorem \ref{core theorem}, which is also an orbit equivalence from $g^{t}$ to $h^{t}$. We define the homeomorphisms $\varphi_{x}:\W^{u,g}(x) \rightarrow \W^{u,h}(\varphi(x))$ for each $x \in T^{1}X$ as in Proposition \ref{induced map}. By the Lebesgue decomposition theorem and Radon-Nikodym theorem there is an $m_{x,g}$-integrable function $J_{x}: \W^{u,g}(x) \rightarrow [0,\infty]$, and a signed measure $\kappa_{x}$ that is mutually singular with respect to $m_{x,g}$, such that we have 
\[
d((\varphi_{x})^{-1}_{*}m_{\varphi(x),h}) = J_{x} dm_{x,g} + d\kappa_{x}. 
\]
We will use the following criterion to produce the conjugacy. We note that we will be using throughout Lemma \ref{conjugacy criterion} and Proposition \ref{prop:orbit to conj} that the Riemannian volumes $m_{x,h}$ on the unstable leaves $\W^{u,h}$ are equivalent to the conditional measures of the SRB measure $m_{h}$ on $\W^{u,h}$. 

\begin{lem}\label{conjugacy criterion}
Suppose that 
\[
m_{g}\left(\left\{x \in T^{1}X: J_{x}(y) = 0 \; \text{for $m_{x,g}$-a.e.~ $y \in \W^{u,g}(x)$}\right\}\right) = 0.
\]
Then there is a conjugacy $\hat{\varphi}$ from $g^{t}$ to $h^{t}$ that is flow related to $\varphi$. 
\end{lem}

\begin{proof}  
For $z \in \W^{u,g}(x)$ we first relate the measures $(\varphi_{z})^{-1}_{*}m_{\varphi(z),h}$ and $(\varphi_{x})^{-1}_{*}m_{\varphi(x),h}$ on $\W^{u,g}(x)$. By Proposition \ref{induced map} we have $\varphi_{z} = h^{\xi(x,z)}\circ \varphi_{x}$ for a certain $\xi(x,z) \in \R$. We then have, using equation \eqref{measure scaling},
\begin{align*}
(\varphi_{z})^{-1}_{*}m_{\varphi(z),h} &= (\varphi_{x})^{-1}_{*}(h^{-\xi(x,z)}_{*}m_{\varphi(x),h}) \\
&= e^{-N\xi(x,z)}(\varphi_{x})^{-1}_{*}m_{\varphi(x),h}.
\end{align*}

We also compute how these measures behave under iteration by $g^{t}$. Let $\alpha$ be the additive cocycle over $g^{t}$ such that for all $x \in T^{1}X$, $t \in \R$ we have $\varphi_{g^{t}x} \circ g^{t} = h^{\alpha(t,x)} \circ \varphi_{x}$. This equation implies that
\begin{align*}
(\varphi_{g^{t}x})^{-1}_{*}m_{\varphi(g^{t}x),h} &= g^{t}_{*}(\varphi_{x}^{-1})_{*}h^{-\alpha(t,x)}_{*}m_{\varphi(g^{t}x),h} \\
&=e^{Nt-N\alpha(t,x)}(\varphi_{x})^{-1}_{*}m_{\varphi(x),h}.
\end{align*}

Now consider another $z \in \W^{u,g}(x)$. As with $x$, there is a signed measure $\kappa_{z}$ and a locally $m_{x,g}$-integrable function $J_{z}: \W^{u,g}(x) \rightarrow [0,\infty]$ such that
\[
d(\varphi_{z})^{-1}_{*}m_{\varphi(z),h} = J_{z}dm_{x,g} + d\kappa_{z}.
\]
By our previous calculations, 
\begin{align*}
(\varphi_{z})^{-1}_{*}m_{\varphi(z),h} &= e^{-N\xi(x,z)}(\varphi_{x})^{-1}_{*}m_{\varphi(x),h} \\
&= e^{-N\xi(x,z)}(J_{x}dm_{x,g}+ d\kappa_{x}).
\end{align*}
We conclude that we may take $J_{z} = e^{-N\xi(x,z)}J_{x}$ and $d\kappa_{z} = e^{-N\xi(x,z)}\kappa_{x}$ in the Radon-Nikodym decomposition for $(\varphi_{z})^{-1}_{*}m_{\varphi(z),h}$. We now define, for $m_{x,g}$-a.e.~ $z \in \W^{u,g}(x)$, 
\[
J(z):=J_{z}(z) = e^{-N\xi(x,z)}J_{x}.
\]
The function $J$ is locally $m_{x,g}$-integrable on $\W^{u,g}(x)$, so as a consequence we have $J(z) < \infty$ for $m_{x,g}$-a.e.~ $z$. By the formula $J(z) = J_{z}(z)$ we can then extend $J$ to a well-defined locally $m_{g}$-integrable function $J: T^{1}X \rightarrow [0,\infty]$ that is finite $m_{g}$-a.e. We also have from our prior calculations that for $x \in T^{1}X$ and $t \in \R$,
\[
J(g^{t}x) = e^{Nt-N\alpha(t,x)}J(x). 
\]
We conclude that $K:=\{x: J(x) \neq 0\}$ is a $g^{t}$-invariant set. 

By ergodicity of $g^{t}$ with respect to $m_{g}$, we have $m_{g}(K) = 0$ or $m_{g}(K) = 1$. Suppose that $m_{g}(K) = 0$. Then for $m_{g}$-a.e.~ $x \in T^{1}X$ and $m_{x,g}$-a.e.~ $z \in \W^{u,g}(x)$, 
\begin{align*}
0 = e^{-N\xi(x,z)}J(z) = e^{-N\xi(x,z)}J_{z}(z) = J_{x}(z).
\end{align*}
Thus for $m_{g}$-a.e.~ $x \in T^{1}X$ and $m_{x,g}$-a.e.~ $z \in \W^{u,g}(x)$, we have $J_{x}(z) = 0$. This contradicts the hypothesis of the proposition.  

Thus $m_{g}(K) = 1$. Since $0 < J(x) < \infty$ for $m_{g}$-a.e.~ $x$, we have for $m_{g}$-a.e.~ $x$,
\[
N^{-1}\log(J(g^{t}x)) - N^{-1}\log J(x) =  t-\alpha(t,x).
\]
Thus the additive cocycles $\alpha(t,x)$ and $t$ are measurably cohomologous over $g^{t}$. By the measurably rigidity of the cohomological equation discussed in Section \ref{thermo section}, these additive cocycles are H\"older cohomologous. Hence we have a H\"older function $\eta: T^{1}X \rightarrow \R$ such that for all $x \in M$ and $t \in \R$, 
\begin{equation}\label{cohomology orbit equivalence}
\eta(g^{t}x) - \eta(x) = t-\alpha(t,x).
\end{equation}
By the definition of an orbit equivalence, $\{\hat{h}^{t}\}_{t \in \R} = \{\varphi  \circ g^{t} \circ \varphi^{-1}\}_{t\in\R}$ defines a continuous flow on $M$ that is a time change of $h^{t}$. Then $\tau(t,x) = \alpha(t,\varphi^{-1}(x))$ defines an additive cocycle over $\hat{h}^{t}$ for which $\hat{h}^{t} = h^{\tau(t,x)}x$. Set $\chi(y) = \eta(\varphi^{-1}(y))$ for $y \in M$. By substituting $x = \varphi^{-1}(y)$ into \eqref{cohomology orbit equivalence} for $ y \in M$, we obtain
\[
\chi(\hat{h}^{t}y)  - \chi(y) = t-\tau(t,y).
\]
Set $\psi(y) = h^{\chi(y)}y$. Then we have
\[
\psi(h^{t}y) = h^{t+\chi(y)}y = h^{\tau(t,y) + \chi(\hat{h}^{t}y)}y = \hat{h}^{t}(\psi(y)).
\]
Thus $\psi \circ h^{t} = \hat{h}^{t} \circ \psi$. Setting $\hat{\varphi} = \psi^{-1} \circ \varphi$, we note that $\hat{\varphi}$ is flow related to $\varphi$ and conclude from the definition of $\hat{h}^{t}$ that $\hat{\varphi} \circ g^{t} = f^{t} \circ \hat{\varphi}$, i.e., $g^{t}$ and $h^{t}$ are conjugate by $\hat{\varphi}$. 
\end{proof}

\begin{prop}\label{prop:orbit to conj}
There is a conjugacy $\varphi_{u}:T^{1}X \rightarrow M$ from $g^{t}$ to $h^{t}$ that is flow related to $\varphi$.
\end{prop}

\begin{proof}
Fix $x \in T^{1}X$. From Proposition \ref{quasisymmetry orbit} we have a quasisymmetric homeomorphism of Ahlfors $N$-regular metric measure spaces,
\[
\psi_{x}:=\varphi_{x} \circ T_{x}: (G, \rho_{G}, m_{G}) \rightarrow (\W^{u,h}(\varphi(x)), \rho_{\varphi(x),h}, m_{\varphi(x),h}).
\]
Since $G$ is a Carnot group, since $\psi_{x}$ is quasisymmetric, and since both $G$ and $\W^{u,h}(\varphi(x))$ are Ahlfors $N$-regular, we may apply \cite[Theorem 5.2]{BKR07} to conclude that $\psi_{x}$ belongs to the Sobolev space $W^{1,N}_{loc}(G;\W^{u,h}(\varphi(x)))$, as defined in \cite{BKR07}. We will not require the particulars of the definition of this Sobolev space; we only require the standard consequence that $\psi_{x}$ is absolutely continuous on $\mathrm{mod}_{N}$-a.e. rectifiable curve in $G$. 

For $y \in G$ we define
\[
\mathcal{J}(y) = \lim_{r \rightarrow 0} \frac{m_{\varphi(x),h}(\psi_{x}(B_{\rho_{G}}(y,r)))}{m_{G}(B_{\rho_{G}}(y,r))}.
\]
As per the discussion in \cite[Section 4]{BKR07}, by the Radon-Nikodym theorem this limit exists and is finite for $m_{G}$-a.e.~ $y \in G$. By \cite[Proposition 4.3]{BKR07}, for every rectifiable curve $\gamma: [0,a] \rightarrow G$ on which $\psi_{x}$ is absolutely continuous we have the distance inequality
\begin{equation}\label{abs cont inequality}
\rho_{\varphi(x),f}(\psi_{x}(\gamma(0)),\psi_{x}(\gamma(a))) \leq C\int_{0}^{a} \mathcal{J}(\gamma(s))^{\frac{1}{N}}\, ds
\end{equation}
for some constant $C \geq 1$ depending only on the quasisymmetry control function $\eta$ of Proposition \ref{quasisymmetry orbit} and the Ahlfors regularity constants of the metric measure spaces $(G,\rho_{G},m_{G})$ and $(\W^{u,h}(\varphi(x)),\rho_{\varphi(x),h},m_{\varphi(x),h})$, which are independent of $x$ by Proposition \ref{Ahlfors regular}. 

We have $d(\varphi_{x}^{-1})_{*}m_{\varphi(x),h} = J_{x}dm_{x,g} + d\kappa_{x}$, where $\kappa_{x}$ is mutually singular with respect to $m_{x,g}$. Hence we can rewrite the limit defining $\mathcal{J}$ as  
\[
\mathcal{J}(y) = \lim_{r \rightarrow 0} \frac{1}{m_{G}(B_{\rho_{G}}(y,r))}\int_{B_{\rho_{G}}(y,r)} J_{x} \circ T_{x}^{-1}\,dm_{G} + \lim_{r \rightarrow 0} \frac{(T_{x}^{-1})_{*}\kappa_{x}(B_{\rho_{G}}(y,r))}{m_{G}(B_{\rho_{G}}(y,r))}
\]
Since $G$ is Ahlfors $N$-regular and $m_{G}$ is a $\sigma$-finite Borel measure, the Lebesgue differentiation theorem holds on $G$ \cite{Hein01}. Hence, for $m_{G}$-a.e.~ $y \in G$, the first limit is $J_{x}(T_{x}^{-1}(y))$ and the second limit is 0. We conclude that $\mathcal{J}(y) = J_{x}(T_{x}^{-1}(y))$ for $m_{G}$-a.e.~ $y \in G$. 

Suppose that the hypothesis of Proposition \ref{conjugacy criterion} does not hold, i.e., suppose that
\[
m_{g}\left(\left\{x \in T^{1}X: J_{x}(y) = 0 \; \text{for $m_{x,g}$-a.e.~ $y \in \W^{u,g}(x)$}\right\}\right) = 0.
\]
Let $x \in T^{1}X$ be a point such that $J_{x}(y) = 0$ for $m_{x,g}$-a.e.~ $y \in \W^{u,g}(x)$. Then $\mathcal{J}(y) = J_{x}(T_{x}^{-1}(y)) = 0$ for $m_{G}$-a.e.~ $y \in G$. 

Write $L \subset TG$ for the left-invariant contact distribution on $G$ which is tangent to $\mathfrak{l}$ at the identity in $G$. Let $y,z \in G$ be given. We join $y$ to $z$ by an embedded smooth curve $\gamma: [0,a] \rightarrow G$ tangent to $L$ that is parametrized by arclength and satisfies $a = \l(\gamma) = \rho_{G}(y,z)$. Let $B$ denote the closed ball of radius $r$ in $\R^{k-1}$; for a sufficiently small $r > 0$ we may extend $\gamma$ to a smooth embedding $\eta: B \times [0,a] \rightarrow G$ such that, for each $p \in B$, $\gamma_{p} = \eta|_{\{p\} \times [0,a]}$ is a smooth embedded curve that is tangent to $L$ and parametrized by arc length, and we have $\gamma_{0} = \gamma$. This can be done, for instance, by considering the tangent vector field to $\gamma$, taking a smooth extension of this vector field in a small neighborhood of $\gamma$ that is tangent to $L$, and then taking the integral curves of this vector field. We write $\Gamma = \{\gamma_{p}:p \in B\}$ for the resulting curve family in $G$. 

We next show that $\mathrm{mod}_{N}(\Gamma) > 0$ (where $N = N_{G}$) by a calculation similar to the one done in the proof of Lemma \ref{positive modulus}. Let $\beta:G \rightarrow [0,\infty]$ be an admissible function for $\Gamma$. Since we are bounding the modulus from below in \eqref{mod defn}, there is no harm in assuming that $\beta \equiv 0$ outside $P:=\eta(B \times [0,a])$. By the admissibility of $\beta$ and the fact that the curves $\gamma_{p}$ are parametrized by arclength we have by applying Fubini's theorem and then H\"older's inequality with conjugate exponent $p$ to $N$,
\begin{align*}
\int_{P} \beta^{N}\, d\mu_{G} &\geq \mu_{G}(P)^{-Np^{-1}}\left(\int_{P} \beta\, d\mu_{G} \right)^{N} \\
&= \mu_{G}(B)^{-Np^{-1}}\int_{B}\int_{\gamma_{p}}\beta\, ds \,d\nu \\
&\geq \mu_{G}(B)^{-Np^{-1}} \nu(B),
\end{align*}
where now $\nu$ is a smooth nonvanishing density on the ball $B$ with respect to the volume $m_{k-1}$ on $\R^{k-1}$, so that $\nu(B) > 0$. We conclude that $\mathrm{mod}_{N}(\Gamma) > 0$. 

Since $\psi_{x}$ is absolutely continuous on $\mathrm{mod}_{N}$-a.e.~ rectifiable curve in $G$, we conclude that $\psi_{x}$ is absolutely continuous on $\gamma_{p}$ for $\mathrm{mod}_{N}$-a.e. curve $\gamma_{p} \in \Gamma$ or, equivalently, for $m_{k-1}$-a.e.~ $p \in B$. We conclude by inequality \eqref{abs cont inequality} that for $m_{k-1}$-a.e.~ $p \in B$, 
\[
\rho_{\varphi(x),f}(\psi_{x}(\gamma_{p}(0)),\psi_{x}(\gamma_{p}(a))) \leq C \int_{0}^{a} \mathcal{J}(\gamma_{p}(s))^{\frac{1}{N}}\, ds. 
\]

By Fubini's theorem, for $m_{k-1}$-a.e.~ $p \in B$ we have  $\mathcal{J}(\gamma_{p}(s)) = 0$ for a.e.~ $s \in [0,a]$. Hence for $m_{k-1}$-a.e.~ $p \in B$,
\[
\rho_{\varphi(x),h}(\psi_{x}(\gamma_{p}(0)),\psi_{x}(\gamma_{p}(a))) = 0.
\]
Since full $m_{k-1}$-measure subsets of $B$ are dense, we conclude by the continuity of $\psi_{x}$ that this also holds for $p = y$. Hence
\begin{equation}\label{path inequality}
\rho_{\varphi(x),h}(\psi_{x}(y),\psi_{x}(z)) = 0.
\end{equation}
It follows that $\psi_{x}(y) = \psi_{x}(z)$ for any $z \in G$ and therefore the image of $\psi_{x}$ is a single point in $\W^{u,h}(\varphi(x))$. This contradicts the fact that $\psi_{x}$ is a homeomorphism. Thus the hypothesis of Proposition \ref{conjugacy criterion} holds and we obtain a conjugacy $\varphi_{u}$ from $g^{t}$ to $h^{t}$ that is flow related to $\varphi$. 
\end{proof}

By analogous arguments replacing $f^{t}$ with $f^{-t}$, Propositions \ref{lem:bound} and \ref{prop:orbit to conj} apply to $h^{t}_{s}$ as well. We conclude that $Q_{f^{-1}} = Q_{X}$ and that the horizontal measure for $h^{t}_{s}$ is the SRB measure $m_{h_{s}}$. Proposition \ref{prop:orbit to conj} then gives a conjugacy $\varphi_{s}$ from $g^{-t}$ to $h_{s}^{t}$.

\section{Horizontal differentiability}\label{sec:pansuderiv}

We continue to write $h^{t}:=h^{t}_{u}$ in this section and write $\varphi:=\varphi_{u}$ for the conjugacy obtained in Proposition \ref{prop:orbit to conj}. Since we will never again refer to the original conjugacy given in the hypotheses of Theorem \ref{core theorem}, this will not cause any confusion.  In this section we will construct a horizontal derivative map $\Phi: L^{u,g} \rightarrow L^{u,h}$ for $\varphi$. 

We let $d_{x}$ be the Riemannian metric on $\W^{u,h}(x)$ coming from the Riemannian structure we have fixed on $E^{u,h}$, and $\bar{d}_{x}$ the flat Riemannian metric on $\mathcal{Q}^{u,h}(x)$ given by its isometric identification with $Q^{u,h}_{x}$ via the charts $e_{x}: Q^{u,h}_{x} \rightarrow \mathcal{Q}^{u,h}(x)$ of Proposition \ref{quotient charts}.  Of course in our notation $d_{x} = d_{y}$ and $\bar{d}_{x} = \bar{d}_{y}$ for $y \in \W^{u,h}(x)$. Then $\pi: \W^{u,h}(x) \rightarrow \mathcal{Q}^{u,h}(x)$ is a Riemannian submersion; in particular we have for $y,z \in \W^{u,h}(x)$, 
\[
\bar{d}_{x}(\pi(y),\pi(z)) \leq d_{x}(y,z),
\]
and 
\[
\bar{d}_{x}(\pi(y),\pi(z)) \geq \inf_{\substack{\pi(p) = \pi(y) \\ \pi(q) = \pi(z)}} d_{x}(p,q).
\]
Since $Dh^{t}_{x}$ is conformal on $Q^{u,h}_{x}$ with expansion factor $e^{t}$ and the charts $e_{x}$ isometrically conjugate this derivative action to the quotient map $\bar{h}^{t}: \mathcal{Q}^{u,h}(x) \rightarrow \mathcal{Q}^{u,h}(h^{t}x)$, we have for each $y ,  z \in \mathcal{Q}^{u,h}(x)$,
\[
\bar{d}_{\bar{h}^{t}x}(\bar{h}^{t}(y),\bar{h}^{t}(z)) = e^{t} \bar{d}_{x} (y,z). 
\]

We show below that the projection $\pi$ is actually globally Lipschitz for the Hamenst\"adt metrics $\rho_{x,h}$ as well. 

\begin{lem}\label{lip proj}
There is a constant $C \geq 1$ independent of $x \in M$ such that, for each $y,z \in \W^{u,h}(x)$, we have
\[
\bar{d}_{x}(\pi(y),\pi(z)) \leq C\rho_{x,h}(y,z).
\]
\end{lem}

\begin{proof}
Let $x \in M$ be given. By the uniform comparability of the Hamenst\"adt metric $\rho_{x,h}$ to the Riemannian metric $d_{x}$ \eqref{Hamenstadt inclusion 1}, \eqref{Hamenstadt inclusion 2}, together with the fact that $\pi$ is a Riemannian submersion, we conclude that there is a constant $c > 0$ independent of $x$ such that for any $y,z \in \W^{u,h}(x)$,
\[
\bar{d}_{x}(\pi(y),\pi(z)) = 1 \Rightarrow \rho_{x,h}(y,z) \geq c.
\]
Now let $y, z \in \W^{u,h}(x)$ be arbitrary and set $r = \bar{d}_{x}(\pi(y),\pi(z))$, $s = -\log r$. Then 
\[
\bar{d}_{h^{t}x}(\pi(h^{s}y),\pi(h^{s}z)) = e^{s}\bar{d}_{x}(\pi(y),\pi(z)) = 1.
\]
Thus we have 
\[
\rho_{h^{s}x,h}(h^{s}y,h^{s}z) \geq c,
\]
and so
\begin{align*}
\rho_{x,h}(y,z) &= e^{-s}\rho_{h^{s}x,h}(h^{s}y,h^{s}z) \\
&\geq ce^{-s} \\
&= c \bar{d}_{x}(\pi(y),\pi(z)).
\end{align*}
\end{proof}

We next give a lemma that characterizes rectifiable curves for the Hamenst\"adt metric on an unstable leaf in the presence of a $u$-splitting; our formulation generalizes previous versions of this lemma used in \cite{Ham91}, \cite{Con}. Below we let $I$ denote a compact subinterval of $\R$.

\begin{lem}\label{horizontal rectifiable}
Let $f^{t}$ be a $u$-smooth Anosov flow on a closed smooth Riemannian manifold $M$ with a $u$-splitting $E^{u} = L^{u} \oplus V^{u}$. Suppose we have a Riemannian structure on $E^{u}$ with associated norm $\|\cdot\|$ such that $\|Df^{t}(v)\| = e^{t}\|v\|$ for all $v \in L^{u}$ and $t \in \R$ and $\mathfrak{m}(Df^{t}_{x}|_{E^{u}_{x}}) \geq e^{t}$ for $x \in M$, $t \geq 0$. Then a continuous curve $\gamma: I \rightarrow \W^{u}(x)$  is rectifiable with respect to the Hamenst\"adt metric $\rho_{x,f}$ if and only if $\gamma$ is Lipschitz with respect to the Riemannian metric $d_{x}$ on $I$ and for a.e.~ $s \in I$ we have $\gamma'(s) \in L^{u}_{\gamma(s)}$. 
\end{lem}

\begin{proof}
We first suppose that $\gamma:I \rightarrow \W^{u}(x)$ is Lipschitz with respect to the metric $d_{x}$ on $\W^{u}(x)$ and that for a.e.~ $s \in I$ we have $\gamma'(s) \in L^{u}_{\gamma(s)}$. By reparametrizing $\gamma$ with respect to $d_{x}$-arc length (which does not affect $\rho_{x,f}$-rectifiability), we may assume that $\|\gamma'(s)\| = 1$ for a.e.~ $s \in I$. Since it suffices to prove local rectifiability we can assume that $\gamma:I \rightarrow \W^{u}(x)$ is injective on $I$ by restricting to smaller closed subintervals. 

Let $a,b \in I$ with $a < b$. For those $s$ for which $\gamma'(s)$ exists with $\gamma'(s) \in L^{u}_{\gamma(s)}$ we have $\|Df^{t}(\gamma'(s))\| = e^{t}\|\gamma'(s)\|$ for all $t \in \R$. This implies that, for $t \geq 0$, we have 
\begin{align*}
d_{f^{t}_{x}}(f^{t}(\gamma(b)),f^{t}(\gamma(a))) &\leq \int_{a}^{b} \|Df^{t}(\gamma'(s))\| \, ds\\
&= e^{t}\int_{a}^{b} \|\gamma'(s)\| \, ds \\
&= e^{t}(b-a),
\end{align*}
since $\gamma$ is parametrized by arc length. Thus by fixing a $T$ such that
\[
d_{f^{T}x}(f^{T}(\gamma(b)),f^{T}(\gamma(a))) = 1,
\]
we must have $e^{T}(b-a) \geq 1$, which implies that $T \geq -\log (b-a)$. This implies that 
\[
\beta(\gamma(a),\gamma(b)) \geq -\log (b-a),
\]
where $\beta$ is defined as in \eqref{beta equation}. This implies that 
\[
\rho(\gamma(a),\gamma(b)) = e^{-\beta(\gamma(a),\gamma(b))} \leq  b-a.
\]
This implies that the restriction of $\rho_{x,f}$ to $\gamma$ in the parametrization of $\gamma$ with respect to arc length is 1-Lipschitz with respect to the Riemannian metric $d_{x}$. It follows immediately that $\gamma$ is $\rho_{x,f}$-rectifiable. 

Now assume that $\gamma$ is $\rho_{x,f}$-rectifiable. By uniform comparability of the Hamenst\"adt metrics and Riemannian metrics at unit scale (\eqref{Hamenstadt inclusion 1}  and \eqref{Hamenstadt inclusion 2} with $r = 1$), there is a constant $c > 0$ such that, for all $x \in M$ and any $y,z \in W^{u,f}(x)$, 
\begin{equation}\label{comparison equation}
d_{x}(y,z) = 1 \Rightarrow \rho_{x,f}(y,z) \geq c.
\end{equation}
By reversing time in the hypothesis $\mathfrak{m}(Df^{t}_{x}|_{E^{u}_{x}}) \geq e^{t}$ for $t \geq 0$, we obtain the estimate $\|Df^{-t}|_{E^{u}}\| \leq e^{-t}$ for $t \geq 0$. By applying $f^{-t}$ to each side of \eqref{comparison equation} (still assuming $d_{x}(y,z)=1$) we conclude that
\begin{align*}
d_{f^{-t}x}(f^{-t}y,f^{-t}z) &\leq e^{-t}d_{x}(y,z) \\
&\leq c^{-1}e^{-t}\rho_{x,f}(y,z) \\
&= c^{-1}\rho_{f^{-t}x,f}(f^{-t}y,f^{-t}z).
\end{align*}
From this we see that $d_{x}(y,z) \leq c^{-1}\rho_{x,f}(y,z)$ for any $x \in M$ whenever $y,z \in \W^{u,f}(x)$ satisfy $d_{x}(y,z) \leq 1$. Thus rectifiability of $\gamma$ with respect to $\rho_{x,f}$ implies that $\gamma$ is also rectifiable with respect to $d_{x}$.  Hence $\gamma$ is Lipschitz with respect to $d_{x}$, and so $\gamma$ is differentiable at a.e.~ $s \in I$. By reparametrizing, we may assume that $\gamma$ is parametrized with respect to arc length for the metric $d_{x}$. Let $\l_{d}$ denote lengths of curves measured in the Riemannian metric $d_{x}$, and let $\l_{\rho}$ denote lengths of curves measured in the Hamenst\"adt metric $\rho_{x,f}$. Let $\kappa$ denote the Lipschitz constant of the curve $\gamma$ in the metric $\rho_{x,f}$.

Assume to a contradiction that there is a positive measure subset $K \subseteq I$ on which $\gamma'(s) \notin L^{u}$. By using Lusin's theorem we can pass to a positive measure compact subset of $K$ (which we will still denote by $K$) on which we may assume that there is a $\theta > 0$ such that $\measuredangle(\gamma'(s), L^{u}_{\gamma(s)}) \geq \theta$ for all $s \in K$. By our domination estimates \eqref{splitting inequality} in the norm $\| \cdot \|$, there is a  constant $\chi > 1$ and a constant $\delta > 0$ such that for every $s \in K$ and $t \geq 0$,
\[
\|Df^{t}(\gamma'(s))\| \geq \delta e^{\chi t}\|\gamma'(s)\|.
\]

Let $s_{0} \in K$ be a Lebesgue density point in $I$ for the set $K$. Let $B_{r}$ denote the interval of radius $r$ centered at $s_{0}$. For $r$ small enough we have 
\[
| B_{r} \cap  K | \geq \frac{|B_{r}|}{2} = r.
\]
Let $\gamma_{r}: B_{r} \rightarrow \W^{u,f}(x)$ denote the restriction of $\gamma$ to $B_{r}$. There is a constant $c > 0$ such that, for those $t \geq 0$ for which $\l_{\rho}(f^{t} \circ \gamma_{r}) \leq 1$, we have using the consequences of \eqref{comparison equation},
\begin{align}\label{first lipschitz}
\l_{d}(f^{t} \circ \gamma_{r}) &\leq c^{-1}\l_{\rho}(f^{t} \circ \gamma_{r}) \\
&= c^{-1} e^{t} \l_{\rho}(\gamma_{r}) \\
&\leq \kappa c^{-1}e^{t}r
\end{align}

On the other hand we have, recalling that $\gamma$ is parametrized with respect to $d_{x}$-arc length,
\begin{align}\label{second lipschitz}
\l_{d}(f^{t} \circ \gamma_{r}) &\geq e^{t}\l_{d}(\gamma_{r}|_{B_{r} \backslash K}) + \delta e^{\chi t} \l_{d}(\gamma_{r}|_{B_{r} \cap K}) \\
&= ce^{t}|B_{r} \backslash K| + \delta e^{\chi t} |B_{r} \cap K| \\
&\geq r\delta e^{\chi t }.
\end{align}
Choose $T$ large enough that $\delta e^{\chi T} > \kappa c^{-1}  e^{T}$. Then choose $r = r(T) > 0$ small enough that $\l_{\rho}(f^{T} \circ \gamma_{r}) \leq 1$. We conclude by combining \eqref{first lipschitz} and \eqref{second lipschitz} that we have
\[
\kappa c^{-1} r e^{T} \geq \l_{d}(f^{t} \circ \gamma_{r}) \geq  r\delta e^{\chi T},
\]
which, after canceling $r$ from each side, gives the inequality $\kappa c^{-1} e^{T} \geq \delta e^{\chi T}$, which contradicts our choice of $T$. We conclude that $\gamma'(s) \in L^{u}_{\gamma(s)}$ for a.e.~ $s \in I$. 
\end{proof}

We will also require a preliminary proposition that is a straightforward generalization of a theorem of Sadovskaya \cite{S15}. Below we let $M$ and $N$ be closed $C^2$ Riemannian manifolds and let $0 < \alpha \leq 1^{-}$. 

\begin{prop}\label{quasi cohomologous}
Let $f^{t}: M \rightarrow M$, $g^{t}: N \rightarrow N$ be transitive $C^2$ Anosov flows orbit equivalent by a homeomorphism $\varphi: M \rightarrow N$. Let $E$, $F$ be $C^{\alpha}$ vector bundles over $M$ and $N$ respectively and let $A^{t}: E \rightarrow E$, $B^{t}: F \rightarrow F$ be uniformly quasiconformal $C^{\alpha}$ linear cocycles over $f^{t}$ and $g^{t}$ respectively. Let $\mu$ be an $f^{t}$-invariant fully supported ergodic probability measure with local product structure. Let $\alpha$ be the additive cocycle over $f^{t}$ such that for all $x \in M$ and $t \in \R$,
\[
\varphi(f^{t}x) = g^{\alpha(t,x)}(\varphi(x)).
\]
Suppose that there is a measurable bundle map $\Phi:E \rightarrow F$ covering $\varphi$ such that for $\mu$-a.e. $x \in M$ we have for every $t \in \R$, 
\[
\Phi_{f^{t}x}\circ A^{t}_{x} = B_{\varphi(x)}^{\alpha(t,x)} \circ \Phi_{x}.
\]
Then $\Phi$ coincides $\mu$-a.e.~ with an orbit equivalence $\bar{\Phi}$ from $A^{t}$ to $B^{t}$ 
\end{prop} 

\begin{proof}
We take a Markov partition for $f^{t}$ as in Section \ref{thermo section}. We thus have a subshift of finite type $\Sigma$ with shift map $\sigma: \Sigma \rightarrow \Sigma$ and a H\"older roof function $\psi: \Sigma \rightarrow (0,\infty)$ such that the suspension flow $h^{t}: \Sigma_{\psi} \rightarrow \Sigma_{\psi}$ is semiconjugate to $f^{t}$ by a finite-to-one H\"older surjection $\kappa: \Sigma_{\psi} \rightarrow M$. Furthermore there is a unique $h^{t}$-invariant ergodic probability measure $\nu$ with local product structure such that $\kappa_{*}\nu = \mu$.  We obtain a Markov partition for $g^{t}$ by pushing the Markov partition for $f^{t}$ forward by $\varphi$, giving that $h^{t}$ is semiconjugate to $g^{t}$ via $\varphi \circ \kappa$. 

We pull everything back to $\Sigma_{\psi}$ using $\kappa$. We let $\t{E} = \kappa^{*}E$ and $\t{F} = (\varphi \circ \kappa)^{*}F$ denote the pullback bundles over $\Sigma_{\psi}$, and define linear cocycles $\t{A}^{t}: \t{E} \rightarrow \t{E}$ and $\t{B}^{t}: \t{F} \rightarrow \t{F}$ over $h^{t}$ by setting $\t{A}^{t}_{x} = A^{t}_{\kappa(x)}$ for $x \in \Sigma_{\psi}$ and similarly setting  $\t{B}^{t}_{x} = B^{t}_{\varphi(\kappa(x))}$. For $\nu$-a.e.~ $x \in \Sigma_{\psi}$ we set $\t{\Phi}_{x} = \Phi_{\kappa(x)}$. Then for $\nu$-a.e.~ $x \in \Sigma_{\psi}$, the equality 
\begin{equation}\label{measurable conjugacy}
\t{\Phi}_{f^{t}x}\circ \t{A}^{t}_{x} = \t{B}_{\varphi(x)}^{\alpha(t,x)} \circ \t{\Phi}_{x},
\end{equation}
holds for all $t \in \R$, by construction. 

Let $\t{E}^{0}$, $\t{F}^{0}$ denote the restrictions of each of these bundles to $\Sigma \times \{0\} \subseteq \Sigma_{\psi}$, which we will identify with $\Sigma$. Set $k = \dim E = \dim F$. Being H\"older vector bundles over a Cantor set $\Sigma$, we have H\"older trivializations $\t{E}^{0} \cong \R^{k}$, $\t{F}^{0} \cong \R^{k}$, which may be obtained for instance by partitioning $\Sigma$ into cylinders on which the bundle is trivial and noting that these cylinders are both open and closed in $\Sigma$. For each $x \in \Sigma$, identifying $x$ with $(x,0) \in \Sigma_{\psi}$, the map
\[
\t{A}^{\psi(x)}_{(x,0)}: \t{E}_{x}^{0} \rightarrow \t{E}_{\sigma(x)}^{0},
\]
gives rise to a linear map $S_{x}: \R^{k} \rightarrow \R^{k}$ in this trivialization, which induces a H\"older linear cocycle $S: \Sigma \times \R^{k} \rightarrow \Sigma \times \R^{k}$ over $\sigma$. Likewise $\t{B}^{\psi(x)}_{(x,0)}$ induces a H\"older linear cocycle $T: \Sigma \times \R^{k} \rightarrow \Sigma \times \R^{k}$ over $\sigma$. We let $\nu^{0}$ denote the $\sigma$-invariant probability measure on $\Sigma$ corresponding to the $h^{t}$-invariant measure $\nu$ on $\Sigma_{\psi}$ \cite{BR75} which also has local product structure. For $\nu^{0}$-a.e.~ $x \in \Sigma$ we then set $\Psi_{x}: \R^{k} \rightarrow \R^{k}$ to be the linear map corresponding to $\Phi_{(x,0)}:\t{E}_{x}^{0} \rightarrow \t{F}_{\varphi(x)}^{0}$ in this trivialization. The measurable orbit equivalence equation \eqref{measurable conjugacy} then gives us for $\nu^{0}$-a.e.~ $x \in \Sigma$, 
\begin{equation}\label{2nd measurable conjugacy}
\Psi_{\sigma(x)}\circ S_{x} = T_{x} \circ \Psi_{x}.
\end{equation}

Both of the linear cocycles $S$ and $T$ over $\sigma$ are H\"older continuous and uniformly quasiconformal, since they were constructed from the H\"older continuous uniformly quasiconformal linear cocycles $\t{A}$ and $\t{B}$ by taking a cross section. Hence by \cite[Theorem 2.7]{S15} the measurable conjugacy $\Psi$ coincides $\nu^{0}$-a.e.~ with a conjugacy $\bar{\Psi}$ between $S$ and $T$. Returning to the bundles $\t{E}^{0}$ and $\t{F}^{0}$ and passing back up to the suspension flow, we obtain a conjugacy (which we will still denote by $\bar{\Psi}:\t{E} \rightarrow \t{F}$) from $\t{A}^{t}$ to $\t{B}^{t}$ which agrees $\nu$-a.e.~ with $\t{\Phi}$. 

The semiconjugacy $\kappa: \Sigma_{\psi} \rightarrow M$ is a continuous bijection off of an $h^{t}$-invariant $\nu$-null set in $\Sigma_{\psi}$. From this we obtain on an $f^{t}$-invariant full $\mu$-measure subset $\bar{M} \subset M$ on which there is an orbit equivalence $\bar{\Phi}:E|_{\bar{M}} \rightarrow F|_{\varphi(\bar{M})}$ from $A^{t}$ to $B^{t}$ which agrees $\mu$-a.e.~ with $\Phi$. Since $\bar{M}$ is dense in $M$, it follows that we may extend $\bar{\Phi}$ continuously to a map $\bar{\Phi}: E \rightarrow F$ giving an orbit equivalence between $A^{t}$ and $B^{t}$. 
\end{proof}

As discussed in Section \ref{subsec:neg curved}, we have a 1-parameter family of expanding automorphisms $A^{t}$ of $G$ such that, writing $\mathfrak{g} = \mathfrak{l} \oplus \mathfrak{v}$ for the Lie algebra of $G$ split into horizontal and vertical subspaces, $DA^{t}$ expands $\mathfrak{l}$ by $e^{t}$ and $\mathfrak{v}$ by $e^{2t}$. Likewise, for an $l$-dimensional vector space $Q$ we have the standard 1-parameter family of expanding linear maps $B^{t}$ on $Q$ given by $B^{t}x = e^{t}x$. We say that a homomorphism $\psi: G \rightarrow Q$ is \emph{homogeneous} if $\psi \circ A^{t} = B^{t} \circ \psi$, i.e., $\psi$ commutes with these dilations. 

\begin{defn}[Pansu differentiable]\label{Pansu diff}
We say that a map $F:G \rightarrow Q$ is \emph{Pansu differentiable} at $x \in G$ if there is a homogeneous homomorphism $DF_{x}: G \rightarrow Q$ such that for all $y \in G$ we have
\[
DF_{x}(y) = \lim_{t \rightarrow 0} B^{1/t}(F(x \cdot A^{t}y) - F(x)).
\]
We then say that $DF_{x}$ is the \emph{Pansu derivative} of $F$ at $x$. 
\end{defn}
See \cite{P89b} for basic properties of the Pansu derivative and its uses. The Pansu derivative may be defined more generally for maps between any pair of Carnot groups, but we will only need the formulation for $G$ and $V$ described above.

Our next proposition gives us a horizontal derivative of the map $\varphi: \W^{u,g}(x) \rightarrow \W^{u,h}(\varphi(x))$  for each $x \in M$.
 
\begin{prop}\label{horizontal diff}
There is a continuous bundle map $\Phi: L^{u,g} \rightarrow L^{u,h}$ covering $\varphi$ such that $\Phi \circ Dg^{t} = Dh^{t} \circ \Phi$ for all $t \in \R$. Furthermore, for each $C^1$ curve $\gamma$ tangent to $L^{u,g}$ we have that $\varphi \circ \gamma$ is a $C^1$ curve tangent to $L^{u,h}$ satisfying $(\varphi \circ \gamma)' = \Phi \circ \gamma'$. 
\end{prop}

\begin{proof}
Fix $x \in M$. We denote by $T_{x}: G \rightarrow \W^{u,g}(x) $ our standard chart for $\W^{u,g}(x)$. Consider the composition of maps 
\[
\Psi: G \xrightarrow{T_{x}} \W^{u,g}(x) \xrightarrow{\varphi} \W^{u,h}(\varphi(x)) \xrightarrow{\pi_{\varphi(x)}} \mathcal{Q}^{u,h}(\varphi(x)) \xrightarrow{e_{\pi(\varphi(x))}^{-1}} Q^{u,h}_{\varphi(x)}.
\]
We consider $\W^{u,g}(x)$ and $\W^{u,h}(\varphi(x))$ with their Hamenst\"adt metrics and $\mathcal{Q}^{u,h}(\varphi(x))$ with its flat metric $\bar{d}_{\varphi(x)}$ induced from the chart $e_{\varphi(x)}$. By Proposition \ref{conjugacy invariant Hamenstadt}, we have that $\varphi$ is Lipschitz. By Lemma \ref{lip proj}, $\pi_{\varphi(x)}$ is also Lipschitz, and by Proposition \ref{quotient charts} $e_{\pi(\varphi(x))}^{-1}$ is an isometry.  We conclude that  $\Psi$ is Lipschitz with respect to the Carnot-Caratheodory metric $\rho_{G}$ on $G$ and the Euclidean metric on $Q^{u,h}_{\varphi(x)}$ induced by the Riemannian structure on $Q^{u,h}$.

By the Pansu-Rademacher-Stepanov theorem for Lipschitz maps between Carnot groups \cite{P89b}, the map $\Psi$ is Pansu differentiable $m_{G}$-a.e.~ on $G$. The Pansu derivative $D\Psi_{y}: G \rightarrow Q^{u,h}_{\varphi(x)}$ at each $y \in G$ is a homogeneous homomorphism. Since $[\mathfrak{l},\mathfrak{l}] = \mathfrak{v}$ in the Lie algebra $\mathfrak{g}$ of $G$, and since $Q^{u,h}_{\varphi(x)}$ is abelian, the vertical subgroup $V$ of $G$ satisfies $V \subseteq \ker D\Psi_{y}$. Hence $D\Psi_{y}$ descends to a homogeneous homomorphism $D\Psi_{y}: G/V \rightarrow Q^{u,h}_{\varphi(x)}$; since both $G/V$ and $Q^{u,h}_{\varphi(x)}$ are abelian, $D\Psi_{y}$ actually gives  a linear map between these two groups.

Since we showed in the proof of Proposition \ref{prop:orbit to conj} that $\varphi \circ T_{x}$ is absolutely continuous along $\mathrm{mod}_{N}$-a.e.~ curve in $G$, we conclude that the map $\Psi$ is also absolutely continuous along $\mathrm{mod}_{N}$-a.e.~ curve in $G$. It follows directly from this that for any left-invariant vector field $Z$ on $G$ that is tangent to $L$, $\Psi$ is absolutely continuous along $\mathrm{mod}_{N}$-a.e. curve in the foliation $\mathcal{Z}$ of $G$ by the integral curves of $Z$ and therefore $\Psi$ is differentiable in the $Z$-direction at $m_{G}$-a.e.~ point in $G$. Letting $Z_{1},\dots,Z_{l}$ be an orthonormal frame of left-invariant vector fields spanning $L$ and combining this with the $m_{G}$-a.e.~ differentiability of $\Psi$ obtained in the previous paragraph, we conclude that at $m_{G}$-a.e.~ $y \in G$ it is the case that $\Psi$ is differentiable at $y$ and that $D\Psi_{y}: L_{y} \rightarrow T_{\Psi(y)}Q_{\varphi(x)}^{u,h}$ is the unique linear transformation with $D\Psi_{y}(Z_{i}(y)) = \mathcal{Z}_{i,y}'(0)$ for $1 \leq i \leq l$, where $\mathcal{Z}_{i,y}$ is the curve $\mathcal{Z}_{i}(y)$ of the foliation $\mathcal{Z}_{i}$ through $y$, parametrized by $\rho_{G}$-arclength, such that $\mathcal{Z}_{i,y}(0) = y$.

It follows that $\pi \circ \varphi: \W^{u,g}(x) \rightarrow \mathcal{Q}^{u,h}(\varphi(x))$ has the same properties: let $Y_{i} = DT_{x}(Z_{i})$, $\mathcal{Y}_{i} = T_{x}(\mathcal{Z}_{i})$, $1 \leq i \leq l$, then at $m_{x,g}$-a.e.~ $y \in \W^{u,g}(x)$ we have that $\pi_{\varphi(x)} \circ \varphi$ is differentiable in each of the directions $Y_{1},\dots,Y_{l}$, that the total horizontal derivative $D_{L}(\pi_{\varphi(x)} \circ \varphi): L^{u,g} \rightarrow T\mathcal{Q}^{u,h}(\varphi(x))$ exists $m_{x,g}$-a.e., and finally that $D_{L}(\pi_{\varphi(x)} \circ \varphi)$ evaluates to the derivative along $Y_{i}$ for $1 \leq i\leq l$ $m_{x,g}$-a.e. We then define for $y \in \W^{u,g}(x)$,
\[
\Phi_{y} = (D\pi_{\varphi(x)})^{-1} \circ D_{L}(\pi_{\varphi(x)} \circ \varphi)_{y}: L^{u,g}_{y} \rightarrow L^{u,h}_{\varphi(y)}. 
\]
Performing this analysis on each leaf $\W^{u,g}(x)$, we obtain for $m_{g}$-a.e. $y \in T^{1}X$ a linear map $\Phi_{y}: L^{u,g}_{y} \rightarrow L^{u,h}_{\varphi(y)}$. From the fact that $\varphi$ is a conjugacy, we deduce that for $m_{g}$-a.e.~ $y \in T^{1}X$,
\begin{equation}\label{invertible semiconjugacy}
\Phi_{g^{t}y} \circ Dg^{t}_{y} = Dh^{t}_{\varphi(y)} \circ \Phi_{y},
\end{equation}
for all $t \in \R$.

We next show that $\Phi_{y}$ is invertible for $m_{g}$-a.e. $y \in T^{1}X$. By the measurable semiconjugacy equation \eqref{invertible semiconjugacy}, the set 
\[
K=\{y \in T^{1}X: \Phi_{y} \; \text{is invertible}\},
 \]
is $g^{t}$-invariant and thus either $m_{g}(K) = 0$ or $m_{g}(K) = 1$. Hence it suffices to show that we cannot have $m_{g}(K) = 0$.

Suppose otherwise, so that for $m_{g}$-a.e. $y \in T^{1}X$ we have that $\Phi_{y}$ is not invertible at $y$. Then $\ker \Phi \subseteq L^{u,g}$ is a measurable $Dg^{t}$-invariant subbundle defined $m_{g}$-a.e. Since $Dg^{t}|_{L^{u,g}}$ is conformal, by \cite[Lemma 2.5]{Bu1} $\ker \Phi$ coincides $m_{g}$-a.e.~ with a $Dg^{t}$-invariant continuous subbundle of $L^{u,g}$ that is invariant under both the stable and unstable holonomies $H^{s,g}$ and $H^{u,g}$ of  $Dg^{t}|_{L^{u,g}}$. The proof of Proposition \ref{exp to conf} shows that such a subbundle of $L^{u,g}$ must be trivial, i.e., we must either have $\ker \Phi_{y} = L^{u,g}_{y}$ or $\ker \Phi_{y} = \{0\}$ for $m_{g}$-a.e.~ $y \in T^{1}X$. If $\ker \Phi_{y} = \{0\}$ for $m_{g}$-a.e.~ $y \in M$ then $\Phi_{y}$ is invertible for $m_{g}$-a.e~ $y \in M$, contradicting that $m_{g}(K) = 0$.  Otherwise we have $\ker \Phi_{y} = L^{u,g}_{y}$ for $m_{g}$-a.e.~ $y \in T^{1}X$.

Let $x \in T^{1}X$ be such that $ \ker \Phi_{y} = L^{u,g}_{y}$ for $m_{g}$-a.e.~ $y \in \W^{u,g}(x)$. It follows that for $\mathrm{mod}_{N}$-a.e. curve in the foliation $\mathcal{Y}_{i}$, $1 \leq i \leq l$, the derivative of $\pi_{\varphi(x)} \circ \varphi$ is $0$ a.e. Since $\pi_{\varphi(x)} \circ \varphi$ is absolutely continuous on these curves, it follows that $\pi \circ \varphi$ is constant on $\mathrm{mod}_{N}$-a.e. curve $\mathcal{Y}_{i}$. By continuity we conclude that $\pi_{\varphi(x)} \circ \varphi$ is constant on every curve in $\mathcal{Y}_{i}$. Since any two points in $\W^{u,g}(x)$ can be connected by a sequence of curves from these $l$ foliations, it follows that $\pi_{\varphi(x)} \circ \varphi$ is constant, which contradicts the fact that it is surjective.

Hence $\Phi_{y}$ is invertible for $m_{g}$-a.e.~ $y \in T^{1}X$. By applying Proposition \ref{quasi cohomologous} to the corresponding time-changed bundle map and horizontal derivative cocycle for the smooth flow $f^{t}$ we obtained $h^{t}$ as a time change of, we conclude that $\Phi$ coincides $m_{g}$-a.e.~ with a continuous bundle map from $Dg^{t}|_{L^{u,g}}$ to $Dh^{t}|_{L^{u,h}}$ satisfying the conjugacy equation \ref{invertible semiconjugacy} $m_{g}$-a.e., hence by continuity satisfying it everywhere. We will continue to denote this bundle map by $\Phi$.  As a consequence, since $\Phi_{y}$ is  continuous in $y$ we obtain that $D\Psi_{y}$ is  continuous in $y$ as well. We conclude that $\Psi: G \rightarrow Q^{u,h}_{\varphi(x)}$ is continuously Pansu differentiable with Pansu derivative $D\Psi$. This implies that for any $C^1$ curve $\gamma$ in $G$ that is tangent to the horizontal distribution $L \subseteq TG$, we have that $\Psi \circ \gamma$ is a $C^1$ curve in $Q^{u,h}_{\varphi(x)}$ with $(\Psi \circ \gamma)' = D\Psi \circ \gamma'$. 

Set $\bar{\Phi}_{x} = D\pi_{\varphi(x)} \circ \Phi_{x}$, $\bar{\varphi}(x) = \pi_{\varphi(x)}(\varphi(x))$ for each $x \in T^{1}X$. For each $x \in M$ the assertions of the previous paragraph show that if $\gamma: I \rightarrow \W^{u,g}(x)$ is a $C^1$ curve tangent to $L^{u,g}$, then $\bar{\varphi} \circ \gamma$ is a $C^1$ curve in $\mathcal{Q}^{u,h}(x)$ with $(\bar{\varphi} \circ \gamma)' = \bar{\Phi} \circ \gamma'$. Since $\varphi: (\W^{u,g}(x),\rho_{x,g}) \rightarrow (\W^{u,h}(\varphi(x)),\rho_{\varphi(x),h})$ is Lipschitz, the curve $\sigma:=\varphi \circ \gamma$ is rectifiable with respect to $\rho_{x,h}$. By Lemma \ref{horizontal rectifiable}, $\sigma$ is a Lipschitz curve in the Riemannian metric $d_{\varphi(x)}$ on $W^{u,h}(\varphi(x))$ with $\sigma'(s) \in L^{u,h}_{\sigma(s)}$ for a.e.~ $s$. Since $\pi_{\varphi(x)} \circ \sigma = \bar{\varphi} \circ \gamma$, we have
\[
D\pi_{\varphi(x)} \circ \sigma' = (\bar{\varphi} \circ \gamma)' = \bar{\Phi} \circ \gamma',
\]
with this equality holding a.e.~ on $I$. Since $\sigma' \in L^{u,h}$ a.e.~ on $I$, we can invert $D\pi_{\varphi(x)}$ in this equation to obtain 
\[
\sigma' = \Phi\circ \gamma',
\]
a.e.~ on $I$. This implies that $\sigma = \varphi \circ \gamma$ is a $C^1$ curve tangent to $L^{u,h}$ with $(\varphi \circ \gamma)' = \Phi \circ \gamma'$. 
\end{proof}

We observe that $\Phi$ intertwines the unstable holonomies on $L^{u,g}$ and $L^{u,h}$ respectively. We let $H^{u,g}$ and $H^{u,h}$ denote the families of unstable holonomies for $Dg^{t}|_{L^{u,g}}$ and $Dh^{t}_{L^{u,h}}$ given by Proposition \ref{existuhol}.

\begin{lem}\label{intertwine holonomies}
For all $x \in T^{1}X$ and $y \in \W^{u,g}(x)$,
\[
\Phi_{y} \circ H^{u,g}_{xy} = H^{u,h}_{\varphi(x)\varphi_{x}(y)} \circ \Phi_{x}.
\]
\end{lem}

\begin{proof}
Note that in the conclusion of \cite[Theorem 2.7]{S15} the resulting conjugacy that is obtained is H\"older continuous. Tracing this back through the proof of Proposition \ref{quasi cohomologous} shows that the orbit equivalence obtained there is H\"older. We conclude from the use of Proposition \ref{quasi cohomologous} in the previous proposition that $\Phi$ itself is H\"older. For $x \in T^{1}X$, $y \in \W^{u,g}(x)$ set 
\begin{equation}\label{intertwine}
H^{u}_{xy} = \Phi_{y}^{-1} \circ H^{u,h}_{\varphi(x)\varphi(y)} \circ \Phi_{x}
\end{equation}
The family of maps $H_{xy}^{u}:L^{u,g}_{x} \rightarrow L^{u,g}_{y}$ satisfies (1) and (2) of Proposition \ref{existuhol} by their construction and the conjugacy property of $\Phi$. The H\"older continuity of $\Phi$ implies that they will also satisfy (3) for an exponent $0 < \alpha \leq 1$, some constant $C$, and a small enough $r > 0$. Since $Dg^{t}|_{L^{u,g}}$ is a smooth linear cocycle on a smooth bundle $L^{u,g}$ and since the action of $Dg^{t}|_{L^{u,g}}$ on that bundle is $1$-fiber bunched by the conformality of $Dg^{t}|_{L^{u,g}}$, we conclude by the uniqueness statement of Proposition \ref{existuhol} that $H^{u}_{xy} = H^{u,g}_{xy}$ for all $x \in T^{1}X$ and $y \in \W^{u,g}(x)$, which implies the desired result by the definition \eqref{intertwine} of $H^{u}$.
\end{proof}

Our next proposition  shows that $\varphi$ maps the foliation $\V^{u,g}$ foliation $\V^{u,h}$. We need a preliminary lemma. Let $\hat{\Phi}: Q^{u,g} \rightarrow Q^{u,h}$ denote the map corresponding to $\Phi:L^{u,g} \rightarrow L^{u,h}$ via the projection isomorphisms $L^{u,*} \rightarrow Q^{u,*}$. Let $\kappa: M \rightarrow  \mathcal{R}Q^{u,h}$ be the Riemannian structure on $Q^{u,f}$ obtained by pushing forward the Riemannian structure $\{\langle \;, \; \rangle\}_{x \in T^{1}X}$ on $L^{u,g}$ coming from the left-invariant Riemannian structure on the horizontal distribution $L\subseteq TG$, i.e., for $v$, $w \in Q^{u,h}$, 
\[
\kappa(v,w) = \langle \hat{\Phi}^{-1}(v),\hat{\Phi}^{-1}(w)\rangle.
\]
Recall that $\nabla$ denotes the $Dh^{t}|_{Q^{u,h}}$-invariant connection on $Q^{u,h}$ constructed in Proposition \ref{holonomy connect}, which is flat and smooth along $\W^{u,h}$. 

\begin{lem}\label{orientation preserve}
The Riemannian structure $\kappa$ is $\nabla$-invariant along $\W^{u,h}$.  
\end{lem}

\begin{proof}
By Lemma \ref{intertwine holonomies} $\Phi$ is equivariant with respect to the unstable holonomies of $Dg^{t}|_{L^{u,g}}$ and $Dh^{t}|_{L^{u,h}}$, and an analogous equivariance property holds for $\hat{\Phi}$ with respect to the unstable holonomies of $Dg^{t}|_{Q^{u,g}}$ and $Dh^{t}|_{Q^{u,h}}$. By Proposition \ref{holonomy connect} the unstable holonomies for $Dh^{t}|_{Q^{u,h}}$ are given by parallel transport by $\nabla$. The lemma follows.
\end{proof}

\begin{prop}\label{vert preserve}
For each $x \in T^{1}X$ we have $\varphi(\V^{u,g}(x)) = \V^{u,h}(\varphi(x))$. 
\end{prop}

\begin{proof} 
We give $Q^{u,h}$ the Riemannian structure $\kappa$ of Lemma \ref{orientation preserve}, whose norm we denote by $|\cdot|$. By construction $\hat{\Phi}:Q^{u,g} \rightarrow Q^{u,h}$ restricts to an orientation-preserving isometry $\hat{\Phi}_{x}: Q^{u,g}_{x} \rightarrow Q^{u,h}_{\varphi(x)}$ on each fiber for this Riemannian structure on $Q^{u,h}$. Since $\hat{\Phi} \circ Dg^{t} = Dh^{t} \circ \hat{\Phi}$ on $Q^{u,g}$, from this isometry we conclude that $|Dh^{t}(v)| =e^{t}|v|$ for all $v \in L^{u,h}$. 

Let $x \in T^{1}X$ and $y \in \V^{u,g}(x)$ be given. Then we must show that $\varphi(y) \in \V^{u,h}(\varphi(x))$. Using the chart $T_{x}$, we view $\W^{u,g}(x)$ as a homogeneous copy of $G$ with $x$ at $0$. Since $ [\mathfrak{l},\mathfrak{l}] = \mathfrak{v}$, there are two orthogonal left-invariant vector fields $X_{1}$ and $X_{2}$ on $G$, which are tangent to $\mathfrak{l}$ at $0$, such that $Y:=[X_{1},X_{2}]$ is a left-invariant vector field on $G$ for which $\exp(Y(0)) = y$, i.e., $y$ is the image of $Y(0)$ by the exponential map $\exp: \mathfrak{g} \rightarrow G$. 

Let $P_{g}$ be the plane in $\R^{l}$ spanned by the projection of $X_{1}(0)$ and $X_{2}(0)$ to $\R^{l}$. Since $[X_{1},X_{2}] = Y$, there is a clockwise oriented rectangle in $P_{g}$ with sides $\bar{\gamma}_{i}$, $1 \leq i \leq 4$ -- each of which are tangent to either the projection of $X_{1}$ or the projection of $X_{2}$ to $\R^{l}$ -- such that the curve $\bar{\gamma}$ lifts to a piecewise smooth curve $\gamma$ tangent to $L^{u,g}$ which has initial point $x$ and endpoint $y$. We let $\gamma_{i}$, $1 \leq i \leq 4$, denote the lifts of the sides of the rectangle. 

Thus $\sigma:=\varphi \circ \gamma$ is a piecewise $C^1$ curve tangent to $L^{u,h}$ that starts at $\varphi(x)$ and ends at $\varphi(y)$. Set $\sigma_{i}:=\varphi \circ \gamma_{i}$. Since the tangent vectors to the curve $\sigma_{i}$ are given by $\Phi \circ \gamma'_{i}$, and since $\Phi$ is an isometry on fibers of $L^{u,g}$ and $L^{u,h}$, we conclude that $\l(\sigma_{i}) =\l(\gamma_{i})$ for $1 \leq i \leq 4$, where $\l$ denotes the length of the curve. 

By Lemma \ref{orientation preserve} $\kappa$ is $\nabla$-invariant. Consequently $\kappa$ projects to a Riemannian structure $\bar{\kappa}$ on $T\mathcal{Q}^{u,h}(\varphi(x))$ which is $\bar{\nabla}$-invariant, where $\bar{\nabla}$ is the projection of $\nabla$ to $T\mathcal{Q}^{u,h}$ constructed prior to Proposition \ref{quotient connect}. Since $\bar{\nabla}$ is torsion-free by Proposition \ref{quotient connect}, we conclude that $\bar{\nabla}$ is the Levi-Civita connection for $\bar{\kappa}$. By the construction of Proposition \ref{quotient charts} we then have an isometric chart $e_{\varphi(x)}: Q^{u,h}_{\varphi(x)} \rightarrow T\mathcal{Q}^{u,h}(\varphi(x))$ given by the exponential map, with $Q^{u,h}$ having the flat metric induced by the inner product $\kappa_{\varphi(x)}$.

Since the projections $\pi_{*}: \W^{u,*}(x) \rightarrow \mathcal{Q}^{u,*}$ (where $\mathcal{Q}^{u,h}$ is equipped with the Riemmanian metric induced by the Riemannian structure $\bar{\kappa}$) are isometric on curves tangent to the $L^{u,*}$ bundle (for $* \in \{g,h\}$), we have that $\bar{\sigma_{i}}:= \pi \circ \sigma_{i}$ satisfies $\l(\bar{\sigma}_{i}) = \l(\bar{\gamma}_{i})$. Let $P_{h}$ be the plane in $Q^{u,h}_{\varphi(x)}$ spanned by the vector fields $De_{\varphi(x)}^{-1} \circ D\pi_{h} \circ \Phi \circ X_{i}$, for $i = 1,2$. These vector fields are parallel under the pullback of $\bar{\nabla}$ to $Q^{u,h}_{\varphi(x)}$, therefore they are constant and tangent to a foliation of $Q^{u,h}_{\varphi(x)}$ by parallel lines. Both $\hat{\Phi}$ and the projections $D\pi_{g}$ and $D\pi_{h}$ also preserve orientation and angles, so $\bar{\sigma}$ gives a clockwise oriented curve tangent to $P_{h}$ such that its sides $\bar{\sigma}_{1}$ and $\bar{\sigma}_{3}$ are parallel lines of the same length in $P_{h}$ and the same is true for $\bar{\sigma}_{2}$ and $\bar{\sigma}_{4}$. Furthermore, for $i=1,2,3$ the lines $\bar{\sigma}_{i}$ and $\bar{\sigma}_{i+1}$ meet at a right angle with the clockwise orientation.  This implies that $\bar{\sigma}$ is simply a rectangle in $P_{h}$. In particular the endpoint of $\bar{\sigma}_{4}$ coincides with the initial point of $\bar{\sigma}_{1}$. This implies that $\varphi(y) \in \V^{u,h}(\varphi(x))$ which completes the proof. 
\end{proof}

Since the vertical foliations are preserved by $\varphi$, we obtain an induced map $\bar{\varphi}: \mathcal{Q}^{u,g} \rightarrow \mathcal{Q}^{u,h}$ of the quotient spaces. 

\begin{lem}\label{smooth quotient map}
The map $\bar{\varphi}: \mathcal{Q}^{u,g} \rightarrow \mathcal{Q}^{u,h}$ is $C^{\infty}$. 
\end{lem}

\begin{proof}
We show that $\bar{\varphi}$ is smooth on each component $\mathcal{Q}^{u,g}(x)$ of $\mathcal{Q}^{u,g}$, $x \in T^{1}X$. From Proposition \ref{horizontal diff} we obtain that $\bar{\varphi}$ is $C^1$, with derivative $D\bar{\varphi} = \bar{\Phi}: T\mathcal{Q}^{u,g}(x) \rightarrow T\mathcal{Q}^{u,h}(\varphi(x))$ induced from the horizontal derivative $\Phi:L^{u,g} \rightarrow L^{u,h}$, so that $\bar{\Phi} \circ D\pi_{g} = D\pi_{h} \circ \Phi$. The map $\Phi$ intertwines the unstable holonomies $H^{u,g}$ and $H^{u,h}$ on $L^{u,g}$ and $L^{u,h}$ respectively, so $\bar{\Phi}$ intertwines the corresponding unstable holonomies $\bar{H}^{u,g}$ and $\bar{H}^{u,h}$ on $T\mathcal{Q}^{u,g}(x)$ and $T\mathcal{Q}^{u,h}(\varphi(x))$ respectively. 

These unstable holonomies are given by the parallel transport of $C^{\infty}$ connections $\nabla^{g}$ and $\nabla^{h}$ on $T\mathcal{Q}^{u,g}(x)$ and $T\mathcal{Q}^{u,h}(\varphi(x))$ respectively, by Proposition \ref{quotient charts}. This implies that $\bar{\varphi}$ maps $\nabla^{g}$ to $\nabla^{h}$. But this implies that 
\[
e_{\varphi(x),h}^{-1} \circ \bar{\varphi} \circ e_{x,g}:Q^{u,g}_{x} \rightarrow Q^{u,h}_{\varphi(x)},
\]
is an affine transformation, which is therefore smooth. Since the charts $e_{x,g}$ and $e_{\varphi(x),h}$ are both smooth we conclude that $\bar{\varphi}$ is smooth as well. 
\end{proof}

\section{Vertical differentiability}\label{sec: vert}
At this point we split into two cases, depending on whether $X$ is complex hyperbolic or $X$ is quaternionic/Cayley hyperbolic. We continue to write $h^{t} = h^{t}_{u}$ and $\varphi = \varphi_{u}$. We will be applying Journ\'e's lemma numerous times in this section, so we state below the form that we will use. Recall that we say that a map $f$ is smooth along a foliation $\W$ with smooth leaves if $\W$ has uniformly $C^{\infty}$ leaves and $f$ is uniformly $C^{\infty}$ along $\W$.

\begin{lem}\cite{J}\label{Journe}
Let $M$ and $N$ be smooth manifolds. Let $\W$ and $\V$ be transverse foliations of $M$ with smooth leaves. Let $f: M \rightarrow N$ be  smooth along $\W$ and smooth along $\V$. Then $f$ is smooth. 
\end{lem}

\subsection{The complex hyperbolic case}\label{complex vert}  This is the simpler of the two cases, as the vertical bundles are 1-dimensional. Recall that the conditionals of the SRB measure $m_{h}$ for $h^{t}$ along $\W^{u,h}$ are equivalent to the Riemannian volume on $\W^{u,h}$, as discussed after Definition \ref{SRB}.

\begin{lem}\label{vertical abs cont}
$\varphi$ is smooth along $\V^{u,g}$.
\end{lem}

\begin{proof}
Let $\nu_{x,g}$, $x \in T^{1}X$ denote the conditional measures of $m_{g}$ on $\V^{u,g}$-leaves, and similarly for $y \in M$ let $\nu_{y,h}$ denote the conditional measures of $m_{h}$ on $\V^{u,h}$-leaves. Then $\nu_{x,g}$ is equivalent to the Riemannian arc length on $\V^{u,g}(x)$ and similarly $\nu_{y,h}$ is equivalent to the Riemannian arc length on $\V^{u,h}(y)$, since $m_{g}$ and $m_{h}$ have conditionals equivalent to the Riemannian volume along $\W^{u,g}$ and $\W^{u,h}$ respectively. 

Since $\varphi_{*}m_{g} = m_{h}$ and $\varphi(\V^{u,g}(x)) = \V^{u,h}(\varphi(x))$ for each $x \in T^{1}X$, we have $\varphi_{*}\nu_{x,g} = c(x)\nu_{\varphi(x),h}$ for $m_{g}$-a.e.~ $x \in T^{1}X$, where $c(x) > 0$ is a measurable function of $x$. We conclude that for $x \in T^{1}X$, $y \in \V^{u,g}(x)$,
\[
\limsup_{r \rightarrow 0} \frac{\sup\{d_{\varphi(x),h}(\varphi(x),\varphi(y)): d_{x,g}(x,y) = r)\}}{r} < \infty,
\]
for $m_{g}$-a.e.~ $x \in T^{1}X$. This implies by the Rademacher-Stepanov theorem that $\varphi$ is differentiable a.e.~ on $\V^{u,g}$ leaves, and thus we have an $m_{g}$-a.e.~ defined derivative map $D\varphi^{V}: V^{u,g} \rightarrow V^{u,h}$ that satisfies $D\varphi^{V}  \circ Dg^{t} = Dh^{t} \circ D\varphi ^{V}$ and satisfies $D\varphi^{V} \neq 0$ $m_{g}$-a.e.~, since $\varphi$ is absolutely continuous along $\V^{u,g}$. 

By Proposition \ref{quasi cohomologous} applied to the corresponding bundle map over the original flow $f^{t}$ like how it was used in Proposition \ref{horizontal diff} followed by arguing as in the proof of Lemma \ref{intertwine holonomies}, we conclude that $D\varphi^{V}$ agrees $m_{g}$-a.e.~ with a H\"older continuous homeomorphism from $V^{u,g}$ to $V^{u,h}$ that is linear on fibers and intertwines the unstable holonomies $P^{u,g}$ and $P^{u,h}$ of $Dg^{t}|_{V^{u,g}}$ and $Dh^{t}|_{V^{u,h}}$ respectively. Since these holonomies are given by parallel transport by a smooth connection in the charts of Proposition \ref{quotient charts}, we conclude that $\varphi$ is affine in these charts and therefore $\varphi$ is smooth along $\V^{u,g}$. 
\end{proof}

\subsection{The quaternionic/Cayley hyperbolic case} For the case of quaternionic and Cayley hyperoblic $X$, the vertical bundles are higher dimensional. Replacing $f^{t}$ with $f^{-t}$ and applying the results of Section \ref{sec:pansuderiv} to $h^{t}_{s}$ and $\varphi_{s}$, we obtain from Proposition \ref{vert preserve} that $\varphi_{s}(\mathcal{V}^{s,g}(x)) = \mathcal{V}^{s,h_{s}}(\varphi_{s}(x))$ for every $x \in T^{1}X$ as well. Since $\varphi$ and $\varphi_{s}$ are flow equivalent, this implies that $\varphi(\V^{cs,g}(x)) = \V^{cs,f}(\varphi(x))$. 

Set $q:=k(X)-l(X)$, $\mathbb{K} \in \{\mathbb{H}, \mathbb{O}\}$. For each $v \in T^{1}X$ we consider a lift $\t{v}$ of $v$ to $T^{1}\mathbf{H}_{\mathbb{K}}^{n}$.  There is a unique totally geodesic submanifold $\t{\mathcal{S}}(\t{v}) \cong \mathbf{H}^{q}_{\R}$ inside of $\mathbf{H}_{\mathbb{K}}^{n}$ of constant negative curvature $-4$ such that $\t{v} \in T^{1}\t{\mathcal{S}}(\t{v})$. Projecting back down, there is a unique totally geodesic submanifold $\mathcal{S}(v)$ of $X$ with $v \in T^{1}\mathcal{S}(v)$. This defines a foliation $T^{1}\mathcal{S}$ of $T^{1}X$ by $(2q-1)$-dimensional submanifolds $T^{1}\mathcal{S}(v)$ that are the unit tangent bundles of these totally geodesic submanifolds $\mathcal{S}(v)$. This foliation is tangent to the subbundle $V^{u,g} \oplus E^{c,g} \oplus V^{s,g}$ and the restriction of $g^{t}$ to each leaf $T^{1}\mathcal{S}(v)$ is the geodesic flow of $\mathcal{S}(v)$. 

\begin{lem}\label{smoothing}
For each $x \in T^{1}X$, $\varphi(T^{1}\mathcal{S}(x))$ is a smooth submanifold of $M$ tangent to $V^{u,f} \oplus E^{c,f} \oplus V^{s,f}$. Hence we have a foliation $\mathcal{K}$ of $M$ with smooth leaves that is tangent to $V^{u,f} \oplus E^{c,f} \oplus V^{s,f}$,  with $\varphi(T^{1}\mathcal{S}(x)) = \mathcal{K}(\varphi(x))$ for each $x \in T^{1}X$. 
\end{lem}

\begin{proof}
As discussed above, for each $x \in T^{1}X$ we have $\varphi(\V^{u,g}(x)) = \V^{u,h}(\varphi(x))$ and $\varphi(\V^{cs,g}(x)) = \V^{cs,f}(\varphi(x))$. We note that
\begin{equation}\label{sum expression}
V^{u,f} \oplus E^{c,f} \oplus V^{s,f} = V^{u,h} \oplus V^{cs,f}.
\end{equation}
Let $x \in T^{1}X$ be given. We restrict to small foliation boxes $U_{g}$ and $U_{f}$ for all of the invariant foliations of $g^{t}$ and both $f^{t}$ and $h^{t}$ in a neighborhood of $x$ and $\varphi(x)$ respectively. Consider a small submanifold $H$ inside of $U_{f}$ which is tangent to  $V^{u,h} \oplus V^{cs,f}$ at $\varphi(x)$ and makes a uniformly small angle with this subbundle inside of $U_{f}$. We let $P: U_{f} \rightarrow H$ denote orthogonal projection. The foliations $\V^{u,h}$ and $\V^{cs,f}$ project to transverse foliations $\bar{\V}^{u,h}$ and $\bar{\V}^{cs,f}$ of $H$. When $U_{f}$ is small enough, each leaf of $\bar{\V}^{u,h}$ on $H$ will intersect at most one leaf of $\bar{\V}^{cs,f}$ on $H$. 

We claim that $P \circ \varphi$ is injective on $U_{g} \cap T^{1}\mathcal{S}(x)$, once $U_{g}$ is chosen small enough. Suppose otherwise, so that there are $y$, $z \in U_{g} \cap T^{1}\mathcal{S}(x)$ with $P(\varphi(y)) = P(\varphi(z))$. We may assume that $y \notin \V^{cu,g}_{U_{g}}(z)$, as $P \circ \varphi$ is a homeomorphism onto its image when restricted to $\V^{cu,g}_{U_{g}}(z)$. Here we recall that the subscript $U_{g}$ indicates taking the connected component of the leaf $\V^{cu,g}(z)$ in $U_{g}$ that contains $z$. Let $w$ be the unique intersection point of $\V^{u,g}_{U_{g}}(y)$ with $\V^{cs,g}_{U_{g}}(z)$ inside of $U_{g} \cap T^{1}\mathcal{S}(x)$, noting that $w \neq y$ since $y\notin \V^{u,g}_{U_{g}}(z)$. Since $P \circ \varphi$ is a homeomorphism onto its image on both $\V^{u,g}_{U_{g}}(y)$ and $\V^{cs,g}_{U_{g}}(z)$, with that image being contained in $\bar{\V}^{u,h}_{U_{f}}(P(\varphi(y)))$ and $\bar{\V}^{cs,f}_{U_{f}}(P(\varphi(z)))$ respectively, we conclude that $P(\varphi(w))$ must be the unique intersection point of $\bar{\V}^{u,h}_{U_{f}}(P(\varphi(y)))$ and $\bar{\V}^{cs,f}_{U_{f}}(P(\varphi(z)))$ inside of $H$. Since $P(\varphi(y)) = P(\varphi(z))$ we conclude that $P(\varphi(w)) = P(\varphi(y))$. But this is impossible because $P \circ \varphi$ is a homeomorphism onto its image when restricted to $\V^{u,g}_{U_{g}}(y)$. 

Hence we may consider $\mathcal{K}_{H}(\varphi(x)) = \varphi(U_{g} \cap T^{1}\mathcal{S}(x))$ as a graph over $H$. This graph is smooth along the transverse foliations $\bar{\V}^{u,h}$ and $\bar{\V}^{cs,f}$ of $H$. We conclude by Lemma \ref{Journe} that this graph is smooth and therefore $\mathcal{K}_{H}(\varphi(x))$ is a smooth submanifold of $M$. Since its tangent space is $(2q-1)$-dimensional and contains the transverse subbundles  $V^{u,h}$ and $V^{cs,f}$, we conclude that $\mathcal{K}_{H}$ is tangent to $V^{u,h} \oplus V^{cs,f}$. 

Since smoothness is a local property, it follows that $\mathcal{K}(y) = \varphi(T^{1}\mathcal{S}(\varphi^{-1}(y)))$ is a smooth submanifold of $M$ tangent to $V^{u,f} \oplus E^{c,f} \oplus V^{s,f}$ for each $y \in M$, by \eqref{sum expression}. The conclusion follows. 
\end{proof}

\begin{lem}\label{vert uniform quasi}
$Df^{t}|_{V^{*,f}}$ is uniformly quasiconformal for $* \in \{s,u\}$.
\end{lem}

\begin{proof}
We prove the claim for $Df^{t}|_{V^{u,f}}$; the claim for $Df^{t}|_{V^{s,f}}$ follows by an analogous argument considering $f^{-t}$.  Fix $x \in M$ and restrict to $\mathcal{K}(x)$. Since $\mathcal{K}$ is a foliation of $M$ with smooth leaves, the embedding of $\mathcal{K}(x)$ into $M$ has uniformly bounded $C^2$ norm. Restricted to $\mathcal{K}(x)$, $\varphi: T^{1}\mathcal{S}(\varphi^{-1}(x)) \rightarrow \mathcal{K}(x)$ gives an orbit equivalence from $f^{t}$ to the geodesic flow of the simply connected Riemannian manifold $\mathcal{S}(\varphi^{-1}(x))$ of pinched negative curvature. By the hypotheses of Theorem \ref{core theorem}, $Df^{t}|_{V^{u,f}}$ and $Df^{t}|_{V^{s,f}}$ are both $1$-fiber bunched and we have $\la_{l+1}^{u}(f,m_{f}) = \la_{k}^{u}(f,m_{f})$.  The claim then follows from Proposition \ref{one exp to conf}.
\end{proof}

We conclude from Lemma \ref{vert uniform quasi} that $Dh^{t}|_{V^{u,h}}$ is uniformly quasiconformal, since $h^{t}$ is obtained as a time change of $f^{t}$ which is smooth along $\W^{cu,f}$. For the following lemma we will use the metric notion of quasiconformality in Euclidean space. For a constant $K\geq 1$, a homeomorphism $\varphi: U_{1} \rightarrow U_{2}$ between open subsets $U_{i}$ of $\R^{q}$ is \emph{$K$-quasiconformal} if for each $x \in U_{1}$ we have
\[
\limsup_{r \rightarrow 0} \frac{\sup\{d(\varphi(x),\varphi(y)):d(x,y) \leq r\}}{\inf\{d(\varphi(x),\varphi(y)):d(x,y) \geq r\}} \leq K. 
\]
Note in particular that bi-Lipschitz maps are quasiconformal.

\begin{lem}\label{phidiffquasi}
$\varphi$ is smooth along $\V^{u,g}$.
\end{lem}

\begin{proof}
We first show that there is a $K \geq 1$ independent of $x \in T^{1}X$ such that $\varphi |_{\V^{u,g}(x)}$ is $K$-quasiconformal. Since $Dh^{t}|_{V^{u,h}}$ is uniformly quasiconformal, we conclude by \eqref{vertical Jacobian} that for $v \in V^{u,h}$ and $t \in \R$ we have
\begin{equation}\label{vertical scaling}
\|Dh^{t}(v)\| \asymp e^{2t}\|v\|.
\end{equation}

We write $d_{x,g}$ for the Riemannian distance on $\V^{u,g}(x)$. Let $r > 0$ be given and let $y,z \in \V^{u,g}(x)$ satisfy $d_{x,g}(y,z) = r$. Setting $t = -(\log r)/2$, we then have $d_{g^{t}x,g}(g^{t}y,g^{t}z) = 1$. By the uniform continuity of $\varphi$, we then conclude that
\[
d_{h^{t}(\varphi(x)),h}(h^{t}(\varphi(y)),h^{t}(\varphi(z)) = d_{h^{t}(\varphi(x)),h}(\varphi(g^{t}y),\varphi(g^{t}z)) \asymp 1,
\]
with multiplicative constant independent of $y, z , r$. Since $h^{t}(\varphi(y)), h^{t}(\varphi(z)) \in \V^{u,h}(h^{t}(\varphi(x)))$, by \eqref{vertical scaling} we then have
\[
d_{\varphi(x),h}(\varphi(y),\varphi(z)) \asymp e^{2t} = r.
\]
for any $y,z \in \V^{u,g}(x)$ that satisfies $d_{x,g}(y,z) = r$. We conclude that $\varphi: (\V^{u,g}(x),d_{x,g}) \rightarrow (\V^{u,h}(x),d_{\varphi(x),h})$ is bi-Lipschitz with Lipschitz constant independent of $x$, and in particular there is a constant $K \geq 1$ independent of $x$ such that $\varphi: \V^{u,g}(x) \rightarrow \V^{u,h}(\varphi(x))$ is $K$-quasiconformal. 

As in Lemma \ref{vertical abs cont}, we let $\nu_{x,g}$ denote the conditional measures of $m_{g}$ on the leaves $\V^{u,g}(x)$ which are equivalent to the Riemannian volume on these leaves. As a consequence of the $K$-quasiconformality, $\varphi|_{\V^{u,g}(x)}$ is absolutely continuous and differentiable $\nu_{x,g}$-a.e.~ on this leaf\cite{V71}, and so we obtain a measurable $m_{g}$-a.e.~ defined derivative map $D\varphi^{V}: V^{u,g} \rightarrow V^{u,h}$ that satisfies $D\varphi^{V} \circ Dg^{t} = Dh^{t} \circ D\varphi^{V}$ and is invertible $m_{g}$-a.e. Reasoning as we did with the measurable bundle map $D\varphi^{V}$ in Lemma \ref{vertical abs cont}, we conclude that $D\varphi^{V}$ agrees $m_{g}$-a.e.~ with a H\"older continuous homeomorphism from $V^{u,g}$ to $V^{u,h}$ that is linear on fibers and intertwines the unstable holonomies $P^{u,g}$ and $P^{u,h}$ of $Dg^{t}|_{V^{u,g}}$ and $Dh^{t}|_{V^{u,h}}$. 

By Proposition \ref{quotient charts} these unstable holonomies are given by parallel transport by a smooth connection along the leaves of $\V^{u,g}$ and $\V^{u,h}$ respectively. By the same argument as in Lemma \ref{smooth quotient map}, we conclude that $\varphi$ is smooth in the coordinates given by these connections and is therefore smooth along $\V^{u,g}$.
\end{proof}

\subsection{Differentiability along the center-unstable foliation} We return to considering both cases together. We will now establish that $\varphi$ is smooth along $\W^{cu,g}$. Differentiability in the flow direction is simple.

\begin{lem}\label{flow dir conj diff}
$\varphi$ is smooth along $\W^{c,g}$. 
\end{lem}

\begin{proof}
The restriction of $h^{t}$ to $\W^{c,h}$ is smooth, and thus we can choose a new Riemannian structure on $\W^{c,h}$, smooth along $\W^{c,h}$, such that the generator $v_{h}$ of $h^{t}$ satisfies $\|v_{h}(x)\| = 1$ for all $x \in M$. For $\delta > 0$ small enough we then have a smooth diffeomorphism $a_{x,h}:(-\delta,\delta) \rightarrow \W^{c,h}(x)$ given by $a_{x,h}(t) = h^{t}x$, and a similarly defined diffeomorphism $a_{x,g}: (-\delta,\delta) \rightarrow \W^{c,g}(x)$ given by $a_{x,g}(t) = g^{t}x$ for $x \in T^{1}X$. 

The conjugacy relation $\varphi \circ g^{t} = h^{t} \circ \varphi$ implies that $a_{\varphi(x),h}^{-1} \circ \varphi \circ a_{x,g}$ is simply the identity map on $(-\delta,\delta)$, hence is smooth. It follows that $\varphi$ is smooth along $\W^{c,g}$.
\end{proof}

We require a simple lemma. Recall that $V$ denotes the left-invariant distribution on $G$ tangent to $\mathfrak{v}$ at the identity. 

\begin{lem}\label{commutator}
Let $F\subset TG$ be a left-invariant distribution with $V \subseteq F$. Let $Z$ be a left-invariant vector field on $G$. Then $P = \mathrm{span}(Z,F)$ is an integrable left-invariant distribution on $G$, tangent to a foliation $\mathcal{P}$. 
\end{lem}

\begin{proof}
It suffices to show that $[Z, Y] \subseteq \mathrm{span}(Z,F)$ for any left-invariant vector field $Y$ tangent to $F$. But in fact $[Z,Y] \subseteq V \subseteq F$ since $G$ is 2-step nilpotent. The conclusion follows. 
\end{proof}

\begin{lem}\label{unstable smooth map}
$\varphi$ is smooth along $\W^{cu,g}$. 
\end{lem}

\begin{proof}
Fix $x \in M$ and consider the restriction $\varphi: \W^{u,g}(x) \rightarrow \W^{u,h}(\varphi(x))$. Identify $\W^{u,g}(x)$ with $G$ via the chart $T_{x}: G \rightarrow \W^{u,g}(x)$ and set $\psi = \varphi \circ T_{x}$. By Lemmas \ref{vertical abs cont} and \ref{phidiffquasi}, $\psi$ is smooth along the vertical foliation $\V$ of $G$ that is tangent to $V$. 

Now suppose that $\psi$ is smooth along a proper foliation $\mathcal{F}$ of $G$ tangent to a left-invariant distribution $F$ on $TG$ with $V \subseteq F$. Let $Z$ be a left-invariant vector field tangent to the horizontal distribution $L$ which is not contained in $F$. Let $P = \mathrm{span}(Z,F)$. By Lemma \ref{commutator} there is a smooth foliation $\mathcal{P}$ of $G$ tangent to $P$.   

By projecting to $G/\mathcal{V}$ we obtain a smooth left-invariant foliation $\bar{\mathcal{P}}$ of $G/\mathcal{V}$. Letting $\bar{\psi}:G/\mathcal{V} \rightarrow \mathcal{Q}^{u,h}(\varphi(x))$ denote the quotient map, by Lemma \ref{smooth quotient map} this map is smooth and therefore the image of $\bar{\mathcal{P}}$ is a smooth foliation $\bar{\mathcal{R}}$ of $\mathcal{Q}^{u,h}(\varphi(x))$. Taking the pre-image under the projection $\pi_{x}: \W^{u,h}(\varphi(x)) \rightarrow \mathcal{Q}^{u,h}(\varphi(x))$, we obtain a smooth foliation $\mathcal{R}$ of $\W^{u,h}(\varphi(x))$ with $\varphi(\mathcal{P}(y))  =\mathcal{R}(\psi(y))$ for each $y \in G$. If $P = TG$ then the foliation $\mathcal{P}$ is trivial and $\mathcal{R}(\psi(y)) = \W^{u,h}(\varphi(x))$ for all $y \in G$. The adaptations for this case in the argument below are straightforward. 

Let $\mathcal{Y}$ denote the image of $\mathcal{Z}$ under $\psi$. By Proposition \ref{horizontal diff} this defines a foliation of $\W^{u,h}(\varphi(x))$ by $C^1$ curves. We restrict to $\mathcal{P}(y)$ for some fixed $y \in G$ and consider the holonomy map $h^{\mathcal{Z}}: \mathcal{F}(y) \rightarrow \mathcal{F}(z)$ for $z \in \mathcal{Z}(y)$. This map is $C^{\infty}$, hence the $\mathcal{Y}$-holonomy map $h^{\mathcal{Y}}$ between $\psi(\mathcal{F}(y))$ and $\psi(\mathcal{F}(z))$ inside of $\mathcal{K}(\psi(y))$ will also be $C^{\infty}$, since $\psi$ is $C^{\infty}$ on $\psi(\mathcal{F}(y))$ and $\psi(\mathcal{F}(z))$. Thus by  \cite[Lemma 31]{BX} $\mathcal{Y}$ defines a smooth foliation of $\mathcal{K}(\psi(y))$, since its holonomies between transversals are smooth. Then $\psi: \mathcal{Z}(y) \rightarrow \mathcal{Y}(\psi(y))$ can be factored as $\psi = \pi_{h}^{-1} \circ \bar{\psi} \circ \pi_{G}$, where $\pi_{G}: G \rightarrow G/\mathcal{V}$ is the smooth projection, $\bar{\psi}$ is smooth by Lemma \ref{smooth quotient map}, and 
\[
\pi_{h}^{-1}: \pi_{h}( \mathcal{Y}(\psi(y)))\rightarrow \mathcal{Y}(\psi(y)),
\]
is smooth since the curve $\mathcal{Y}(\psi(y))$ is smooth inside of $\mathcal{K}(\psi(y))$.

We conclude that $\psi$ is smooth along $\mathcal{Z}$ considered as a subfoliation of a fixed leaf $\mathcal{P}(y)$  of $\mathcal{P}$. By Lemma \ref{Journe} this implies that $\psi$ is smooth along $\mathcal{P}$. This completes the inductive step. We repeat this process until we reach $P = TG$, at which point we conclude that $\psi$ is smooth and therefore $\varphi: \W^{u,g}(x) \rightarrow \W^{u,h}(\varphi(x))$ is smooth. Thus $\varphi$ is smooth along $\W^{u,g}$. Since $\varphi$ is smooth along $\W^{c,g}$ by Lemma \ref{flow dir conj diff}, by applying Lemma \ref{Journe} again we conclude that $\varphi$ is smooth along $\W^{cu,g}$. 
\end{proof}

\subsection{Conclusion}
We will now complete the proof of Theorem \ref{core theorem}. Let $h^{t}_{u}$ and $h^{t}_{s}$ be the respective unstable and stable synchronizations constructed in Section \ref{subsec:synchro}, which are smooth along $\W^{cu,f}$ and $\W^{cs,f}$ respectively. Let $\W^{u,h_{u}}$ and $\W^{u,h_{s}}$ denote the respective invariant subfoliations of $\W^{cu,f}$ and $\W^{cs,f}$ for $h^{t}_{u}$ and $h^{t}_{s}$. We constructed conjugacies $\varphi_{u}$ and $\varphi_{s}$ from $g^{t}$ to  $h^{t}_{u}$ and $h^{t}_{s}$ respectively in Proposition \ref{prop:orbit to conj}, and established in the previous section that $\varphi_{u}$ is smooth along $\W^{cu,g}$ and $\varphi_{s}$ is smooth along $\W^{cs,g}$; the result for $\varphi_{s}$ follows from arguments exactly analogous to those used for $\varphi_{u}$, replacing $f^{t}$ with $f^{-t}$ and $g^{t}$ with $g^{-t}$ as usual. 

\begin{lem}\label{foliation smooth}
The foliations $\W^{cu,f}$ and $\W^{cs,f}$ are $C^{\infty}$. 
\end{lem}

\begin{proof}
Let $U$ be a foliation box for all of the invariant foliations of $f^{t}$. Consider the $\W^{cs,f}$-holonomy map $h^{cs,f}:\W^{u,h_{u}}_{U}(x) \rightarrow \W^{u,h_{u}}_{U}(y)$ for $x \in U$ and $y \in \W^{cs,f}_{U}(x)$. Since $\varphi_{u}$ conjugates $g^{t}$ to $h^{t}_{u}$, we have $h^{cs,f} = \varphi_{u} \circ h^{cs,g} \circ \varphi_{u}^{-1}$, where $h^{cs,g}$ denotes the $\W^{cs,g}$-holonomy map between  $\varphi_{u}^{-1}(\W^{u,h_{u}}_{U}(x))$ and $\varphi_{u}^{-1}(\W^{u,h_{u}}_{U}(x))$. Since $\W^{cs,g}$ is a smooth foliation of $T^{1}X$ and $\varphi_{u}$ is smooth along $\W^{u,g}$, we conclude that $h^{cs,f}$ is $C^{\infty}$. Since this holds for any pair of transversals to $\W^{cs,f}$ from the foliation $\W^{u,h_{u}}$, we conclude by \cite[Lemma 31]{BX} that $\W^{cs,f}$ is $C^{\infty}$. An analogous proof using $h^{t}_{s}$ and $\varphi_{s}$ instead of $h^{t}_{u}$ and $\varphi_{u}$ establishes that $\W^{cu,f}$ is $C^{\infty}$ as well. 
\end{proof}

We also establish an analogous smoothness property for the vertical foliations. 

\begin{lem}\label{vert foliation smooth}
The foliations $\V^{cu,f}$ and $\V^{cs,f}$ are $C^{\infty}$. 
\end{lem}

\begin{proof}
We prove the claim for $\V^{cs,f}$, with the proof for $\V^{cu,f}$ being analogous. Let $U$ be a foliation box for all of the invariant foliations of $f^{t}$. We claim that if $x \in U$, $y \in \W^{u,h_{u}}_{U}(x)$, and $z \in \V^{cs,f}_{U}(x)$ then, provided $x$, $y$, and $z$ are close enough, there is a unique intersection point $\{w\} = \W^{u,h_{u}}_{U}(z) \cap \V^{cs,f}_{U}(y)$. Since $\varphi^{-1}_{u}$ takes $\W^{u,h_{u}}$ to $\W^{u,g}$ and takes $\V^{cs,f}$ to $\V^{cs,g}$, it suffices to prove the corresponding claim for the foliations $\W^{u,g}$ and $\V^{cs,g}$ of $T^{1}X$. But these foliations are jointly integrable (see Section \ref{subsec:neg curved}), and so the claim follows. 

Hence we have a well-defined holonomy map $h^{\V^{cs,f}}:\W^{u,h_{u}}_{U}(x) \rightarrow \W^{u,h_{u}}_{U}(z)$ along $\V^{cs,f}$. As in Lemma \ref{foliation smooth}, we can express this map as $h^{\V^{cs,f}} = \varphi_{u} \circ h^{\V^{cs,g}} \circ \varphi_{u}^{-1}$, from which we conclude that $h^{\V^{cs,f}}$ is smooth since $\V^{cs,g}$ is a smooth foliation of $T^{1}X$ and $\varphi_{u}$ is smooth along $\W^{u,h_{u}}$. Appealing to \cite[Lemma 31]{BX} again, we conclude that the bundle $\V^{cs,f}$ is smooth along $\W^{u,h_{u}}$. Since $V^{cs,f}$ is also smooth along $\W^{cs,f}$, we conclude by Lemma \ref{Journe} that $V^{cs,f}$ is a smooth subbundle of $TM$ and therefore $\V^{cs,f}$ is a smooth foliation. 
\end{proof}

As a consequence of Lemmas \ref{foliation smooth} and \ref{vert foliation smooth}, both $E^{cu,f}$ and $V^{cu,f}$ are $C^{\infty}$ subbundles of $TM$. It follows that $\bar{Q}^{u,f} = E^{cu,f}/V^{cu,f}$ is a $C^{\infty}$ bundle over $M$, with $C^{\infty}$ holonomies $H^{s,f}$ and $H^{u,f}$ given by Proposition \ref{holonomy connect}, since $Df^{t}|_{\bar{Q}^{u,f}}$ is conformal. Hence the Riemannian structure $\delta$ on $\bar{Q}^{u,f}$ constructed at the beginning of Section \ref{subsec:synchro} is $C^{\infty}$, and consequently the potential $\bar{\zeta}_{f}$ of \eqref{unstable quotient potential} is $C^{\infty}$. We conclude that $h^{t}_{u}$ is a $C^{\infty}$ time change of $f^{t}$. Applying the same reasoning to $E^{cs,f}$ and $V^{cs,f}$, we conclude that $h^{t}_{s}$ is a $C^{\infty}$ time change of $f^{t}$ as well. Hence $h^{t}_{u}$ is a smooth time change of $h^{t}_{s}$. 

We have the conjugacy relations
\[
\varphi_{u} \circ g^{t} = h^{t}_{u} \circ \varphi_{u},
\]
and
\[
\varphi_{s} \circ g^{-t} = h^{t}_{s} \circ \varphi_{s},
\]
valid for all $t \in \R$. Let $\xi:T^{1}X \rightarrow T^{1}X$ denote the flip map $\xi(v) = -v$ sending a tangent vector to its negation. Then
\[
\xi \circ g^{t} = g^{-t} \circ \xi,
\]
and therefore if we set $\hat{h}^{t}_{s} = h^{-t}_{s}$ and $\hat{\varphi}_{s} = \varphi_{s} \circ \xi$ then
\[
\hat{\varphi}_{s} \circ g^{t} = \hat{h}^{t}_{s} \circ \hat{\varphi}_{s}
\] 

Letting $\psi = \varphi_{u} \circ \hat{\varphi}_{s}^{-1}$, we obtain 
\begin{equation}\label{final conjugacy}
\psi \circ \hat{h}^{t}_{s} = h^{t}_{u} \circ \psi. 
\end{equation}
The map $\psi:M \rightarrow M$ fixes the $\W^{c,f}$ foliation. We can thus, after passing to a finite cover of $M$ if necessary, write $\psi(x) = h^{\eta(x)}_{u}x$ for some continuous function $\eta: M \rightarrow \R$. Since $h^{t}_{u}$ is a smooth time change of $\hat{h}^{t}_{s}$, we can write $\hat{h}^{t}_{s}x = h^{\tau(t,x)}_{u}x$ for a smooth additive cocycle $\tau$ over $\hat{h}^{t}_{s}$. Plugging this into equation \eqref{final conjugacy} gives
\[
h^{\eta(\hat{h}^{t}_{s}x)+\tau(t,x)}_{u}x = h^{t+\eta(x)}_{u}x.
\]
For any $x$ not lying on a periodic orbit of $h^{t}_{u}$, the map $t \rightarrow h^{t}_{u}x$ is a homeomorphism of $\R$ onto the $h^{t}_{u}$-orbit of $x$. Hence for such $x$ it follows that for all $t \in \R$, 
\[
\eta(\hat{h}^{t}_{s}x)- \eta(x) = t - \tau(t,x).
\]
Since non-periodic points for $h^{t}_{u}$ are dense, this equation holds for every $x \in M$ and $t \in \R$. Thus the additive cocycle $\tau$ over $\hat{h}^{t}_{s}$ is cohomologous to $t$ via the transfer function $\eta$. Since both $t$ and $\tau$ are smooth, standard Livsic theory \cite{LS72} implies that $\eta$ is also smooth. We conclude that $\psi$ is smooth. 

Hence we may write $\varphi_{u} = \psi \circ \hat{\varphi}_{s}$, with $\psi:M \rightarrow M$ smooth. Since $\hat{\varphi}_{s}$ is smooth along $\W^{cs,g}$, it then follows that $\varphi_{u}$ is smooth along $\W^{cs,g}$ as well. Then $\varphi_{u}$ is smooth along the transverse foliations $\W^{cu,g}$ and $\W^{s,g}$ of $T^{1}X$; a final application of Lemma \ref{Journe} then establishes that $\varphi_{u}$ itself is smooth. We may then take $\hat{\varphi} = \varphi_{u}$ in Theorem \ref{core theorem}, which completes the proof.

\section{Mostow rigidity}\label{mostow mod} We explain in this section how Corollary \ref{mostow rigidity} follows directly from Theorem \ref{periodic rigid} without the use of the minimal entropy rigidity theorem implicit in Theorem \ref{minimal rigidity}. We treat the nonconstant negative curvature case here, since it is the one that we consider in this paper; the constant negative curvature case is analogous and follows naturally by simplifying the arguments of the paper. 

Let's now restrict to the setting of Corollary \ref{mostow rigidity}. Let $\kappa: X \rightarrow Y$ be the given homotopy equivalence and let $\varphi:T^{1}X \rightarrow T^{1}Y$ be the orbit equivalence from $g^{t}_{X}$ to $g^{t}_{Y}$ built from $\kappa$ by the construction in Section \ref{subsec:geom rigidity}.

We first claim that $X$ and $Y$ must have homothetic universal covers. Let $G_{X}$ and $G_{Y}$ denote the respective (at most) 2-step nilpotent Lie groups giving the structure of the unstable manifolds of $g^{t}$ and $f^{t}$ respectively. We have standard charts $T_{x}: G_{X} \rightarrow \W^{u,g_{X}}(x)$ and $S_{y}: G_{Y} \rightarrow \W^{u,g_{Y}}(y)$  giving each of these unstable manifolds the structure of a homogeneous space for these groups. The orbit equivalence $\varphi$ gives rise to a quasisymmetric homeomorphism $\varphi_{x}: (\W^{u,g_{X}}(x),\rho_{x,g_{X}}) \rightarrow (\W^{u,g_{Y}}(\varphi(x)),\rho_{\varphi(x),g_{Y}})$ for each $x \in T^{1}X$, by Proposition \ref{quasisymmetry orbit}, which gives a quasisymmetric homeomorphism $\psi_{x} = S_{\varphi(x)}^{-1} \circ \varphi_{x} \circ T_{x}$ from $G_{X}$ to $G_{Y}$, when $G_{X}$ and $G_{Y}$ are equipped with their respective CC-metrics (or the Euclidean metric, in the case that $X$ and/or $Y$ has constant negative curvature). 

Set $k = \dim X - 1$. Since both $G_{X}$ and $G_{Y}$ have a positive modulus family of curves, by Theorem \ref{Tyson theorem} we must have $\mathrm{Hd}(G_{X},\rho_{G_{X}}) =  \mathrm{Hd}(G_{Y},\rho_{G_{Y}})$. The Hausdorff dimensions of $G_{X}$ and $G_{Y}$ are $N_{G_{X}} = l(X) + 2(k-l(X))$ and $N_{G_{Y}} = l(Y) + 2(k-l(Y))$ respectively, where if $X$ has universal cover homothetic to $\mathbf{H}^{n}_{\mathbb{K}}$ then 
\[
l(X) = k - \dim_{\R}\mathbb{K} + 1,
\]
and the same for $Y$. Since $N_{G_{X}} = N_{G_{Y}}$, we conclude that $l(X) = l(Y)$ and therefore both $X$ and $Y$ have universal covers homothetic to the same symmetric space $\mathbf{H}^{n}_{\mathbb{K}}$. We then set $G = G_{X} = G_{Y}$. 

By rescaling the metrics on $X$ and $Y$ we can then assume that their universal covers are isometric to $\mathbf{H}_{\mathbb{K}}^{n}$. We will be working in the nonconstant negative curvature case, as we mentioned above, so we will assume that $\mathbb{K} \in \{\C,\mathbb{H},\mathbb{O}\}$. With this rescaling we then have $\vec{\la}(g_{Y}) = \vec{\la}(g_{X})$. In particular $g^{t}_{Y}$ satisfies the hypotheses of Theorem \ref{periodic rigid}, since all $g^{t}_{Y}$-invariant ergodic probability measures $\nu$ have the same Lyapunov spectrum $\vec{\la}(g_{Y})$.

Proceeding into the body of the proof of Theorem \ref{core theorem}, we observe that in Section \ref{subsec:synchro} the synchronization of $g^{t}_{Y}$ is trivial, i.e., we may take $g^{t}_{Y} = h^{t}_{u} = h^{-t}_{s}$. This is because $Dg^{t}_{Y}$ is already conformal with scaling factor $e^{t}$ on $L^{u,g_{Y}}$ in the Riemannian structure on $L^{u,g_{Y}}$ coming from the left-invariant Riemannian metric on $G$. Then Proposition \ref{prop:orbit to conj} gives us a conjugacy $\hat{\varphi}: T^{1}X \rightarrow T^{1}Y$ from $g^{t}_{X}$ to $g^{t}_{Y}$ that is flow related to $\varphi$. The work of Sections \ref{sec:pansuderiv} and \ref{sec: vert} shows that $\hat{\varphi}$ is $C^{\infty}$. 
 
For each $x \in T^{1}X$ we claim that, in the notation of Section \ref{subsec:neg curved}, we have
\begin{equation}\label{flow related use}
\p \t{\kappa} = P_{\hat{\varphi}(x)} \circ \hat{\varphi} \circ P_{x}^{-1}: \p \mathbf{H}^{n}_{\mathbb{K}} \backslash\{\theta_{-}(x)\} \rightarrow \p \mathbf{H}^{n}_{\mathbb{K}} \backslash \{\theta_{-}(\hat{\varphi}(x))\}.
\end{equation}
Writing $P_{\varphi(x)}: \W^{cu,g_{Y}}(x) \rightarrow \p \mathbf{H}^{n}_{\mathbb{K}} \backslash\{\theta_{-}(\varphi(x))\}$ for the $\theta_{+}$ projection to $\p \mathbf{H}^{n}_{\mathbb{K}}$ on the full center-unstable manifold $ \W^{cu,g_{Y}}(x)$, by the construction of $\varphi$ we have
\[
\p \t{\kappa} = P_{\varphi(x)} \circ \varphi \circ P_{x}^{-1},
\]
on $\p \mathbf{H}^{n}_{\mathbb{K}} \backslash\{\theta_{-}(x)\}$. Since $\hat{\varphi}$ is flow related to $\varphi$, $\hat{\varphi}(x)$ is tangent to the same geodesic as $\varphi(x)$ for any $x \in T^{1}X$. Since the projections $P_{x}$ are projections along geodesics, this implies that 
\[
P_{\hat{\varphi}(x)} \circ \hat{\varphi} \circ P_{x}^{-1} = P_{\varphi(x)} \circ \varphi \circ P_{x}^{-1},
\]
from which equation \eqref{flow related use} follows. 

We conclude from \eqref{flow related use} that $\p \t{\kappa}$ is smooth on $\p \mathbf{H}^{n}_{\mathbb{K}} \backslash\{\theta_{-}(x)\}$ since $\hat{\varphi}$ is smooth and the projections $P_{x}$ are smooth. Taking a point $z \in T^{1}X$ with $\theta_{-}(z) \neq \theta_{-}(x)$, the same argument shows that $\p \t{\kappa}$ is smooth on $\p \mathbf{H}^{n}_{\mathbb{K}} \backslash\{\theta_{-}(z)\}$, and putting these two claims together shows that $\p \t{\kappa}$ is smooth on $\p \mathbf{H}^{n}_{\mathbb{K}}$.

We may think of $\p \mathbf{H}^{n}_{\mathbb{K}}$ as the one point compactification $\hat{G} = G \cup \{\infty\}$; the horizontal distribution $L$ then extends to a smooth distribution on $\hat{G}$ \cite{Bou18}. The inverse projection $P_{x}^{-1}: \p \mathbf{H}^{n}_{\mathbb{K}} \backslash\{\theta_{-}(x)\} \rightarrow \W^{u,g_{X}}(x)$ maps $L$ to $L^{u,g_{X}}$ on $\W^{u,g_{X}}(x)$, and likewise $P_{\varphi(x)}^{-1}$ maps $L$ to $L^{u,g_{Y}}$ on $\W^{u,g_{Y}}(\varphi(x))$. Since $D\hat{\varphi}(L^{u,g_{X}}) = L^{u,g_{Y}}$ and $D\hat{\varphi}$ is isometric on $L^{u,g_{X}}$ (as shown in Proposition \ref{vert preserve}), since the projection $P_{x}$ is conformal on $L^{u,g_{X}}$ and $P_{\varphi(x)}$ is conformal on $L^{u,g_{Y}}$ we conclude that $D(\p \t{\kappa})(L) = L$ and $D(\p \t{\kappa})$ is conformal on $L$.

It follows that $\p\t{\kappa}$ is conformal in the conformal structure on $\p \mathbf{H}^{n}_{\mathbb{K}}$ induced by the CC-metric on $G$, and therefore $\p \t{\kappa}$ is induced by a unique isometry $\t{\sigma}: \mathbf{H}^{n}_{\mathbb{K}} \rightarrow  \mathbf{H}^{n}_{\mathbb{K}}$ which is a bounded distance from $\t{\kappa}$ and equivariant with respect to the actions of $\pi_{1}(X)$ and $\pi_{1}(Y)$ on $\mathbf{H}^{n}_{\mathbb{K}}$ \cite{P89b}. The map $\t{\sigma}$ descends to the desired isometry $\sigma: X \rightarrow Y$ as desired. Lastly, since $\t{\kappa}$ and $\t{\sigma}$ are at a bounded distance from one another, they are homotopic by a straight line homotopy along geodesics which will descend to a homotopy from $\kappa$ to $\sigma$. 
\bibliographystyle{plain}
\bibliography{HorConf}

\begin{thebibliography}{10}

\bibitem{Abr59}
L.~M. Abramov.
\newblock On the entropy of a flow.
\newblock {\em Dokl. Akad. Nauk SSSR}, 128:873--875, 1959.

\bibitem{A69}
D.~V. Anosov.
\newblock {\em Geodesic flows on closed {R}iemann manifolds with negative
  curvature}.
\newblock Proceedings of the Steklov Institute of Mathematics, No. 90 (1967).
  Translated from the Russian by S. Feder. American Mathematical Society,
  Providence, R.I., 1969.

\bibitem{AVW}
A~Avila, M.~Viana, and A.~Wilkinson.
\newblock Absolute continuity, {L}yapunov exponents and rigidity {I}: geodesic
  flows.
\newblock {\em J. Eur. Math. Soc. (JEMS)}, 17(6):1435--1462, 2015.

\bibitem{BGS}
W.~Ballmann, M.~Gromov, and V.~Schroeder.
\newblock {\em Manifolds of nonpositive curvature}, volume~61 of {\em Progress
  in Mathematics}.
\newblock Birkh\"auser Boston, Inc., Boston, MA, 1985.

\bibitem{BKR07}
Z.~M. Balogh, P.~Koskela, and S.~Rogovin.
\newblock Absolute continuity of quasiconformal mappings on curves.
\newblock {\em Geom. Funct. Anal.}, 17(3):645--664, 2007.

\bibitem{BCG2}
G.~Besson, G.~Courtois, and S.~Gallot.
\newblock Minimal entropy and {M}ostow's rigidity theorems.
\newblock {\em Ergodic Theory Dynam. Systems}, 16(4):623--649, 1996.

\bibitem{BG09}
J.~Bochi and N.~Gourmelon.
\newblock Some characterizations of domination.
\newblock {\em Math. Z.}, 263(1):221--231, 2009.

\bibitem{Bol}
J.~Boland.
\newblock On rigidity properties of contact time changes of locally symmetric
  geodesic flows.
\newblock {\em Discrete Contin. Dynam. Systems}, 6(3):645--650, 2000.

\bibitem{Bou18}
M.~Bourdon.
\newblock Mostow type rigidity theorems.
\newblock In {\em Handbook of group actions. {V}ol. {IV}}, volume~41 of {\em
  Adv. Lect. Math. (ALM)}, pages 139--188. Int. Press, Somerville, MA, 2018.

\bibitem{Bow73}
R.~Bowen.
\newblock Symbolic dynamics for hyperbolic flows.
\newblock {\em Amer. J. Math.}, 95:429--460, 1973.

\bibitem{BR75}
R.~Bowen and D.~Ruelle.
\newblock The ergodic theory of {A}xiom {A} flows.
\newblock {\em Invent. Math.}, 29(3):181--202, 1975.

\bibitem{BK87}
M.~Brin and Yu. Kifer.
\newblock Dynamics of {M}arkov chains and stable manifolds for random
  diffeomorphisms.
\newblock {\em Ergodic Theory Dynam. Systems}, 7(3):351--374, 1987.

\bibitem{Bu1}
C.~Butler.
\newblock Rigidity of equality of {L}yapunov exponents for geodesic flows.
\newblock {\em J. Differential Geom.}, 109(1):39--79, 2018.

\bibitem{BX}
Clark Butler and Disheng Xu.
\newblock Uniformly quasiconformal partially hyperbolic systems.
\newblock {\em Ann. Sci. Ec. Norm. Super.}, 2016.
\newblock to appear.

\bibitem{Con}
C.~Connell.
\newblock Minimal {L}yapunov exponents, quasiconformal structures, and rigidity
  of non-positively curved manifolds.
\newblock {\em Ergodic Theory Dynam. Systems}, 23(2):429--446, 2003.

\bibitem{CFS82}
I.~P. Cornfeld, S.~V. Fomin, and Ya.~G. Sina\u{\i}.
\newblock {\em Ergodic theory}, volume 245 of {\em Grundlehren der
  Mathematischen Wissenschaften [Fundamental Principles of Mathematical
  Sciences]}.
\newblock Springer-Verlag, New York, 1982.
\newblock Translated from the Russian by A. B. Sosinski\u{\i}.

\bibitem{Cr}
C.~Croke.
\newblock Rigidity for surfaces of nonpositive curvature.
\newblock {\em Comment. Math. Helv.}, 65(1):150--169, 1990.

\bibitem{Fer04}
R.~Feres.
\newblock A differential-geometric view of normal forms of contractions.
\newblock In {\em Modern dynamical systems and applications}, pages 103--121.
  Cambridge Univ. Press, Cambridge, 2004.

\bibitem{Fla95}
L.~Flaminio.
\newblock Local entropy rigidity for hyperbolic manifolds.
\newblock {\em Comm. Anal. Geom.}, 3(3-4):555--596, 1995.

\bibitem{Ghys87}
\'Etienne Ghys.
\newblock Flots d'{A}nosov dont les feuilletages stables sont
  diff\'erentiables.
\newblock {\em Ann. Sci. \'Ecole Norm. Sup. (4)}, 20(2):251--270, 1987.

\bibitem{GKS18}
A.~Gogolev, B.~Kalinin, and V.~Sadovskaya.
\newblock Local rigidity of lyapunov spectrum for toral automorphisms.
\newblock 2018.
\newblock Preprint, arXiv:1808.06249.

\bibitem{Gro91}
M.~Gromov.
\newblock Foliated {P}lateau problem. {II}. {H}armonic maps of foliations.
\newblock {\em Geom. Funct. Anal.}, 1(3):253--320, 1991.

\bibitem{GK98}
M.~Guysinsky and A.~Katok.
\newblock Normal forms and invariant geometric structures for dynamical systems
  with invariant contracting foliations.
\newblock {\em Math. Res. Lett.}, 5(1-2):149--163, 1998.

\bibitem{Ham90}
U.~Hamenst\"adt.
\newblock Entropy-rigidity of locally symmetric spaces of negative curvature.
\newblock {\em Ann. of Math. (2)}, 131(1):35--51, 1990.

\bibitem{Ham91}
U.~Hamenst{\"a}dt.
\newblock A geometric characterization of negatively curved locally symmetric
  spaces.
\newblock {\em J. Differential Geom.}, 34(1):193--221, 1991.

\bibitem{Ham97}
U.~Hamenst{\"a}dt.
\newblock Cocycles, {H}ausdorff measures and cross ratios.
\newblock {\em Ergodic Theory Dynam. Systems}, 17(5):1061--1081, 1997.

\bibitem{Ham99}
U.~Hamenst{\"a}dt.
\newblock Cocycles, symplectic structures and intersection.
\newblock {\em Geom. Funct. Anal.}, 9(1):90--140, 1999.

\bibitem{Has4}
B.~Hasselblatt.
\newblock A new construction of the {M}argulis measure for {A}nosov flows.
\newblock {\em Ergodic Theory Dynam. Systems}, 9(3):465--468, 1989.

\bibitem{Has2}
B.~Hasselblatt.
\newblock Regularity of the {A}nosov splitting and of horospheric foliations.
\newblock {\em Ergodic Theory Dynam. Systems}, 14(4):645--666, 1994.

\bibitem{HW99}
B.~Hasselblatt and A.~Wilkinson.
\newblock Prevalence of non-{L}ipschitz {A}nosov foliations.
\newblock {\em Ergodic Theory Dynam. Systems}, 19(3):643--656, 1999.

\bibitem{H94}
N.~Haydn.
\newblock Canonical product structure of equilibrium states.
\newblock {\em Random Comput. Dynam.}, 2(1):79--96, 1994.

\bibitem{Hein01}
J.~Heinonen.
\newblock {\em Lectures on analysis on metric spaces}.
\newblock Universitext. Springer-Verlag, New York, 2001.

\bibitem{HH}
E.~Heintze and H.-C. Im~Hof.
\newblock Geometry of horospheres.
\newblock {\em J. Differential Geom.}, 12(4):481--491 (1978), 1977.

\bibitem{Her}
L.~Hern{\'a}ndez.
\newblock K\"ahler manifolds and {$1/4$}-pinching.
\newblock {\em Duke Math. J.}, 62(3):601--611, 1991.

\bibitem{HPS}
M.~W. Hirsch, C.~C. Pugh, and M.~Shub.
\newblock {\em Invariant manifolds}.
\newblock Lecture Notes in Mathematics, Vol. 583. Springer-Verlag, Berlin-New
  York, 1977.

\bibitem{J}
J.-L. Journ{\'e}.
\newblock A regularity lemma for functions of several variables.
\newblock {\em Rev. Mat. Iberoamericana}, 4(2):187--193, 1988.

\bibitem{Kal}
B.~Kalinin.
\newblock Liv\v sic theorem for matrix cocycles.
\newblock {\em Ann. of Math. (2)}, 173(2):1025--1042, 2011.

\bibitem{KS2}
B.~Kalinin and V.~Sadovskaya.
\newblock Linear cocycles over hyperbolic systems and criteria of conformality.
\newblock {\em J. Mod. Dyn.}, 4(3):419--441, 2010.

\bibitem{KS}
B.~Kalinin and V.~Sadovskaya.
\newblock Cocycles with one exponent over partially hyperbolic systems.
\newblock {\em Geom. Dedicata}, 167:167--188, 2013.

\bibitem{Kat82}
A.~Katok.
\newblock Entropy and closed geodesics.
\newblock {\em Ergodic Theory Dynam. Systems}, 2(3-4):339--365 (1983), 1982.

\bibitem{HK}
A.~Katok and B.~Hasselblatt.
\newblock {\em Introduction to the modern theory of dynamical systems},
  volume~54 of {\em Encyclopedia of Mathematics and its Applications}.
\newblock Cambridge University Press, Cambridge, 1995.
\newblock With a supplementary chapter by Katok and Leonardo Mendoza.

\bibitem{Led84}
F.~Ledrappier.
\newblock Propri\'{e}t\'{e}s ergodiques des mesures de {S}ina\"{\i}.
\newblock {\em Inst. Hautes \'{E}tudes Sci. Publ. Math.}, (59):163--188, 1984.

\bibitem{Liv}
A.~N. Liv{\v{s}}ic.
\newblock Cohomology of dynamical systems.
\newblock {\em Izv. Akad. Nauk SSSR Ser. Mat.}, 36:1296--1320, 1972.

\bibitem{LS72}
A.~N. Liv{\v{s}}ic and Ja.~G. Sina{\u\i}.
\newblock Invariant measures that are compatible with smoothness for transitive
  {$C$}-systems.
\newblock {\em Dokl. Akad. Nauk SSSR}, 207:1039--1041, 1972.

\bibitem{Mel19}
K.~Melnick.
\newblock Non-stationary smooth geometric structures for contracting measurable
  cocycles.
\newblock {\em Ergodic Theory Dynam. Systems}, 39(2):392--424, 2019.

\bibitem{Mos73}
G.~D. Mostow.
\newblock {\em Strong rigidity of locally symmetric spaces}.
\newblock Princeton University Press, Princeton, N.J.; University of Tokyo
  Press, Tokyo, 1973.
\newblock Annals of Mathematics Studies, No. 78.

\bibitem{Osel}
V.~I. Oseledec.
\newblock A multiplicative ergodic theorem. {C}haracteristic {L}japunov,
  exponents of dynamical systems.
\newblock {\em Trudy Moskov. Mat. Ob\v s\v c.}, 19:179--210, 1968.

\bibitem{O}
J.-P. Otal.
\newblock Le spectre marqu\'e des longueurs des surfaces \`a courbure
  n\'egative.
\newblock {\em Ann. of Math. (2)}, 131(1):151--162, 1990.

\bibitem{Pan85}
P.~Pansu.
\newblock Quasiconformal mappings and manifolds of negative curvature.
\newblock In {\em Curvature and topology of {R}iemannian manifolds ({K}atata,
  1985)}, volume 1201 of {\em Lecture Notes in Math.}, pages 212--229.
  Springer, Berlin, 1986.

\bibitem{P89b}
P.~Pansu.
\newblock M\'etriques de {C}arnot-{C}arath\'eodory et quasiisom\'etries des
  espaces sym\'etriques de rang un.
\newblock {\em Ann. of Math. (2)}, 129(1):1--60, 1989.

\bibitem{Par86}
W.~Parry.
\newblock Synchronisation of canonical measures for hyperbolic attractors.
\newblock {\em Comm. Math. Phys.}, 106(2):267--275, 1986.

\bibitem{Pes76}
Ja.~B. Pesin.
\newblock Characteristic {L}japunov exponents, and ergodic properties of smooth
  dynamical systems with invariant measure.
\newblock {\em Dokl. Akad. Nauk SSSR}, 226(4):774--777, 1976.

\bibitem{Pla72}
J.~Plante.
\newblock Anosov flows.
\newblock {\em Amer. J. Math.}, 94(3):729--754, 1972.

\bibitem{R78}
D.~Ruelle.
\newblock An inequality for the entropy of differentiable maps.
\newblock {\em Bol. Soc. Brasil. Mat.}, 9(1):83--87, 1978.

\bibitem{S05}
V.~Sadovskaya.
\newblock On uniformly quasiconformal {A}nosov systems.
\newblock {\em Math. Res. Lett.}, 12(2-3):425--441, 2005.

\bibitem{S15}
V.~Sadovskaya.
\newblock Cohomology of fiber bunched cocycles over hyperbolic systems.
\newblock {\em Ergodic Theory Dynam. Systems}, 35(8):2669--2688, 2015.

\bibitem{SY18}
R.~Saghin and J.~Yang.
\newblock Lyapunov exponents and rigidity of anosov automorphisms and skew
  products.
\newblock 2018.
\newblock Preprint, arXiv:1802.08266.

\bibitem{Sa58}
S.~Sasaki.
\newblock On the differential geometry of tangent bundles of {R}iemannian
  manifolds.
\newblock {\em T\^{o}hoku Math. J. (2)}, 10:338--354, 1958.

\bibitem{Tys}
J.~Tyson.
\newblock Quasiconformality and quasisymmetry in metric measure spaces.
\newblock {\em Ann. Acad. Sci. Fenn. Math.}, 23(2):525--548, 1998.

\bibitem{V71}
J.~V{\"a}is{\"a}l{\"a}.
\newblock {\em Lectures on {$n$}-dimensional quasiconformal mappings}.
\newblock Lecture Notes in Mathematics, Vol. 229. Springer-Verlag, Berlin-New
  York, 1971.

\bibitem{Vel}
R.~Velozo.
\newblock Characterization of uniform hyperbolicity for fiber-bunched cocycles.
\newblock 2018.
\newblock Preprint, arXiv:1807.11335.

\bibitem{FY}
S.-T. Yau and F.~Zheng.
\newblock Negatively {$\frac14$}-pinched {R}iemannian metric on a compact
  {K}\"ahler manifold.
\newblock {\em Invent. Math.}, 103(3):527--535, 1991.

\end{thebibliography}
\end{document}